\documentclass[twoside,11pt,reqno]{amsart}

\usepackage[margin=1.1in]{geometry}
\usepackage{amsmath} 
\usepackage{amssymb} 
\usepackage{amsopn}
\usepackage{amsthm} 
\usepackage{amsfonts} 
\usepackage{mathrsfs}
\usepackage{stmaryrd}
\usepackage{colonequals}
\usepackage{tikz-cd}
\usepackage{enumitem}
\usepackage{cite}
\usepackage{verbatim}
\usepackage[pdfpagelabels,pdftex]{hyperref}

\usepackage{setspace}
\usepackage{amssymb}
\usepackage{euscript}
\usepackage{cite}
\usepackage[all]{xy}
\usepackage{tabularx}

\theoremstyle{plain}
\usepackage{amsfonts}
\usepackage{graphicx}
\usepackage{amsmath}
\hypersetup{
  pdftitle={},
  pdfauthor={},
  pdfsubject={},
  pdfkeywords={},
  colorlinks=false,
  breaklinks=true,
  bookmarksopen=true,
  bookmarksnumbered=true,
  pdfpagemode=UseOutlines,
  plainpages=false}

\DeclareMathOperator{\GL}{GL}
\DeclareMathOperator{\SL}{SL}
\DeclareMathOperator\ord{ord}
\DeclareMathOperator\vol{vol}
\DeclareMathOperator\Char{char}
\DeclareMathOperator\Gal{Gal}

\DeclareMathOperator\red{red}
\DeclareMathOperator\Ind{Ind}
\DeclareMathOperator\cInd{c-Ind}
\DeclareMathOperator\diag{diag}
\DeclareMathOperator\Nm{Nm}
\DeclareMathOperator\pr{pr}
\DeclareMathOperator\Fr{Fr}
\DeclareMathOperator\Hom{Hom}
\DeclareMathOperator\Tr{Tr}

\newcommand{\sm}{{\,\smallsetminus\,}}
\newcommand\adm{\mathrm{adm}}
\newcommand\rat{\mathrm{rat}}
\newcommand\from{\colon}
\newcommand\bG{\mathbb G}
\newcommand\bU{\mathbb U}
\newcommand\QQ{\mathbb Q}

\newcommand\bT{\mathbb T}
\newcommand\rd{\mathrm{red}}
\newcommand\bW{\mathbb W}
\newcommand\FF{\mathbb F}
\newcommand\cA{\mathcal A}
\newcommand\cB{\mathcal B}
\newcommand\cI{\mathcal I}

\newcommand\bZ{\mathbb Z}
\newcommand\bK{\mathbb K}

\newcommand\cT{\mathcal T}

\newcommand\sX{\mathscr{X}}
\newcommand\cG{\mathcal{G}}
\newcommand\cW{\mathcal{W}}
\newcommand\sG{\mathscr G}
\DeclareMathOperator\JL{JL}

\newcommand\bN{\mathbb N}

\newcommand\bx{\mathbf{x}}
\newcommand\e{\mathbf{e}}

\newcommand\cO{\mathcal O}
\newcommand\sL{\mathscr L}
\newcommand\bA{\mathbb A}
\newcommand\cU{\mathcal U}

\newcommand\trans{\intercal}

\makeatletter
\newtheorem*{rep@theorem}{\rep@title}
\newcommand{\newreptheorem}[2]{%
\newenvironment{rep#1}[1]{%
 \def\rep@title{#2 \ref{##1}}%
 \begin{rep@theorem}}%
 {\end{rep@theorem}}}
\makeatother

\newtheorem{thm}{Theorem}[section]
\newtheorem*{thm*}{Theorem}

\newtheorem{prop}[thm]{Proposition}

\newtheorem{cor}[thm]{Corollary}

\newtheorem*{cor*}{Corollary}
\newtheorem{lm}[thm]{Lemma}

\newtheorem{theorem}[thm]{Theorem}
\newtheorem{lemma}[thm]{Lemma}

\newtheorem{proposition}[thm]{Proposition}
\newtheorem{proposition-definition}[thm]{Proposition-Definition}
\newtheorem*{conj*}{Conjecture}

\theoremstyle{remark}
\theoremstyle{remark}\newtheorem*{claim*}{Claim}

\newreptheorem{thm}{Theorem}
\newreptheorem{prop}{Proposition}
\newreptheorem{cor}{Corollary}

\theoremstyle{definition}
\newtheorem{Def}[thm]{Definition}

\newtheorem{ex}[thm]{Example}
\newtheorem{definition}[thm]{Definition}

\theoremstyle{remark}
\newtheorem{rem}[thm]{Remark}

\newtheorem{remark}[thm]{Remark}

\newenvironment{pro*}[1][Proof]{{\it{#1:}} }{}

\newcounter{absatzcounter}[section]
\numberwithin{equation}{section}

\begin{document}

\title{Affine Deligne--Lusztig varieties at infinite level}
\author{Charlotte Chan and Alexander Ivanov}
\address{Department of Mathematics \\
Princeton University \\
Fine Hall, Washington Road \\
Princeton, NJ 08544-1000 USA}
\email{charchan@mit.edu}
\address{Mathematisches Institut \\ Universit\"at Bonn \\ Endenicher Allee 60 \\ 53115 Bonn, Germany}
\email{ivanov@math.uni-bonn.de}

\maketitle

\begin{abstract}
We initiate the study of affine Deligne--Lusztig varieties with arbitrarily deep level structure for general reductive groups over local fields. We prove that for $\GL_n$ and its inner forms, Lusztig's semi-infinite Deligne--Lusztig construction is isomorphic to an affine Deligne--Lusztig variety at infinite level. We prove that their homology groups give geometric realizations of the local Langlands and Jacquet--Langlands correspondences in the setting that the Weil parameter is induced from a character of an unramified field extension. In particular, we resolve Lusztig's 1979 conjecture in this setting for minimal admissible characters.
\end{abstract}

\tableofcontents

\newpage

\section{Introduction}

In their fundamental paper \cite{DeligneL_76}, Deligne and Lusztig gave a powerful geometric approach to the construction of representations of finite reductive groups. To a reductive group $G$ over a finite field $\mathbb{F}_q$ and a maximal $\mathbb{F}_q$-torus $T \subseteq G$, they attach a variety given by the set of Borel subgroups of $G$ lying in a fixed relative position (depending on $T$) to their Frobenius translate. This variety has a $T$-torsor called the \textit{Deligne--Lusztig variety}. The Deligne--Lusztig variety has commuting actions of $G$ and $T$, and its $\ell$-adic \'etale cohomology realizes a natural correspondence between characters of $T(\mathbb{F}_q)$ and representations of $G(\mathbb{F}_q)$. 

Two possible ways of generalizing this construction to reductive groups over local fields are to consider subsets cut out by Deligne--Lusztig conditions in the semi-infinite flag manifold (in the sense of Feigin--Frenkel \cite{FeiginF_90}) or in affine flag manifolds of increasing level. The first approach is driven by an outstanding conjecture of Lusztig \cite{Lusztig_79} that the semi-infinite Deligne--Lusztig set has an algebro-geometric structure, one can define its $\ell$-adic homology groups, and the resulting representations should be irreducible supercuspidal. This conjecture was studied in detail in the case of division algebras by Boyarchenko and the first named author in \cite{Boyarchenko_12, Chan_DLI, Chan_DLII}, and ultimately resolved in this setting in \cite{Chan_siDL}. Prior to the present paper, Lusztig's conjecture was completely open outside the setting of division algebras.

The second approach is based on Rapoport's affine Deligne-Lusztig varieties \cite{Rapoport_05}, which are closely related to the reduction of (integral models of) Shimura varieties. Affine Deligne--Lusztig varieties for arbitrarily deep level structure were introduced and then studied in detail for $\GL_2$ by the second named author in \cite{Ivanov_15_ADLV_GL2_unram, Ivanov_15_ADLV_GL2_ram, Ivanov_18_wild}, where it was shown that their $\ell$-adic cohomology realizes many irreducible supercuspidal representations for this group.

The goals of the present paper are to show that these constructions
%\vspace{-.5em}
\begin{enumerate}[label=(\Alph*),leftmargin=*]
\item
are isomorphic for all inner forms of $\GL_n$ and their maximal unramified elliptic torus
\item
realize the local Langlands and Jacquet--Langlands correspondences for supercuspidal representations coming from unramified field extensions
\end{enumerate}
%\vspace{-.5em}

The first goal is achieved by computing both sides and defining an explicit isomorphism between Lusztig's semi-infinite construction and an inverse limit of coverings of affine Deligne-Lusztig varieties. In particular, this allows us to use the known scheme structure of affine Deligne--Lusztig varieties to define a natural scheme structure on the semi-infinite side, which was previously only known in the case of division algebras. This resolves the algebro-geometric conjectures of \cite{Lusztig_79} for all inner forms of $\GL_n$. %We give a description of its connected components, which is based on the corresponding results for certain affine Deligne--Lusztig varieties due to Viehmann \cite{Viehmann_08}. 

To attain the second goal, we study the cohomology of this infinite-dimensional variety using a wide range of techniques.  To show irreducibility of certain eigenspaces under the torus action, we generalize a method of Lusztig \cite{Lusztig_04,Stasinski_09} to quotients of parahoric subgroups which do not come from reductive groups over finite rings. We study the geometry and its behavior under certain group actions to prove an analogue of cuspidality for representations of such quotients. To obtain a comparison to the local Langlands correspondence, we use the Deligne--Lusztig fixed-point formula to determine the character on the maximal unramified elliptic torus and use characterizations of automorphic induction due to Henniart \cite{Henniart_92,Henniart_93}. In particular, for minimal admissible characters, we resolve the remaining part of Lusztig's conjecture (supercuspidality) for all inner forms of $\GL_n$.

We now give a more detailed overview. Let $K$ be a non-archimedean local field with finite residue field $\FF_q$, let $\breve K$ be the completion of the maximal unramified extension of $K$ and let $\sigma$ denote the Frobenius automorphism of $\breve K/K$. For any algebro-geometric object $X$ over $K$, we write $\breve{X} \colonequals X(\breve K)$ for the set of its $\breve K$-points. 
Let $\mathcal G$ be a connected reductive group over $K$. For simplicity assume that $\mathcal G$ is split. For $b \in \breve{\mathcal G}$, let $J_b$ be the $\sigma$-stabilizer of $b$
\[ J_b(R) \colonequals \{ g \in \mathcal G(R \otimes_K \breve K) \colon g^{-1}b\sigma(g) = b \}  
\]
for any $K$-algebra $R$. Then $J_b$ is an inner form of a Levi subgroup of $\mathcal G$, and if $b$ is basic, $J_b$ is an inner form of $\mathcal G$. Let $\cT$ be a maximal split torus in $\mathcal G$. For an element $w$ in the Weyl group of $(\mathcal G, \cT)$, let %$T_w$ be the form of $T$ given by
\[ T_w(R) \colonequals \{ t \in \cT(R \otimes_K \breve K) \colon t^{-1}\dot{w}\sigma(t) = \dot{w} \}  \]
\noindent for any $K$-algebra $R$, where $\dot{w}$ is a lift of $w$ to $\breve{\mathcal G}$. 

The semi-infinite Deligne--Lusztig set $X_{\dot w}^{DL}(b)$ is the set of all Borel subgroups of $\breve{\mathcal{G}}$ in relative position $w$ to their $b \sigma$-translate. It has a cover
\begin{equation*}
\dot X_{\dot w}^{DL}(b) \colonequals \{g \breve U \in \breve{\mathcal{G}}/\breve U : g^{-1} b \sigma(g) \in \breve U \dot w \breve U\} \subseteq \breve{\mathcal{G}}/\breve U
\end{equation*}
with a natural action by $J_b(K) \times T_w(K)$, and this set coincides with Lusztig's construction \cite{Lusztig_79}. On the other hand, for arbitrarily deep congruence subgroups $J \subseteq \breve{\mathcal{G}}$, one can define affine Deligne--Lusztig sets of higher level $J$,
\[
X_x^J(b) \colonequals \{gJ \in \breve{\mathcal{G}}/J \colon g^{-1}b\sigma(g) \in JxJ \} \subseteq \breve{\mathcal{G}}/J,
\]
where $x$ is a $J$-double coset in $\breve{\mathcal{G}}$. 
%If $J' \subseteq J$ is normal, $X_{x'}^{J'}(b) \rightarrow X_x^J(b)$ will under some conditions be a torsor under a certain subgroup $(J/J')_x$ of $J/J'$, hence carrying a natural $J_b(K) \times (J/J')_x$-action. 
Under some technical conditions on $x$, we prove that these sets can be endowed with a structure of an $\FF_q$-scheme (Theorem \ref{prop:loc_closed_ADLV_covers}). We remark that when $K$ has mixed characteristic, $\breve{\mathcal{G}}/J$ is a ind-(perfect scheme), so $X_x^J(b)$ will also carry the structure of a perfect scheme. 

We now specialize to the following setting. Consider $\breve{\mathcal{G}} = \GL_n(\breve K)$ and $G = J_b(K)$ for some basic $b \in \GL_n(\breve K)$ so that $G$ is an inner form of $\GL_n(K)$. Let $w$ be a Coxeter element so that $T \colonequals T_w(K) \cong L^\times$ for the degree-$n$ unramified extension $L$ of $K$. Let $G_{\cO}$ be a maximal compact subgroup of $G$ and let $T_{\cO} = T \cap G_{\cO} \cong \cO_L^{\times}$. We consider a particular tower of affine Deligne--Lusztig varieties $\dot X_{\dot w_r}^m(b)$ for congruence subgroups of $\breve{\mathcal{G}}$ indexed by $m$, where the image of each $\dot w_r$ in the Weyl group is $w$. We form the inverse limit $\dot X_w^\infty(b) = \varprojlim_{r>m\geq 0} \dot X_{\dot w_r}^m(b)$, which carries a natural action of $G \times T$. %We give a simple description of $\dot X_w^{DL}(b)$ and $\dot X_w^\infty(b)$ in terms of the isocrystal $(\breve K^n, b \sigma)$. Let $V_b^{\adm}$ denote the set of all elements which do not lie in a proper sub-isocrystal.

\begin{thm*}[\ref{cor:infty_level_adlv}]
There is a $(G \times T)$-equivariant map of sets
\begin{equation*}
\dot X_w^{DL}(b) \stackrel{\sim}{\longrightarrow} \dot X_w^\infty(b).
\end{equation*}
In particular, this gives $\dot X_w^{DL}(b)$ the structure of a scheme over $\overline \FF_q$.
%
%There is a commutative diagram with $J_b(K)$-equivariant maps:
%\begin{equation*}
%\begin{tikzcd}\label{diag:character_isos_diag}
%X_b^{(U)}(b) \ar{d}{T_w(K)} & \{x \in V_b^{\rm adm} \colon \det(g_b(x)) \in K^{\times} \} \ar{l}[above]{\sim} \ar{rr}{\sim} \ar{d} & & \dot{X}_b^{\infty}(b) \ar{d}{T_w(\mathcal{O}_K)} \\
%X_w^{(B)}(b) & V_b^{\rm adm}/\breve K^{\times} \ar{l}[above]{\sim} & V_b^{\rm adm}/\mathcal{O}_{\breve K}^{\times} \ar{r}{\sim} \ar{l}[above]{\mathbb{Z}} & X_w^{\infty}(b),
%\end{tikzcd}
%\end{equation*}
%where $g_b(x)$ denotes the matrix with columns $(b\sigma)^i(x)$ for $0 \leq i \leq n-1$. In particular, this endows $X_b^{(U)}(b)$, $X_w^{(B)}(b)$ with natural pro- (perfect) scheme over $\mathbb{F}_q$.
\end{thm*}

We completely determine the higher level affine Deligne--Lusztig varieties $\dot X_{\dot w_r}^m(b)$. They are $(\cO_L/\mathfrak{p}_L^{m+1})^{\times}$-torsors over the schemes $X_{\dot w_r}^m(b)$, which are interesting in their own right. In particular, $X_{\dot w_r}^0(b)$ provide examples of explicitly described Iwahori-level affine Deligne--Lusztig varieties. We prove the following.

\begin{thm*}[\ref{thm:scheme_structure}]
The scheme $X_{\dot w_r}^m(b)$ is a disjoint union, indexed by $G/G_\cO$, of classical Deligne--Lusztig varieties for the reductive quotient of $G_{\cO} \times T_{\cO}$ times finite-dimensional affine space. 
\end{thm*}
% 
% \begin{thm*}[\ref{thm:scheme_structure}]
% $\dot X_{\dot w_r}^m(b)$ is an affine fibration over a disjoint union, indexed by $G/G_\cO$, of classical Deligne--Lusztig varieties for the reductive quotient of $G \times T$.
% \end{thm*}
%
%\begin{thm*}[see Theorem \ref{thm:scheme_structure}]
%Let $b$ be the special representative (see Section \ref{sec:two b}) of the basic conjugacy class $[b]$ with $0 \leq \kappa_{\GL_n}(b) = \kappa < n$. Write $n' = \gcd(\kappa,n)$ and $n_0=n/n'$ For $r > m \geq 0$, the $\FF_q$-scheme $X_{\dot w_r}^m(b) \cong \bigsqcup_{G/G_{\cO}} \Omega_{\FF_{q^{n_0}}}^{n'-1} \times \mathbb{A}$, where $\mathbb{A}$ is a finite dimensional affine space over $\FF_q$ and $\Omega_{\FF_{q^{n_0}}}^{n'-1} = \mathbb{P}^{n'-1} \sm \mathbb{P}^{n'-1}(\FF_{q^{n_0}})$.  The morphism $\dot X_{\dot{w}_r}^m(b) \rightarrow X_{\dot w_r}^m(b)$ is a finite \'etale $(\cO_L/\mathfrak{p}_L^{m+1})^{\times}$-torsor. In particular, all these schemes are smooth.
%\end{thm*}

The disjoint union decomposition is deduced from Viehmann \cite{Viehmann_08}. We point out the similarity between the Iwahori level varieties $X_{\dot w_r}^0(b)$ and those considered by G\"ortz and He \cite[e.g.\ Proposition 2.2.1]{GoertzH_15}, though in our setting, the elements $\dot w_r$ can have arbitrarily large length in the extended affine Weyl group.

One of the key insights throughout our paper is the flexibility of working with different representatives $b$ of a $\sigma$-conjugacy class. For example, when $G = \GL_n(K)$, switching between $b = 1$ and $b$ being a Coxeter element allows us to use techniques that are otherwise inaccessible.

Having established the isomorphism $\dot X_w^{DL}(b) \stackrel{\sim}{\longrightarrow} \dot X_w^\infty(b)$, the main objective in the rest of the paper is to study the virtual $G$-representation
\begin{equation*}
R_T^G(\theta) \colonequals \textstyle\sum\limits_i (-1)^i H_i(\dot X_w^\infty(b), \overline \QQ_\ell)[\theta]
\end{equation*}
for smooth characters $\theta \from T \to \overline \QQ_\ell^\times$, where $[\theta]$ denotes the subspace where $T$ acts by $\theta$. We write $|R_T^G(\theta)|$ to denote the genuine representation when one of $\pm R_T^G(\theta)$ is genuine.

One could try to calculate $R_T^G(\theta)$ by calculating the cohomology of the affine Deligne--Lusztig varieties $\dot X_{\dot w_r}^m(b)$. These finite-level varieties have somewhat strange descriptions (see the equivalence relation $\dot \sim_{b,m,r}$ in Section \ref{sec:Comparison_unif_DLV_ADLV}), though it is conceivable that one could use the results of Part \ref{part:aut ind and JL} to study the cohomology of these higher-level affine Deligne--Lusztig varieties. 

Instead of passing through affine Deligne--Lusztig varieties, we approximate our infinite-level object $\dot X_w^{DL}(b)$ by using an analogue of Deligne--Lusztig varieties for parahoric subgroups, which are easier to explicitly describe than affine Deligne--Lusztig varieties. Using the decomposition of $\dot{X}_w^{\infty}(b)$ into $G$-translates of $G_{\cO}$-stable components (as in Theorem \ref{thm:scheme_structure}), the computation of the cohomology of $\dot{X}_b^{\infty}(b)$ reduces to the computation for one such component, which can in turn be written as an inverse limit $\varprojlim_h X_h$ of finite-dimensional varieties $X_h$, each endowed with an action of level-$h$ quotients $G_h \times T_h$ of $G_\cO \times T_\cO$. We write $R_{T_h}^{G_h}(\theta)$ for the virtual $G_h$-representation corresponding to $\theta \from T_h \to \overline \QQ_\ell^\times$. We note that $X_1$ is a classical Deligne-Lusztig variety for the reductive subquotient of $T_\cO$ in the reductive quotient of $G_{\cO}$. %We will use a range of techniques to study $X_h$ and its cohomology.

However, the infinite-level object $\dot X_w^\infty(b)$ has a very natural description, so we proceed by defining another tower of finite-dimensional objects $X_h$, which are analogues of Deligne--Lusztig varieties for parahoric subgroups. Using the Deligne--Lusztig fixed-point formula, we compute (part of) the character of $R_{T_h}^{G_h}(\theta)$ on $T_h$, which when combined with Henniart's characterizations \cite{Henniart_92,Henniart_93} of automorphic induction yields:

\begin{thm*}[\ref{t:LLC JLC gen}]
Let $\theta \from T \to \overline \QQ_\ell^\times$ be a smooth character. If $|R_T^G(\theta)|$ is irreducible supercuspidal, then the assignment $\theta \mapsto |R_T^G(\theta)|$ is a geometric realization of automorphic induction and the Jacquet--Langlands correspondence.
\end{thm*}

Proving that $|R_T^G(\theta)|$ is irreducible supercuspidal involves two main steps: proving that $|R_{T_h}^{G_h}(\theta)|$ is irreducible and proving its induction to $G$ (after extending by the center) is irreducible. In  \cite{Lusztig_04}, Lusztig studies the irreducibility of $R_{T_h}^{G_h}(\theta)$ for reductive groups over finite rings under a regularity assumption on $\theta$. In our setting, this regularity assumption corresponds to $\theta$ being \textit{minimal admissible}. We extend Lusztig's arguments to the non-reductive setting to handle the non-quasi-split inner forms of $\GL_n(K)$ and prove that $R_{T_h}^{G_h}(\theta)$ is irreducible under the same regularity assumption on $\theta$ (Theorem \ref{t:alt sum Xh}). In this context, we prove a cuspidality result (Theorem \ref{t:Gh cuspidal}) for $|R_{T_h}^{G_h}(\theta)|$, which allows us to emulate the arguments from \cite[Proposition 6.6]{MoyP_96} that inducing classical Deligne--Lusztig representations gives (depth zero) irreducible supercuspidal representations of $p$-adic groups. This approach was carried out in the $\GL_2$ case for arbitrary depth in \cite[Propositions 4.10, 4.22]{Ivanov_15_ADLV_GL2_unram}. Note that the $|R_T^G(\theta)|$ can have arbitrarily large depth, depending on the level of the smooth character $\theta$.

\begin{thm*}[\ref{t:RTG irred}]
If $\theta \from T \to \overline \QQ_\ell^\times$ is minimal admissible, then $|R_T^G(\theta)|$ is irreducible supercuspidal.
\end{thm*}

\subsection{Outline}

This paper is divided into four parts. The first part of the article is devoted to purely geometric properties of the Deligne--Lusztig constructions for arbitrary reductive groups over local fields. In Sections \ref{Def:DL_sets} and \ref{sec:adlv_and_covers}, we define and recall the two types of Deligne--Lusztig constructions. The main result of this part is Theorem \ref{prop:loc_closed_ADLV_covers}, where we prove that, under a technical hypothesis, affine Deligne--Lusztig sets of arbitrarily deep level can be endowed with a scheme structure. After Part \ref{part:DL}, we work only in the context of the inner forms of $\GL_n(K)$.

We begin Part \ref{part:GLn DL} with a discussion of the group-theoretic constructions we will use at length throughout the rest of the paper (Section \ref{sec:inner_forms}). We emphasize the importance of the seemingly innocuous Section \ref{sec:two b}, where we define two representatives $b$ for each basic $\sigma$-conjugacy class of $\GL_n(\breve K)$. In Section \ref{sec:comparison_DL_ADLV_isocrystal}, we define the affine Deligne--Lusztig varieties $\dot X_{\dot w_r}^m(b)$, construct an isomorphism between $\dot X_w^\infty(b)$ and $\dot X_{w}^{DL}(b)$ using the isocrystal $(\breve K^n, b\sigma)$, and explicate the scheme structure of $\dot X_w^\infty(b)$. In Section \ref{s:schemes_Xh}, we introduce a family of smooth finite-type schemes $X_h$ whose limit is a component of $\dot X_w^\infty(b)$ corresponding to $G_{\cO}$ and study its geometry. This plays the role of a Deligne--Lusztig variety for subquotients of $G$ (see Proposition \ref{p:Xh Ah-}). 

In Part \ref{part:cohomology}, we calculate the cohomology $R_{T_h}^{G_h}(\theta)$ under a certain regularity assumption on $\theta$. We prove irreducibility (Theorem \ref{t:alt sum Xh}) using a generalization of \cite{Lusztig_04, Stasinski_09} discussed in Section \ref{sec:lusztig_lemma}. We prove a result about the restriction of $R_{T_h}^{G_h}(\theta)$ to the ``deepest part'' of unipotent subgroups (Theorem \ref{t:Gh cuspidal}) which can be viewed as an analogue of cuspidality for $G_h$-representations. This is a long calculation using fixed-point formulas.

Finally, in Part \ref{part:aut ind and JL}, we combine the results of the preceding two parts to deduce our main theorems about $R_T^G(\theta)$, the homology of the affine Deligne--Lusztig variety at infinite level $\dot X_w^\infty(b)$. We review the methods of Henniart \cite{Henniart_92, Henniart_93} in Section \ref{s:henniart}, define and discuss some first properties of the homology of $\dot X_w^\infty(b)$ in Section \ref{s:inf ADLV}, and prove the irreducible supercuspidality of $R_T^G(\theta)$ for minimal admissible $\theta$ in Section \ref{s:minimal}.

\subsection*{Acknowledgements} 
The first author was partially supported by NSF grants DMS-0943832 and DMS-1160720, the ERC starting grant 277889, the DFG via P.\ Scholze's Leibniz Prize, and an NSF Postdoctoral Research Fellowship, Award No.\ 1802905. In addition, she would like to thank the Technische Universit\"at M\"unchen and Universit\"at Bonn for their hospitality during her visits in 2016 and 2018. The second author was partially supported by European Research Council Starting Grant 277889 ``Moduli spaces of local G-shtukas'', by a postdoctoral research grant of the DFG during his stay at University Paris 6 (Jussieu), and by the DFG via P.\ Scholze's Leibniz Prize. 
The authors thank Eva Viehmann for very enlightening discussions on this article, especially for the explanations concerning connected components, and also thank Laurent Fargues for his observation concerning the scheme structure on semi-infinite Deligne--Lusztig sets. Finally, the authors thank the anonymous referee for numerous careful and insightful comments.

%
%%%%%%%%%%%%%%%%%%%%%%%%%%%%%%%%%%%%%%%%%%%%%%%%%%%%%%%%%%%%%
%%%%%%%%%%%%%%%%%%%%%%%%%%%%%%%%%%%%%%%%%%%%%%%%%%%%%%%%%%%%%

\newpage

\section{Notation} \label{sec:notation}

Throughout the paper we will use the following notation. Let $K$ be a non-archimedean local field with residue field $\FF_q$ of prime characteristic $p$, and let $\breve{K}$ denote the completion of a maximal unramified extension of $K$. We denote by $\mathcal{O}_K$, $\mathfrak{p}_K$ (resp.\ $\cO$, $\mathfrak{p}$) the integers and the maximal ideal of $K$ (resp.\ of $\breve K$). The residue field of $\breve K$ is an algebraic closure $\overline{\FF}_q$ of $\FF_q$. We write $\sigma$ for the Frobenius automorphism of $\breve K$, which is the unique $K$-automorphism of $\breve K$, lifting the $\FF_q$-automorphism $x \mapsto x^q$ of $\overline{\FF}_q$. Finally, we denote by $\varpi$ a uniformizer of $K$ (and hence of $\breve K$) and by $\ord = \ord_{\breve K}$ the valuation of $\breve K$, normalized such that $\ord(\varpi) = 1$.

If $K$ has positive characteristic, we let $\bW$ denote the ring scheme over $\FF_q$ where for any $\FF_q$-algebra $A$, $\bW(A) = A[\![\pi]\!]$. If $K$ has mixed characteristic, we let $\bW$ denote the $K$-ramified Witt ring scheme over $\FF_q$ so that $\bW(\FF_q) = \cO_K$ and $\bW(\overline \FF_q) = \cO$. Let $\bW_h = \bW/V^h \bW$ be the truncated  ring scheme, where $V \from \bW \to \bW$ is the Verschiebung morphism. For any $1 \leq r \leq h$, we write $\bW_h^r$ to denote the kernel of the natural projection $\bW_h \to \bW_r$. As the Witt vectors are only well behaved on perfect $\FF_q$-algebras, algebro-geometric considerations when $K$ has mixed characteristic are taken up to perfection. We fix the following convention. 

\medskip

\noindent \textbf{Convention.} If $K$ has mixed characteristic, whenever we speak of a scheme (resp.\ ind-scheme) over its residue field $\FF_q$, we mean a \emph{perfect scheme} (resp.\ \emph{ind-(perfect scheme)}), that is a set-valued functor on perfect $\FF_q$-algebras, representable by the perfection of a scheme (resp. ind-scheme). %Analogously, ``of finite type'' means ``of perfect-(finite type)'', etc. 

\medskip

For results on perfect schemes we refer to \cite{Zhu_17, BhattS_17}. Note that passing to perfection does not affect the $\ell$-adic \'etale cohomology; thus for purposes of this paper, we could in principle pass to perfection in all cases. However, in the equal characteristic case working on non-perfect rings does not introduce complications, and we prefer to work in this slightly greater generality.

\medskip

Fix a prime $\ell \neq p$ and an algebraic closure $\overline \QQ_\ell$ of $\QQ_\ell$. The field of coefficients of all representations is assumed to be $\overline \QQ_\ell$ and all cohomology groups throughout are compactly supported $\ell$-adic \'etale cohomology groups.

\subsection{List of terminology}

Our paper introduces some notions for a general group $G$ (Part \ref{part:DL}) and then studies these notions for $G$ an inner form of $\GL_n$ (Parts \ref{part:GLn DL} through \ref{part:aut ind and JL}). The investigations for $G$ an inner form of $\GL_n$ involve many different methods. For the reader's reference, we give a brief summary of the most important notation introduced and used in Parts \ref{part:GLn DL} through \ref{part:aut ind and JL}.

\begin{enumerate}[wide=0pt,labelsep=60pt,align=parleft,leftmargin=60pt]
\item[$L$] the degree-$n$ unramified extension of $K$. Its ring of integers $\cO_L$ has a unique maximal ideal $\frak p_L$ and its residue field is $\cO_L/\frak p_L \cong \FF_{q^n}$. For any $h \geq 1$, we write $U_L^h = 1 + \frak p_L^h$
%\item[$\breve K$] the maximal unramified extension of $K$. Its ring of integers $\cO$ has a unique maximal ideal $\frak p$ with residue field $\overline \FF_q$
%\item[$\bW$] If $\Char K > 0$, we let $\bW(A) = A[\![\varpi]\!]$ for any $\FF_q$-algebra $A$. If $\Char K = 0$, we let $\bW$ denote the ramified Witt ring scheme over $\FF_q$ associated to $K$. In both cases, we have $\bW(\FF_q) = \cO_K$ and $\bW(\FF_{q^n}) = \cO_L$.
%\item[$\widetilde G$] $\GL_n(\breve K)$, where $\breve K$ is the maximal unramified extension of $K$ (section \ref{sec:basic_sigma_conj_classes})
%\item[$\widetilde T$] the maximal split torus in $\GL_n(\breve K)$ given by the diagonal matrices
\item[{$[b]$}] fixed basic $\sigma$-conjugacy class of $\GL_n(\breve K)$. Typically we take representatives $b$ of $[b]$ to be either the Coxeter-type or special representative (Section \ref{sec:two b})
\item[$\kappa$] $\kappa_{\GL_n}([b])$, where $\kappa_{\GL_n}$ is the Kottwitz map. We assume that $0 \leq \kappa \leq n-1$ and set $n' = \gcd(n, \kappa)$, $n_0 = n/n'$, $k_0 = \kappa/n'$
%\item[$J_b(K)$] $\colonequals \{g \in \widetilde G : g^{-1} b \sigma(g) = b\}$, the $\sigma$-stabilizer of $\sigma$-conjugacy class $[b]$ in $\widetilde G$ (section \ref{s:sigma stabilizers})
%\item[$T_w(K)$] $\colonequals \{t \in \widetilde T : t^{-1} \dot w \sigma(t) = \dot w\}$, the $\sigma$-stabilizer of $w$ in $\widetilde T$ (section \ref{s:sigma stabilizers})
\item[$F$] twisted Frobenius morphism $F \from \GL_n(\breve K) \to \GL_n(\breve K)$ given by $F(g) = b \sigma(g) b^{-1}$ 
% abusing notation, we also write $F = b \sigma$, a linear operator on $V = \breve K^{\oplus n}$.
\item[$G$] $= J_b(K) = \GL_n(\breve K)^F \cong \GL_{n'}(D_{k_0/n_0})$, where $D_{k_0/n_0}$ is the division algebra with Hasse invariant $k_0/n_0$
\item[$T$] $= L^\times$, an unramified elliptic torus in $G$
%\item[$V$] $= \breve K^n$ the $\breve K$-vector space in which $\GL_n(\breve K)$ naturally acts
% \item[$V_b^{\adm}$] $= \{x \in V : \text{the $n$ vectors $\{(b\sigma)^i(x)\}_{i=0}^{n-1}$ span $V$}\}$ (Section \ref{sec:adm_subset_cyclic_lattices})
% \item[$g_b(x)$] $n\times n$-matrix with columns $\{(b\sigma)^i(x)\}_{i=0}^{n-1}$
\item[$g_b^\rd(x)$] $(n \times n)$-matrix whose $i$th column is $\varpi^{-\lfloor (i-1) k_0/n_0 \rfloor} (b\sigma)^{i-1}(x)$ with $x \in V$ (Equation \eqref{d:g_b red})
%\item[$\sL_0$] $= \cO^n \subseteq V$ the $\cO$-lattice in $V$ with stabilizer $\GL_n(\cO)$
%\item[$\sL_{0,b}^{\adm}$] $= \{x \in \sL_0 : \det g_b^{\red}(x) \in \cO_K^{\times}\}$ (Section \ref{sec:adm_subset_cyclic_lattices})
%\item[$\sL_b(x)$] the $\cO$-lattice generated by $x \in V_b^{\adm}$
% \item[$I^m$] $m$th Iwahori congruence subgroup, with pro-$p$ unipotent radical $\dot I^m$ (section \ref{sec:Comparison_unif_DLV_ADLV})
% \item[$g_{b,r}(x)$] $n \times n$ matrix with $i$th column $\varpi^{r(i-1)}F^{i-1}(x)$, where $r \in \bZ_{\geq 0}$; write $g_b(x) = g_{b,0}(x)$
%\item[$\sim_{b,m,r}$] equivalence relation on $V_b^{\adm}$, with pro-unipotent analogue $\dot \sim_{b,m,r}$ (section \ref{sec:Comparison_unif_DLV_ADLV})
%\item[$\dot w_r$] a fixed lift of the cycle $(12 \cdots n)$ with $\val(\det(\dot w_r)) = \kappa$ (section \ref{sec:Comparison_unif_DLV_ADLV})
%\item[$X_{w}^{(B)}$] $= \{x \widetilde B \in \widetilde G/\widetilde B : x^{-1} b \sigma(x) \in \widetilde B \dot w_0 \widetilde B\}$  is a semi-infinite Deligne--Lusztig variety (section \ref{sec:DL_sets})
\item[$\dot X_{\dot w}^{DL}(b)$] a semi-infinite Deligne--Lusztig variety, with a natural action of $G \times T$ (Section \ref{sec:DL_sets})
% \item[$\dot X_{\dot w}^{DL}(b)$] $= \{x \breve U \in \breve G/\breve U : x^{-1} b \sigma(x) \in \breve U \dot w \breve U\}$ is a semi-infinite Deligne--Lusztig variety and is a torsor over $X_{w}^{(DL)}(b)$ by a maximal elliptic unramified torus (section \ref{sec:DL_sets})
%\item[$X_{\dot w_r}^m(b)$] $= \{x I^m \in \widetilde G/I^m : x^{-1} b \sigma(x) \in I^m \dot w_r I^m\}$ is the affine Deligne--Lusztig variety of level $I^m$ associated to $\dot w_r$ (section \ref{sec:adlv_and_covers})
\item[$\dot X_{\dot w_r}^m(b)$] an affine Deligne--Lusztig variety with a natural action of $G \times T$ (Section \ref{sec:Comparison_unif_DLV_ADLV})
% \item[$\dot X_{\dot w_r}^m(b)$] $= \{x \dot I^m \in \widetilde G/\dot I^m : x^{-1} b \sigma(x) \in \dot I^m \dot w_r \dot I^m\}$ is the affine Deligne--Lusztig variety of level $\dot I^m$ associated to $\dot w_r$ and is a torsor over $X_{\dot w_r}^m(b)$ by $(\cO_l/\varpi^m)^\times$ (section \ref{sec:Comparison_unif_DLV_ADLV})
%\item[$X_w^\infty(b)$] $= \varprojlim\limits_{r > m} X_{\dot w_r}^m(b) = V_b^{\adm}/\cO^\times$ (definition \ref{def:ADLV_infinity})
\item[$\dot X_w^\infty(b)$] $= \varprojlim\limits_{r > m} \dot X_{\dot w_r}^m(b) = V_b^{\adm,\rat,\dot w_0} \cong V_b^{\adm,\rat} = \{x \in V_b^{\adm} : \det g_b(x) \in K^\times\}$ an affine Deligne--Lusztig variety at the infinite level, with a natural $G \times \cO_L^\times$-action (Corollary \ref{cor:algebraicity_of_gbrx_maps})
% \item[$b$] a representative of $[b]$; typically we take $b$ to be either the Coxeter-type or special representative (section \ref{sec:two b})
%\item[$\sL_{0,b}^{\adm}$] an integral analogue of $V_b^{\adm}$ (definition \ref{d:conn comps})
\item[$\dot X_w^\infty(b)_{\sL_0}$] $= \sL_{0,b}^{\adm,\rat,\dot w_0} \cong \sL_{0,b}^{\adm,\rat} = \{ x \in \sL_0 \colon \det g_b^{\rd}(x) \in \cO_K^\times\}$ is the union of connected components of $X_w^\infty(b)$ associated to the lattice $\sL_0$ (Definition \ref{d:conn comps})
% \item[$g_b^\rd(x)$] $n \times n$ matrix whose $i$th column is $\varpi^{-\lfloor (i-1) k_0/n_0 \rfloor} F^{i-1}(x)$ (definition \ref{d:g_b red})
%\item[$\bW_h$] truncated ramified Witt vectors $\bW/V^h \bW$, where $V \from \bW \to \bW$ is the Verschiebung map
%\item[$\widetilde G_h$] %$= \bJ_b(\bW_h(\overline \FF_q))$, where $\bJ_b$ is an $\cO_K$-model $J_b$
%\item[$V_h^s$] If $\{V_h : h \geq 1\}$ is a family of varieties with natural maps $V_h \to V_s$ for $s \leq h$, we let $V_h^s = \ker(V_h \to V_s)$.
\item[$G_h$] $= \bG_h(\FF_q) = (\breve G_{\bx,0}/\breve G_{\bx,(h-1)+})^F$ where $F(g) = b \sigma(g) b^{-1}$ for $b$ the Coxeter-type or special representative. $G_h$ is a subquotient of $G$ (Section \ref{sec:integral_models})
%\item[$\widetilde T_h$] the $\overline \FF_q$-points of the diagonal torus $\bT_h \subset \bG_h$
\item[$T_h$] $= \bT_h(\FF_q) \cong \cO_L^\times/U_L^h$
\item[$X_h$] a quotient of $\dot X_{\dot w_r}^m(b)_{\sL_0}$ for any $r > m \geq 0$ (Section \ref{s:Xh def}). It has a $(G_h \times T_h)$-action and is a finite-ring analogue of a Deligne--Lusztig variety (Proposition \ref{p:Xh Ah-})
%\item[$X_{h,r}$] a stratum in the Drinfeld stratification of $X_h$, indexed by the divisors $r$ of $n' = \gcd(\kappa,n)$ (section \ref{s:stratification})
%\item[$Z_h$] $= X_{h,n'}$, the unique closed stratum in the Drinfeld stratification
%\item[$\cA$] with or without various decorations, these are normed indexing sets contained in $\cA^+$, which coordinatizes $\bG_h^1$. These sets are used in the calculation of the cohomology of $Z_h^1$ (section \ref{s:index})
%\item[{$[a], [a]_{n_0}$}] for $a \in \bZ$, let $1 \leq [a] \leq n$ denote its residue modulo $n$ and let $1 \leq [a]_{n_0} \leq n_0$ denote its residue modulo $n_0$
\item[$R_{T_h}^{G_h}(\theta)$] $= \sum_i (-1)^i H_c^i(X_h, \overline \QQ_\ell)[\theta]$, where $H_c^i(X_h, \overline \QQ_\ell)[\theta] \subset H_c^i(X_h, \overline \QQ_\ell)$ is the subspace where $T_h$ acts by $\theta \from T_h \to \overline \QQ_\ell^\times$
%\item[$r_\theta$] for a character $\theta \from T_h \to \overline \QQ_\ell^\times$, the number $r_\theta$ is the nonvanishing cohomological degree of $H_c^i(Z_h, \overline \QQ_\ell)[\theta]$. It is equal to the explicitly determined $r_\chi$, where $\chi = \theta|_{T_h^1}$ (theorem \ref{t:coh Zh1}).
\item[$R_T^G(\theta)$] $= \sum_i (-1)^i H_i(\dot X_w^\infty(b), \overline \QQ_\ell)[\theta] = \sum_i (-1)^i H_i(\dot X_w^{DL}(b), \overline \QQ_\ell)[\theta]$, where the homology groups of the scheme $\dot X_w^\infty(b)$ are defined in Section \ref{s:inf ADLV} and where $[\theta]$ denotes the subspace where $T$ acts by $\theta \from T \to \overline \QQ_\ell^\times$
\item[$\sX$] the set of all smooth characters of $L^\times$ that are in general position; i.e., they have trivial stabilizer in $\Gal(L/K)$ (Part \ref{part:aut ind and JL})
%\item[$\sX^0$] the set of all characters of $L^\times$ such that $\theta|_{U_L^1}$ factors through the norm map $U_L^1 \to U_K^1$. Note that $\sX^0 \subset \sX$. (section \ref{s:depth zero})
\item[$\sX^{\rm min}$] the set of all characters of $L^\times$ that are minimal admissible (Section \ref{s:minimal})
%\item[$\sX^+$] the set of all characters of $L^\times$ that are in pro-unipotent general position; i.e., characters whose restriction to the pro-unipotent radical $U_L^1$ have trivial $\Gal(L/K)$-stabilizer. Note that $\sX^+ \subset \sX$. (definition \ref{d:in gen pos})
\end{enumerate}
The action of $G \times T$ on each of the schemes $\dot X_{\dot w_r}^m(b), \dot X_w^\infty(b), \dot X_w^{DL}(b)$ is given by $x \mapsto g x t$. These actions descend to an action of $G_h \times T_h$ on $X_h$.

\newpage

%The aim of this part of the paper is to prove a comparison result relating two geometric constructions: affine Deligne--Lusztig varieties at infinite level and semi-infinite Deligne--Lusztig varieties.

%*******************************************************************************************************************************************************
%*******************************************************************************************************************************************************

\part{Deligne--Lusztig constructions for $p$-adic groups}\label{part:DL}

In this part we discuss two analogues of Deligne--Lusztig constructions attached to a reductive group over $K$: semi-infinite Deligne--Lusztig sets and affine Deligne--Lusztig varieties at higher level. We begin by fixing some notation.

\mbox{}

Let $G$ be a connected reductive group over $K$. Let $S$ be a maximal $\breve K$-split torus in $G$. By \cite[5.1.12]{BruhatT_72} it can be chosen to be defined over $K$. Let $T = \mathscr{Z}_G(S)$ and $\mathscr N_G(S)$ be the centralizer and normalizer of $S$, respectively. By Steinberg's theorem, $G_{\breve K}$ is quasi-split, hence $T$ is a maximal torus. The Weyl group $W$ of $S$ in $G$ is the quotient $W = \mathscr{N}_G(S)/T$ of the normalizer of $S$ by its centralizer. By \cite[Theorem 21.2]{Borel_91}, every connected component of $\mathscr{N}_G(S)$ meets $\breve G$, so $W = \mathscr{N}_G(S)(\breve K)/\breve T$. In particular, the action of the absolute Galois group of $K$ on $W$ factors through a $\Gal(\breve K/K)$-action.

For a scheme $X$ over $K$, the loop space $LX$ of $X$ is the functor on $\FF_q$-algebras given by $LX(R) = X(\mathbb{W}(R)[\varpi^{-1}])$. For a scheme $\mathfrak{X}$ over $\mathcal{O}$, the space of positive loops $L^+\mathfrak{X}$ of $\mathfrak{X}$ is the functor on $\FF_q$-algebras given by $L^+\mathfrak{X}(R) = \mathfrak{X}(\mathbb{W}(R))$, and the functor $L^+_h$ of truncated positive loops is given by $L^+_h \mathfrak{X} (R) = \mathfrak{X}(\mathbb{W}_h(R))$.

For any algebro-geometric object $X$ over $K$, we write $\breve X$ for the set of its $\breve K$-rational points.

%***********************************************************************************************************************
%***********************************************************************************************************************

\section{Semi-infinite Deligne--Lusztig sets in $G/B$}\label{sec:DL_sets} 

Assume that $G$ is quasi-split. Pick a $K$-rational Borel $B \subseteq G$ containing $T$ and let $U$ be the unipotent radical of $B$. We have the following direct analogue of classical Deligne--Lusztig varieties  \cite{DeligneL_76}.

\begin{Def}\label{Def:DL_sets} 
Let $w \in W$, $\dot w \in \mathscr N_G(S)(\breve K)$ a lift of $w$, and $b \in \breve G$. The \emph{semi-infinite Deligne--Lusztig sets} $X_w^{DL}(b), \dot X_w^{DL}(b)$ are
\begin{align*} 
X_w^{DL}(b) &= \{ g \in \breve G/ \breve B : g^{-1}b\sigma(g) \in \breve B w \breve B\}, \\
\dot X_{\dot{w}}^{DL}(b) &= \{ g \in \breve G/\breve U : g^{-1}b\sigma(g) \in \breve U \dot w \breve U \}.  
\end{align*}
There is a natural map $\dot X_{\dot{w}}^{DL}(b) \rightarrow X_w^{DL}(b)$, $g\breve U \mapsto g\breve B$.
\end{Def}

For $b \in \breve{G}$, we denote by $J_b$ the $\sigma$-stabilizer of $b$, which is the $K$-group defined by
\[ 
J_b(R) \colonequals \{ g \in G(R \otimes_K \breve K) : g^{-1} b \sigma(g) = b \} 
\]
for any $K$-algebra $R$ (cf. \cite[1.12]{RapoportZ_96}). Then $J_b$ is an inner form of the centralizer of the Newton point $b$ (which is a Levi subgroup of $G$). In particular, if $b$ is \emph{basic}, i.e., the Newton point of $b$ is central, then $J_b$ is an inner form of $G$. Let $w \in W$ and let $\dot{w} \in \mathscr{N}_G(S)(\breve K)$ be a lift. We denote by $T_w$ the $\sigma$-stabilizer of $\dot w$ in $T$, which is the $K$-group defined by
\[ 
T_w(R) \colonequals \{ t \in T(R \otimes_K \breve K) : t^{-1}\dot{w}\sigma(t) = \dot{w} \}. 
\]
for any $K$-algebra $R$. As $T$ is commutative, this only depends on $w$, not on $\dot{w}$. 

\begin{lm}\label{lm:actions_on_DL_sets} Let $b\in \breve{G}$ and let $w \in W$ with lift $\dot{w} \in \mathscr{N}_G(S)(\breve K)$. 
\begin{enumerate}[label=(\roman*)]
\item Let $g \in \breve G$. The map $x\breve{B} \mapsto gx\breve{B}$ defines a bijection $X_w^{DL}(b) \stackrel{\sim}{\rightarrow} X_w^{DL}(g^{-1}b\sigma(g))$.
\item Let $g \in \breve G$ and $t \in \breve T$. The map $x\breve{U} \mapsto gxt\breve{U}$ defines a bijection $\dot X_{\dot{w}}^{DL}(b) \stackrel{\sim}{\rightarrow} \dot X_{t^{-1}\dot{w}\sigma(t)}^{DL}(g^{-1}b\sigma(g))$.
\item[(iii)] There are actions of $J_b(K)$ on $X_w^{DL}(b)$ given by $(g, x\breve{B}) \mapsto gx\breve{B}$ and of $J_b(K) \times T_w(K)$ on $\dot X_{\dot{w}}^{DL}(b)$ given by $(g,t,x\breve{U}) \mapsto gxt\breve{U}$. They are compatible with $\dot X_{\dot{w}}^{DL}(b) \rightarrow X_w^{DL}(b)$, and if this map is surjective, then $\dot X_{\dot{w}}^{DL}(b)$ is a right $T_w(K)$-torsor over $X_w^{DL}(b)$.
\end{enumerate}
\end{lm}

\begin{proof}
(i) and (ii) follow from the definitions by immediate computations. (iii) follows from (i) and (ii). 
\end{proof}

\begin{rem}\label{rem:semiinfs_remarks} \mbox{}
\begin{enumerate}[label=(\roman*)]
\item Whereas the classical Deligne--Lusztig varieties are always non-empty, $X_w^{DL}(b)$ is non-empty if and only if the $\sigma$-conjugacy class $[b]$ of $b$ in $G(\breve K)$ intersects the double coset $\breve B w \breve B$. For example, if $G = \GL_n$ ($n \geq 2$) and $b$ is superbasic, then $X_1^{DL}(b) = \varnothing$, as was observed by E. Viehmann.
\item L. Fargues pointed out the following way to endow the semi-infinite Deligne--Lusztig set $X_w^{DL}(1)$ (and $\dot X_{\dot w}^{DL}(b)$ if $T_w$ is elliptic) with a scheme structure: assume that $G$ (and $B$) come from a reductive group over $\cO_K$ (again denoted $G$), such that $G/B$ is a projective $\mathcal{O}_K$-scheme. Then
\[
(G/B)(\breve K) = (G/B)(\cO) = \mathop{\underleftarrow{\lim} }_r (G/B)(\cO/\mathfrak{p}^r).
\]
Now $(G/B)(\cO/\mathfrak{p}^r) = L^+_r(G/B)(\overline{\FF}_q)$ is a finite dimensional $\FF_q$-scheme via $L^+_r$. For a given element $w$ in the finite Weyl group, the corresponding Deligne--Lusztig condition is given by a finite set of open and closed conditions in $G/B$ which involve $\sigma$. The closed conditions cut a closed, hence projective, subscheme of $G/B$, and replacing $G/B$ by this closed subscheme $Z$, we may assume that there are only open conditions. These define an open subscheme $Y_r$ in each $L^+_r Z$. Set $X_w^{DL}(1)_r \colonequals \pr_r^{-1}(Y_r)$, where $\pr_r \colon L^+Z \rightarrow L^+_rZ$ is the projection. This gives $X_w^{DL}(1)_r$ the structure of an open subscheme of $L^+Z$ and $X_w(1) = \bigcup_{r = 1}^{\infty} X_w^{DL}(1)_r$ is now an (ascending) union of open subschemes of $L^+Z$. Note that since the transition morphisms are not closed immersions, this union does not define an ind-scheme. Now if $w$ is such that $T_w$ is elliptic, then $T_w(K)$ is compact modulo $Z(K)$, where $Z$ is the center of $G$, and $\dot X^{DL}_w(1)$---being a $T_w(K)$-torsor over $X_w^{DL}(1)$---is a scheme.

However, this scheme structure appears to be the ``correct'' one only on the subscheme $X_w^{DL}(1)_1$, as the action of $G(K) = J_1(K)$ on $X_w^{DL}(1)$ cannot in general be an action by algebraic morphisms (whereas the action of $G(\cO_K)$ on $X_w^{DL}(1)_1$ is). This will become clear from the ${\rm SL}_2$-example discussed in Section \ref{sec:example_SL2} below.
% Recently, Scholze suggested to use the following functorial definition for $X_w(b)$: for an $\FF_q$-algebra define $X_w(b)(R)$ as the pull-back in
% \begin{equation*}
% \begin{tikzcd}\label{diag:Scholzes_definition_Xwb}
% X_w(b) \ar{r} & O(w)(\bW(R)[\frac{1}{\varpi}]) \ar{d} \\
% (G/B)(\bW(R)[\frac{1}{\varpi}]) & 
% \end{tikzcd}
% \end{equation*}
% \item 
% We remark that the definition of a semi-infinite Deligne--Lusztig set we give here is in the context of inner forms $J_b$ of $G$. In general, not all inner forms of $G$ can be obtained in this way; for example, $\SL_n$ has no nontrivial $\sigma$-conjugacy classes. 
% % We choose to work in this setting as this is the standard framework for affine Deligne--Lusztig varieties.
% %In the context of classical Deligne--Lusztig theory \cite{DeligneL_76} any reductive group $\mathbb{F}_q$ is automatically quasi-split, whereas here we additionally put this assumption on $G$. 
% %However we are interested in representations of $K$-points of the reductive group $J_b$, realized by $\dot X_{\dot w}^{DL}(b)$ (once a scheme structure is provided). 
% %Up to some central extension, any reductive group over $K$ is of the form $J_b$ for some quasi-split $G$ and $b \in \breve G$, thus the quasi-splitness assumption is not problematic.
\hfill $\Diamond$
\end{enumerate}
\end{rem}

Finally we investigate the relation of $\dot X_{\dot w}^{DL}(b)$ with Lusztig's constructions from \cite{Lusztig_79,Lusztig_04}. In fact, consider the map $F \colon \breve G \to \breve G, g \mapsto b \sigma(g) b^{-1}$. 
Assuming that $(w,b)$ satisfies $w\breve{B} = b\sigma(\breve{B})$, so that $w \breve B b^{-1} = F(\breve B)$,
\begin{align*} X_w^{DL}(b) &= \{ g\breve{B} \in \breve{G}/\breve{B} : g^{-1} b \sigma(g) \in \breve{B}w\breve{B} \} \\
% &= \{ g\breve{B} \in \breve{G}/\breve{B} : g^{-1}F(g) \in \breve{B}w\breve{B}b^{-1} \} \\
&= \{ g\breve{B} \in \breve{G}/\breve{B} : g^{-1}F(g) \in \breve{B}F(\breve{B}) \} \\
&= \{ g \in \breve{G} : g^{-1}F(g) \in F(\breve{B})\} / (\breve{B} \cap F(\breve{B})) \\
% &= \{ g \in \breve{G} : g^{-1}F(g) \in F(\breve{U}) \} / (T^F (\breve{U} \cap F(\breve{U}))).
\end{align*}
Similarly, assuming that $(\dot{w},b)$ satisfies $\dot{w}\breve{U} = b\sigma(\breve{U})$, so that $\dot{w}\breve{U}b^{-1} = F(\breve{U})$,
\begin{align*}
\dot X_{\dot{w}}^{DL}(b) = \{g \in \breve G : g^{-1}F(g) \in F(\breve{U})\} / (\breve{U} \cap F(\breve{U})).
\end{align*} 
This is precisely the definition of the semi-infinite Deligne--Lusztig set in \cite{Lusztig_79}. It was studied by Boyarchenko \cite{Boyarchenko_12} and the first named author \cite{Chan_DLI, Chan_DLII, Chan_siDL} in the case when $G = \GL_n$ and $b$ superbasic, i.e., $J_b(K)$ are the units of a division algebra over $K$, where it admits an ad hoc scheme structure.

In the setting of Part 2 of this paper (see Theorem \ref{thm:structure_result}), it will turn out that $X_w^{DL}(b) = \{g \in \breve G : g^{-1}F(g) \in F(\breve U)\}/(\breve T^F (\breve U \cap F(\breve U))) = \dot X_{\dot w}^{DL}(b)/\breve T^F$. This is quite nontrivial. In the finite field setting \cite[Definition 1.17(i)]{DeligneL_76}, this is true because the Lang map $g \mapsto g^{-1}F(g)$ is surjective. In the setting of $p$-adic groups (even in our $\GL_n$ setting), the Lang map is no longer surjective. However, a corollary of Theorem \ref{thm:structure_result} is that for any $x \in X_w^{DL}(b)$, there exists a representative $g \in \breve G$ such that $g^{-1}F(g) = tu \in F(\breve B)$ with $t$ in the image of the Lang map on $\breve T$.

%*******************************************************************************************************************************************************
%*******************************************************************************************************************************************************

\section{Affine Deligne--Lusztig varieties and covers} \label{sec:adlv_and_covers}

Let the notation be as in the beginning of part \ref{part:DL}. In this section we recall from \cite{Ivanov_15_ADLV_GL2_unram} the definition of affine Deligne--Lusztig varieties of higher level, and prove that they are locally closed in the affine Grassmannian (Theorem \ref{prop:loc_closed_ADLV_covers} and Corollary \ref{cor:lftschemes}).

\subsection{Affine Grassmannian}
We will use representability results on affine Grassmannians attached to $G$, which were proven by Pappas--Rapoport \cite{PappasR_08} in the equal characterisic case, and by Zhu \cite{Zhu_17} and Bhatt-Scholze \cite{BhattS_17} in the mixed characteristic case. Let $\cG$ be a smooth affine $\cO_K$-scheme with generic fiber $G$ and with connected special fiber. The functor $L^+\cG$ is represented by an (infinite-dimensional) affine group scheme over $\FF_q$. The functor $LG$ is represented by a strict ind-scheme of ind-finite type; that is, $LG$ can be written as a direct limit of schemes of finite type, with transition morphisms being closed immersions. 

The affine Grassmannian associated with $\cG$ is the fpqc-sheaf $LG/L^+\cG$, which is the sheafification for the fpqc-topology of the functor on $\FF_q$-algebras given by 
\[
R \mapsto LG(R)/L^+\cG(R).
\]
It possesses the following representability properties.

\begin{thm}(cf. \cite[Theorem 1.4]{PappasR_08}, \cite[Theorem 1.4]{Zhu_17} and \cite[Corollary 9.6]{BhattS_17})\label{thm:IndSchemeAffineGrassmannian}
The fpqc-sheaf $LG/L^+\cG$ on $\FF_q$-algebras is represented by a strict ind-scheme. The quotient morphism $LG \rightarrow LG/L^+\cG$ has sections locally for the \'etale topology (i.e., ${\rm Spec}(R) \times_{LG/L^+\cG} LG \cong {\rm Spec}(R) \times_{{\rm Spec}(\FF_q)} L^+\cG$ for each point of $LG/L^+\cG$ with values in a strictly henselian ring $R$).
\end{thm}

Moreover, if $\cG$ is parahoric, then $LG/L^+\cG$ is ind-proper, but we will not use this in the following. In general, the affine Grassmannian is not reduced. We have $LG/L^+\cG(\overline{\FF}_q) = \breve G / \cG(\cO)$.

\subsection{Level subgroups}

Let $\Phi = \Phi(G_{\breve K}, S)$ denote the set of roots of $S$ in $G_{\breve K}$ and let $U_{\alpha}$ denote the root subgroup for $\alpha \in \Phi$. Put $U_0 \colonequals T$. 

Let ${\bf x}$ be a point in the apartment of $S$ inside the Bruhat--Tits building of the adjoint group of $G$ over $\breve{K}$. Attached to it, there is a valuation of the root datum of $G$ in the sense of Bruhat--Tits \cite{BruhatT_72}. In particular, for each $\alpha \in \Phi$, it induces a descending filtration $\breve U_{\alpha, r}$ on $\breve U_{\alpha}$ with $r \in \widetilde{\mathbb{R}}$, where $\widetilde{\mathbb{R}} \colonequals \mathbb{R} \cup \{r+ \colon r \in \mathbb{R}\} \cup \{\infty\}$ is the ordered monoid as in \cite[6.4.1]{BruhatT_72}. Further, a choice of an admissible schematic filtration on tori (in the sense of \cite[\S4]{Yu_02}) also defines a descending filtration $\breve U_{0, r} := \breve U_{0,r}$ on $\breve U_0$. (If $G$ is either simply connected or adjoint, or split over a tamely ramified extension, this filtration coincides with the Moy--Prasad filtration, and hence is independent on the choice.) For any concave function $f \colon \Phi \cup \{0\} \rightarrow \widetilde{\mathbb{R}}_{\geq 0} \sm \{\infty \}$, let $\breve{G}_f$ denote the subgroup of $\breve{G}$ (depending on ${\bf x}$) generated by $U_{\alpha,f(\alpha)}$ ($\alpha \in \Phi \cup \{0\}$). In \cite[Theorem 8.3]{Yu_02} it is shown that there exists a smooth affine group scheme $\cG_f$ over $\cO$ with generic fiber $G$, satisfying $\cG_f(\cO) = \breve G_f$. Moreover, assume that ${\bf x}$ is stable under the action of $\sigma$ on the adjoint building. Then $\cG_f$ descends to a smooth affine group scheme over $\cO_K$, again denoted $\cG_f$ \cite[\S9.1]{Yu_02}. 

\begin{prop}\label{prop:representability}
Let $f,g  \colon \Phi \cup \{0\} \rightarrow \widetilde{\mathbb{R}}_{\geq 0} \sm \{\infty \}$ be two concave functions with $g \geq f$.
\begin{itemize}
\item[(i)] $L^+\cG_g$ is a closed subgroup scheme of $L^+\cG_f$.
\end{itemize}
Assume that $\cG_g$ is normal in $\cG_f$, and that $L^+\cG_g$ is pro-unipotent.
\begin{itemize}
\item[(ii)] The fpqc quotient sheaf $L^+\cG_f / L^+\cG_g$ is representable by a smooth affine $\FF_q$-group scheme. The morphism $L^+\cG_f \rightarrow L^+\cG_f / L^+\cG_g$ splits Zariski-locally on the target.
\item[(iii)] The fpqc sheaf morphism $LG/L^+\cG_g \rightarrow LG/L^+\cG_f$ is represented in the category of ind-schemes. It is thus an $L^+\cG_f / L^+\cG_g$-torsor in the category of ind-schemes. It admits sections locally for the \'etale topology on $LG/L^+\cG_f$.
\end{itemize}
\end{prop}
\begin{proof} When $\cG_g, \cG_f$ are parahoric models of $G$, part (i) is shown in \cite[Proposition 8.7(a)]{PappasR_08}. In the general case, (i) follows by the same argument. To see (ii), first observe that $L^+\cG_g \hookrightarrow L^+\cG_f$ is a monomorphism of sheaves (although $\cG_g \rightarrow \cG_f$ is not an immersion if $f \neq g$), as $\cG_g$ is obtained from $\cG_f$ by a series of dilatations (see \cite[\S3]{BoschLR_90}) of closed subschemes in the special fiber. Put $\cG_f^{(0)} = \cG_f$ and for $h \geq 1$, let $\cG_f^{(h)}$ be the dilatation of $\cG_f^{(h-1)}$ along the unit section of the special fiber. Then $L^+\cG_f^{(h)} = \ker(L^+\cG_f \rightarrow L^+_h\cG_f)$ (cf. \cite[p.\ 414]{Zhu_17}). We can find an $h \geq 1$, such that the natural morphism $\cG_f^{(h)} \rightarrow \cG_f$ factors through $\cG_g \rightarrow \cG_f$. This gives closed immersions $L^+\cG_f^{(h)} \hookrightarrow L^+\cG_g \hookrightarrow L^+\cG_f$. Applying the natural morphism of functors $L^+ \rightarrow L^+_h$ to the arrow $\cG_g \rightarrow \cG_f$, and using that $L^+ \rightarrow L^+_h$ is surjective when evaluated at a flat $\cO_K$-scheme, we thus obtain the following commutative diagram of fpqc-sheaves on $\FF_q$-algebras, with exact rows and columns:
\begin{equation*}
\begin{tikzcd}
&&L^+\cG_f^{(h)} \ar[hookrightarrow]{dl} \ar[hookrightarrow]{d} && \\
0 \ar{r} & L^+\cG_g \ar{r} \ar[twoheadrightarrow]{d} & L^+\cG_f \ar{r} \ar[twoheadrightarrow]{d} & L^+\cG_f/L^+\cG_{g} \ar{r} \ar[twoheadrightarrow]{d} & 0 \\ 
& L^+_h\cG_{g} \ar{r}{\alpha} & L^+_h\cG_f \ar{r} & L^+_h\cG_f/{\rm im}(\alpha) \ar{r} & 0 
\end{tikzcd}
\end{equation*}
%\centerline{
%\begin{xy}\label{diag:character_isos_diag}
%\xymatrix{
%&&L^+\cG_f^{(h)} \ar@{^(->}[dl] \ar@{^(->}[d] && \\
%0 \ar[r] & L^+\cG_g \ar[r] \ar@{->>}[d] & L^+\cG_f \ar[r] \ar@{->>}[d] & L^+\cG_f/L^+\cG_{{\bf x},g} \ar[r] \ar@{->>}[d] & 0 \\ 
%  & L^+_h\cG_{{\bf x},g} \ar[r]^{\alpha} & L^+_h\cG_f \ar[r] & L^+_h\cG_f/{\rm im}(\alpha) \ar[r] & 0 
%% X_b^{(U)}(b) \ar[d]^{T_w(K)} & \{x \in V_b^\adm \colon \det(g_b(x)) \in K^{\times} \} \ar[l]_-{\sim} \ar[rr]^-{\sim} \ar[d] & & \dot{X}_b^{\infty}(b) \ar[d]^{T_w(\mathcal{O}_K)} \\
%% X_w^{(B)}(b) & V_b^\adm/\breve K^{\times} \ar[l]_-{\sim} & V_b^\adm/\cO^{\times} \ar[r]^-{\sim} \ar[l]_-{\mathbb{Z}} & X_w^{\infty}(b)
%}
%\end{xy}
%}
Now a diagram chase shows that the right vertical map is a monomorphism. Hence it is an isomorphism. We have presented $L^+\cG_f/L^+\cG_g$ as a quotient of two finite dimensional smooth affine group schemes. The last claim of (ii) follows as in the proof of \cite[Proposition 8.7(b)]{PappasR_08}.

Finally, we prove (iii). It is clear that the morphism of fpqc sheaves $p \colon LG / L^+\cG_g \rightarrow LG/L^+\cG_f$ is an $L^+\cG_f/L^+\cG_g$-torsor. A (sheaf-)torsor under an affine group scheme is always relatively representable, so we deduce from (ii) that for any scheme $T$ and any morphism $t \colon T \rightarrow LG / L^+\cG_f$, the pullback $p_t \colon T \times_{LG/L^+\cG_f} LG / L^+\cG_g \rightarrow T$ is a morphism of schemes. This implies that $LG / L^+\cG_g \rightarrow LG/L^+\cG_f$ is a morphism of ind-schemes. The last claim follows from Theorem \ref{thm:IndSchemeAffineGrassmannian}.
% From Theorem \ref{thm:IndSchemeAffineGrassmannian}, it follows that \'etale locally on $(LG/L^+\cG_f)$, $LG / L^+\cG_g \cong (LG/L^+\cG_f) \times_{\FF_q} (L^+\cG_f / L^+\cG_g)$. As \'etale descend is effective for morphisms of schemes, it follows from part (ii) that $LG / L^+\cG_g \rightarrow LG / L^+\cG_f$ is  a morphism of ind-schemes.
\end{proof}

\subsection{Affine Deligne--Lusztig varieties of higher level}
Until the end of Section \ref{sec:adlv_and_covers}, we fix a $\sigma$-stable ${\bf x}$ as above, and a $\sigma$-stable Iwahori subgroup $I \subseteq \breve G$, whose corresponding alcove in the building contains ${\bf x}$. There is a function $f_I \colon \Phi \cup \{0\} \rightarrow \widetilde{\mathbb{R}}_{\geq 0} \sm \{\infty \}$ satisfying $\breve G_{{\bf x},f_I} = I$, and we have the corresponding integral model $\cI := \cG_{{\bf x}, f_I}$. The extended affine Weyl group of $S$ in $G$ is $\widetilde{W} = \mathscr{N}_G(S)(\breve F)/\mathscr{N}_G(S)(\breve F) \cap I$. %By a \emph{level subgroup} we mean a subgroup of $I$ of the form $\breve G_f$ with $f\geq f_I$, resp. the corresponding integral model $\cG_f$.

% For more details we refer to \cite[\S 6.4]{BruhatT_72} and \cite{Yu_02}. By a \emph{level subgroup} of $I$ we mean a subgroup of the form $\breve G_f$, where ${\bf x}$ is assumed to lie in the closure of $\underline{a}_I$.

% Let $I$ be an $\sigma$-stable Iwahori subgroup of $\breve G$, whose corresponding alcove $\underline{a}_I$ in the Bruhat--Tits building $\mathscr{B}$ of $G$ over $\breve K$ is contained in the apartment of $S$. The extended affine Weyl group of $S$ is $\widetilde{W} = \mathscr{N}_G(S)(\breve F)/\mathscr{N}_G(S)(\breve F) \cap I$. The affine flag variety $\breve G/I$ is a proper ind-scheme of ind-finite type (recall the convention in Section \ref{sec:notation}). 
In \cite{Rapoport_05} Rapoport introduced an \emph{affine Deligne--Lusztig variety} attached to elements $w \in \widetilde{W}$ and $b \in \breve{G}$,
\[ 
X_w(b) = \{ gI \in \breve G/I: g^{-1}b\sigma(g) \in IwI \} \subseteq \breve G/I = (LG/L^+\cI)(\overline{\FF}_q).
\]
It is a locally closed subset of $LG/L^+\cI$, hence it inherits the reduced induced sub-ind-scheme structure (see also Theorem \ref{prop:loc_closed_ADLV_covers} below). It is even a scheme locally of finite type over $\FF_q$. Covers of $X_w(b)$ were introduced by the second named author in \cite{Ivanov_15_ADLV_GL2_unram}. We recall the definition (cf. \cite[Sections 2.1-2.2]{Ivanov_15_ADLV_GL2_ram} for a discussion in a more general setup). 

\begin{Def}\label{d:ADLV}
Let $b \in \breve{G}$. Let $f \colon \Phi \cup \{0\} \rightarrow \widetilde{R}_{\geq 0} \sm \{\infty\}$ be a concave function such that $\breve G_f$ is $\sigma$-stable. Let $x \in \breve G_f \backslash \breve G /\breve G_f$ be a double coset. Then we define the corresponding \emph{affine Deligne--Lusztig set of level $f$},
\[
X_x^f(b) \colonequals \{g\breve G_f \in \breve G/\breve G_f: g^{-1}b\sigma(g) \in \breve G_f x \breve G_f \} \subseteq \breve G/\breve G_f = (LG/L^+\cG_f)(\overline{\FF}_q).
\]
\end{Def}

If $J = \breve G_f$ satisfies the assumptions in the definition, we sometimes also write $X_x^J(b)$ for $X_x^f(b)$. We will prove that $X_x^f(b)$ is in certain cases a locally closed subset of $LG/L^+\cG_f$ (Theorem \ref{prop:loc_closed_ADLV_covers}). There is a natural $J_b(K)$-action by left multiplication on $X_x^f(b)$ for all $f$ and all $x$. If $f' \geq f$ and $x^{\prime} \in \breve G_{f'}\backslash \breve G / \breve G_{f'}$ lies over $x \in \breve G_f \backslash \breve G/\breve G_f$, then the natural projection $\breve{G}/\breve G_{f'} \twoheadrightarrow \breve{G}/\breve G_f$ restricts to a map $X_{x^{\prime}}^{f'}(b) \rightarrow X_x^f(b)$. Concerning the right action, we have the following lemma.

\begin{lm}\label{lm:actions_on_ADLV}
Let $J' = \breve G_{f'}$ and $J = \breve G_f$ be two subgroups as in Definition \ref{d:ADLV}, such that $J'$ is a normal subgroup of $J$. Let $x' \in J'\backslash \breve G / J'$ lie over $x \in J\backslash \breve G / J$ and let $b \in \breve G$.
\begin{enumerate}[label=(\roman*)]
\item Any $i \in J$ defines an $X_x^J(b)$-isomorphism $X_{x'}^{J'}(b) \rightarrow X_{i^{-1}x'\sigma(i)}^{J'}(b)$ given by $g J' \mapsto g i J'$. 
\item If $X_{x'}^{J'}(b) \to X_x^J(b)$ is surjective, then $X_{x'}^{J'}(b)$ is set-theoretically a $(J/J')_{x'}$-torsor over $X_x^J(b)$, where
\[ 
(J/J')_{x'} \colonequals \{ i \in J : i^{-1} x' \sigma(i) = x' \}/J'. 
\]
\end{enumerate}
\end{lm}

\begin{proof}
Since $J'$ is normal in $J$, we see that $i J' x' J' \sigma(i)^{-1} = J' i x' \sigma(i)^{-1} J'$. This implies (i). For (ii) we need to show that $(J/J')_{x'}$ acts faithfully and transitively on the the fibers of $\varphi \from X_{x'}^{J'}(b) \to X_x^J(b)$. By definition, $\varphi^{-1}(g J) = \{gh J' : \text{$h \in J$ and $(gh)^{-1} b \sigma(gh) \in J' x' J'$}\}$. The claim follows from normality of $J'$ in $J$ and the definition of $(J/J')_{x'}$.
\end{proof}

%************************************************************************************************************************************************************
%************************************************************************************************************************************************************

\subsection{Scheme structure} \label{sec:scheme_structure_on_ADLV}
% The goal of this section is to prove that under a technical assumption on $x$, the subset $X_x^J(b) \subseteq \breve G/J$ is locally closed, and hence is an $\overline{\FF}_q$-scheme locally of finite type. 
We need some notation. Write $\widehat{\Phi} \colonequals \Phi \cup \{0\}$. Let $\Phi_{\mathrm{aff}}$ denote the set of affine roots of $S$ in $G$ and let $\widehat \Phi_\mathrm{aff}$ be the disjoint union of $\Phi_\mathrm{aff}$ with the set of all pairs $(0,r)$ with $r \in \widetilde{\mathbb{R}}_{< \infty}$, for which the filtration step $\breve{U}_{0, r}/\breve{U}_{0, r+}$ is non-trivial. There is a natural projection $p \colon \widehat\Phi_{\mathrm{aff}} \twoheadrightarrow \widehat \Phi$, mapping an affine root to its vector part and $(0,r)$ to $0$. We extend the action of $\widetilde{W}$ on $\Phi, \Phi_\mathrm{aff}$ to an action on $\widehat{\Phi}, \widehat \Phi_\mathrm{aff}$ by letting it act trivially on $0$ and all $(0,r)$.

By \cite{Yu_02}, for any $\alpha \in \widehat{\Phi}$ and $r \in \tilde{\mathbb{R}}_{\geq 0} \sm \{\infty\}$, there is an $\cO$-scheme $\cU_{\alpha,r}$ satisfying $\cU_{\alpha,r}(\cO) = \breve U_{\alpha,r}$ whose generic fibre is $U_{\alpha, \breve K}$. If $f \colon \widehat\Phi \rightarrow \widetilde{\mathbb{R}}_{\geq 0} \sm \{\infty \}$ is concave, the schematic closure of $U_{\alpha}$ in $\cG_f$ is $\cU_{\alpha, f(\alpha)}$. If $r < s$ in $\widetilde{\mathbb{R}}_{< \infty}$, there is a unique morphism of group schemes $\cU_{\alpha,s} \rightarrow \cU_{\alpha,r}$ which induces the natural inclusion $\breve{U}_{\alpha, s} \hookrightarrow \breve{U}_{\alpha, r}$ on $\cO$-points. Let $L_{[r,s)}\cU_{\alpha}$ be the fpqc quotient sheaf
\[ 
L_{[r,s)}\cU_{\alpha} = L^+\cU_{\alpha,r}/L^+\cU_{\alpha,s}. 
\]
It is represented by a finite-dimensional group scheme over $\overline{\FF}_q$. 

\begin{lm} \label{lm:BT_lemma} Let $f \colon \widehat{\Phi} \rightarrow \widetilde{\mathbb{R}}_{\geq 0} \sm \{\infty\}$ be a concave function such that $\breve G_f \subseteq I$ is a normal subgroup. Then there is an isomorphism of $\overline{\FF}_q$-schemes
\[ 
\textstyle\prod\limits_{\alpha \in \widehat\Phi} L_{[f_I(\alpha),f(\alpha))}\cU_{\alpha} \rightarrow L^+\cI/L^+\cG_f,
\] 
which on geometric points is given by $(a_{\alpha})_{\alpha \in \widehat\Phi} \rightarrow \prod_{\alpha} \tilde{a}_{\alpha}$, where $\tilde{a}_{\alpha}$ is any lift of $a_{\alpha}$ to $\breve{U}_{\alpha,f_I(\alpha)}$ and the product can be taken in any order. 
\end{lm}
\begin{proof} The conclusion of \cite[6.4.48]{BruhatT_72} also holds for the Iwahori subgroup, i.e., for the function $f_I$ (this follows from the Iwahori decomposition). Thus there is a bijection  
\[ 
\textstyle\prod\limits_{\alpha \in \widehat\Phi} L^+\cU_{\alpha,f_I(\alpha)}(\overline{\FF}_q) \rightarrow I, \quad (a_{\alpha})_{\alpha \in \Phi \cup \{0\}} \rightarrow \prod_{\alpha} a_{\alpha}, 
\]
given by multiplication in any order, and a similar statement for $I,f_I$ replaced by $\breve G_f,f$. The statement of the lemma on geometric points follows from these bijections by normality of $\breve G^f$ in $I$. Now, the map $(a_{\alpha})_{\alpha \in \widehat\Phi} \rightarrow \prod_{\alpha} \tilde{a}_{\alpha}$ in the lemma is an algebraic morphism between smooth varieties that is bijective on geometric points and hence an isomorphism. 
\end{proof}

Let $x \in \widetilde{W}$. We give an explicit parametrization of the set of double cosets $\breve G_f \backslash IxI / \breve G_f$ in certain cases. For simplicity, we abuse the notation in the following few lemmas and write $x$ again for any lift of $x$ to $\breve G$. We say also that $(\alpha,m) \in \widehat \Phi_\mathrm{aff}$ \emph{occurs} in a subgroup $J$ of $\breve G$, if $\breve{U}_{\alpha,m}$ is contained in $J$. Then $(\alpha,m)$ occurs in $\breve{G}_f$ if and only if $m \geq f(\alpha)$. Let $\widehat \Phi_\mathrm{aff}(J) \subseteq \widehat \Phi_\mathrm{aff}$ denote the set of all pairs $(\alpha,m)$ occurring in $J$. If $J^{\prime} \subseteq J$ is a normal subgroup, let $\widehat \Phi_\mathrm{aff}(J/J^{\prime}) \colonequals \widehat \Phi_\mathrm{aff}(J) \sm \widehat \Phi_\mathrm{aff}(J^{\prime})$. 

Let $f \colon \widehat\Phi \rightarrow \widetilde{\mathbb{R}}_{\geq 0} \sm \{\infty\}$ be a concave function such that $\breve G_f \subseteq I$ is a normal subgroup. For $x \in \widetilde{W}$, we can divide the set of all affine roots $\Phi_\mathrm{aff}(I/\breve G_f)$ into three disjoint parts $A_x$, $B_x$, $C_x$:
\begin{align}
\nonumber A_x &= \{(\alpha,m) \in \widehat \Phi_\mathrm{aff}(I/\breve G_f) \colon x.(\alpha,m) \not\in \widehat \Phi_\mathrm{aff}(I) \} \\
\label{eq:def_ABC_subsets} B_x &= \{ (\alpha,m) \in \widehat \Phi_\mathrm{aff}(I/\breve G_f) \colon x.(\alpha,m) \in \widehat \Phi_\mathrm{aff}(I/\breve G_f) \} \\
\nonumber C_x &= \{ (\alpha,m) \in \widehat \Phi_\mathrm{aff}(I/\breve G_f) \colon x.(\alpha,m) \in \widehat \Phi_\mathrm{aff}(\breve G_f) \}
\end{align}
%Observe that all affine roots lying over the $0 \in \Phi \cup \{0\}$ lie in $B_x$ (this seems to be only true for nice \breve G_f. In general the action of x permute the roots lying over 0).
% Let $p \colon \widehat \Phi_\mathrm{aff} \rightarrow \widehat{\Phi}$ denote the  natural projection mapping an affine root onto its vector part.

% Let $p_f \colon \widehat \Phi_\mathrm{aff}(I/\breve G_f) \rightarrow \widehat{\Phi}$ denote the restriction of $p \colon \widehat{\Phi}_{\rm aff} \twoheadrightarrow \widehat{\Phi}$ to $\widehat \Phi_\mathrm{aff}(I/\breve G_f)$.

\begin{lm} \label{lm:description_of_DNK} Let $f \colon \widehat\Phi \rightarrow\widetilde{\mathbb{R}}_{\geq 0} \sm \{\infty\}$ be a concave, such that $\breve G_f \subseteq I$ is a normal subgroup. Let $x \in \widetilde{W}$. Assume that $p(A_x)$, $p(B_x)$ and $p(C_x)$ are mutually disjoint, and that the same is true for $p(A_{x^{-1}}), p(B_{x^{-1}}), p(C_{x^{-1}})$. Then there is a well-defined bijective map
\[
\textstyle\prod\limits_{\alpha \in p(A_{x^{-1}})} L_{[f_I(\alpha),f(\alpha))} \cU_{\alpha}(\overline{\FF}_q) \times \prod\limits_{\alpha \in p(B_x)} L_{[f_I(\alpha),f(\alpha))} \cU_{\alpha}(\overline{\FF}_q) \times \prod\limits_{\alpha \in p(A_x)} L_{[f_I(\alpha),f(\alpha))} \cU_{\alpha}(\overline{\FF}_q) \rightarrow \breve G_f \backslash IxI / \breve G_f \\
\]
given by $((a_{\alpha})_{\alpha \in p(A_{x^{-1}})}, (b_{\alpha})_{\alpha \in p(B_x)}, (a_{\alpha})_{\alpha \in p(A_x)}) \mapsto \prod_{\alpha \in p(A_{x^{-1}})} \tilde{a}_{\alpha} \cdot x \cdot \prod_{\alpha \in p(B_x)} \tilde{b}_{\alpha} \cdot \prod_{\alpha \in p(A_x)} \tilde{a}_{\alpha}$, where $\tilde{a}_{\alpha}$ is any lift of $a_{\alpha}$ to an element of $\breve{U}_{\alpha,f_I(\alpha)}$, and similarly for $\tilde b_{\alpha}, b_{\alpha}$. This endows the set $\breve G_f \backslash IxI/\breve G_f$ with the structure of a reduced $\overline{\FF}_q$-scheme of finite type.
%(a_{x^{-1}}, b_x, a_x) a_{x^{-1}}xb_xa_x
\end{lm}

\begin{proof} That the claimed map is well-defined follows from Lemma \ref{lm:BT_lemma}. We have an obvious surjective map $I/\breve G_f \times I/\breve G_f \rightarrow \breve G_f \backslash IxI / \breve G_f$, given by $(i \breve G_f, j \breve G_f) \mapsto \breve G_f i x j \breve G_f$. By Lemma \ref{lm:BT_lemma}, we may write any element of the left $I/\breve G_f$ as product $a_{x^{-1}}  b_{x^{-1}} c_{x^{-1}}$, where $a_{x^{-1}} = \prod_{\alpha \in p(A_{x^{-1}})} a_{\alpha}$, etc. 
%Analogously, we may write any element in the right $I/\breve G_f$ as the product $c_x \cdot b_x \cdot a_x$, with $c_x = \prod_{\alpha \in p(C_x)} c_{\alpha}$, etc. 
Thus any element of $\breve G_f \backslash IxI / \breve G_f$ may be written in the form
\begin{equation}\label{eq:element_of_If_DNK} \breve G_f \tilde{a}_{x^{-1}} \tilde{b}_{x^{-1}} \tilde{c}_{x^{-1}} \cdot x \cdot j \breve G_f, \end{equation}
for some $j \in I$, where $\tilde{(\cdot)}$ denotes an arbitrary lift of an element to the root subgroup. Bringing $\tilde{b}_{x^{-1}} \tilde{c}_{x^{-1}}$ to the right side of $x$ changes it to $x^{-1}\tilde{b}_{x^{-1}} \tilde{c}_{x^{-1}} x$, which is a product of elements of certain filtration steps of root subgroups, all of which lie in $I$ by definition of $B_{x^{-1}}, C_{x^{-1}}$. Thus we may eliminate $\tilde{b}_{x^{-1}} \tilde{c}_{x^{-1}}$ from \eqref{eq:element_of_If_DNK}. Now, by Lemma \ref{lm:BT_lemma}, we may write any element of the right $I/\breve G_f$ as the product $c_x  b_x a_x$, with $c_x = \prod_{\alpha \in p(C_x)} c_{\alpha}$, etc. That is, any element of $\breve G_f \backslash IxI / \breve G_f$ may be written as 
\begin{equation}\label{eq:element_of_If_DNK_2} \breve G_f \tilde{a}_{x^{-1}} \cdot x \cdot \tilde{c}_x \tilde{b}_x \tilde{a}_x \breve G_f, \end{equation}
for some lifts $\tilde{c}_x, \tilde{b}_x, \tilde{a}_x$ of $c_x, b_x, a_x$. Bringing $\tilde{c}_x$ to the left side of $x$ in \eqref{eq:element_of_If_DNK_2}, makes it to $x^{-1} \tilde{c}_x x$, which is a product of elements of certain filtration steps of root subgroups, all of which lie in $\breve G_f$ by definition of $C_{x}$. By normality of $\breve G_f$, we may eliminate $\tilde{c}_{x}$ from the \eqref{eq:element_of_If_DNK_2}. It finally follows that we may write any element of $\breve G_f \backslash IxI / \breve G_f$ as a product
\begin{equation}\label{eq:element_of_If_DNK_3} \breve G_f \tilde{a}_{x^{-1}} \cdot x \cdot \tilde{b}_x \tilde{a}_x \breve G_f, \end{equation}
with $\tilde{a}_{x^{-1}}$, $\tilde{b}_x$, $\tilde{a}_x$ as above. This shows the surjectivity of the map in the lemma. It remains to show injectivity. 

Suppose there are tuples $(a_{x^{-1}}, b_x, a_x)$ and $(a_{x^{-1}}^{\prime}, b_x^{\prime}, a_x^{\prime})$ giving the same double coset, i.e., $\tilde{a}_{x^{-1}}  x  \tilde{b}_x  \tilde{a}_x = i \tilde{a}_{x^{-1}}^{\prime} x \tilde{b}_x^{\prime} \tilde{a}_x^{\prime} j$ for some $i,j \in \breve G_f$. This equation is equivalent to
\[  x^{-1} (\tilde{a}_{x^{-1}}^{\prime})^{-1} i \tilde{a}_{x^{-1}} x = \tilde{b}_x^{\prime} \tilde{a}_x^{\prime} j  \tilde{a}_x^{-1} \tilde{b}_x^{-1}. \] 
Here, the right hand side lies in $I$, hence it follows that $(\tilde{a}_{x^{-1}}^{\prime})^{-1} i \tilde{a}_{x^{-1}} \in I \cap xIx^{-1}$. We now apply Lemma \ref{lm:BT_lemma}: any element of $I/\breve G_f$ can be written uniquely as a product $s_{x^{-1}} r_{x^{-1}}$ with $s_{x^{-1}} = \prod_{\alpha \in p(A_{x^{-1}})} s_{\alpha}$ and $r_{x^{-1}} = \prod_{\alpha \in p(B_{x^{-1}} \cup C_{x^{-1}})} r_{\alpha}$ with $s_{\alpha}, r_{\alpha} \in L_{[f_I(\alpha),f(\alpha))}U_{\alpha}(\overline{\FF_q})$. By definition, the affine roots in $A_{x^{-1}}$ are precisely those affine roots in $\widehat \Phi_\mathrm{aff}(I/\breve G_f)$ which do not occur in $I \cap xIx^{-1}$. Hence we see that the image of the composed map $I \cap xIx^{-1} \hookrightarrow I \twoheadrightarrow I/\breve G_f$ is equal to the set of all elements of $I/\breve G_f$ with $s_{x^{-1}} = 1$ in the above decomposition. Now we have inside $I/\breve G_f$ (so in particular, the element $i \in \breve G_f$ can be ignored)
\[ 
a_{x^{-1}} = a_{x^{-1}} \cdot 1 = a_{x^{-1}}^{\prime}  \cdot (a_{x^{-1}}^{\prime})^{-1} i a_{x^{-1}}, 
\]
which gives two decompositions of the element $a_{x^{-1}} \in I/\breve G_f$. By uniqueness of such a decomposition, we must have $a_{x^{-1}}^{\prime} = a_{x^{-1}}$. Now analogous computations (first done for $a_x^{\prime}, a_x$ and then for $b_x^{\prime}, b_x$) show that we also must have $a_x^{\prime} = a_x$ and $b_x^{\prime} = b_x$. This finishes the proof of injectivity. 
\end{proof}

The Schubert cell attached to $x \in \widetilde{W}$ is the reduced subscheme of $LG/L^+\cI$ whose underlying set of $\overline{\FF}_q$-points is $IxI/I \subseteq \breve G/I$. We denote it by $C(x)$. As $\overline{\FF}_q$-schemes $C(x) \cong \bA^{\ell(x)}$, where $\ell(x)$ is the length of the element $x$ in $\widetilde{W}$. We now consider the reduced subscheme of $LG/L^+\cG_f$, whose underlying set of $\overline{\FF}_q$-points is $IxI/\breve G_f \subseteq \breve G/\breve G_f$ and we denote it by $C_f(x)$. The \'etale $L^+\cI/L^+\cG_f$-torsor $LG/L^+\cG_f \rightarrow LG/L^+\cI$ pulls back to the \'etale $L^+\cI/L^+\cG_f$-torsor $C_f(x) \rightarrow C(x)$.

% The $\cI/\cG_f$-torsor $\breve G/\breve G_f \twoheadrightarrow \breve G/I$ can be trivialized over the Schubert cell $IxI/I ( \cong \mathbb{A}^{\ell(x)})$, hence a choice of any section $IxI/I \rightarrow IxI/\breve G_f$ together with the action of $I/\breve G_f$ on the fibers of $IxI/\breve G_f \twoheadrightarrow IxI/I$ gives the following parametrization of $IxI/\breve G_f$ (the bijectivity on $\overline{\FF}_q$-points is seen in the same straightforward way as in Lemma \ref{lm:description_of_DNK}).  

\begin{lm} \label{lm:param_of_Schuber_cell_preimage}
Let $f \colon \widehat\Phi \rightarrow \widetilde{\mathbb{R}}_{\geq 0} \sm \{ \infty\}$ be a concave function such that $\breve G_f \subseteq I$ is a normal subgroup. Let $x \in \widetilde{W}$. The \'etale $L^+\cI/L^+\cG_f$-torsor $C_f(x) \rightarrow C(x)$ is trivial. If $p^{-1}(p(A_{x^{-1}})) \cap \widehat\Phi_\mathrm{aff}(I/\breve G_f) = A_{x^{-1}}$, then there is an isomorphism of $\overline{\FF}_q$-schemes
\[
\textstyle\prod\limits_{\alpha \in p(A_{x^{-1}})} L_{[f_I(\alpha),f(\alpha))} \cU_{\alpha} \times L^+\cI/L^+\cG_f \rightarrow C_f(x)
\]
given by $((a_{\alpha})_{\alpha \in p(A_{x^{-1}})}, i) \mapsto \prod_{\alpha \in p(A_{x^{-1}})} \tilde{a}_{\alpha} \cdot x \cdot i\breve G_f$, where $\tilde{a}_{\alpha} \in \breve U_{\alpha, f_I(\alpha)}$ is any lift of $a_{\alpha}$.
\end{lm}

\begin{proof}
The group $L^+\cI/L^+\cG_f$ has a composition series with all subquotients either $\bG_a$ or $\bG_m$. The cohomology of both vanishes on an affine space. This proves that $C_f(x) \rightarrow C(x)$ is a trivial torsor. The explicit isomorphism is proven in the same way as in Lemma \ref{lm:description_of_DNK}.
\end{proof}

\begin{lm}\label{lm:geom_quotient} Under the assumptions of Lemma \ref{lm:description_of_DNK}, the projection $p \colon C_f(x) \twoheadrightarrow \breve G_f \backslash I x I / \breve G_f$ is a geometric quotient in the sense of Mumford for the left multiplication action of $L^+\cG_f$ on $C_f(x)$. Here $\breve G_f \backslash IxI/\breve G_f$ is endowed with the structure of an $\FF_q$-scheme using the parametrization from Lemma \ref{lm:description_of_DNK}.
\end{lm}

\begin{proof}
First note that from the assumptions of Lemma \ref{lm:description_of_DNK} it follows that $p^{-1}(p(A_{x^{-1}})) \cap \widehat\Phi_\mathrm{aff}(I/\breve G_f) = A_{x^{-1}}$ (as $A_{x^{-1}} \dot\cup B_{x^{-1}} \dot\cup C_{x^{-1}} = \widehat\Phi_{\rm aff}(I/\breve G_f)$), so we may use Lemma \ref{lm:param_of_Schuber_cell_preimage}. The action of $L^+\cG_f$ on $C_f(x)$ factors through a finite-dimensional quotient (any subgroup $J \subseteq \breve G_f \cap x \breve G_f x^{-1}$ which is normal in $\breve G_f$ acts trivially on $C_f(x)$). Now, since $p$ is a surjective orbit map, $\breve G_f \backslash I x I / \breve G_f$ is normal and the irreducible components of $C_x(f)$ are open. Thus by \cite[Proposition 6.6]{Borel_91}, it remains to show that $p$ is a separable morphism of varieties. But this is true since, in terms of the parameterizations given in Lemmas \ref{lm:description_of_DNK} and \ref{lm:param_of_Schuber_cell_preimage}, it is given by $(a_{x^{-1}}, i = c_x b_x a_x) \mapsto (a_{x^{-1}}, b_x, a_x)$.
\end{proof}

% For split $G$, where the Iwahori level sets are known to be locally closed in $\breve G/I$, we obtain the following result.

\begin{thm}\label{prop:loc_closed_ADLV_covers} Assume $G$ is split. Let $f \colon \widehat\Phi \rightarrow \widetilde{\mathbb{R}}_{\geq 0} \sm \{ \infty\}$ be a concave function such that $\breve G_f \subseteq I$ is a normal subgroup. Let $\dot{x}$ be an $\breve G_f$-double coset in $\breve{G}$ with image $x$ in $\widetilde{W}$. Assume that $p(A_x)$, $p(B_x)$ and $p(C_x)$ are mutually disjoint, and that the same is true for $p(A_{x^{-1}}), p(B_{x^{-1}}),p(C_{x^{-1}})$, where $A,B,C$  are as in \eqref{eq:def_ABC_subsets}.
% Assume that $p_f^{-1}(p_f(A_x)) = A_x$ and that the same is true for $B_x,C_x,A_{x^{-1}}, B_{x^{-1}},C_{x^{-1}}$, where $A,B,C$  are as in \eqref{eq:def_ABC_subsets}. 
For $b \in \breve{G}$ arbitrary, $X_{\dot{x}}^f(b)$ is locally closed in $\breve G/\breve G_f$.
\end{thm}

\begin{proof} By Lemma \ref{lm:geom_quotient}, the theorem is now a special case of \cite[Proposition 2.4]{Ivanov_15_ADLV_GL2_ram}.  We recall the proof. It is well known that the Iwahori-level sets $X_x^I(b) = X_x(b)$ are locally closed in $LG/L^+\cI$. Let $\widetilde X$ be the pullback of $X_x^I(b)$ along $LG/L^+\cG_f \rightarrow LG/L^+\cI$. As $X_w(b)$ are schemes locally of finite type over $\overline{\FF}_q$, the same is true for $\widetilde{X}$. By Theorem \ref{thm:IndSchemeAffineGrassmannian}, the map $\beta \colon LG \rightarrow LG/L^+\cG_f$ admits sections \'etale locally. Let $U \rightarrow \widetilde X$ be \'etale such that there is a section $s \colon U \rightarrow \beta^{-1}(U)$ of $\beta$. Consider the composition
\[ 
\psi \colon U \rightarrow \beta^{-1}(U) \times U \rightarrow LG/L^+\cG_f, 
\]
where the first map is $g \mapsto (s(g^{-1}),b\sigma(g))$ and the second map is the restriction of the left multiplication action of $LG$ on $LG/L^+\cG_f$. As $U$ lies over $\widetilde{X}$, this composition factors through $C_f(x) \rightarrow LG/L^+\cG_f$. Denote the resulting morphism by $\psi_0 \colon U \rightarrow C_f(x)$. Let $p \colon C_f(x) \rightarrow \breve G_f \backslash I x I / \breve G_f$ denote the geometric quotient from Lemma \ref{lm:geom_quotient}. The composition $p \circ \psi_0$ is independent of the choice of the section $s$. It sends an $\overline{\FF}_q$-point $g\breve G_f$ to the double coset $\breve G_f g^{-1} b \sigma(g) \breve G_f$. Thus \'etale locally, $X_{\dot{x}}^f(b)$ is just the preimage of the point $\dot{x}$ under $p \circ \psi_0$. The theorem now follows by using \'etale descent for closed subschemes.  
\end{proof}

\begin{cor}\label{cor:lftschemes}
Under the assumptions as in Theorem \ref{prop:loc_closed_ADLV_covers}, $X_{\dot x}^f(b)$ endowed with the induced reduced sub-ind-scheme structure is a scheme locally of finite type over $\overline{\FF}_q$.
\end{cor}
\begin{proof} 
$X_w^I(b)$ is a scheme locally of finite type over $\overline{\FF}_q$. Since $p \colon LG/L^+\cG_f \rightarrow LG/L^+\cI$ is a morphism of ind-schemes which is a torsor under the finite-dimensional affine group scheme $L^+\cI/L^+\cG_f$ (by Proposition \ref{prop:representability}), it follows that $\widetilde{X} := p^{-1}(X_w^I(b))$ is also a scheme locally of finite type over $\overline{\FF}_q$. Now the proof of Theorem \ref{prop:loc_closed_ADLV_covers} shows that $X_x^f(b)$ has the same property.
\end{proof}

%*******************************************************************************************************************************************************
%*******************************************************************************************************************************************************

\newpage

\part{Geometry of Deligne--Lusztig varieties for inner forms of $\GL_n$}\label{part:GLn DL}

From now and until the end of the paper, we fix an integer $n \geq 1$ and study in detail the constructions in Part \ref{part:DL} for $\GL_n(K)$ and its inner forms. Inner forms of $\GL_n$ over $K$ can be naturally parametrized by $\frac{1}{n} \mathbb{Z}/\mathbb{Z}$. Fix an integer $0 \leq \kappa < n$, put $n^{\prime} = \gcd(\kappa,n),$ and let $n_0$, $k_0$ be the non-negative integers such that
\begin{equation*}
n = n' n_0, \qquad \kappa = n' k_0.
\end{equation*}
The group of $K$-points of the inner form corresponding to $\kappa/n$ is isomorphic to $G \colonequals \GL_{n^{\prime}}(D_{k_0/n_0})$, where $D_{k_0/n_0}$ denotes the central division algebra over $K$ with invariant $k_0/n_0$. Let $\cO_{D_{k_0/n_0}}$ denote the ring of integers of $D_{k_0/n_0}$ and set $G_\cO \colonequals \GL_{n'}(\cO_{D_{k_0/n_0}})$. Note that $G_\cO$ is a maximal compact subgroup of $G$. 

We let $L$ denote the unramified extension of $K$ of degree $n$, and write $\cO_L$ for its integers, $\mathfrak{p}_L$ for the maximal ideal in $\cO_L$. For $h\geq 1$, we write $U_L^h = 1 + \mathfrak{p}_L^h$ for the $h$-units of $L$. 

Up to conjugacy there is only one maximal unramified elliptic torus $T \subseteq G$. We have $T \cong L^\times$. Moreover, we say a smooth character $\theta \colon L^{\times} \rightarrow \overline{\mathbb{Q}}_{\ell}$ has level $h \geq 0$, if $\theta$ is trivial on $U_L^{h+1}$ and non-trivial on $U_L^h$.

We let $V$ be an $n$-dimensional vector space over $\breve K$ with a fixed $K$-rational structure $V_K$. Fix a basis $\{e_1, \ldots, e_n\}$ of $V_K$. This gives an identification of $\GL(V_K)$ with $\GL_n$ over $K$. Set $\sL_0$ to be the $\cO$-lattice generated by $\{e_1, \ldots, e_n\}$.

\section{Inner forms of $\GL_n$}\label{sec:inner_forms}

%*******************************************************************************************
%*******************************************************************************************

\subsection{Presentation as $\sigma$-stabilizers of basic elements}

For $b \in \GL_n(\breve K)$, recall from Section \ref{sec:DL_sets} the $\sigma$-stabilizer $J_b$ of $b$. Then $J_b$ is an inner form of the centralizer of the Newton point $b$ (which is a Levi subgroup of $\GL_n$). In particular, if $b$ is \emph{basic}, i.e.\ the Newton point of $b$ is central, then $J_b$ is an inner form of $\GL_n$, and every inner form of $\GL_n$ arises in this way. If 
\[
\kappa = \kappa_{\GL_n}(b) \colonequals \ord\circ\det(b),
\] 
then $J_b$ is the inner form corresponding to $\kappa/n$ modulo $\mathbb{Z}$. Note that $\kappa_{\GL_n}$ is the Kottwitz map 
\begin{equation*}
\kappa_{\GL_n} \colon B(\GL_n(\breve K)) \colonequals \{\text{$\sigma$-conj classes in $\GL_n(\breve K)$}\} \rightarrow \mathbb{Z}
\end{equation*} 
and induces a bijection between the set of basic $\sigma$-conjugacy classes and $\mathbb{Z}$. Consider
\[
F \colon \GL_n(\breve K) \rightarrow \GL_n(\breve K), \qquad g \mapsto b\sigma(g)b^{-1}.
\]
This is a twisted Frobenius on $\GL_n(\breve K)$ and $J_b$ is the $K$-group corresponding to this Frobenius on $\GL_n(\breve K)$. In particular, if $b$ is in the basic $\sigma$-conjugacy class with $\kappa_{\GL_n}(b) = \kappa$, then
\begin{equation*}
G = \GL_{n'}(D_{k_0/n_0}) \cong \GL_n(\breve K)^F = J_b(K).
\end{equation*}

%*******************************************************************************************************************************************************
%*******************************************************************************************************************************************************

\subsection{Two different choices for $b$}\label{sec:two b}

We will need to choose representatives $b$ of the basic $\sigma$-conjugacy class $[b]$ with $\kappa_{\GL_n}(b) = \kappa$. Depending on the context, we will work with either a \textit{Coxeter-type representative} or a \textit{special representative}.

%*******************************************************************************************************************************************************
%*******************************************************************************************************************************************************

\subsubsection{Coxeter-type representatives}\label{sec:coxeter_representatives}

Set 
\begin{equation*}
b_0 \colonequals \left(\begin{matrix} 0 & 1 \\ 1_{n-1} & 0 \end{matrix}\right), \qquad \text{and} \qquad t_{\kappa,n} \colonequals \begin{cases}
\diag(\underbrace{1, \ldots, 1}_{n-\kappa},\underbrace{\varpi, \ldots, \varpi}_{\kappa}) & \text{if $(\kappa,n) = 1$,} \\
\diag(\underbrace{t_{k_0,n_0}, \ldots, t_{k_0,n_0}}_{n'}) & \text{otherwise.}
\end{cases}
\end{equation*}
Fix an integer $e_{\kappa,n}$ such that $(e_{\kappa,n},n)=1$ and $e_{\kappa,n} \equiv k_0$ mod $n_0$. (It is clear that $e_{\kappa,n}$ exists.)  If $\kappa$ divides $n$, (i.e.\ $k_0 = 1$), always take $e_{\kappa,n} = 1$. 

\begin{Def}
The \textit{Coxeter-type representative} attached to $\kappa$ is $b_0^{e_{\kappa,n}} \cdot t_{\kappa,n}.$
\end{Def}

The main advantage of this choice is that the maximal torus of $\GL_n(\breve K)$ consisting of diagonal matrices gives an unramified elliptic torus of $J_b$ (as the image of $b$ in the Weyl group of the diagonal torus is a cycle of length $n$). Thus when we use the explicit presentation $G = J_b(K)$ for the Coxeter-type $b$, then our unramified elliptic torus $T \subseteq G$ is the diagonal torus.

%*******************************************************************************************************************************************************
%*******************************************************************************************************************************************************

\subsubsection{Special representatives}\label{sec:special_representatives}

\begin{Def}
The \textit{special representative} attached to $\kappa$ is the block-diagonal matrix of size $n \times n$ with $(n_0 \times n_0)$-blocks of the form $\begin{pmatrix} 0&\varpi \\1_{n_0 - 1} & 0 \end{pmatrix}^{k_0}$. 
\end{Def}

Note that the special representative and the Coxeter-type representative agree if $(\kappa,n) = 1$ (see the proof of Lemma \ref{lm:GLO conj}), though in general they may differ (for example, when $\kappa = 0$, the special representative is the identity and the Coxeter representative is $b_0$).

\begin{rem}\label{rem:comparison_to_Viehmann_08}
If $b$ is the special representative, $b\sigma$ acts on the standard basis $\{e_i\}_{i=1}^n$ of $V$ in the same way as in \cite[Section 4.1]{Viehmann_08} the operator $F$ considered there acts on the basis $\{e_{j,i,l}\}_{j,i,l}$. To be more precise, in our situation, there is only one $j$ (that is $j=1$) as the isocrystal $(V, b\sigma)$ is isoclinic. Then our basis element $e_i$ for $1\leq i \leq n$ corresponds to Viehmann's basis element $e_{1,i'+1,l}$, where $i = i'n_0 + l$ for $0 \leq i' < n'$, $0\leq l < n_0$. \hfill $\Diamond$
\end{rem}

\begin{rem}
If $(\kappa, n) = 1$, the special representative $b$ is a length-$0$ element of the extended affine Weyl group of $\GL_n$ and therefore is a standard representative in the sense of \cite[Section 7.2]{GHKR10}. In general, $b$ is block-diagonal with blocks consisting of the standard representative of size $n_0 \times n_0$ and determinant $k_0$. \hfill $\Diamond$
\end{rem}

%*******************************************************************************************************************************************************
%*******************************************************************************************************************************************************

\subsubsection{Properties of the representatives}

\begin{lm} 
Let $\breve T_{\rm diag}$ denote the maximal torus of $\GL_n(\breve K)$ given by the subgroup of diagonal matrices. Then the  Coxeter-type and special representatives lie in the normalizer $N_{\GL_n(\breve K)}(\breve T_{\rm diag})$. Moreover, both representatives are basic elements whose Newton polygon has slope $\kappa/n$.
\end{lm}
\begin{proof}
The first statement is clear. For $b \in N_{\GL_n(\breve K)}(\breve T_{\rm diag})$, the Newton point can be computed as $\frac{1}{a} v_{b^a}$, where $a \in \mathbb{Z}_{>0}$ is appropriate such that $b^a \in \breve T_{\rm diag}$. Thus the second statement follows from an easy calculation (for the Coxeter type, it uses the condition on $e_{\kappa,n}$).
\end{proof}

Let $b,b' \in \GL_n(\breve K)$. We say $b, b'$ are \emph{integrally $\sigma$-conjugate} if there is $g \in \GL_n(\cO)$ such that $g^{-1}b\sigma(g) = b'$. 

\begin{lm}\label{lm:GLO conj} 
The Coxeter-type and special representatives attached to $\kappa/n$ are integrally $\sigma$-conjugate.
\end{lm}

\begin{proof}
If $\kappa$ is coprime to $n$, then necessarily $e_{\kappa,n} = \kappa$. We have $b_0^{-1}\diag(1,\ldots, 1, \varpi) b_0 = \diag(1, \ldots, 1, \varpi, 1)$ and it follows that $(b_0 \cdot t_{1,n})^\kappa = b_0^\kappa \cdot t_{\kappa,n}$, so the special representative and the Coxeter representative agree. Now assume that $(\kappa,n) = n' > 1$. For convenience of notation, let $b$ denote the Coxeter-type representative and let $b'$ denote the special representative. Recall that $b = b_0^{e_{\kappa,n}} \cdot t_{\kappa,n} = \left(\begin{smallmatrix} 0 & 1 \\ 1_{n-1} & 0 \end{smallmatrix}\right)^{e_{\kappa,n}} \cdot \diag(t_{k_0,n_0}, \ldots, t_{k_0,n_0})$ and that $b'$ is the block-diagonal matrix with $\left(\begin{smallmatrix} 0 & \varpi \\ 1_{n_0-1} & 0 \end{smallmatrix}\right)^{k_0}$ in each block. We would like to produce a $g \in \GL_n(\cO)$ such that $b \sigma(g) = g b'$. Write
\begin{equation*}
\text{$g = \Big( g_1 \, \big| \, g_2 \, \big| \, \cdots \, \big| \, g_n\Big),$ where
$g_i = \left(\begin{smallmatrix} g_{i,1} \\ g_{i,2} \\ \vdots \\ g_{i,n} \end{smallmatrix}\right)$.}
\end{equation*}
Then the first $n_0$ columns of the equation $b \sigma(g) = g b'$ is the equality of $n \times n_0$ matrices
\begin{equation*}
\Big( b \sigma(g_1) \, \big| \, b\sigma(g_2) \, \big| \, \cdots \, \big| \, b\sigma(g_{n_0})\Big) = \Big(g_{k_0+1} \, \big| \, g_{k_0+2} \, \big| \, \cdots \, \big| \, g_{n_0} \, \big| \, \varpi g_{1} \, \big| \, \varpi g_2 \, \big| \, \cdots \, \big| \, \varpi g_{k_0}\Big).
\end{equation*}
In other words, we have
\begin{equation}\label{e:gi}
g_i =
\begin{cases}
\varpi^{-1} \cdot b\sigma(g_{[i-k_0]_{n_0}}) & \text{if $1 \leq i \leq k_0,$} \\
b\sigma(g_{[i-k_0]_{n_0}}) & \text{if $k_0 + 1 \leq i \leq n_0,$}
\end{cases}
\end{equation}
and hence
\begin{equation}\label{e:gk0}
\varpi^{k_0} g_{k_0} = \varpi^{k_0-1} b \sigma(g_{n_0}) = \varpi^{k_0-1} b^2\sigma^2(g_{n_0-k_0}) = \cdots = b^{n_0} \sigma^{n_0}(g_{k_0}),
\end{equation}
where the exponent of $\varpi$ changes periodically according to \eqref{e:gi}. Observe that since $b_0^{-1}t_{\kappa,n}b_0$ is the block-diagonal matrix with blocks equal to $\left(\begin{smallmatrix} 0 & 1 \\ 1_{n_0-1} & 0 \end{smallmatrix}\right)^{-1}t_{k_0,n_0}\left(\begin{smallmatrix} 0 & 1 \\ 1_{n_0-1} & 0 \end{smallmatrix}\right)$, we see that $\varpi^{-k_0} b^{n_0}$ is a permutation matrix of order $n'$; writing $e' \colonequals [e_{\kappa,n} n_0]_n$, we have $\varpi^{-k_0}b^{n_0} = \left(\begin{smallmatrix} 0 & 1_{e'} \\ 1_{n-e'} & 0  \end{smallmatrix}\right)$ and $(e',n) = n_0$. Write
\begin{equation*}
g_{k_0} = \left(\begin{smallmatrix} x_1 \\ x_2 \\ \vdots \\ x_n \end{smallmatrix}\right).
\end{equation*}
Then from \eqref{e:gk0} it follows that
\begin{align*}
\left(\begin{smallmatrix} x_1 \\ x_2 \\ \vdots \\ x_{e'} \\ x_{e'+1} \\ \vdots \\ x_n \end{smallmatrix}\right) = \left(\begin{smallmatrix} \sigma^{n_0}(x_{n-e'+1}) \\ \sigma^{n_0}(x_{n-e'+2}) \\ \vdots \\ \sigma^{n_0}(x_n) \\ \sigma^{n_0}(x_1) \\ \vdots \\ \sigma^{n_0}(x_{n-e'}) \end{smallmatrix}\right).
\end{align*}
In particular, the vector $g_{k_0}$ is determined by $x_1, x_2, \ldots, x_{n_0}$. 

Let $\overline\alpha \in \FF_{q^n}^\times$ be an element such that $\overline\alpha, \sigma(\overline\alpha), \ldots, \sigma^{n-1}(\overline\alpha)$ are linearly independent over $\FF_q$. Let $\alpha \in \cO^\times$ be \textit{any} lift of $\overline \alpha$. Let $g_{k_0}$ be the vector associated to the choice $x_{n_0-k_0+1} = \alpha$ and $x_i = 0$ otherwise. We next show that for this choice of $g_{k_0}$, each of the columns $g_1, \ldots, g_{n_0}$ (all of which are determined by $g_{k_0}$ by \eqref{e:gi}) have coefficients in $\cO$ and that moreover the entries are either zero or in $\cO^\times$. For any positive integer $j$, we know that 
\begin{equation*}
b\sigma(\e_{[i_0-e_{\kappa,n}j]_{n}}) = \begin{cases}
\varpi \cdot \e_{[i_0-e_{\kappa,n}(j+1)]_{n}} & \text{if $[i_0-e_{\kappa,n}j]_{n_0} \geq n_0 - k_0 + 1$} \\
\e_{[i_0-e_{\kappa,n}(j+1)]_{n}} & \text{if $[i_0 - e_{\kappa,n}j]_{n_0} \leq n_0 - k_0$}
\end{cases}
\end{equation*}
where $\e_i$ denotes the $i$th elementary column vector. Comparing this to \eqref{e:gi}, we see that the condition for when the scalar $\varpi^{-1}$ appears in the equation for $g_{[(j+1)k_0]_{n_0}}$ coincides with the condition for when the scalar $\varpi$ appears in the equation for $b\sigma(\e_{[i_0-e_{\kappa,n}j]_{n}})$ when $[e_{\kappa,n}(j+1)]_{n_0} = n_0+1-[i_0-e_{\kappa,n}j]_{n_0}$. Since $e_{\kappa,n} \equiv k_0$ modulo $n_0$ by definition, this implies $i_0 = n_0-k_0+1$, and it follows now that the entries of $g_1, \ldots, g_{n_0}$ are either zero or in $\cO^\times$.

We may repeat the above argument for the next $n_0$ columns of the equation $b\sigma(g) = g b'$. We then obtain that for $n_0+1 \leq i \leq 2n_0$, the columns $g_i$ of $g$ are uniquely determined by $g_{n_0+k_0}$, and that the vector $g_{n_0+k_0}$ is uniquely determined by its first $n_0$ entries $x_1, \ldots, x_{n_0}$. Let $g_{n_0+k_0}$ be the vector associated to the choice $x_{n_0-k_0+1} = \sigma(\alpha)$ and $x_i = 0$ otherwise. Repeating this for the remaining columns of $g$, with $g_{in_0+k_0}$ being determined by setting $x_{n_0-k_0+1} = \sigma^i(\alpha)$ for $1 \leq i \leq n'-1$, we have now obtained an $n \times n$ matrix $g$ satisfying $b\sigma(g) = gb'$ whose entries are either $0$ or in $\cO^\times$.

To complete the proof that $b,b'$ are integrally $\sigma$-conjugate, it remains to show that the $\cO$-valued $n \times n$ matrix $g$ lies in $\GL_n(\cO)$; that is, it remains to show that $\det(g) \in \cO^\times$.
%, for which it is sufficient to show that the determinant of the reduction $\overline g$ modulo $\varpi$ is nonzero. 
To see this, observe that by construction, the rows and columns of $g$ can be permuted so that it is block-diagonal with $i$th ($1 \leq i \leq n_0$) block equal to the matrix
%\begin{equation*}
%\left(\begin{matrix}
%\alpha & \sigma(\alpha) & \cdots & \sigma^{n_0-1}(\alpha) \\
%\sigma(\alpha) & \sigma^2(\alpha) & \cdots & \alpha \\
%\vdots & \vdots & \ddots & \vdots \\
%\sigma^{n_0-1}(\alpha) & \alpha & \cdots & \sigma^{n_0-2}(\alpha)
%\end{matrix}\right)
%\end{equation*}
\begin{equation*}
\sigma^{(i-1)n'}\left(\begin{matrix}
\alpha & \sigma(\alpha) & \cdots & \sigma^{n'-1}(\alpha) \\
\sigma(\alpha) & \sigma^2(\alpha) & \cdots & \alpha \\
\vdots & \vdots & \ddots & \vdots \\
\sigma^{n'-1}(\alpha) & \alpha & \cdots & \sigma^{n'-2}(\alpha)
\end{matrix}\right).
\end{equation*}
The reduction-modulo-$\varpi$ of this matrix is a Vandermonde matrix, and since $\overline \alpha$ is such that $\overline \alpha, \sigma(\overline \alpha), \ldots, \sigma^{n-1}(\overline \alpha) \in \FF_{q^n}$ are linearly dependent over $\FF_q$, it has nonzero determinant. Hence $\det(g) \in \cO^\times$.
\end{proof}

\subsection{Integral models}\label{sec:integral_models}\label{sec:definition of GGh}

Let $\cB^{\red} \colonequals \cB^{\red}(\GL_n, \breve K)$ be the reduced building of $\GL_n$ over $\breve K$. For any point $x \in \cB^{\red}$, the Moy--Prasad filtration is a collection of subgroups $\breve G_{x,r} \subset \GL_n(\breve K)$ indexed by real numbers $r \geq 0$ \cite[Section 3.2]{MoyP_96}. We write $\breve G_{x,r+} = \cup_{s > r} \breve G_{x,s} \subset \GL_n(\breve K)$.

Let $\cA^{\red}$ denote the apartment of $\cB^{\red}$ associated to the maximal split torus given by the subgroup of diagonal matrices in $\GL_n(\breve K)$ and let $b$ be the Coxeter-type representative so that $b$ acts on $\cA^{\red}$ with a unique fixed point $\bx \in \cA^{\red}$. By construction,  each $\breve G_{\bx,r}$ is stable under the Frobenius $F(g) = b \sigma(g) b^{-1}$ and $\breve G_{\bx,0}^F \cong G_\cO$. 

We now define $\bG$ to be the smooth affine group scheme over $\FF_q$ such that 
\begin{equation*}
\bG(\overline \FF_q) = \breve G_{\bx, 0}, \qquad \bG(\FF_q) = \breve G_{\bx,0}^F.
\end{equation*}
For $h \in \bZ_{\geq 1}$, we define $\bG_h$ to be the smooth affine group scheme over $\FF_q$ such that
\begin{equation*}
\bG_h(\overline \FF_q) = \breve G_{\bx,0}/\breve G_{\bx,(h-1)+}, \qquad \bG_h(\FF_q) = \breve G_{\bx,0}^F/\breve G_{\bx,(h-1)+}^F.
\end{equation*}
We have a well-defined determinant morphism
\begin{equation*}
\det \from \bG_h \to \bW_h^\times.
\end{equation*}
Define $\bT_h$ to be the subgroup scheme of $\bG_h$ defined over $\FF_q$ given by the diagonal matrices. Set:
\begin{equation*}
G_h \colonequals \bG_h(\FF_q), \qquad T_h \colonequals \bT_h(\FF_q).
\end{equation*}
Note that $\bG_h(\FF_q)$ is a subquotient of $G$ and $\bT_h(\FF_q) \cong (\cO_L/\varpi^h)^\times \cong \bW_h^\times(\FF_{q^n})$ is a subquotient of the unramified elliptic torus $T$ of $G$.

We remark that each $\breve G_{\bx,r}$ is also stable under the Frobenius $F(g) = b \sigma(g) b^{-1}$ for the special representative $b$ and that $\breve G_{\bx,0}^F \cong G_\cO$. Thus we also can regard $\bG_h$ as a group scheme over $\FF_q$ as above with $\bG_h(\FF_q)$ a subquotient of $J_b(K)$ with $b$ being the special representative. However, the induced $\FF_q$-rational structure on $\bT_h$ gives that $\bT_h(\FF_q) \cong (\bW_h^\times(\FF_{q^{n_0}}))^{\times n'}$, which is \textit{not} a subquotient of any elliptic torus in $G$. 

Explicitly, $\bG_h(\overline \FF_q)$ is the group of invertible $n \times n$-matrices, whose $n_0\times n_0$-blocks are matrices $(a_{ij})_{1\leq i,j \leq n_0}$ with $a_{ii} \in \cO/\mathfrak{p}^h$, $a_{ij} \in \cO/\mathfrak{p}^{h-1}$ ($\forall i>j$), $a_{ij} \in \mathfrak{p}/\mathfrak{p}^h$ ($\forall i<j$). For example, for $n_0=3$, the $n_0 \times n_0$-blocks are
\[ \left(\begin{smallmatrix}
\cO/\mathfrak{p}^h & \mathfrak{p}/\mathfrak{p}^h  & \mathfrak{p}/\mathfrak{p}^h \\
\cO/\mathfrak{p}^{h-1} & \cO/\mathfrak{p}^h  & \mathfrak{p}/\mathfrak{p}^h \\
\cO/\mathfrak{p}^{h-1} & \cO/\mathfrak{p}^{h-1} & \cO/\mathfrak{p}^h 
\end{smallmatrix}\right).
\]

The following lemma describes the $F$-fixed part of the Weyl group of $\bT_1$ in $\bG_1$ explicitly. Note that $b^{n_0}\varpi^{-k_0}$ is a permutation matrix in $\GL_n(\breve K)$.

\begin{lm}\label{lm:descr_normalizer} Let $b$ be the Coxeter-type representative. We have
\begin{enumerate}[label=(\roman*)]
\item We have ${\rm N}_{\bG_h}(\bT_h)/\bT_h = {\rm N}_{\bG_1}(\bT_1)/\bT_1 = S_{n'} \times \dots \times S_{n'}$ ($n_0$ copies).
\item ${\rm N}_{G_h}(T_h)/T_h = ({\rm N}_{\bG_h}(\bT_h)/\bT_h)^F = \langle b^{n_0}\varpi^{-k_0} \rangle \cong \Gal(L/K)[n']$,
the $n'$-torsion subgroup of $\Gal(L/K)$.
\end{enumerate}
\end{lm}
\begin{proof}
Part (i) is clear by the explicit description of $\bG_h$. To prove (ii), we need to make the action of $F$ on ${\rm N}_{\bG_h}(\bT_h)/\bT_h$ explicit. Indeed, $F$ is an automorphism of order $n$, it permutes the copies of $S_{n'}$ cyclically, and each of the copies is stabilized by $F^{n_0}$. We can think of the first $S_{n'}$ as permutation matrices with entries $0$ and $1$ in $\GL(\langle e_i \colon i \equiv 1 \pmod{n_0} \rangle) \cong \GL_{n'}$. Then the $F^{n_0}$-action $S_{n'}$ comes from the conjugation by $b^{n_0}$ on $\GL(\langle e_i \colon i \equiv 1 \pmod{n_0} \rangle)$. But $b^{n_0}$ is the order-$n'$ cycle $e_1 \mapsto e_{1 + n_0} \mapsto \dots \mapsto e_{1 + n_0(n'-1)} \mapsto e_1$, and the subgroup of $S_{n'}$ stable by it is $\langle b^{n_0}\varpi^{-k_0} \rangle$. We can identify it with $\Gal(L/K)[n']$ by sending $b^{n_0}\varpi^{-k_0}$ to $\sigma^{n_0}$ (see also Lemma \ref{lm:defin_of_V}).
\end{proof}

\subsection{Twisted polynomial rings}\label{sec:twisted poly}

Let $L_0$ be the degree-$n_0$ unramified extension of $K$ and consider the twisted polynomial ring $L_0 \langle \Pi_0 \rangle$ determined by the commutation relation $a \cdot \Pi_0 = \Pi_0 \cdot \sigma^{l_0}(a)$, where $l_0$ is the integer in the range $1 \leq l_0 \leq n_0$ with $l_0 k_0 = 1$ modulo $n_0$. 

On the other hand, consider the Frobenius map $F_0 \from M_{n_0}(\breve K) \to M_{n_0}(\breve K)$ defined by $F_0(g) = \left(\begin{smallmatrix} 0 & \varpi \\ 1_{n_0-1} & 0 \end{smallmatrix}\right)^{k_0} \sigma(g) \left(\begin{smallmatrix} 0 & \varpi \\ 1_{n_0-1} & 0 \end{smallmatrix}\right)^{-k_0}.$ The diagonal matrices in $M_{n_0}(\breve K)^{F_0}$ are exactly of the form
\begin{equation*}
D_0(a) \colonequals \diag(a, \sigma^{l_0}(a), \ldots, \sigma^{(n_0-1)l_0}(a)), \qquad \text{for $a \in L_0$}.
\end{equation*}

By a direct calculation, we see that we can define an isomorphism
\begin{equation*}
L_0\langle \Pi_0 \rangle/(\Pi_0^{n_0} - \varpi^{k_0}) \to M_{n_0}(\breve K)^{F_0} = D_{k_0/n_0}
\end{equation*}
by sending $a \in L_0$ to $D_0(a)$, and sending $\Pi_0$ to $\left(\begin{smallmatrix} 0 & \varpi \\ 1_{n_0-1} & 0 \end{smallmatrix}\right)$. Note that under this identification, the ring of integers $\cO_{D_{k_0/n_0}}$ of $D_{k_0/n_0}$ is $\cO_{L_0}\langle \Pi_0 \rangle/(\Pi_0^{n_0} - \varpi^{k_0})$.

\subsection{Cartan decomposition}\label{sec:cartan}

Let $b$ be the special representative and let $\Pi_0 = \left(\begin{smallmatrix} 0 & \varpi \\ 1_{n_0-1} & 0 \end{smallmatrix}\right)$. We use the description of $D_{k_0/n_0}$ in Section \ref{sec:twisted poly}. Let $\breve T_{\rm diag}$ be the subgroup of diagonal matrices in $\GL_n(\breve K)$. Then the set of $F$-fixed points of the cocharacters $X_*(\breve T_{\rm diag})^F$ is given by
\[ X_{\ast}(\breve T_{\rm diag})^F = \{\nu = (\nu_1,\dots,\nu_1,\nu_2,\dots,\nu_2, \dots,\nu_{n'},\dots,\nu_{n'}) \colon \nu_i \in \bZ \},\] 
where each $\nu_i$ repeated $n_0$ times. Let $X_{\ast}(\breve T_{\rm diag})_{{\rm dom}}^F \subset X_*(\breve T_{\rm diag})^F$ be the subset consisting of $\nu$ with $\nu_1 \leq \nu_2 \leq \dots \leq \nu_{n'}$. For $\nu \in X_{\ast}(\breve T_{\rm diag})^F$, we write $\Pi_0^{\nu}$ for the $n \times n$ block-diagonal matrix whose $i$th $n_0\times n_0$-block is $\Pi_0^{\nu_i}$. The Cartan decomposition of $G = \GL_{n'}(D_{k_0/n_0})$ with respect to the maximal compact subgroup $G_\cO = \GL_{n'}(\cO_{D_{k_0/n_0}})$ is given by 
\[
G = \bigsqcup_{\nu \in X_{\ast}(\breve T_{\rm diag})_{F,{\rm dom}}} G_\cO \Pi_0^{\nu} G_\cO.
\]
Note that $\Pi_0^\nu$ normalizes $G_\cO$ if and only if all $\nu_i$ are equal, and $\Pi_0^\nu$ centralizes $G_\cO$ if and only if all $\nu_i$ are equal and divisible by $n_0$ so that
\begin{equation*}
N_{G}(G_\cO)/Z_G G_\cO \cong \bZ/n_0 \bZ.
\end{equation*}

\subsection{Reductive quotient $\bG_1$}\label{sec:reductive_quotient_of_JJb}
Let $b$ be either the Coxeter-type or the special representative. The group $\bG_1$ is equal to the reductive quotient of $\bG$. Recall the $\cO$-lattice $\sL_0$ and its basis $\{e_i\}_{i=1}^n$ from the beginning of Part \ref{part:GLn DL}. The following lemma describes the reductive quotient in terms of $\sL_0$. Its proof reduces to some elementary explicit calculations, so we omit it.

\begin{lm}\label{lm:defin_of_V} \mbox{} Let $c,d \in \mathbb{Z}$ with $k_0 c + n_0 d = 1$.
\begin{enumerate}[label=(\roman*)]
\item We have $(b\sigma)^c \varpi^d(\sL_0) \subseteq \sL_0$, and $(b\sigma)^c \varpi^d(\sL_0)$ is independent of the choice of $c,d$.\footnote{$(b\sigma)^c \varpi^d(\sL_0)$ coincides with the operator defined in \cite[Equation (4.3)]{Viehmann_08}.} The quotient space 
\[ 
\overline{V} \colonequals \sL_0/(b\sigma)^c \varpi^d(\sL_0)
\] 
is $n^{\prime}$-dimensional $\overline{\FF}_q$-vector space. The images of $\{e_i\}_{i \equiv 1 \pmod{n_0}}$ form a basis of $\overline{V}$. %(ordered by increasing $i$)
\item The map $(b\sigma)^{n_0} \varpi^{-k_0}$ induces a $\sigma^{n_0}$-linear automorphism $\overline{\sigma_b}$ of $\overline{V}$, equipping it with a $\FF_{q^{n_0}}$-linear structure. If $b$ is the special representative, the $\sigma^{n_0}$-linear operator $\overline{\sigma_b}$ of $\overline{V}$ is given by $e_i \mapsto e_i$ for $1 \leq i \leq n$ with $i \equiv 1 \pmod{n_0}$. If $b$ is Coxeter-type, then it is given by $e_{1 + n_0i} \mapsto e_{1 + n_0(i+e_{\kappa,n})}$.
\item We have a canonical identification 
\[\bG_1 = {\rm Res}_{\FF_{q^{n_0}}/\FF_q} \GL_{n'}\overline V. \]
\end{enumerate}
\end{lm}

\subsection{Isocrystals} 

We recall that an $\overline{\FF}_q$-isocrystal is an $\breve K$-vector space together with an $\sigma$-linear isomorphism. For $b\in \GL_n(\breve K)$, we have the isocrystal $(V,b\sigma)$. Assume now that $b$ is basic with $\kappa_G(b) = \kappa$. Then $(V,b\sigma)$ is isomorphic to the direct sum of $n'$ copies of the simple isocrystal with slope $k_0/n_0$. We observe that $(V,b\sigma)$ up to isomorphy only depends on the $\sigma$-conjugacy class $[b]$, and that its group of automorphisms is $G = J_b(K)$.

%*****************************************************************************************
%*****************************************************************************************

\section{Comparison in the case $\GL_n$, $b$ basic, $w$ Coxeter} \label{sec:comparison_DL_ADLV_isocrystal}

We will compare the two Deligne--Lusztig type constructions from Part \ref{part:DL} in this special situation and describe both explicitly using the isocrystal $(V,b\sigma)$. In Section \ref{sec:adm_subset_cyclic_lattices} and \ref{sec:Comparison_unif_DLV_ADLV}, we let $b \in \GL_n(\breve K)$ be any basic element with $\kappa_{\GL_n}(b) = \kappa$. From Section \ref{sec:connected_components} onwards, we take $b$ to be the special representative defined in Section \ref{sec:special_representatives}.
%*****************************************************************************************
%*****************************************************************************************

\subsection{The admissible subset of $(V,b\sigma)$}\label{sec:adm_subset_cyclic_lattices}
We will describe the various Deligne--Lusztig varieties using certain subsets of $V$, which we now define. Let $x \in V$. Put
\begin{align*}
g_b(x) &= \text{matrix in $M_n(\breve K)$ with columns $x$, $b\sigma(x)$, $\dots$, $(b\sigma)^{n-1}(x)$} \\
V_b^\adm &= \{x \in V : \det g_b(x) \in \breve K^{\times} \} \\
V_b^{\adm, \rat} &= \{x \in V : \det g_b(x) \in  K^{\times} \}
\end{align*} 
If $g^{-1} b' \sigma(g) = b$, then the isomorphism of isocrystals $(V,b\sigma) \rightarrow (V,b'\sigma)$, $x \mapsto gx$, maps $V_b^{\adm}$ to $V_{b'}^{\adm}$. In particular, $J_b(K)$ acts on $V_b^{\adm}$ by left multiplication. Moreover, $L^{\times}$ acts on $V_b^{\rm adm, rat}$ by scaling. Note also that $x \in V$ lies in $V_b^\adm$ if and only if the $\cO$-submodule of $V$ generated by $x,(b\sigma)(x), \dots, (b\sigma)^{n-1}(x)$ is an $\cO$-lattice. 
% We denote this lattice by $\mathscr{L}_b(x)$ and call it an $b\sigma$-\emph{cyclic lattice with generator} $x$. If $b^{\prime} = g^{-1} b \sigma(g)$, then $g\mathscr{L}_{b^{\prime}}(x) = \mathscr{L}_b(gx)$. 

We have the following useful lemma, which essentially follows from basic properties of Newton polygons. Its simple proof was explained to the authors by E.\ Viehmann.

\begin{lm}\label{lm:sigma_cyclic_implies_stable}
Let $x \in V_b^{\adm}$. The $\cO$-lattice generated by $\{(b\sigma)^i(x)\}_{i=0}^{n-1}$ is $b\sigma$-stable, i.e., there exist unique elements $\lambda_i \in \cO$ such that $(b\sigma)^n(x) = \sum_{i = 0}^{n-1} \lambda_i (b\sigma)^i(x)$. Moreover, $\ord(\lambda_0) = \kappa_G(b)$.
\end{lm}

\begin{proof}
The Newton polygon of $(V,b\sigma)$ is the straight line segment connecting the points $(0,0)$ and $(n,\kappa)$ in the plane. Now, let $K[\Sigma]$ be the non-commutative ring defined by the relation $a\Sigma = \Sigma \sigma(a)$, and let $\Sigma$ act on $V$ by $b\sigma$. Then the Newton polygon of the characteristic polynomial of $x$ (which is an element of $K[\Sigma]$) is equal to the Newton polygon of $(V, b\sigma)$ (see e.g.\ \cite{Beazley_09}). Observe that any $x \in V_b^\adm$ generates $V$ as a $K[\Sigma]$-module. Then the point $(i,\ord(a_i))$ in the plane, where $a_i$ is the coefficient of $\Sigma^{n-i}$ in the characteristic polynomial, lies over that Newton polygon. This simply means $\ord(a_i) \geq \frac{i\kappa}{n} \geq 0$, as $\kappa \geq 0$. Hence $\Sigma^n(v) = \sum_{i=1}^n a_i \Sigma^{n-i}(x)$ lies in the $\cO$-lattice generated by $x, \Sigma(x), \dots, \Sigma^{n-1}(x)$. This proves the first assertion. The second statement follows as $(n, \ord(a_n))$ has necessarily to be the rightmost vertex of the Newton polygon, which is $(n,\kappa)$. 
\end{proof}

\begin{ex} 
For $b = 1$, the set $V_1^\adm$ is just the Drinfeld upper halfspace. If $(\kappa, n) = 1$, then $V_b^\adm = V \sm \{0\}$ as $(V,b\sigma)$ has no proper non-trivial sub-isocrystals.
\end{ex}

\subsection{Set-theoretic description} \label{sec:Comparison_unif_DLV_ADLV}

We need the following notation:

\begin{itemize}
\item[$\bullet$] Let $T_{\diag}$ denote the diagonal torus of $\GL_n$ and $W$ its Weyl group.
\item[$\bullet$] Let $w$ be the image in $W$ of the element $b_0$ from Section \ref{sec:coxeter_representatives}. Then the form $T_w := T_{\diag,w}$ of $T_{diag}$ (as in Section \ref{sec:DL_sets}) is elliptic with $T_w(K) \cong L^{\times}$ and has a natural model over $\cO_K$, again denoted $T_w$, with $T_w(\cO_K) \cong \cO_L^{\times}$.
\item[$\bullet$] $I^m$ (with $m \geq 0$) denotes the preimage under the projection $\GL_n(\cO) \twoheadrightarrow \GL_n(\cO/\varpi^{m+1}\cO)$, of all lower triangular matrices in $\GL_n(\cO/\varpi^{m+1}\cO)$ whose entries under the main diagonal lie in $\varpi^m\cO/\varpi^{m+1}\cO$
\item[$\bullet$] $\dot{I}^m$ (with $m \geq 0$) denotes the subgroup of $I^m$ consisting of all elements whose diagonal entries are congruent $1$ modulo $\varpi^{m+1}$
\item[$\bullet$] $X_{\ast}^m(b)$, $\dot{X}_{\ast}^m(b)$ denote affine Deligne--Lusztig varieties of level $I^m$, $\dot{I}^m$ respectively (for appropriate $\ast$)
\item[$\bullet$] For $r \geq 0$ and $x \in V_b^\adm$, let $g_{b,r}(x) \in \GL_n(\breve K)$ denote the matrix whose $i$th column is $\varpi^{r(i-1)}(b\sigma)^{i-1}(x)$. We have $g_b(x) = g_{b,0}(x)$.
\item[$\bullet$] For $r, m \geq 0$, define $\sim_{b,m,r}$ and $\dot{\sim}_{b,m,r}$ on $V_b^\adm$ by 
\begin{eqnarray*}
x \,\sim_{b,m,r}\, y \in V_b^\adm \, &\Leftrightarrow& \, y \in g_{b,r}(x) \cdot \begin{pmatrix} \cO^{\times}&\mathfrak{p}^m & \dots & \mathfrak{p}^m \end{pmatrix}^\trans, \\
x \,\dot{\sim}_{b,m,r}\, y \in V_b^\adm \, &\Leftrightarrow& \, y \in g_{b,r}(x) \cdot \begin{pmatrix} 1+ \mathfrak{p}^{m+1} &\mathfrak{p}^m & \dots & \mathfrak{p}^m \end{pmatrix}^\trans.
\end{eqnarray*}
It follows from Lemma \ref{lm:comp_for_the_equiv_rels} that $\sim_{b,m,r}$ and $\dot{\sim}_{b,m,r}$ are equivalence relations.
\item[$\bullet$] For $r \geq 0$, set $\dot{w}_r = b_0 \varpi^{(-r,\dots,-r,\kappa + (n-1)r)} \in \GL(\breve K)$
% \[ 
% \dot{w}_r = \begin{pmatrix} &&&&&\varpi^{\kappa + (n-1)r} \\ \varpi^{-r}&&&&& \\ &\varpi^{-r}&&&& \\ &&\varpi^{-r}&&& \\ &&&\ddots&& \\ &&&&\varpi^{-r}& \end{pmatrix}  \in \widetilde{G}, 
% \]
and denote again by $\dot{w}_r$ the image of $\dot{w}_r$ in all the sets $I^m\backslash\GL_n(\breve K)/I^m$ and $\dot{I}^m\backslash \GL_n(\breve K)/\dot{I}^m$ for $m \geq 0$. The image of $\dot{w}_r$ in $W$ is the Coxeter element $w$.
\end{itemize}
Furthermore, we define
\begin{equation}
V_b^{\adm,\rat,\dot w_0} \colonequals \{x \in V : \sigma(\det g_b(x)) = \det(\dot w_0) \det(b)^{-1} \det(g_b(x))\}.\label{e:V_b w}
\end{equation}
Observe that $\det(\dot w_0) \det(b)^{-1} \in \cO^\times$, so picking any $\alpha \in \cO^\times$ such that $\sigma(\alpha \sigma(\alpha) \sigma^2(\alpha) \cdots \sigma^{n-1}(\alpha)) = \det(\dot w_0) \det(b)^{-1}$ induces a $(J_b(K) \times L^\times)$-equivariant isomorphism $V_b^{\adm, \rat} \to V_b^{\adm, \rat, \dot w_0}$ given by $x \mapsto \alpha x$.

\begin{rem}\label{rem:lft_scheme_str_for_speical_wrm} We will study the scheme structure on $X_{\dot{w}_r}^m(b)$, $\dot X_{\dot{w}_r}^m(b)$ in detail below in Section \ref{sec:scheme_structure_ADLV_Jb}. But we want to point out already here that both are locally closed in $\GL_n(\breve K)/I^m$, $\GL_n(\breve K)/\dot I^m$, hence are reduced $\FF_q$-schemes locally of finite type. Indeed, $\dot I^m$ is normal in $I$ and the image of $\dot{w}_r$ in $\widetilde{W}$ satisfies the assumptions of Theorem \ref{prop:loc_closed_ADLV_covers} (see Lemma \ref{lm:action_of_wrdot_on_roots} below), hence it follows that $\dot X_{\dot{w}_r}^m(b) \subseteq \breve{\GL_n}/\dot I^m$ is locally closed. The same argument does not apply to $X_{\dot{w}_r}^m(b)$ as $I^m \subseteq I$ is not normal. Still $X_{\dot{w}_r}^m(b) \subseteq \GL_n(\breve K)/I^m$ is locally closed. Indeed, let $p \colon \GL_n(\breve K)/I^m \rightarrow \GL_n(\breve K)/I$ denote the natural projection. As we will see below in Proposition \ref{p:conn comps}, the Iwahori level variety $X_{\dot w_r}^0(b) = \bigsqcup_{G/G_{\cO}} g.X^0_{\dot w_r}(b)_{\sL_0} \subseteq \GL_n(\breve K)/I$ is the scheme-theoretic disjoint union of translates of a certain locally closed subset $X^0_{\dot w_r}(b)_{\sL_0}$. It thus suffices to show that $X_{\dot w_r}^m(b)_{\sL_0} = X_{\dot w_r}^m(b) \cap p^{-1}(X^0_{\dot w_r}(b)_{\sL_0}) \subseteq p^{-1}(X^0_{\dot w_r}(b)_{\sL_0})$ is locally closed. But this follows from the explicit coordinates on $X_{\dot w_r}^m(b)_{\sL_0}$ given in the proof of Theorem \ref{thm:scheme_structure}.
\hfill $\Diamond$
\end{rem}

\begin{lm}\label{lm:action_of_wrdot_on_roots}
Let $w_r$ denote the image of $\dot w_r$ in $\widetilde W$. Then $w_r$ and $\dot I_m$ satisfy the assumptions of Theorem \ref{prop:loc_closed_ADLV_covers} (with respect to the Iwahori subgroup $I^0$). 
\end{lm}
\begin{proof}
We use the same notation as in Theorem \ref{prop:loc_closed_ADLV_covers}. We have to show that the subsets $p(A_{w_r})$, $p(B_{w_r})$, $p(C_{w_r})$ of $\widehat\Phi$ are disjoint and that the same holds for $p(A_{w_r^{-1}})$, $p(B_{w_r^{-1}})$, $p(C_{w_r^{-1}})$.
Write $\widehat\Phi = \{\alpha_{i,j} \colon 1\leq i \neq j \leq n\} \cup \{0\}$ where $\alpha_{i,j}$ corresponds to the $i,j$-th entry of a matrix. Let $w$ be the image of $w_r$ in the finite Weyl group $W$. Then $w$ acts on $\Phi$ by $w.\alpha_{i,j} = \alpha_{i+1,j+1}$, $w.0 = 0$, where we consider $i,j$ as elements of $\bZ/n\bZ$. Then $\widehat\Phi_{\rm aff}(I/\dot I^m)$ is equal to
\[ 
\{(\alpha_{i,j}, \lambda) \colon \alpha_{i,j} \in \Phi \text{ and } 0\leq \lambda \leq m-1 \text{ (if $i>j$)} \text{ resp. } 1 \leq \lambda \leq m \text{ (if $i<j$)}  \} \cup \{(0,\lambda) \colon 0\leq \lambda \leq m\}.
\]
Now one computes that $w_r.(0,\lambda) = (0,\lambda)$ and $w_r.(\alpha_{i,j},\lambda) = (\alpha_{i+1,j+1},\lambda)$ if either ($n > i > j$) or ($n > j > i$). Further, $w_r.(\alpha_{n,j},\lambda) = (\alpha_{1,j+1}, \lambda+nr+\kappa)$ and $w_r.(\alpha_{i,n},\lambda) = (\alpha_{i+1,1}, \lambda - nr - \kappa)$. 
Thus we deduce that $p(B_{w_r}) = \{\alpha_{i,j} \in \Phi \colon n>i>j \text{ or } n>j>i \} \cup \{0\}$, $p(A_{w_r}) = \{\alpha_{i,n} \in \Phi \colon 1\leq i\leq n-1\}$ and $p(C_{w_r}) = \{\alpha_{n,j} \in \Phi \colon 1\leq j\leq n-1\}$. Similarly, we compute that $p(B_{w_r}^{-1}) = \{\alpha_{i,j} \in \Phi \colon i>j>1 \} \cup\{0\}$, $p(A_{w_r^{-1}}) = \{\alpha_{1,j} \in \Phi \colon 1\leq j\leq n-1\}$ and $p(C_{w_r^{-1}}) = \{\alpha_{i,1} \in \Phi \colon 1\leq i\leq n-1\}$.
\end{proof}

Recall from Section \ref{sec:DL_sets} that $G = J_b(K)$ acts on $X_w^{DL}(b)$ and $\dot X_{\dot w_0}^{DL}(b)$ by left multiplication. By Theorem \ref{thm:structure_result} below, the maps $\dot X_{\dot w_0}^{DL}(b) \to X_w^{DL}(b)$ and $\dot X_{\dot w_r}^m(b) \to X_{\dot w_r}^m(b)$ are both surjective. Hence the former is a  $T_w(K)$-torsor via the right-multiplication action of $T_w(K)$ on $\dot X_{\dot w_0}^{DL}(b)$ (by Lemma \ref{lm:actions_on_DL_sets}(iii)) and the latter is a  $(I^m/\dot I^m)_{\dot w_r} \cong T_w(\cO_K/\varpi^{m+1})$-torsor via the right-multiplication action of $I^m/\dot I^m$ on $\dot X_{\dot w_r}^m(b)$. 

%In our situation the map $\dot X_{\dot w_0}^{DL}(b) \rightarrow X_w^{DL}(b)$ is surjective (as will follow from Theorem \ref{thm:structure_result} below) and hence a $T_w(K)$-torsor via right multiplication action of $T_w(K)$ on $\dot X_{\dot w_0}^{DL}(b)$ (by Lemma \ref{lm:actions_on_DL_sets}(iii)). Analogously, $G$ acts by left multiplication on $X_{\dot w_r}^m(b), \dot X_{\dot w_r}^m(b)$ and $\dot X_{\dot w_r}^m(b) \rightarrow X_{\dot w_r}^m(b)$ (it follows from the theorem below that this map is surjective) is
%% as $\dot X_{\dot w_r}^m(b) \rightarrow X_{\dot w_r}^m(b)$ is surjective (this follows from Lang's theorem for the $\FF_q$-group $I^m/\dot I^m$) 
%a $(I^m/\dot I^m)_{\dot w_r} \cong T_w(\cO_K/\varpi^{m+1})$-torsor via right multiplication action of $I^m/\dot I^m$ on $\dot X_{\dot w_r}^m(b)$. 
%% Finally, $G$ acts by left multiplication on $V_b^{\adm}$ and $T_{diag,w}(K) \cong L^{\times}$ acts on $V_b^{\adm}$ by scaling (by multiplying a vector with the upper left entry of a matrix in $T_{diag,w}(K)$).

% The following theorem gives a comparison between these objects: $X_w^{(U)}(b)$ is isomorphic to the affine Deligne--Lusztig variety at infinite level $\dot X_{\dot w_r}^\infty(b)$. We give an explicit isomorphism, which is compatible with the action of $J_b(K) \times T_w(K)$. We note here that the $T_w(\cO_K)$-action on $\dot X_{\dot w_r}^\infty(b)$ can be promoted to a $T_w(K)$-action since $T_w(K) \cong L^\times \cong \pi^\bZ \times \cO_L^\times \cong \pi^\bZ \times T_w(\cO_K)$.

%The following theorem is inspired by \cite[Section 2.2]{DeligneL_76}. 

\begin{thm}\label{thm:structure_result}
\begin{enumerate}[label=(\roman*)]
 \item There is a commutative diagram of sets
%\centerline{
%\begin{xy}\label{diag:character_isos_diag}
%\xymatrix{
%\{x \in V_b^\adm \colon \det(g_b(x)) \in K^{\times} \} \ar[d] \ar[r]^-{\sim} & X_{\dot{w}_0}^{(U)}(b) \ar[d]^{T_{w}(K)} \\
%V_b^\adm/\breve K^{\times} \ar[r]^-{\sim} & X_{w}^{(B)}(b)
%}
%\end{xy}
%}
\begin{equation*}
\begin{tikzcd}\label{diag:character_isos_diag}
V_b^{\adm, \rat} \ar{d} \ar{r}{\sim} & X_{\dot{w}_0}^{DL}(b) \ar{d}{T_{w}(K)} \\
V_b^\adm/\breve K^{\times} \ar{r}{\sim} & X_{w}^{DL}(b)
\end{tikzcd}
\end{equation*}
in which horizontal arrows are $G\times T_w(K)$-equivariant isomorphisms. Moreover, the vertical arrows in the diagram are surjective.

\item Assume that $r \geq m \geq 0$. There is a commutative diagram of sets 
%\[
%\xymatrix@!R=3pc{
%\left\{ x \in V_b^\adm \colon {\begin{matrix} \det(g_{b,r}(x))\, {\rm mod} \, \varpi^{m+1} \\ \text{is $\sigma$-rational} \end{matrix}} \right\} \Big/ \dot{\sim}_{b,m,r} \ar[d] \ar[r]^-{\sim} & \dot{X}_{\dot{w}_r}^{m}(b)(\bar{k}) \ar[d]^{T_{w}(\mathcal{O}_K/\varpi^{m+1}\mathcal{O}_K)} \\
%V_b^\adm/\sim_{b,m,r} \ar[r]^-{\sim} & X_{\dot{w}_r}^m(b)(\bar{k})
%}
%\]
\begin{equation*}
\begin{tikzcd}
V_b^{\adm,\rat} / \dot{\sim}_{b,m,r} \ar{d} \ar{r}{\sim} & \dot{X}_{\dot{w}_r}^{m}(b)(\overline{\FF}_q) \ar{d}{T_{w}(\mathcal{O}_K/\varpi^{m+1}\mathcal{O}_K)} \\
V_b^\adm / \sim_{b,m,r} \ar{r}{\sim} & X_{\dot{w}_r}^m(b)(\overline{\FF}_q)
\end{tikzcd}
\end{equation*}
% 
% \centerline{
% \begin{xy}\label{diag:character_isos_diag}
% \xymatrix{
% \{ x \in V_b^\adm \colon \det(g_{b,r}(x))\, {\rm mod} \, \varpi^{m+1} \\ \text{is $\sigma$-rational} \}/\dot{\sim}_{b,m,r} \ar[d] \ar[r]^-{\sim} & \dot{X}_{\dot{w}_r}^{m}(b)(\bar{k}) \ar[d]^{T_{w}(\mathcal{O}_K/\varpi^{m+1}\mathcal{O}_K)} \\
% V_b^\adm/\sim_{b,m,r} \ar[r]^-{\sim} & X_{\dot{w}_r}^m(b)(\bar{k})
% }
% \end{xy}
% }
in which horizontal arrows are $G \times T_w(\cO_K/\varpi^{m+1})$-equivariant isomorphisms. Moreover, the vertical arrows in the diagram are surjective.
\end{enumerate}
\end{thm}

Before proving the theorem, we need some preparations. Observe that by Lemmas \ref{lm:actions_on_DL_sets} and \ref{lm:actions_on_ADLV} in the proof of Theorem \ref{thm:structure_result}, we may replace $b$ by an $\sigma$-conjugate element of $\breve{G}$. 

% this is not really necessary: 
%Therefore, we may (and from now on do) assume that
%\[ b = \dot{w}_0 \]

%% \[ b = \begin{pmatrix} &&&&&\varpi^{\kappa} \\ 1&&&&& \\ &1&&&& \\ &&1&&& \\ &&&\dots&& \\ &&&&1& \end{pmatrix} \in \widetilde{G}. \]

\begin{lm}\label{lm:comp_for_the_equiv_rels} Let $r > 0$. Let $x,y \in V_b^\adm$. Then 
\begin{eqnarray}
\label{e:sim iwahori}
x \sim_{b,m,r} y &\Leftrightarrow& g_{b,r}(x)I^m = g_{b,r}(y)I^m, \\
\label{e:sim dot iwahori}
x \dot{\sim}_{b,m,r} y &\Leftrightarrow& g_{b,r}(x)\dot{I}^m = g_{b,r}(y)\dot{I}^m.
\end{eqnarray} 
\end{lm}

\begin{proof}
Indeed, $g_{b,r}(y) \in g_{b,r}(x)I^m$ is equivalent to 
\begin{align*}
y &\in x \cO^{\times} + \varpi^{m+r}b\sigma(x) \cO + \dots + \varpi^{m + r(n-1)} (b\sigma)^{n-1}(x) \cO \\
\varpi^r (b\sigma)(y) &\in \varpi^{m+1} x \cO + \varpi^r b\sigma(x) \cO^{\times} + \varpi^{m+2r} (b\sigma)^2(x) \cO \dots + \varpi^{m+(n-1)r}(b\sigma)^{n-1}(x) \cO \\
&\;\;\vdots \\
\varpi^{r(n-1)} (b\sigma)^{n-1}(y) &\in \varpi^{m+1}x\cO + \dots + \varpi^{m+1+r(n-2)} (b\sigma)^{n-2}(x) \cO + \varpi^{r(n-1)}(b\sigma)^{n-1}(x) \cO^{\times}.
\end{align*}
By definition, the first equation is equivalent to $x \sim_{b,m,r} y$. But once the first equation holds, then the $(i+1)$th equation must also hold by applying $\varpi^{ri} (b\sigma)^i$ to the first equation and using Lemma \ref{lm:sigma_cyclic_implies_stable}. Hence \eqref{e:sim iwahori} follows, and a similar proof gives \eqref{e:sim dot iwahori}.
\end{proof}

\begin{lm}\label{lm:computation_for_gbrx}
Let $r \geq 0$ and $x \in V_b^\adm$. Then 
\[ 
b\sigma(g_{b,r}(x)) = g_{b,r}(x) \dot{w}_r A, 
\]
where $A \in \GL_n(\breve K)$ is a matrix, which can differ from the identity matrix only in the last column. Moreover, the lower right entry of $A$ lies in $\cO^{\times}$, and if $r > m \geq 0$, then $A \in I^m$.
\end{lm}

\begin{proof}
By definition, we have
\begin{align*}
b \sigma(g_{b,r}(x)) &= \begin{pmatrix} b\sigma(x) & \varpi^r (b\sigma)^2(x) & \cdots & \varpi^{r(n-2)}(b\sigma)^{n-1}(x) & \varpi^{r(n-1)} (b\sigma)^n(x) \end{pmatrix}, \\
g_{b,r}(x)\dot{w}_r &=  \begin{pmatrix} b\sigma(x) & \varpi^r (b\sigma)^2(x) & \cdots & \varpi^{r(n-2)}(b\sigma)^{n-1}(x) & \varpi^{r(n-1) + \kappa_G(b)} x \end{pmatrix},
\end{align*}
As the first $n-1$ columns of these matrices coincide, it follows that $A$ can at most differ from the identity matrix in the last column. By Lemma \ref{lm:sigma_cyclic_implies_stable}, we may write
\begin{align*}
(b\sigma)^n(x) &= \sum_{i=0}^{n-1} \alpha_i \cdot (b\sigma)^i(x) \\
&= \frac{\alpha_0}{\varpi^{r(n-1)+\kappa_G(b)}} \cdot \varpi^{r(n-1)+\kappa_G(b)} x + \sum_{i=1}^{n-1} \frac{\alpha_i}{\varpi^{r(i-1)}} \cdot \varpi^{r(i-1)} (b\sigma)^i(x),
%\alpha_1 \cdot F(x) + \frac{\alpha_2}{\varpi^r} \cdot \varpi^r F^2(x) + \cdots + \frac{\alpha_{n-1}}{\varpi^{r(n-2)}} \cdot \varpi^{r(n-2)} F^{n-1}(x),
\end{align*}
where $\alpha_0, \ldots, \alpha_{n-1} \in \cO$ and $\ord(\alpha_0) = \kappa$. By construction, the last column of $A$ is 
\begin{equation*}
\left(\varpi^{r(n-1)} \alpha_1, \varpi^{r(n-2)} \alpha_2, \varpi^{r(n-3)} \alpha_3, \ldots, \varpi^r \alpha_{n-1}, \frac{\alpha_0}{\varpi^{\kappa_G(b)}}\right)^\trans.
\end{equation*}
We then see that the lower right entry of $A$ is $\frac{\alpha_0}{\varpi^{\kappa}} \in \cO^\times$ and that if $r \geq m+1$, then all the entries above $\frac{\alpha_0}{\varpi^{\kappa}}$ lie in $\varpi^{m+1} \cO$ and $A \in I^m$.
\end{proof}

\begin{proof}[Proof of Theorem \ref{thm:structure_result}]
(i): As in \cite[\S1]{DeligneL_76}, the sets $X_w^{DL}(b)$ do not depend on the choice of the Borel subgroup, so we may choose $B \subseteq \GL_n$ to be the Borel subgroup of the upper triangular matrices and $U$ its unipotent radical. Lemma \ref{lm:computation_for_gbrx} for $r = 0$ implies the existence of the map 
\[ 
V_b^\adm \rightarrow X_w^{DL}(b), \quad x \mapsto g_b(x)\breve{B}. 
\]
We claim this map is surjective. Let $g\breve{B} \in X_w^{DL}(b)$, i.e., $g^{-1} b \sigma(g) \in \breve{B} \dot{w}_0 \breve{B}$. Replacing $g$ by another representative in $g\breve{B}$ if necessary, we may assume that $b \sigma(g) \in g \dot{w}_0 \breve{B}$. Moreover, this assumption does not change, whenever we replace $g$ by another representative $g^{\prime} = gc$ with $c \in \breve{B} \cap {}^{\dot w_0} \breve{B}$ (here ${}^{\dot w_0} \breve{B} = \dot w_0\breve{B}\dot w_0^{-1}$). A direct computation shows that replacing $g$ by $gc$ for an appropriate $c \in B \cap {}^{\dot{w}_0} \breve{B}$, we find a representative $g$ of $g\breve{B}$ with columns $g_1, g_2, \dots, g_n$ satisfying $g_{i+1} = b\sigma(g_i)$ for $i = 1, \dots, n-1$. This means precisely $g = g_b(x)$. All this shows the surjectivity claim. For $x, y \in V_b^\adm$, one has $g_b(x) \breve{B} = g_b(y) \breve{B}$ if and only if $x,y$ differ by a constant in $\breve K^{\times}$. This shows the lower horizontal isomorphism in part (i) of the 
theorem.

We now construct the upper isomorphism. We may write an element of $\dot{g}\breve{U} \in \breve{G}/\breve{U}$ lying over $g_b(x)\breve{B} \in X_w^{DL}(b)$ as $\dot{g}\breve{U} = g_b(x)t\breve{U}$ for some $t \in \breve{T}$. 
%The $\sigma$-linearity of $F$ implies $F(g_b(x)t) = F(g_b(x))\sigma(t)$. 
Using Lemma \ref{lm:computation_for_gbrx} (and the notation from there) we see that
\[ 
\dot{g}^{-1} b \sigma(\dot{g}) = t^{-1} g_b(x)^{-1}b \sigma(g_b(x)) \sigma(t) = t^{-1} \dot{w}_0 A \sigma(t) 
\]
Now write $A = \lambda A_0$ with $A_0 \in \breve U$ and $\lambda = \diag(1,\dots,1,\lambda_0)$ a diagonal matrix with $\lambda_0 \in \cO^\times$. We then see that $\dot{g}^{-1} b \sigma(\dot{g}) = t^{-1}\dot w_0 \lambda \sigma(t)\cdot {}^{\sigma(t)}A_0$. Hence $g_b(x)t \in X_{\dot{w}_0}^{DL}(b)$ if and only if $\dot w_0^{-1}t^{-1}\dot w_0 \sigma(t) = \lambda^{-1}$. Thus, if we write $t_0, t_1, \dots, t_{n-1} \in \breve K^{\times}$ for the diagonal entries of $t$, then we have $t_{i+1} = \sigma(t_i)$ for $0 \leq i \leq n-2$. We may assume this. In particular, it implies that $g_b(x)t = g_b(xt_0)$. In other words, replacing $x$ by $xt_0$, we may assume that $\dot{g} = g_b(x)$. It remains to determine all $x \in V_b^\adm$, for which $g_b(x)\breve{U} \in X_{\dot{w}_0}^{DL}(b)$, i.e., $g_b(x)^{-1} b\sigma(g_b(x)) \in \breve{U} \dot{w}_0 \breve{U}$. Comparing the determinants on both sides we deduce that $\sigma(\det g_b(x)) = \det(\dot w_0) \det(b)^{-1} \det g_b(x)$ (i.e.\ that $x \in V_b^{\adm,\rat,\dot w_0}$) as a necessary condition. Assume this holds. With notations as in Lemma \ref{lm:computation_for_gbrx}, we deduce $\det(A) = 1$. Moreover, since $A$ can differ from the identity matrix in at most the last column by Lemma \ref{lm:computation_for_gbrx}, $\det(A) = 1$ is equivalent to $A \in \breve{U}$. All this, together with the fact that $V_b^{\adm,\rat,\dot w_0} \cong V_b^{\adm,\rat}$ (see \eqref{e:V_b w}) shows the upper isomorphism in part (i). The commutativity of the diagram and $J_b(K)$-equivariance of the involved maps are clear from the construction. The surjectivity claim for the vertical maps is shown in exactly the same way as the analogous claim in part (ii) below.

(ii): Lemma \ref{lm:computation_for_gbrx} for $r > m \geq 0$ implies the existence of the map 
\[
V_b^\adm \rightarrow X_{\dot{w}_r}^m(b), \quad x \mapsto g_{b,r}(x)I^m. 
\]
We claim it is surjective. Let $gI^m \in X_{\dot{w}_r}^m(b)$, i.e., $g^{-1}b\sigma(g) \in I^m \dot{w}_r I^m$. Replacing $g$ by another representative of $gI^m$ if necessary, we may assume that $b\sigma(g) \in g \dot{w}_r I^m$. Moreover, this assumption does not change, whenever we replace $g$ by another representative $g^{\prime} = gj$ with $j \in I^m \cap {}^{\dot{w}_r} I^m$. In the rest of the proof, we call such transformations \emph{allowed}. We compute
\[ 
I^m \cap {}^{\dot w_r} I^m = \begin{pmatrix} 
\cO^{\times} & \mathfrak{p}^{rn+m} & \cdots & \cdots & \mathfrak{p}^{rn+m} \\ 
\mathfrak{p}^m & \cO^{\times} & \mathfrak{p}^{m+1} & \cdots & \mathfrak{p}^{m+1} \\
\vdots&\ddots&\ddots&\ddots&\vdots \\
\mathfrak{p}^m &\cdots &\mathfrak{p}^m&\cO^{\times}&\mathfrak{p}^{m+1} \\
\mathfrak{p}^m & \cdots & \cdots  &\mathfrak{p}^m & \cO^{\times}
\end{pmatrix} 
\]
% Here is the special case $m=0$:
% \[ I \cap {}^{w_r} I = \begin{pmatrix} 
% \cO^{\times} & \mathfrak{p}^{nr} & &  & \dots & \mathfrak{p}^{nr} \\ 
% \cO & \cO^{\times} & \mathfrak{p} & & \dots & \mathfrak{p} \\
% &&\dots&&& \\
% &&\dots&&& \\
% \cO&&&\cO&\cO^{\times}&\mathfrak{p} \\
% \cO & & \dots &  &\cO & \cO^{\times}
% \end{pmatrix} \]
(on the main diagonal entries can lie in $\cO^{\times}$, under the main diagonal in $\mathfrak{p}^m$, in the first row, beginning from the second entry, in $\mathfrak{p}^{rn + m}$, and above the main diagonal, except for the first row, in $\mathfrak{p}^{m+1}$). Let $g_1, \dots, g_n$ denote the columns of $g$, seen as elements of $V$. Then $g\dot{w}_r \in b\sigma(g)I^m$ is equivalent to the following $n$ equations:
\begin{align*}
\nonumber g_2 &\in \varpi^r b\sigma(g_1)\cO^{\times} + \varpi^{r+m}b\sigma(g_2) \cO + \dots + \varpi^{r+m} b\sigma(g_n) \cO \\
\nonumber g_3 &\in \varpi^{r+m+1} b\sigma(g_1) \cO + \varpi^r b\sigma(g_2)\cO^{\times} + \varpi^{r+m} b\sigma(g_3) \cO + \dots + \varpi^{r+m} b\sigma(g_n) \cO \\
\nonumber & \;\;\vdots \\
\nonumber g_n &\in \varpi^{r+m+1} b\sigma(g_1) \cO + \dots + \varpi^{r+m+1} b\sigma(g_{n-2})\cO + \varpi^r b\sigma(g_{n-1}) \cO^{\times} + \varpi^{r+m} b\sigma(g_n) \cO \\
\nonumber \varpi^{rn+m+\kappa} g_1 &\in \varpi^{r + 2m + 1} b\sigma(g_1) \cO + \dots + \varpi^{r + 2m + 1} b\sigma(g_{n-1})\cO + \varpi^{r+m} b\sigma(g_n) \cO^{\times}. 
\end{align*}
A linear algebra exercise shows that after some allowed transformations these equations can be rewritten as 
\begin{align*}
\nonumber g_2 &\in \varpi^r b\sigma(g_1)\cO^{\times} \\
\nonumber g_3 &\in \varpi^r b\sigma(g_2)\cO^{\times} \\
\nonumber & \;\;\vdots \\
\nonumber g_n &\in \varpi^r b\sigma(g_{n-1}) \cO^{\times} \\
\nonumber \varpi^{r(n-1)+\kappa}g_1 &\in \varpi^{m + 1}b\sigma(g_1) \cO + \dots + \varpi^{m + 1} b\sigma(g_{n-1})\cO + b\sigma(g_n) \cO^{\times}. 
\end{align*}
This shows that $g = g_{b,r}(g_1)$, and hence the claimed surjectivity. Lemma \ref{lm:comp_for_the_equiv_rels} shows that the lower map in part (ii) is an isomorphism. Exactly as in the proof of (i), one shows that the claim of (ii) is true if one replaces the upper left entry by $\left\{ x \in V_b^\adm \colon {\begin{matrix} \det(g_{b,r}(x))\, {\rm mod} \, \varpi^{m+1} \\ \text{is fixed by $\sigma$} \end{matrix}} \right\}$. As $x\, \dot{\sim}_{b,m,r}\, xu$ for any $u \in 1 + \mathfrak{p}^{m+1}$, the original claim of (ii) follows from this modified claim along with the surjectivity of the map $1 + \mathfrak{p}^{m+1} \rightarrow 1 + \mathfrak{p}^{m+1}$, $u \mapsto \prod_{i=0}^{n-1} \sigma^i(u)$, and the fact that $\det g_b(x) \in K^{\times} \Leftrightarrow \det g_{b,r}(x) \in K^{\times}$. 

It remains to show that the vertical arrows in the diagram in (ii) are surjective. It suffices to handle the left arrow. Let $x \in V^{\rm adm}_b$. By definition of $\sim_{b,m,r}$, for any $\lambda \in \cO^\times$ we have $\lambda x \sim_{b,m,r} x$. Now we have $\det g_b(\lambda x) = \varpi^a \cdot \prod_{i=0}^{n-1}\sigma^i(\lambda) \cdot \det g_b(x)$ for an appropriate $a \in \bZ$. Now the map $\lambda \mapsto \prod_{i=0}^{n-1}\sigma^i(\lambda) \colon \cO^\times \rightarrow \cO^\times$ is surjective, hence by rescaling $x$ with an appropriate $\lambda \in \cO^\times$ (thus not changing its class in $V_b^{\rm adm}/\sim_{b,m,r}$) we can arrange that $\det g_b(x) \in K^\times$, i.e., that $x\in V_b^{\rm adm, rat}$. \qedhere
\end{proof}

The natural projection maps $X_{\dot w_r}^{m+1}(b) \to X_{\dot w_r}^m(b)$ and $\dot X_{\dot w_r}^{m+1}(b) \to \dot X_{\dot w_r}^m(b)$ are obviously morphisms of schemes. However, Theorem \ref{thm:structure_result} implies that there are $G$- and $G \times T_w(\cO_K/\varpi^{m+1})$-equivariant maps of sets (on $\overline \FF_q$-points)
\begin{equation}
X_{\dot w_{r+1}}^m(b) \to X_{\dot w_r}^m(b), \qquad \text{and} \qquad \dot X_{\dot w_{r+1}}^m(b) \to \dot X_{\dot w_r}^m(b)
\end{equation}
induced by $g_{b,r+1}(x)\mapsto g_{b,r}(x)$. In Section \ref{sec:scheme_structure_ADLV_Jb}, we explicate the scheme structure on $X_{\dot w_r}^m(b)$, $\dot X_{\dot w_r}^m(b)$ and prove that these maps of sets are actually morphisms of schemes (Theorem \ref{thm:scheme_structure}). Taking Theorem \ref{thm:scheme_structure} for granted at the moment, we have a notion of an \textit{affine Deligne--Lusztig variety at infinite level}.

\begin{Def}\label{d:inf ADLV}
Define the (infinite-dimensional) $\overline \FF_q$-scheme
\begin{equation*}
X_w^{\infty}(b) \colonequals \varprojlim\limits_{r,m \colon r > m} X_{\dot{w}_r}^m(b) \qquad \text{and} \qquad \dot{X}_w^{\infty}(b) \colonequals \varprojlim\limits_{r,m \colon r > m} \dot{X}_{\dot{w}_r}^m(b).
\end{equation*}
Both have actions by $G$ and the natural $G$-equivariant map $\dot{X}_w^{\infty}(b) \rightarrow X_w^{\infty}(b)$ is a $T_w(\cO_K)$-torsor. 
\end{Def}

Passing to the infinite level in Theorem \ref{thm:structure_result} gives the following result.

\begin{thm}\label{cor:infty_level_adlv} \mbox{}
There is a commutative diagram of sets with $G$-equivariant maps:
%\centerline{
%\begin{xy}\label{diag:character_isos_diag}
%\xymatrix{
%X_b^{(U)}(b) \ar[d]^{T_w(K)} & \{x \in V_b^\adm \colon \det(g_b(x)) \in K^{\times} \} \ar[l]_-{\sim} \ar[rr]^-{\sim} \ar[d] & & \dot{X}_b^{\infty}(b) \ar[d]^{T_w(\mathcal{O}_K)} \\
%X_w^{(B)}(b) & V_b^\adm/\breve K^{\times} \ar[l]_-{\sim} & V_b^\adm/\cO^{\times} \ar[r]^-{\sim} \ar[l]_-{\mathbb{Z}} & X_w^{\infty}(b)
%}
%\end{xy} \\
\begin{equation*}
\begin{tikzcd}\label{diag:character_isos_diag}
\dot X_w^{DL}(b) \ar{d}{T_w(K)} & V_b^{\adm,\rat} \ar{l}[above]{\sim} \ar{rr}{\sim} \ar{d} & & \dot{X}_w^{\infty}(b) \ar{d}{T_w(\mathcal{O}_K)} \\
X_w^{DL}(b) & V_b^\adm/\breve K^{\times} \ar{l}[above]{\sim} & V_b^\adm/\cO^{\times} \ar{r}{\sim} \ar{l}[above]{\mathbb{Z}} & X_w^{\infty}(b)
\end{tikzcd}
\end{equation*}
%}
The upper horizontal maps are $T_w(\cO_K)$-equivariant. This extends the natural $T_w(\cO_K)$-action on $\dot X_w^{\infty}(b)$ to a $T_w(K)$-action.
\end{thm}

Using the set-theoretic isomorphism in Theorem \ref{cor:infty_level_adlv}, we will see in Section \ref{sec:scheme_structure_ADLV_Jb} that by endowing $V_b^{\adm}$ with the natural scheme structure over $\overline \FF_q$ coming from the lattice $\sL_0$, we can view the semi-infinite Deligne--Lusztig sets $X_w^{DL}(b), \dot X_w^{DL}(b)$ as (infinite-dimensional) $\overline \FF_q$-schemes. Moreover, every isomorphism in Theorem \ref{cor:infty_level_adlv} is an isomorphism of $\overline \FF_q$-schemes (Corollary \ref{cor:scheme_structure_on the_inverse_limit}).

%************************************************************************************************************************************************************
%************************************************************************************************************************************************************

\subsection{Connected components}\label{sec:connected_components}
To ``minimize'' powers of the uniformizer, we define
\begin{equation}\label{d:g_b red}
g_b^\rd(v) \colonequals \left( v \Bigm\vert \frac{1}{\varpi^{\lfloor k_0/n_0 \rfloor}} b\sigma(v) \Bigm\vert \frac{1}{\varpi^{\lfloor 2k_0/n_0 \rfloor}} (b\sigma)^2(v) \Bigm\vert \cdots \Bigm\vert \frac{1}{\varpi^{\lfloor (n-1)k_0/n_0 \rfloor}} (b\sigma)^{n-1}(v) \right)
\end{equation}
to be the $n \times n$ matrix whose $i$th column is $\frac{1}{\varpi^{\lfloor (i-1)k_0/n_0 \rfloor}} \cdot (b\sigma)^{i-1}(v)$ for $v \in V$. Observe that 
\begin{equation*}
g_b(v) = g_b^\rd(v) \cdot D_{k,n},
\end{equation*}
where $D_{k,n}$ is the diagonal matrix whose $(i,i)$th entry is $\varpi^{\lfloor (i-1)k_0/n_0\rfloor}$. 

\begin{Def}\label{d:conn comps}
For any basic $b$ (with $\kappa_{\GL_n}(b) = \kappa$) which is integrally $\sigma$-conjugate to the special representative as in Section \ref{sec:special_representatives}, we define 
\begin{align*}
\sL_{0,b}^{\adm} &\colonequals \left\{v \in \mathscr{L}_0 : \det g_b^{\rd}(v) \in \cO^\times \right\}, \\ 
\sL_{0,b}^{\adm, \rat} &\colonequals \left\{v \in \mathscr{L}_0 : \det g_b^{\rd}(v) \in \cO_K^\times \right\}, \\
\sL_{0,b}^{\adm,\rat,\dot w_0} &\colonequals \left\{v \in \mathscr L_0 : \sigma(\det g_b^{\rd}(v)) = \det(\dot w_0) \det(b)^{-1} \det g_b^{\rd}(v) \in \cO^\times\right\}.
% , & \dot X_{\dot w_r}^m(b)_{\sL_0} &\colonequals \sL_{0,b}^{\adm,\rat}/\dot \sim_{b,m,r},
\end{align*}
% Further, we define $\dot X_{\dot w_r}^m(b)_{\sL_0} \subseteq \dot X_{\dot w_r}^m(b)$ and $X_{\dot w_r}^m(b)_{\sL_0} \subseteq X_{\dot w_r}^m(b)$ as the image of $\sL_{0,b}^{\rm adm, rat}$ and $\sL_{0,b}^{\rm adm}$ under the maps in Theorem \ref{thm:structure_result}(ii).
As in \eqref{e:V_b w}, note that $\sL_{0,b}^{\adm,\rat,\dot w_0} \cong \sL_{0,b}^{\adm,\rat}$ by rescaling appropriately by an element of $\cO^\times$.
\end{Def}

Let now $b$ be the special representative with $\kappa_G(b) = \kappa$. As $G_{\cO} \subseteq \GL_n(\cO) = {\rm Stab}(\sL_0)$ inside $\GL_n(\breve K)$, we see that $\sL_{0,b}^{\rm adm, rat}, \, \sL_{0,b}^{\rm adm}$ are stable under $G_{\cO} \times T_w(\cO_K)$. We have 

\begin{equation}\label{eq:L0adm_explicit}
\sL_{b,0}^{\adm} = \left\{v = \sum_{i=1}^n a_i e_i \in \mathscr{L}_0 \colon \begin{gathered} \text{$a_i \in \cO$ for $1 \leq i \leq n$; $\{ a_i e_i \!\!\!\!\pmod \varpi\}_{i \equiv 1 \!\!\!\!\pmod{n_0}}$} \\ \text{generate the $\overline\FF_q$-vector space $\overline V$}  \end{gathered} \right\},
\end{equation}
where $\overline V$ is as in Section \ref{sec:reductive_quotient_of_JJb} (compare \cite[Lemma 4.8]{Viehmann_08}). For $x \in \sL_{b,0}^{\adm}$ with reduction-modulo-$\varpi$ $\overline x$, define $\overline{g_b}(\bar{x}) \in {\GL}_{n'}(\overline{\FF}_q)$ to be the matrix from reducing every entry of $g_b^{\rm red}(x)$ modulo $\varpi$ and deleting the $i$th row and $j$th column for all $(i,j) \not\equiv (1,1)$ modulo $n_0$.

\begin{lm}\label{lm:disjoint_union_Vbadm} 
We have a disjoint decomposition
\[
V_b^{\rm adm} = \coprod_{g \in G/G_\cO} g(\sL_{0,b}^{\rm adm})
\]
\end{lm}

\begin{proof}
Let $c,d$ be as in Lemma \ref{lm:defin_of_V}. On $V$ we have the operators considered in \cite[4.1]{Viehmann_08}: $b\sigma$, $\varpi(b\sigma)^{-1}$, $\varpi^d(b\sigma)^c$, $\varpi^{-\kappa_0}(b\sigma)^{n_0} = \sigma^{n_0}$ (in \cite{Viehmann_08}, these operators are denoted $F,V,\pi_1,\sigma_1$ respectively). For $v \in V_b^\adm$ we may consider the smallest $\cO$-submodule $P(v)$ of $V$, containing $v$ and stable under these four operators (this is a indeed lattice: for $\kappa>0$ see \cite[p. 354, paragraph before Lemma 4.10]{Viehmann_08}; for $\kappa = 0$, Lemma 6.1 shows that this is the lattice generated by $\{\sigma^i(v)\}_{i=0}^{n-1}$). Further we have
\[ 
\sL_{0,b}^{\rm adm} = \left\{ v \in \sL_0 \colon P(v) = \sL_0 \right\}.
\]
This follows from the explicit description of both sides: \eqref{eq:L0adm_explicit} above and \cite[Lemma 4.8 and beginning of \S4.4]{Viehmann_08}. 
Next, the set of lattices in $V$ stable under the four operators above is in bijection with $G/G_\cO$ via $gG_\cO \mapsto g\sL_0$. Indeed, for $\kappa > 0$ this follows from Lemma 4.8 and 4.10 of \cite{Viehmann_08} (note that there Viehmann only considers bi-infinitesimal $p$-divisible groups, so that the slope(s) of $V$ must be strictly between $0$ and $1$); and for $\kappa = 0$, this is clear from the description of the affine Grassmannian of $\GL_n$ in terms of lattices. 

Any $g \in G$ commutes with the four operators above, hence $P(gv) = gP(v)$ and hence $g(\sL_{0,b}^{\rm adm}) = \{v \in \mathscr{M} \colon P(v) = \mathscr{M}\}$ if $\mathscr{M} = g(\sL_0)$. Finally, $\bigcup_{\mathscr{M}} \{v \in \mathscr{M} \colon P(v) = \mathscr{M}\}$ (where $\mathscr{M}$ runs through all lattices stable under the four operators) is disjoint and equal to $V_b^\adm$.
\end{proof}

For each $h \in G/G_\cO$, let $X_{\dot w_r}^m(b)_{h(\sL_0)}$ be the subset of $X_{\dot w_r}^m(b)$ consisting of all points which under the isomorphism $X_{\dot w_r}^m(b)(\overline \FF_q) \cong V_b^\adm/\sim_{b,m,r}$ of Theorem \ref{thm:structure_result} correspond to $h(\sL_{0,b}^\adm)$.

The Iwahori level variety $X_{\dot w_r}^0(b)$ is known to be locally closed in the ind-scheme $\GL_n(\breve K)/I$ and locally of finite type over $\overline\FF_q$. The following result (as well as the preceeding lemma) is based on ideas from \cite{Viehmann_08}, which were explained to the authors by E. Viehmann.

\begin{prop}\label{p:conn comps}
Let $r > 0$. For any $h\in G$, $X_{\dot w_r}^0(b)_{h(\sL_0)}$ is a closed and open subset of $X_{\dot w_r}^0(b)$. In other words, $X_{\dot w_r}^0(b) = \coprod_{h \in G/G_\cO} X_{\dot w_r}^0(b)_{h(\sL_0)}$ is a scheme-theoretic disjoint union decomposition. 
\end{prop}
\begin{proof} The proof follows along the same lines as \cite[Lemma 4.16]{Viehmann_08}. First we show that $X_{\dot w_r}^0(b)_{\sL_0} \subseteq X_{\dot w_r}^0(b)$ is closed. To this end, recall that points of $\GL_n(\breve K)/I$ can be interpreted as complete chains of $\cO$-lattices in $\breve K^n$. Let $\sL_\bullet = \{\sL_0 \supseteq \sL_1 \supseteq \dots \supseteq \sL_{n-1} \supseteq \varpi\sL_0\}$ be the standard chain (stabilized by $I$). Then $gI$ corresponds to the lattice chain $g\sL_\bullet = \{g\sL_i\}_{i=0}^{n-1}$. We claim that there is an integer $C > 0$ only depending on $n,\kappa,r$ such that
\begin{equation}\label{eq:claim_closed_Xwrb0}
X_{\dot w_r}^0(b)_{\sL_0} = \{ \mathscr{M}_\bullet = (\mathscr{M}_i)_{i=0}^{n-1} \in X_{\dot w_r}^0(b) \colon \mathscr{M}_0 \subseteq \sL_0 \text{ and } \vol(\mathscr{M}_0) = C \}.
\end{equation}
Let $\mathscr{M}_\bullet \in X_{\dot w_r}^0(b)_{\sL_0}$. By construction, there is some $v \in \sL_{0,b}^{\rm adm}$ such that $\mathscr{M}_\bullet = g_{b,r}(v)\sL_\bullet$. Then $\mathscr{M}_0 = g_{b,r}(v)(\sL_0)$ is the lattice generated by $\{\varpi^{ri}(b\sigma)^i(v)\}_{i=0}^{n-1}$, which is contained in $\sL_0$. Moreover, 
\[
\vol(\mathscr{M}_0) = \vol(g_{b,r}(v)(\sL_0)) = \ord \det g_{b,r}(v) = C + \ord\det g_b^{\rm red}(v) = C,
\]
with $C > 0$ some explicit constant depending on $n,\kappa,r$. Conversely, let $\mathscr{M}_\bullet \in X_{\dot w_r}^0(b)$ be such that $\mathscr{M}_0 \subseteq \sL_0$ and $\vol(\mathscr{M}_0) = C$ with the same $C$ as above. By Theorem \ref{thm:structure_result}, there is some $v\in V_b^{\adm}$ with $\mathscr{M}_\bullet = g_{b,r}(v)\sL_\bullet$. Then $\mathscr{M}_0 = g_{b,r}(v)\sL_0$ is the lattice generated by $\{\varpi^{ri}(b\sigma)^{i}(v)\}_{i=0}^{n-1}$. In particular, we must have $v \in \mathscr{M}_0 \subseteq \sL_0$ and
\[
\ord\det(g_b^{\rm red}(v)) = \ord\det(g_{b,r}(v)) - C = \vol(\mathscr{M}_0) - C = 0.
\]
This proves claim \eqref{eq:claim_closed_Xwrb0}. From this claim it follows that $X_{\dot w_r}^0(b)_{\sL_0}$ is the intersection in $\GL_n(\breve K)/I$ of $X_{\dot w_r}^0(b)$ with the preimage under 
\[ 
\GL_n(\breve K)/I \rightarrow \GL_n(\breve K)/{\rm Stab (\sL_0)} \cong \{\text{$\cO$-lattices in $\breve K^n$} \}, \quad \mathscr{M}_\bullet = \{\mathscr{M_i}\}_{i=0}^{n-1} \mapsto \mathscr{M}_0 
\]
of the closed (see \cite[Remark 4.15]{Viehmann_08}) subset of lattices with fixed volume and contained in $\sL_0$. Hence $X_{\dot w_r}^0(b)_{\sL_0}$ is closed in $X_{\dot w_r}^0(b)$. 

Applying the $G$-action on $X_{\dot w_r}(b)$ we deduce that also $X_{\dot w_r}^0(b)_{h(\sL_0)}$ is closed in $X_{\dot w_r}^0(b)$ for each $h \in G$. Now the closed subvarieties $X_{\dot w_r}^0(b)_{h(\sL_0)}$ form a disjoint cover of $X_{\dot w_r}^0(b)$, as this holds on geometric points by Lemma \ref{lm:disjoint_union_Vbadm} and Theorem \ref{thm:structure_result}. As $X_{\dot w_r}^0(b)$ is locally of finite type, this disjoint union is locally finite. Hence $X_{\dot w_r}^0(b)_{h(\sL_0)}$ is also open.
\end{proof}

\begin{rem}
We explain the differences between our setting and that of \cite[\S4]{Viehmann_08}. Viehmann proved a similar decomposition for an open subset $\mathcal{S}_1$ of some minuscule affine Deligne--Lusztig varieties at the hyperspecial level (in particular, in the bi-infinitesimal case, which in our notation corresponds to $\kappa > 0$). A point of $\mathcal{S}_1$ corresponds to a Dieudonn\'e lattice $\mathscr{M}$ in $V$, that is a lattice stable by the operators $b\sigma$ and $p(b\sigma)^{-1}$. Such a lattice also possesses a single generator $v \in V_b^\adm$, but the difference between $\vol(\mathscr{M})$ and $\vol(P(\mathscr{M}))$ (where $P(\mathscr{M})$ is the smallest lattice containing $\mathscr{M}$ and stable under the four operators as in the proof of Lemma \ref{lm:disjoint_union_Vbadm}), is quite inexplicit, and it is a complicated task \cite[Theorem 4.11(ii)]{Viehmann_08} to show that this difference is constant. In our situation the lattice (chain) corresponding to a point of $X_{\dot w_r}^0(b)$ is completely explicit, and the difference of volumes is immediate to compute. Note also that we work with the cocharacter $(-r,\ldots, -r,(n-1)r)$ for $r \geq 1$, and since this is never minuscule, there is no direct comparison between our setting and the varieties from \cite{Viehmann_08}. \hfill $\Diamond$
%This seems to be related to the following two facts: (1) we work with $r \geq 1$ (in particular, the cocharacter $(-r,\dots,-r,(n-1)r)$ is never minuscule, so there is no direct comparison between our setting and the varieties from \cite{Viehmann_08}), and (2) we work in specialized cases of Iwahori (or higher) level, and our choices seem to have resulted in varieties which have simpler behavior than some hyperspecial cases.
%(2) we work in the Iwahori (or higher) levels, where the varieties we consider seem to have simpler behavior than at hyperspecial level.
\end{rem}

\begin{cor} \label{cor:conn cpts any good b}
Let $b \in \breve G$ be integrally $\sigma$-conjugate to the special representative attached to $\kappa$. Then the conclusion of Proposition \ref{p:conn comps} holds for this $b$. 
% Moreover, we have a similar statement for $\dot X_{\dot{w}_r}^m(b)$:
% \[
% \dot X_{\dot{w}_r}^m(b) = \bigsqcup_{g \in J_b(K)/\bJ_b(\cO_K)} g \cdot \dot X_{\dot w_r}^m(b)_{\sL_0}.
% \]
\end{cor}

\begin{proof}
If $h \in \GL_n(\cO)$ is such that $b = h^{-1}b_{sp}\sigma(h)$, where $b_{sp}$ is the special representative, then $g \mapsto h^{-1}g$ defines an isomorphism $X_{\dot w_r}^m(b_{sp}) \stackrel{\sim}{\longrightarrow} X_{\dot w_r}^m(b)$. Further, $g_b^{\rd}(v) = h^{-1}g_{b_{sp}}^{\rd}(hv)$ and the corollary follows from the commutativity of the obvious diagram.
\end{proof}

By Lemma \ref{lm:GLO conj}, Corollary \ref{cor:conn cpts any good b} applies to the Coxeter-type representatives from Section \ref{sec:coxeter_representatives}.

%******************************************************************************************************************************************************
%******************************************************************************************************************************************************

\subsection{The structure of $X_{\dot w}^m(b)$}\label{sec:scheme_structure_ADLV_Jb}

Let $b$ be the special representative with $\kappa_{\GL_n}(b) = \kappa$. Let $X_{\dot w_r}^m(b)_{\sL_0}$ be as defined before Proposition \ref{p:conn comps}. 
The following auxiliary elements of $\GL_n(\breve K)$ will be used in this subsection only. For $r\geq 1$, put $\mu_r = (1,r,2r,\dots,(n-1)r) \in X_{\ast}(T_{\rm diag})$. For an integer $a$, let $0\leq [a]_{n_0}<n_0$ denote its residue modulo $n_0$. Let $v_1 \in \GL_{n_0}(\breve K)$ be the permutation matrix whose only non-zero entries are concentrated in the entries $(1 + [(i-1)k_0]_{n_0}, i)$ and are all equal to $1$. Let $v \in {\rm GL}_n(\breve K)$ denote the block-diagonal matrix, whose diagonal $n_0 \times n_0$-blocks are each equal to $v_1$. We begin with the following key proposition.

\begin{prop}\label{prop:ADLV_contained_in_Schubert_cell}
For $r\geq 1$, the Iwahori level variety $X_{\dot w_r}^0(b)_{\sL_0}$ is contained in the Schubert cell $Iv D_{\kappa,n}\mu_rI/I \subseteq \GL_n(\breve K)/I$. In particular, $X_{\dot w_r}^m(b)_{\sL_0} \subseteq Iv D_{\kappa,n}\mu_r I/I^m \subseteq {\rm GL}_n(\breve K)/I^m$.
\end{prop}

\begin{proof}
We have to show that for $x \in \sL_0^{{\rm adm}}$ one has $Ig_{b,r}(x)I = IvD_{\kappa,n}\mu_rI$, i.e., that by successively multiplying by elements from $I$ on the left and right side we can bring $g_{b,r}(x) = g_b^{\red}(x)D_{\kappa,n}\mu_r$ to the form $vD_{\kappa,n}\mu_r$. For $1\leq i \leq n'$, we call a matrix in $\GL_n(\breve K)$ $i$-\emph{nice}, if the following two conditions hold: 
\begin{enumerate}[label=(\roman*),leftmargin=*]
\item
each of its $n^{\prime 2}$ blocks of size $n_0 \times n_0$ has the following shape: in its $\ell$th column ($1\leq \ell\leq n_0$), the entries above the $(1 + [(\ell-1)k_0]_{n_0}, \ell)$th entry lie in $\mathfrak p$ and the other entries lie in $\cO$; 
\item
for $1\leq \ell \leq n_0$, the $(1 + [(\ell-1)k_0]_{n_0}, \ell)$th entry of the $(i,i)$th $(n_0\times n_0)$-block lies in $\cO^{\times}$. 
\end{enumerate}

The inductive algorithm to prove the lemma is as follows: put $A_1 \colonequals g_b^{\red}(x)$ and let $1 \leq i \leq n'$. Assume that by modifying $g_b^{\red}(x)D_{\kappa,n}\mu_r$ (by multiplication from left and right with $I$) we have constructed the $i$-nice matrix $A_i$, such that $Ig_b^{\red}(x)D_{\kappa,n}\mu_rI = I A_iD_{\kappa,n}\mu_r I$ and such that the first $i-1$ rows and $i-1$ columns of $n_0 \times n_0$-blocks of $A_iD_{\kappa,n}\mu_r$ coincide with $vD_{\kappa,n}\mu_r$ up to $\cO^{\times}$-multiplies of the non-zero entries. Now we do the following steps:
\begin{itemize}
\item[(1)] 
\textit{Annihilate the entries of the $(i,i)$th $n_0 \times n_0$-block of $A_i$ lying over $(1 + [(\ell-1)k_0]_{n_0}, \ell)$th entry (for each $1 \leq \ell \leq n_0$).} 

By assumption, the $(1 + [(\ell-1)k_0]_{n_0}, \ell)$th entry lies in $\cO^{\times}$. By multiplying upper triangular unipotent elements from $I$ (with non-diagonal entries in $\mathfrak p$) from the left to $A_iD_{\kappa,n}\mu_r$ (i.e., performing elementary row operations on matrices), we obtain a nice matrix $A_i'$ (uniquely determined by $A_i$) whose entries have the same images in $\cO/\mathfrak p$ as those of $A_i$. Moreover, $IA_iD_{\kappa,n}\mu_rI = IA_i'D_{\kappa,n}\mu_rI$. 

%Do this by exploiting that $(1 + [(\ell-1)k_0]_{n_0}, \ell)$th entry lies in $\cO^{\times}$, by multiplying upper triangular unipotent elements from $I$ (with non-diagonal entries in $\mathfrak p$) from the left to $A_iD_{\kappa,n}\mu_r$ (i.e., performing elementary row operations on matrices).  As a result we obtain a (uniquely determined by $A_i$) nice matrix $A_i'$, whose entries have the same images in $\cO/\mathfrak p$ as those of $A_i$, and which  satisfies $IA_iD_{\kappa,n}\mu_rI = IA_i'D_{\kappa,n}\mu_rI$. 
\end{itemize}
Put $A_{i,0}' \colonequals A_i'$. For $\ell = 1, 2, \dots, n_0$ do successively the following step: 
\begin{itemize}
\item[(2)${}_{\ell}$] 
\textit{Annihilate the $(n_0(i-1)+\ell)$th column and $(n_0(i-1)+1 + [(\ell-1)k_0]_{n_0})$th row of $A_{i,\ell - 1}'$.}
%(``killing $(n_0(i-1)+\ell)$th column and $(n_0(i-1)+1 + [(\ell-1)k_0]_{n_0})$th row of $A_{i,\ell - 1}'$'')

By assumption, the $(n_0(i-1)+1 + [(\ell-1)k_0]_{n_0}, n_0(i-1)+\ell)$th entry of the $i$-nice matrix $A_{i,\ell-1}'$ lies in $\cO^{\times}$. By multiplying $A_{i,\ell-1}'D_{\kappa,n}\mu_r$ successively from the left  by lower triangular matrices from $I$ which have $1$'s on the main diagonal and only further non-zero entries in the $n_0(i-1)+1 + [(\ell-1)k_0]_{n_0}$th column, we can kill all entries of the $n_0(i-1) + \ell$th column of $A_{i,\ell -1}'$ except for the $(n_0(i-1)+1 + [(\ell-1)k_0]_{n_0}, n_0(i-1)+\ell)$th entry itself, which remains unchanged. After this we can, using the $(n_0(i-1)+1 + [(\ell-1)k_0]_{n_0}, n_0(i-1)+\ell)$th entry, easily eliminate all entries $n_0(i-1)+1 + [(\ell-1)k_0]_{n_0}$th row except for $(n_0(i-1)+1 + [(\ell-1)k_0]_{n_0}, n_0(i-1)+\ell)$th entry itself, which remains unchanged (by multiplying $A_{i,\ell-1}'D_{\kappa,n}\mu_r$ from the right with unipotent upper triangular matrices in $I$). This does not change the rest of the matrix, because $n_0(i-1) + \ell$th column contains precisely one non-zero entry.

\end{itemize}
As an output we obtain the matrix $A_{i+1} \colonequals A_{i,n_0}'$ which we claim is $(i+1)$-nice. Assume for now that this is true. Proceeding the described algorithm for all $1\leq i\leq n'$, we obtain the matrix $A_{n'+1}$, which differs from $v$ only by some diagonal matrix with entries in $\cO^{\times}$, so that $IA_{n'+1}D_{\kappa,n}\mu_rI = IvD_{\kappa,n}\mu_rI$ is now clear.

Observe that when looking modulo $\mathfrak p$, the step (2)${}_{\ell}$ in the algorithm   for a single $\ell$ affects the $(1 + [(\ell-1)k_0]_{n_0}, \ell)$th entry of the $(i+1,i+1)$th $n_0\times n_0$-block, but does not affect the entries $(1 + [(\ell'-1)k_0]_{n_0}, \ell')$th ($\forall \ell' \neq \ell$) of the same block. In particular, the steps (2)${}_{\ell}$ can be applied in any order of the $\ell$'s, and when the (2)${}_{\ell_0}$ is applied first to $A_i'$ (to kill its $(n_0(i-1)+\ell_0)$th column and $(n_0(i-1)+1 + [(\ell_0-1)k_0]_{n_0})$th row) giving the matrix $A_{i,\ell_0}''$, then the $(1 + [(\ell_0 -1)k_0]_{n_0}, \ell_0)$th entry of $(i+1,i+1)$th $n_0\times n_0$-block of $A''_{i,\ell_0}$ already coincides modulo $\mathfrak p$ with the same entry of $A_{i+1}$.

We now show that for $1 \leq i \leq n$, the matrix $A_i$ appearing in the algorithm is $i$-nice. (By induction we may assume that $A_{i'}$ is $i'$-nice for $1\leq i' < i$, which is sufficient to run the first $i-1$ steps of the algorithm to obtain $A_i$). For $1 \leq j \leq i' \leq n$, $1\leq \ell\leq n'$, let $\alpha_{i',j,\ell} \in \cO/\mathfrak p$ denote the residue modulo $\mathfrak p$ of the $(1 + [(\ell-1)k_0]_{n_0}, \ell)$th entry of the $(j,j)$th $n_0\times n_0$-block of $A_{i'}$. Note that $\alpha_{i',j,\ell} = \alpha_{i'',j,\ell}$ for all $1\leq j \leq i' \leq i''$. Indeed, if $j<i'$, this is obvious as the first $i'-1$ diagonal blocks of $A_{i'}$ and $A_{i''}$ coincide. If $j=i'$ observe that the $(1 + [(\ell-1)k_0]_{n_0}, \ell)$th entries (for all $1\leq \ell\leq n_0$) of the $(i',i')$th $n_0 \times n_0$-block of $A_{i'}$ can only be affected by step (1) of the algorithm, which does not change the residue modulo $\mathfrak p$.

Recall the image $\bar x = (\bar x_1, \dots, \bar x_{n'})^T$ of $x$ in $\overline{V}$ and the corresponding matrix $\overline{g_b}(\bar{x}) \in {\GL}_{n'}(\overline{\FF}_q)$ defined in Section \ref{sec:connected_components}. For $1\leq i \leq n'$, let $m_i \in \overline{\FF}_q$ denote the determinant of the upper left $i\times i$-minor of $\overline{g_b}(\bar{x})$. By Lemma \ref{lm:nonvanishing_minors}, $m_i \in \overline{\FF}_q^{\times}$ for all $i$. We claim that for $1\leq \ell\leq n_0$, 
\begin{equation}\label{eq:claim_for_residues}
\alpha_{i,j,\ell} = 
\begin{cases} 
\sigma^{\ell - 1}(m_1) & \text{if $j=1$} \\
\sigma^{\ell-1}(\frac{m_j}{m_{j-1}}) & \text{if $2 \leq j \leq i$} 
\end{cases}                                                                                                                                                                                                                                                                                                                                                                                                                                                                                                                                                                                                                                         \end{equation}
By induction we may assume that this holds for all $1 \leq i' < i$, from which \eqref{eq:claim_for_residues} follows for all $j<i$. It thus remains to compute $\alpha_{i,i,\ell}$. Note that for $1\leq \ell \leq n_0$, the $(1 + [(\ell-1)k_0]_{n_0}, \ell)$-entry of $A_1 = g_b^{\red}(x)$ is equal to is equal to $\sigma^{\ell-1}(x_{1,0}) = \sigma^{\ell-1}(\bar x_1)$. This finishes the case $i=1$. Assume $i \geq 2$ and fix some $1 \leq \ell \leq n_0$. By the observation above, $\alpha_{i,i,\ell}$ is equal to the residue modulo $\mathfrak p$ of the $(1 + [(\ell-1)k_0]_{n_0}, \ell)$th entry of the $(i,i)$th $n_0\times n_0$-block of the matrix $A_{i-1,\ell}''$, obtained from $A_{i-1}'$ by directly applying step (2)${}_{\ell}$. 

For $X \in \GL_n(\breve K)$, let $M(X)$ denote the $(n_0(i-1) + 1) \times (n_0(i-1) + 1)$-minor of $X$ obtained by removing from $X$ all columns with numbers $\{j \colon j > n_0(i-1) \text{ and } j \neq n_0(i-1) + \ell \}$ and all rows with numbers $\{s \colon s > n_0(i-1) \text{ and } s \neq n_0(i-1) + 1 + [(\ell-1)k_0]_{n_0} \}$. We compute:
\[
\alpha_{i,i,\ell}\prod_{\lambda = 1}^{n_0} \sigma^{\lambda - 1}(m_{i-1}) \equiv \det M(A_{i-1,\ell}'') = \det M(A_{i-1}') = \det M(g_b^{\red}(x)) \mod \mathfrak{p}.
\]
The first equality follows from the explicit form of $A_{i-1,\ell}''$ and by the induction hypothesis on the $\alpha_{i,j,\ell}$'s. The remaining equalities are true as every operation in the algorithm does not change the determinant of the matrices. On the other side, a simple calculation shows that
\[
\det M(g_b^{\red}(x)) \equiv \frac{\sigma^{\ell}(m_i)}{\sigma^{\ell}(m_{i-1})} \prod_{\lambda = 1}^{n_0} \sigma^{\lambda - 1}(m_{i-1}) \mod \mathfrak p.
\]
This finishes the proof of \eqref{eq:claim_for_residues}, and thus of the proposition.
\end{proof}

\begin{lm}\label{lm:nonvanishing_minors} Let $x \in \sL_{0,b}^{\rm adm}$ and let $\bar x \in \overline{V}$ denote its image. For $1\leq i \leq n'$, let $m_i$ denote the upper left $(i \times i)$-minor of $\overline{g_b}(\bar x) \in \GL_{n'}(\overline{\FF}_q)$. Then $m_i \in \overline{\FF}_q^{\times}$ for all $i$.
\end{lm}
\begin{proof}
Replacing $\FF_{q^{n_0}}$ by $\FF_q$ we may assume that $n_0 = 1$, $n' = n$. We have $\overline{g_b}(\bar x) = (\bar{x}_i^{q^{j-1}})_{1\leq i,j \leq n}$ and $\det \overline{g_b}(\bar x) \in \overline{\FF}_q^{\times}$. Clearly, $m_1 = \bar{x}_1 \neq 0$. Let $2 \leq i \leq n$. By induction we may assume that $m_{i'} \in \overline{\FF}_q^{\times}$ for all $1\leq i' < i$. Suppose $m_i = 0$. This means that the $i$ vectors $v_j = (x_j^{q^{k-1}})_{k=1}^i \in \overline{\FF}_q^i$ ($1\leq j \leq i$) are linearly $\overline{\FF}_q$-dependent. Note that the first $i-1$ of these vectors are $\overline{\FF}_q$-independent, as already the vectors $(x_j^{q^{k-1}})_{k=1}^{i-1} \in \overline{\FF}_q^{i-1}$ ($1\leq j\leq i-1$) are $\overline{\FF}_q$-independent, which in turn follows from the induction hypothesis $m_{i-1} \neq 0$. This shows that there exist $\lambda_1, \dots, \lambda_{i-1} \in \overline{\FF}_q$ with $\sum_{j = 1}^{i-1} \lambda_j v_j = v_i$. From this we deduce two systems of linear equations which uniquely determine the $\lambda_j$'s: (1) $\sum_{j=1}^{i-1} \lambda_j (x_j^{q^{k-1}})_{k=1}^{i-1} = (x_i^{q^{k-1}})_{k=1}^{i-1}$ as well as (2) $\sum_{j=1}^{i-1} \lambda_j (x_j^{q^{k-1}})_{k=2}^i = (x_i^{q^{k-1}})_{k=2}^i$. Note that (2) is obtained from (1) by raising all coefficients to the $q$th power. For $1\leq j\leq i-1$ let $m_{i-1}^{(j)}$ denote the minor $m_{i-1}$, in which $j$th row is replaced by $(x_i^{q^{k-1}})_{k=1}^{i-1}$. Then (1) gives $\lambda_j = m_{i-1}^{-1} m_{i-1}^{(j)}$, whereas (2) gives $\lambda_j = (m_{i-1}^{-1} m_{i-1}^{(j)})^q$ for each $1\leq j \leq i-1$. Thus $\lambda_j \in \FF_q$. This gives a non-trivial $\FF_q$-relation between the $x_1,\dots, x_i$, and hence also between the first $i$ rows of $\overline{g_b}(\bar x)$, i.e., $\det \overline{g_b}(\bar x) = 0$, contradicting the assumption.
\end{proof}

Let 
\begin{equation}\label{eq:Drinfeld_upper_half_space_def}
\Omega_{\FF_{q^{n_0}}}^{n'-1} \colonequals \mathbb{P}(\overline{V}) \sm \bigcup_{\substack{H \subseteq \overline{V} \\ \FF_{q^{n_0}}-\text{rational hyperplane} }} H
\end{equation}
be $n'-1$-dimensional Drinfeld's upper half-space over $\FF_{q^{n_0}}$.

\begin{thm}\label{thm:scheme_structure}
Let $b$ be the special representative with $\kappa_{\GL_n}(b) = \kappa$. Let $r > m \geq 0$. Then we have a decomposition of $\FF_q$-schemes 
\begin{equation*}
X_{\dot w_r}^m(b) \cong \bigsqcup_{G/G_{\cO}} \Omega_{\FF_{q^{n_0}}}^{n'-1} \times \mathbb{A},
\end{equation*}
where $\mathbb{A}$ is a finite dimensional affine space over $\FF_q$ (with dimension depending on $r,m$). The morphism $\dot X_{\dot{w}_r}^m(b) \rightarrow X_{\dot w_r}^m(b)$ is a finite \'etale $T_w(\cO_K/\varpi^{m+1})$-torsor. In particular, all these schemes are smooth.
\end{thm}

\begin{proof}
To prove the first statement of the theorem it suffices, by applying Proposition \ref{p:conn comps}, to show that $X_{\dot w_r}^m(b)_{\sL_0}$ is a locally closed subset of $\GL_n(\breve K)/I^m$ isomorphic to $\Omega_{\FF_{q^{n_0}}}^{n'-1} \times \mathbb{A}$. By Proposition \ref{prop:ADLV_contained_in_Schubert_cell}, $X_{\dot w_r}^m(b)_{\sL_0} \subseteq IvD_{\kappa, n}\mu_r(\varpi)I/I^m$. So, it suffices to show that as a subset of $IvD_{\kappa, n}\mu_r(\varpi) I/I^m$, $X_{\dot w_r}^m(b)_{\sL_0}$ is locally closed and with its induced reduced sub-scheme-structure isomorphic to $\Omega_{\FF_{q^{n_0}}}^{n'-1} \times \mathbb{A}$. For simplicity, we treat the case that $\kappa= 0$ or $1$, so that $v= D_{\kappa,n} = 1$. The general case is done in exactly the same way, but is slightly more technical due to the presence of the permutation matrix $v$ (the corresponding technical details are very similar to those appearing in the proof of Proposition \ref{prop:ADLV_contained_in_Schubert_cell}, which we proved in full generality).

First, by Lemma \ref{lm:param_of_Schuber_cell_preimage}, $I\mu_r(\varpi)I/I^m$ is isomorphic to an affine space and we fix the following coordinates on it: let $\alpha_{ij}$ ($1\leq i\neq j \leq n$) denote the root of $T_{\rm diag}$ corresponding to $(i,j)$th matrix entry. Then
\begin{equation}\label{eq:explicit_parametrization_of_a_Schubert_cell}
\psi \colon I \mu_r(\varpi) I/I^m \stackrel{\sim}{\rightarrow} C:= \underbrace{\prod_{i=2}^n L_{[0,(i-1)r)} U_{\alpha_{i1}}}_{A:=} \times \underbrace{\prod_{n \geq i > j > 1} L_{[0,(i-j)r)}U_{\alpha_{ij}}}_{B:=} \times I/I^m
\end{equation}
(the products can be taken in any fixed order; each factor -- including $I/I^m$ -- is an affine space over $\overline{\FF}_q$) is a parametrization of $I\mu_r(\varpi)I/I^m$, whose inverse sends $(a_{i1})_{i=2}^n, (a_{ij})_{i>j\neq 1}, g$ to $\prod_{i=2}^n a_{i1} \cdot \prod_{i,j} a_{ij} \cdot \dot \mu_r(\varpi) \cdot g I^m$. It suffices to show that there is an open subset $U \subseteq A$, isomorphic to $\Omega_{\FF_{q^{n_0}}}^{n'-1} \times \mathbb{A}$, such that $X_{\dot w_r}(b)_{\sL_0} \subseteq I \mu_r(\varpi) I/I^m$ is the graph of some morphism $f \colon U \rightarrow B \times I/I^m$. Indeed, then it follows that $X_{\dot w_r}(b)_{\sL_0}$ is a locally closed subset of $I \mu_r(\varpi) I/I^m$, which (endowed with the induced reduced sub-scheme structure) gets isomorphic to $U$ via the projection $p \colon C \cong A \times B \times I/I^m \rightarrow A$ to the first factor. First, we define $U$. Therefore, consider the natural projection
\begin{equation}\label{eq:projection_from_A_to_roots}
A \twoheadrightarrow \textstyle\prod\limits_{i=2}^{n'} L_{[0,1)}U_{\alpha_{1 + n_0(i-1),1}} \cong \{ [v] \in \mathbb{P}(\overline{V})  \colon v = \sum_i v_i e_{1 + n_0(i-1)} \in \overline{V}, v_1 \neq 0\}
\end{equation}
where the latter isomorphism is $(v_i)_{i=2}^{n'} \mapsto [1:v_2: \dots :v_{n'}]$ and let $U$ be the preimage of $\Omega_{\FF_{q^{n_0}}}^{n'-1}$, which is a subspace of the right hand side via \eqref{eq:Drinfeld_upper_half_space_def}. (If $n'=1$, $\overline{V}$ is one-dimensional, and the right hand side of \eqref{eq:projection_from_A_to_roots} as well as $\Omega_{\FF_{q^{n_0}}}^{n'-1}$ is a point). It is clear that $U \cong \Omega_{\FF_{q^{n_0}}}^{n'-1} \times \mathbb{A}$.

Next, we determine the image of $X_{\dot w_r}^m(b)_{\sL_0}$ under $\psi$. Let $g_{b,r}(x)I^m = g_b^{\red}(x)\mu_r(\varpi) I^m \in X_{\dot w_r}(b)_{\sL_0}$, i.e., $x \in \sL_0^{\rm adm}$. This point does not change if we multiply $g_{b,r}(x)$ by an element of $I^m$ from the right, or equivalently, if we multiply $g:= g_b^{\red}(x)$ by an element of 
\begin{equation}\label{eq:mur_conj_of_Im}
{}^{\mu_r(\varpi)}I^m = \mu_r(\varpi) I^m \mu_r(\varpi)^{-1} = \begin{pmatrix} 
\cO^{\times} & \mathfrak{p}^{m+1-r} & \cdots & \cdots & \mathfrak{p}^{m+1-(n-1)r} \\ 
\ast & \cO^{\times} & \mathfrak{p}^{m+1-r} & \cdots & \mathfrak{p}^{m+1 - (n-2)r} \\
\vdots&\ddots&\ddots&\ddots&\vdots \\
\ast &\cdots &\ast&\cO^{\times}&\mathfrak{p}^{m+1-r} \\
\ast & \cdots & \cdots  &\ast & \cO^{\times}\end{pmatrix}
\end{equation}
from the right (the entries marked with $\ast$ are uninteresting for us). First, we multiply $g$ by the diagonal matrix $i_1 = \diag(x_1, \sigma(x_1), \dots,\sigma^{n-1}(x_1))^{-1} \in {}^{\mu_r(\varpi)}I^m$ (note that $x_1$ invertible by \eqref{eq:L0adm_explicit}; see also Lemma \ref{lm:nonvanishing_minors}), achieving that all entries of $gi_1$ are of the form $\sigma^i(\frac{x_j}{x_1})$ resp. $\varpi\sigma^i(\frac{x_j}{x_1})$ and, in particular, the first column of $gi_1$ has the entries $1,\frac{x_2}{x_1}, \dots,\frac{x_n}{x_1}$. Next, using that $r > m$, so that all $\mathfrak{p}^{m+1-r}\supseteq \cO$ in \eqref{eq:mur_conj_of_Im}, we may eliminate all $(1,j)$th ($2\leq j \leq n$) entries of $gi_1$, i.e., we find some $i_1' \in {}^{\mu_r(\varpi)}I^m$ (with entries different from $0$ and $1$ only in the first row), such that $gi_1i_1'$ has the same first column as $gi_1$ and $0$'s in all entries of the first row except the first one. Also, all entries of $gi_1i_1'$ lie in $\cO$ and are some functions in $\sigma^i(\frac{x_j}{x_1})$ (as the same is true for $gi_1$). Moreover, the $(2,2)$ entry of $gi_1i_1'$ must now be in $\cO^\times$ (otherwise $g\mu_r(\varpi)I^m \not\in I\mu_r(\varpi)I/I^m$), and we can iterate the procedure: rescale the second column such that $(2,2)$th entry equals $1$, then kill all entries $(2,j)$ with $3\leq j \leq n$, etc. At the end we obtain a matrix $g' = gi_1i_1'i_2i_2' \dots i_{n-1}i_{n-1}'$, such that $g'\mu_r(\varpi)I^m = g\mu_r(\varpi)I^m$ and 
\[
g' = \begin{pmatrix}
      1 & 0 & \dots & \dots & \dots & 0 \\
      \frac{x_2}{x_1} & 1 & 0 & \dots & \dots & 0 \\
      \frac{x_3}{x_1} & \ast & 1 & \dots & \dots & 0 \\
      \vdots & \ddots & \ddots & \ddots & \ddots & \vdots\\
      \frac{x_{n-1}}{x_1} & \ast & \dots & \ast & 1 & 0 \\
      \frac{x_n}{x_1} & \ast & \dots & \ast & \ast & 1
     \end{pmatrix} = 
    \begin{pmatrix}
      1 & 0 & \dots & \dots & \dots & 0 \\
      \frac{x_2}{x_1} & 1 & 0 & \dots & \dots & 0 \\
      \frac{x_3}{x_1} & 0 & 1 & \dots & \dots & 0 \\
      \vdots & \ddots & \ddots & \ddots & \ddots & \vdots\\
      \frac{x_{n-1}}{x_1} & 0 & \dots & 0 & 1 & 0 \\
      \frac{x_n}{x_1} & 0 & \dots & 0 & 0 & 1
     \end{pmatrix} 
     \begin{pmatrix}
      1 & 0 & \dots & \dots & \dots & 0 \\
      0 & 1 & 0 & \dots & \dots & 0 \\
      0 & \ast & 1 & \dots & \dots & 0 \\
      \vdots & \ddots & \ddots & \ddots & \ddots & \vdots\\
      0 & \ast & \dots & \ast & 1 & 0 \\
      0 & \ast & \dots & \ast & \ast & 1
     \end{pmatrix} 
\]
where the entries marked by $\ast$ all lie in $\cO$ and are functions of $\frac{x_2}{x_1}$, $\dots$, $\frac{x_n}{x_1}$. We can regard the first matrix in the product as an element of $A$ (as in \eqref{eq:explicit_parametrization_of_a_Schubert_cell}), and moreover from \eqref{eq:L0adm_explicit} it follows that it lies in the open subset $U \subseteq A$. It gets clear now that with respect to the parametrization in \eqref{eq:explicit_parametrization_of_a_Schubert_cell}, $\psi(X_{\dot w_r}^m(b)_{\sL_0})$ consists of points of the form $(u,f_0(u), 1\cdot I^m)$ with $u \in U \subseteq A$ and $f_0(u) \in B$, where $f_0 \colon U \rightarrow B$ is some morphism (which determines the entries $\ast$ in terms of $\frac{x_2}{x_1}, \dots, \frac{x_n}{x_1}$). Thus $f \colon U \rightarrow B \times I/I^m$, defined by $u \mapsto (f_0(u), 1\cdot I^m)$ is the required morphism we wished to construct. This finishes the proof of the first claim of the theorem.

For the second claim in the theorem, we could repeat the above arguments with $I\mu_r(\varpi)I/\dot I^m$ replacing $I\mu_r(\varpi)I/I^m$. Alternatively we can argue as follows: by Theorem \ref{prop:loc_closed_ADLV_covers} (see Remark \ref{rem:lft_scheme_str_for_speical_wrm} and Lemma \ref{lm:action_of_wrdot_on_roots}), $\dot X_{\dot w_r}^m(b)$ is locally closed in $G(\breve K)/\dot I^m$, and hence a scheme locally of finite type over $\overline{\mathbb{F}}_q$ (Corollary \ref{cor:lftschemes}), and by Proposition \ref{prop:representability}(iii) the morphism $G(\breve K)/\dot I^m \rightarrow G(\breve K)/I^m$ is representable and has sections \'etale-locally. It thus follows that $\dot X_{\dot w_r}^m(b) \rightarrow X_{w_r}^m(b)$ is a honest morphism of schemes locally of finite type over $\overline{\mathbb{F}}_q$. Moreover, by the explicit description on geometric points in Theorem \ref{thm:structure_result} it is surjective and a torsor under $T_w(\cO_K/\varpi^{m+1})$.\qedhere

\end{proof}

\begin{cor}\label{cor:algebraicity_of_gbrx_maps}
Let $r' > m' > 0$, $r > m > 0$ be two pairs of integers with $r' \geq r$, $m' \geq m$. Then all maps $X_{\dot w_{r'}}^{m'}(b) \rightarrow X_{\dot w_r}^m(b)$, $\dot X_{\dot w_{r'}}^{m'}(b) \rightarrow \dot X_{\dot w_r}^m(b)$, $\dot X_{\dot w_{r'}}^{m'}(b) \rightarrow X_{\dot w_r}^m(b)$ induced by $g_{b,r'}(x) \mapsto g_{b,r}(x)$ are morphisms of schemes. In particular, $X_w^{\infty}(b)$ and $\dot X_w^{\infty}(b)$ are schemes over $\FF_q$.
\end{cor}
\begin{proof}
With respect to the coordinates on $X_{\dot w_r}^m(b)$, $\dot X_{\dot w_r}^m(b)$ in the proof of Theorem \ref{thm:scheme_structure}, these maps are simply induced by the natural projections $L_{[0,\nu')}U_{\alpha_{j,1}} \rightarrow L_{[0,\nu)}U_{\alpha_{j,1}}$ for $\nu' \geq \nu$ and $L_{[0,m'+1)}\bG_m \rightarrow L_{[0,m+1)}\bG_m$ for $m' \geq m$.
\end{proof}

We are now ready to endow all objects in the diagram in Theorem \ref{cor:infty_level_adlv} with scheme structures and compare them. The set $\sL_{0,b}^{{\rm adm, \rat}}$ has an obvious scheme structure as a closed subset of the infinite dimensional affine space $\sL_0$ over $\FF_q$. Analogously, the natural embedding $\sL_0^{\rm adm}/\cO^{\times} \subseteq \sL_0/\cO^{\times} = L^+\mathbb{P}(\sL_0)(\overline{\FF}_q)$, where $L^+\mathbb{P}(\sL_0)$ is an infinite-dimensional $\FF_q$-scheme, endows $\sL_0^{\rm adm}/\cO^{\times}$ with the structure of an open subscheme. We endow $V_b^{{\rm adm, rat}}$ and $V_b^{{\rm adm}}/\cO^{\times}$ with the scheme structure of a disjoint union: 
\begin{equation*}
V_b^{{\rm adm, rat}} = \bigsqcup_{g \in G/G_{\cO}} g.\sL_{0,b}^{{\rm adm,rat}} \qquad \text{and} \qquad V_b^{\adm}/\cO^{\times} = \bigsqcup_{g \in G/G_{\cO}} g.\left(\sL_{0,b}^{{\rm adm}}/\cO^{\times}\right).
\end{equation*}
Since the action of $\varpi$ on $V_b^{\adm}/\cO^{\times}$ just permutes the connected components, the quotient $V^{\adm}/\breve K^{\times}$ inherits the scheme structure $V^{\adm}/\breve K^{\times} = \bigsqcup_{G/Z(G)G_{\cO}} g.\left(\sL_{0,b}^{\adm}/\cO^{\times}\right)$.

\begin{cor}\label{cor:scheme_structure_on the_inverse_limit}
The maps of sets $V_b^{{\rm adm, rat}} \stackrel{\sim}{\rightarrow} \dot X_w^{\infty}(b)$, $V_b^{{\rm adm}}/\cO^{\times} \stackrel{\sim}{\rightarrow} X_w^{\infty}(b)$ from Theorem \ref{cor:infty_level_adlv} are isomorphisms of $\FF_q$-schemes. We endow $X_w^{DL}(b)$, $\dot X_w^{DL}(b)$ with the scheme structure via the isomorphisms in the diagram in Theorem \ref{cor:infty_level_adlv}.
\end{cor}
\begin{proof}
To show the first isomorphism, it suffices to prove that for the special representative $b$, we have an isomorphism of schemes $\sL_{0,b}^{{\rm adm, rat}} \stackrel{\sim}{\rightarrow} \dot X_w^{\infty}(b)_{\sL_0}$. Rescaling by an appropriate element of $\cO^\times$, we may replace $\sL_{0,b}^{\adm,\rat}$ by $\sL_{0,b}^{\adm,\rat,\dot w_0}$. With notation as in the proof of Theorem \ref{thm:scheme_structure}, the coordinates on the inverse limit $\dot X_w^{\infty}(b)_{\sL_0}$ are given by $((a_{\alpha_{j,1}})_{j=2}^n,c_1) \in L^+U_{\alpha_{j,1}} \times L^+\bG_m$ and the map $\sL_{0,b}^{\adm, \rat, \dot w_0} \to \dot X_w^\infty(b)_{\sL_0}$ is given by $(x_i)_{i=1}^n \mapsto (\frac{x_j}{x_1})_{j=2}^n, x_1$. This is an isomorphism. The second isomorphism is proven similarly.
\end{proof}

%******************************************************************************************************************************************************
%******************************************************************

\subsection{Example ${\rm SL}_2$, $w$ Coxeter, $b = 1$}\label{sec:example_SL2} 

It is instructive to explicate the scheme structure on $X_w^{DL}(1)$ from Remark \ref{rem:semiinfs_remarks}(ii) and compare it to the one obtained via affine Deligne--Lusztig varieties (a similar description applies in a number of further cases, in particular for ${\rm GL}_n$ or ${\rm GSp}_{2n}$ and $w$ Coxeter). We have $\SL_2/B = \mathbb{P}^1$ and
\[
X_w^{DL}(1) = \mathbb{P}^1(\breve K) \sm \mathbb{P}^1(K).
\]
It is thus given by the open condition $\det\left(\begin{smallmatrix} x & \sigma(x) \\ y & \sigma(y) \end{smallmatrix}\right) = x \sigma(y) - \sigma(x) y \not= 0$ inside
\[
L\mathbb{P}^1_K(\overline{\FF}_q) = L^+\mathbb{P}^1_{\mathcal{O}_K}(\overline{\FF}_q) = \mathbb{P}^1(\cO) = \{ [x:y] \colon x,y \in \cO, \text{at least one of $x,y$ lies in $\cO^{\times}$} \}
\]
(where $[x:y] = [x^{\prime}:y^{\prime}]$ if and only if there exists $a \in \cO^{\ast}$ with $ax = x^{\prime}$, $ay = y^{\prime}$) and
\[
X_w^{DL}(1)_r = \{ [x:y] \in L^+\mathbb{P}^1_{\mathcal{O}_K}(\overline{\FF}_q) \colon \sigma(x) y - x \sigma(y) \not\equiv 0 \mod \mathfrak{p}^r \}.
\]
It is clear (from the version of Theorem \ref{thm:structure_result} for $\SL_2$) that if $g \in \SL_2(K) \sm \SL_2(\cO)$, then $g.X_w^{DL}(1)_1 \cap X_w^{DL}(1)_1 = \varnothing$. Moreover, $X_w^{DL}(1)_1 \subseteq X_w^{DL}(1)$ is dense open. This means that $g$ maps a  dense open subset of $X_w^{DL}(1)$ onto a subset which lies in its boundary and hence cannot be dense. Thus $g$ cannot be an automorphism of the scheme $X_w^{DL}(1)$, and the action of $G$ on $X_w^{DL}(1)$ with the above scheme structure is not algebraic.

The subsets $Y_r$ of $L^+_r\mathbb{P}^1_{\mathcal{O}_K}(\overline{\FF}_q) = \mathbb{P}^1(\cO/\mathfrak{p}^r)$ can easily be computed to be
\begin{align*}
Y_1 &=  \Omega_{\FF_q}  \\
Y_2 &= (\Omega_{\FF_q} \times \mathbb{A}^1) \,\, \sqcup \,\, \bigsqcup_{\lambda \in \mathbb{P}^1(\FF_q)} \Omega_{\FF_q}, \\
& \;\; \vdots \\
Y_r &= (Y_{r-1} \times \mathbb{A}_{\FF_q}^1) \,\, \sqcup \,\, \bigsqcup \,\, \Omega_{\FF_q},
\end{align*}
where the last union is taken over all hyperspecial vertices in the Bruhat--Tits building of ${\rm SL}_2$ over $K$ such that the minimal gallery connecting this vertex to the one stabilized by $\SL_2(\mathcal{O}_K)$ has length $2r-1$. The unions are disjoint set-theoretically but not scheme-theoretically, since for example the preimage of $Y_1$ in $Y_2$ is open and not closed.

On the other hand, we can explicate the way in which $X_w^{DL}(1)$ is built from finite-dimensional pieces as dictated by Theorems \ref{thm:structure_result}, \ref{thm:scheme_structure}. In fact, $X_w^{DL}(1)$ is an inverse limit of the affine Deligne-Lusztig varieties of increasing level
\[
X_{\dot{w}_m}^m(1) \cong \bigsqcup_{\SL_2(K)/\SL_2(\mathcal{O}_K)} \Omega_{\FF_q}^1 \times \mathbb{A},
\]
where $\Omega_{\FF_q}^1 = \mathbb{P}^1_{\FF_q} \sm \mathbb{P}^1(\FF_q)$ is the Drinfeld upper half-plane over $\FF_q$, $\dot{w}_m$ are lifts of $w$ whose length in the affine Weyl group has to grow with $m$, and $\bA$ is some finite dimensional affine space over $\FF_q$, whose dimension depends on $m$ and $\dot{w}_m$ and goes to $\infty$ when $m \rightarrow \infty$.

%******************************************************************************************************************************************************
%******************************************************************

\section{A family of finite-type varieties $X_h$}\label{s:schemes_Xh}

In this section, we study the geometry of a family of finite-type varieties $X_h$ for $h \geq 1$ which have natural projection maps $X_h \to X_{h-1}$. These varieties are more tractable than (components of) the affine Deligne--Lusztig varieties $\dot X_{\dot w_r}^m(b)_{\sL_0}$, but we can see that after passing to the limit, these two families at infinite level are the same:
\begin{equation}\label{eq:two_prolimit_presentations}
\varprojlim_{r,m \colon r > m} \dot X_{\dot w_r}^m(b)_{\sL_0} = \sL_{0,b}^{\adm,\rat,\dot w_0} \cong \sL_{0,b}^{\adm,\rat} = \varprojlim_h X_h.
\end{equation}
Our work in this section will prepare us for Part \ref{part:cohomology}, where we will study the cohomology of $X_h$ as representations of $G_h \times T_h$.

We remark that $X_h$ will depend on whether we choose $b$ to be the Coxeter-type representative or the special representative, but they are isomorphic as $\FF_{q^n}$-schemes for the same reason as in Corollary \ref{cor:conn cpts any good b}. The flexibility of choosing this representative $b$ allows us to use a wide range of techniques to understand $X_h$ and its cohomology. We will see this theme throughout Part \ref{part:cohomology}.

\subsection{Ramified Witt vectors}\label{sec:witt}

Recall the schemes $\bW$, $\bW_h$ from Section \ref{sec:notation} (see \cite[18.6.13, 25.3.18]{Hazewinkel_78} for more details on the construction of ramified Witt vectors).
% If $K$ has positive characteristic, we let $\bW$ denote the ring scheme over $\FF_q$ where for any $\FF_q$-algebra $A$, $\bW(A) = A[\![\pi]\!]$. If $K$ has mixed characteristic, we let $\bW$ denote the $K$-ramified Witt ring scheme over $\FF_q$ so that $\bW(\FF_q) = \cO_K$ and $\bW(\overline \FF_q) = \cO$. Let $\bW_h = \bW/V^h \bW$ be the truncated  ring scheme, where $V \from \bW \to \bW$ is the Verschiebung morphism. For any $1 \leq r \leq h$, we write $\bW_h^r$ to denote the kernel of the natural projection $\bW_h \to \bW_r$.
We will need to coordinatize $\bW$ in order to make an explicit computations about the variety $X_h$. If $A$ is a perfect $\FF_q$-algebra, the elements of $\bW(A)$ can be written in the form $\sum_{i \geq 0} [x_i] \varpi^i$, where $[x_i]$ is the Teichm\"uller lift of $x_i \in A$ if $\Char K = 0$ and $[x_i] = x_i$ if $\Char K > 0$. (Note that the perfectness assumption is only necessary when $\Char K = 0$.) We identify $\bW$ with $\bA^{\bZ_{\geq 0}}$ and identify $\bW_h$ with $\bA^h$ under this choice of coordinates. We recall the following lemma about the ring structure of $\bW$ with respect to these coordinates.

\begin{lemma}\label{l:witt}
Let $A$ be a perfect $\FF_q$-algebra. 
\begin{enumerate}[label=(\roman*)]
\item
The coefficient of $\varpi^i$ in $(\sum_{i \geq 0} [a_i] \varpi^i) + (\sum_{i \geq 0} [b_i] \varpi^i)$ is
\begin{equation*}
[a_i + b_i + c_i], \qquad \text{where $c_i \in A[a_j^{1/q^N}, b_j^{1/q^N} : j < i, \, N \in \bZ_{\geq 0}]$.}
\end{equation*}
\item
The coefficient of $\varpi^i$ in $(\sum_{i \geq 0} [a_i] \varpi^i)(\sum_{i \geq 0} [b_i] \varpi^i)$ is 
\begin{equation*}
\left[\textstyle\sum\limits_{j = 0}^i a_j b_{i-j} + c_i \right], \qquad \text{where $c_i \in A[a_{i_1}^{e_1/q^N}b_{i_2}^{e_2/q^N} : i_1+i_2 < i, \, e_1, e_2, N \in \bZ_{\geq 0}]$}
\end{equation*}
\end{enumerate}
In both cases, we call $c_i$ the ``minor contribution.'' Note that if $\Char K > 0$, then the minor contribution is identically zero. In particular, for any given $i$, the $i$th minor contribution does only depend on $a_j, b_j$ with $j<i$.
\end{lemma}

This lemma says that up to ``minor contributions'', working in coordinates with the Witt vectors is the same as working in coordinates in $\FF_q[\![t]\!]$. This allows us to uniformly perform calculations in the mixed and equal characteristic settings. We will implicitly use Lemma \ref{l:witt} in Section \ref{sec:fibers_as_hypersurface} and Section \ref{s:cuspidality}. 

\begin{rem}\label{r:witt}
Note in particular that by Lemma \ref{l:witt}(ii), the coefficient of $\varpi^i$ in the product $(\sum_{i \geq 0}[a_i] \varpi^i)(\sum_{i \geq 0} [b_i] \varpi^i)$ is of the form $[a_0 b_i + e_i]$ where $e_i$ is independent of $b_i$. For this reason, the minor contributions never play a role in our formulae as we study $X_h$ as a subvariety of $X_{h-1} \times \bA^N$ (for some $N$), and so the minor contributions only contribute to unspecified ``constant'' terms (see $c$ in Proposition \ref{prop:polynomial_P_describing_the_fiber}). 

We also point out a case where the minor contribution vanishes (this is used in Proposition \ref{p:Wh fixed}). Let $A$ be a perfect $\FF_q$-algebra and pick an integer $h \geq 1$. Then the product $(1 + \sum_{i \geq h} [a_i] \varpi^i)(\sum_{i \geq 0} [b_i] \varpi^i) \in [b_0] + [b_1]\varpi + \dots + [b_{h-1}]\varpi^{h-1} + [a_h + a_0 b_h]\varpi^h + \varpi^{h+1} \bW(A)$. Indeed, it suffices to compute modulo $\varpi^{h+1}$ (that is, in the $(h+1)$-truncated ramified Witt vectors $\bW_{h+1}(A)$), where we have
\begin{align*}
(1 + \varpi^h[a_h])\cdot \sum_{i=0}^h [b_i]\varpi^i &= \sum_{i=0}^{h-1} [b_i]\varpi^i + \varpi^h([b_h] + [b_0 a_h]) = \sum_{i=0}^{h-1} [b_i]\varpi^i + [b_h + b_0 a_h]\varpi^h.
\end{align*}
\end{rem}

\subsection{The scheme $X_h$} \label{sec:dim_smooth}

Fix a $0 \leq \kappa < n$ and let $b$ be either the Coxeter-type or special representative with $\kappa_G(b) = \kappa$ as in Section \ref{sec:two b}. Define the $\cO$-submodule of $\sL_0$,
\[
\sL_0^{(h)} \colonequals \bigoplus_{\substack{1 \leq i \leq n \\  i \equiv 1 \!\!\!\! \pmod{n_0}}} \varpi^h\sL_0 \oplus \bigoplus_{\substack{1 \leq i \leq n \\  i \not\equiv 1 \!\!\!\! \pmod{n_0}}} \varpi^{h-1}\sL_0.
\]
%It is stable under $(b\sigma)^{k_0}\varpi^{-n_0}$. 
Under the conventions set in Section \ref{sec:witt}, any $x \in \sL_0/\sL_0^{(h)}$ can be written as 
\begin{equation}\label{eq:affine_coordinates_for_Xh}
x = \sum_{\substack{1\leq i \leq n \\ i \equiv 1 \!\!\!\! \pmod{n_0}}} \sum_{\ell=0}^{h-1} [x_{i,\ell}]\varpi^{\ell} e_i + \sum_{\substack{1\leq i \leq n \\  i \not\equiv 1 \!\!\!\!\pmod{n_0}}} \sum_{\ell=0}^{h-2} [x_{i,\ell}]\varpi^{\ell} e_i \qquad (x_{i,\ell} \in \overline{\mathbb{F}}_q).
\end{equation}
This identifies $\sL_0/\sL_0^{(h)}$ with $\bA^{n(h-1) + n'}_{\FF_{q^n}}$. Observe that if $b$ is Coxeter-type, then although $\sL_0^{(h)}$ is stable under $(b\sigma)^{n_0}\varpi^{-k_0}$, the $\FF_{q^{n_0}}$-rational structure given by this Frobenius on $\sL_0/\sL_0^{(h)}$ does not agree with the $\FF_{q^{n_0}}$-rational structure on $\bA^{n(h-1)+n'}$ given by the standard $\FF_{q^{n_0}}$-Frobenius.

\begin{definition}\label{def:Xh}
For $h \geq 1$, define 
\begin{equation*}
X_h(\overline \FF_q) \colonequals  \sL_{0,b}^{\adm,\rat} / \sL_0^{(h)} = \text{image of $\sL_{0,b}^{\adm, \rat}$ in $\sL_0/\sL_0^{(h)}$}
\end{equation*}
and let $X_h \subset \bA^{n(h-1)+n'}$ be the $\FF_{q^n}$-subscheme whose $\FF_{q^n}$-rational structure comes from the standard $\FF_{q^n}$-Frobenius on $\bA^{n(h-1)+n'}$.
\end{definition}

As $\det(g_b^{\rd}(\cdot)) \colon X_h \rightarrow (\mathcal{O}_K/\varpi^h)^{\times}$ is a morphism onto a discrete scheme, we have the scheme-theoretic disjoint decomposition
\begin{equation}\label{eq:dec_of_Xh_conn_cpts}
X_h = \bigsqcup_{a \in (\mathcal{O}_K/\varpi^h)^{\times}} g_a.X_h^{\det \equiv 1},
\end{equation}
where $X_h^{\det \equiv 1}$ consists of all $x \in X_h$ with $\det g_b^{\rd}(x) \equiv 1 \pmod{\varpi^h}$, and $g_a \in G_h$ is any matrix with determinant $a$.

\begin{prop}\label{p:dim Xh}\mbox{}
$X_h$ is a smooth affine scheme of dimension $(n-1)(h-1) + (n'-1)$. 
\end{prop}

\begin{proof}
The proof is very similar to that of \cite[Proposition 3.10]{Chan_siDL}. Choose $b$ to be the Coxeter-type or special representative. 
% Recall that $X_h$ is a subvariety of $\bA^{(n-1)h + n'-1}$ with coordinates indexed by $x_i \in \bW_h$ if $i \equiv 1 \pmod{n_0}$ and $x_i \in \bW_{h-1}$ if $i \not\equiv 1 \pmod{n_0}$. 
Write $\det(g_b^{\rm red}(x_1, \ldots, x_n)) = [g_s]_{0 \leq s \leq h-1}$. It is enough to prove the assertions for the open and closed subset $X_h^{\det \equiv 1}$, which is defined by the equations $g_0 = 1$ and $g_s = 0$ for $1 \leq s \leq h-1$.

To prove that $X_h^{\det \equiv 1}$ is a smooth affine scheme of dimension $(n-1)(h-1)+n'-1$, it suffices to show that for any point $X_h^{\det \equiv 1}$, there exists a nonsingular $h \times h$ submatrix of the Jacobian matrix $J$. First let $g_b^{\rm red}(x_1, \ldots, x_n) \in X_h^{\det \equiv 1}$. Then for some $x_r$, the determinant of the $(n-1) \times (n-1)$ minor obtained by deleting the $1$st column and the $r$th row is nonzero modulo $\varpi$---denote by $d[x_r]$ the reduction of this determinant. From \eqref{d:g_b red}, observe that $x_{r,i}$ only contributes to $g_s$ if $i \leq s$, 
\begin{equation*}
g_s = 
d[x_r] x_{r,s} + (\text{terms w/ $q$th powers of $x_{r,s}$, and $x_{i,j}$ for $(i,j) \neq (r,s)$}).
\end{equation*}
Reorder the rows of $J$ so that the first $h$ rows correspond to the coordinates $x_{r,0}, \ldots, x_{r,h-1}$ of $x_r \in \bW_h$. Since we are working in characteristic $p$, we have 
\begin{equation*}
\frac{\partial g_s}{\partial x_{1,i}} = \begin{cases}
d[x_r] \neq 0 & \text{if $i = s$,} \\
0 & \text{if $i > s$,} \\
\text{?} & \text{if $i < s$.}
\end{cases}
\end{equation*}
This submatrix of $J$ is an upper triangular matrix with nonzero determinant. Hence we have shown that $X_h^{\det \equiv 1}$ is a smooth complete intersection of dimension $(n-1)(h-1) + n'-1$.
\end{proof}

%******************************************************************************************************************************************************
%******************************************************************************************************************************************************

\subsection{Relation to classical Deligne--Lusztig varieties} \label{sec:projection_to_reductive_quotient}

Recall that for $\overline V = \sL_0/\sL_0^{(1)}$ we have that $\bG_1 = {\rm Res}_{\FF_{q^{n_0}}/\FF_q}\GL(\overline V)$ (see Section \ref{sec:definition of GGh}, \ref{sec:reductive_quotient_of_JJb}). The scheme $X_1$ is a classical Deligne-Lusztig variety corresponding to the maximal nonsplit torus $\FF_{q^n}^\times$ in $\bG_1(\FF_q) = \GL_{n'}(\FF_{q^{n_0}})$. We get a commutative diagram
\begin{equation*}
\begin{tikzcd}
\sL_0 \ar[twoheadrightarrow]{r} & \sL_0/\sL_0^{(h)} \ar[twoheadrightarrow]{r} & \overline{V} \ar[hookleftarrow]{r} & \overline{V} \smallsetminus \{0\} \ar[twoheadrightarrow]{r} & \mathbb P(\overline{V}) \\
\sL_{0,b}^{\adm, \rat} \ar[hookrightarrow]{u} \ar[twoheadrightarrow]{r} & X_h \ar[hookrightarrow]{u} \ar[twoheadrightarrow]{rr} && X_1 \ar[hookrightarrow]{u} \ar[twoheadrightarrow]{r} & \Omega_{\overline{V}} \ar[hookrightarrow]{u}
\end{tikzcd}
\end{equation*}
where $\Omega_{\overline{V}}$ is isomorphic to the Drinfeld upper half-space $\mathbb{P}(\overline{V}) \sm \mathbb{P}(\overline{V})(\FF_{q^{n_0}})$ and $X_1$ is a $\FF_{q^n}^\times$-torsor over $\Omega_{\overline V}$. (If $b$ is the special representative, $\Omega_{\overline V}$ is literally the Drinfeld upper half-space.)

For $v \in \overline{V}$ define $\overline{g_b}(v)$ to be the $(n' \times n')$-matrix whose $i$th column is $\overline{\sigma_b}^{i-1}(v)$ (written with respect to the basis $\{e_i\}_{i \equiv 1 \pmod{n_0}}$ of $\overline{V}$ from Lemma \ref{lm:defin_of_V}). Then  
\[
X_1 = \{ v \in \overline{V} \colon \det \overline{g_b}(v) \in \FF_{q^{n_0}}^{\times} \}.
\]
%Then $X_{1,V}$ is isomorphic to the classical Deligne-Lusztig variety attached to the reductive quotient of $\bJ_b(\mathcal{O})$, whose image in $\mathbb{P}(V)$ is just the Drinfeld upper half space $\Omega_V$ (with respect to the $\mathbb{F}_{q^{n_0}}$-structure on $V$). 

\begin{ex}\label{ex:classical DL}
If $\kappa = 0$, then $\overline{V} = \sL_0/\varpi\sL_0$, $\overline{\sigma_b}^{i-1} = b\sigma$ and $X_1$ is the Deligne-Lusztig variety for $\GL_n(\FF_q)$ associated to the maximal nonsplit torus $\FF_{q^n}^\times$. If $\kappa,n$ are coprime, then $\overline{V}$ is one-dimensional and $X_1$ is a finite set of points and can be identified with $\FF_{q^n}^\times$.
\end{ex}

%******************************************************************************************************************************************************
%******************************************************************************************************************************************************

\subsection{The projection $X_h \rightarrow X_{h-1}$ and its fibers} \label{sec:fibers_as_hypersurface}

%We will several times use the explicit structure of the fibers of the natural projection $X_h \rightarrow X_{h-1}$. In practice it will suffice to work with the (slightly simplier) fibers of $X_h^{\det \equiv 1} \rightarrow X_{h-1}^{\det \equiv 1}$.

Let $h \geq 2$. We will actually work with an intermediate scheme: $X_h \twoheadrightarrow X_{h-1}^+ \twoheadrightarrow X_{h-1}$. By Sections \ref{sec:witt}, \ref{sec:dim_smooth}, the quotient $\sL_0/\varpi^{h-1} \sL_0$ can be identified with the affine space $\bA^{n(h-1)}$. Define $X_{h-1}^+$ to be the $\FF_{q^n}$-subscheme of $\bA^{n(h-1)}$ defined by
\begin{equation*}
X_{h-1}^+(\overline \FF_q) \colonequals \sL_{0,b}^{\adm, \rat}/\varpi^{h-1} \sL_0 = \text{image of $\sL_{0,b}^{\adm,\rat}$ in $\sL_0/\varpi^{h-1} \sL_0$}.
\end{equation*}
Observe that 
\begin{equation}\label{e:+}
X_{h-1}^+ = X_{h-1} \times \bA^{n-n'},
\end{equation}
since the coordinates $x_{i,h-2}$ for $i \not\equiv 1 \pmod{n_0}$ do not contribute to $\det(g_b^{\red}(x))$ modulo $\varpi^{h-1}$. Furthermore, $X_h$ is a closed subscheme of $X_{h-1}^+ \times \bA^{n'}$, and under this embedding
\begin{equation*}
X_h \hookrightarrow X_{h-1}^+ \times \bA^{n'},
\end{equation*}
we may write $x = (\widetilde x, x_{1,h-1}, x_{n_0+1,h-1}, \ldots, x_{n_0(n'-1)+1,h-1})$ for $x \in X_h$ and its image $\tilde x \in X_{h-1}^+$.
More precisely, we have the following technical proposition, which will be used in Section \ref{s:cuspidality}.

\begin{prop}\label{prop:polynomial_P_describing_the_fiber}
Let $h \geq 2$.
\begin{enumerate}[label=(\roman*)]
\item
$X_h$ is the closed subscheme of $X_{h-1}^+ \times \bA^{n'}$ cut out by the polynomial
\begin{equation*}
P \colonequals P_0^q - P_0,
\end{equation*}
where $[P_0]$ is the coefficient of $\varpi^{h-1}$ in $\det(g_b^{\red}(\cdot))$.
\item
Let $b$ be the special representative. Then
\begin{equation*}
P_0(x) = 
c(\widetilde x) + \sum_{i=0}^{n_0 - 1} P_1(x)^{q^i}
\end{equation*} 
where $\sum_{i=0}^{n_0 - 1} P_1^{q^i}$ exactly consists of all terms of $P_0$ that depend on the coordinates $x_{1,h-1}, x_{n_0 + 1, h-1}, \ldots, x_{n_0(n'-1)+1,h-1}$ and $c$ is a morphism $X_{h-1}^+ \to \bA^1$. In particular, $X_h$ is the closed subscheme of $X_{h-1}^+ \times \bA^{n'}$ cut out by the equation
\begin{equation*}
P_1^{q^{n_0}} - P_1 = c - c^q.
\end{equation*}
\item
Let $b$ be the special representative. Explicitly, the polynomial in (ii) is given by
\begin{equation*}
P_1 = 
\sum_{1 \leq i,j \leq n^{\prime}} m_{ji} x_{1 + n_0(i-1),h-1}^{q^{(j-1)n_0}},
\end{equation*}
where $m \colonequals (m_{ji})_{j,i}$ is the adjoint matrix of $\overline{g_b}(\bar{x})$ and $\bar x$ denotes the image of $x$ in $\overline{V} = \sL_0/\sL_0^{(1)}$. Explicitly, $m \cdot \overline{g_b}(\bar{x}) = \det\overline{g_b}(\bar{x}) \cdot 1_{n'}$ and the $(j,i)$th entry of $m$ is $(-1)^{i+j}$ times the determinant of the $(n^{\prime}-1)\times (n^{\prime}-1)$ matrix obtained from $\overline{g_b}(\bar{x})$ by deleting the $i$th row and $j$th column.
\end{enumerate}
\end{prop}

\begin{proof}
An explicit calculation shows that $P_0 = c + \sum_{i=0}^{n_0 - 1} \sigma^i(P_1)$, with $P_1$ as claimed if $b$ is the special representative. Note that in the mixed characteristic setting, we use Lemma \ref{l:witt}(ii) to see that minor contributions only appear in the $c(\widetilde x)$ term. From this the proposition easily follows.
\end{proof}

%******************************************************************************************
%******************************************************************************************

\subsection{Level compatibility on the cohomology of $X_h$}\label{s:level compat}

\begin{prop}\label{p:Wh fixed}
Let $h \geq 2$. The action of $\ker(T_h \rightarrow T_{h-1}) = \bW_h^{h-1}(\FF_{q^n})$ on $X_h$ preserves each fiber of the map $X_h \to X_{h-1}$, the induced morphism $X_h/\bW_h^{h-1}(\FF_{q^n}) \to X_{h-1}$ is smooth, and each of its fibers is isomorphic to $\bA^{n-1}$.
\end{prop}

\begin{proof}
Let $b$ be the special representative and let $x \in X_h$ be coordinatized as in Section \ref{sec:dim_smooth}. Then $x_{i,0} \neq 0$ for $i \equiv 1 \pmod{n_0}$. By (a slight variant of) Proposition \ref{prop:polynomial_P_describing_the_fiber}, $X_h$ is the closed subscheme of $X_{h-1} \times \bA^n$ given by $P = 0$, where $\bA^n$ has the coordinates $\{y_i\}_{i=1, \ldots, n}$, where $y_i = x_{i,h-1}$ if $i \equiv 1 \pmod{n_0}$ and $y_i = x_{i,h-2}$ if $i \not\equiv 1 \pmod{n_0}$. Note that the natural $\bW_h^{h-1}(\FF_{q^n})$-action on $X_h$ extends to the action on $X_{h-1} \times \bA^n$ over $X_{h-1}$ given by  \[
1 + [\lambda]\varpi^{h-1} \colon \begin{cases} x_{i,h-1} \mapsto x_{i,h-1} + x_{i,0} \lambda &\text{if $i \equiv 1 \pmod{n_0}$,} \\ x_{i,h-2} \mapsto x_{i,h-2} & \text{otherwise,}\end{cases}
\] 
where $\lambda \in \FF_{q^n}$. (Note that we use Remark \ref{r:witt} here.) Consider the morphism
\begin{equation*}
f \from X_{h-1} \times \bA^n \to X_{h-1} \times \bA^n, \qquad y_i \mapsto \begin{cases}
\left(\frac{y_1}{x_{1,0}}\right)^{q^n} - \frac{y_1}{x_{1,0}} & \text{if $i = 1$,} \\
y_i - \frac{x_{i,0} y_1}{x_{1,0}} & \text{if $i > 1$, $i \equiv 1 \pmod{n_0}$,} \\
y_i & \text{if $i \not\equiv 1 \pmod{n_0}$.}
\end{cases}
\end{equation*}
This morphism factors through the surjection $X_{h-1} \times \bA^n \to X_{h-1} \times \bA^n/\bW_h^{h-1}(\FF_{q^n})$ so that it is a composition
\begin{equation*}
X_{h-1} \times \bA^n \to X_{h-1} \times \bA^n/\bW_h^{h-1}(\FF_{q^n}) \stackrel{\sim}{\to} X_{h-1} \times \bA^n,
\end{equation*}
where the second map must in fact be an isomorphism. Since $\bW_h^{h-1}(\FF_{q^n})$ is a $p$-group, \cite[Proposition 3.6]{Chan_DLII} implies that $P((y_i)_{i=1,\ldots, n}) = P'(f(y_i)_{i=1,\ldots,n})$ for some $P' \from X_{h-1} \times \bA^n \to \bA^1$. Now $X_h/\bW_h^{h-1}(\FF_{q^n})$ is the closed subscheme of $X_h \times \bA^n/\bW_h^{h-1}(\FF_{q^n})$ defined by $P' = 0$. We therefore have a commutative diagram
\begin{equation}\label{diag:factor_out_Wh_action}
\begin{tikzcd}
X_{h-1} \times \bA^n \ar{r} & X_{h-1} \times \bA^n / \bW_h^{h-1}(\FF_{q^n}) \ar{r}{\sim} & X_{h-1} \times \bA^n \\
X_h \ar{r} \ar[hookrightarrow]{u} & X_h / \bW_h^{h-1}(\FF_{q^n}) \ar{r}{\sim} \ar[hookrightarrow]{u} & \{P^{\prime} = 0 \} \ar[hookrightarrow]{u},
\end{tikzcd}
\end{equation}
Since $P$ is a degree-$q^n$ polynomial in $x_{1,h-1}$, we know that $P^{\prime}$ must be at most degree one in $y_1$. A calculation shows that the coefficient of $y_1$ is the function on $X_{h-1}$ given by $x \mapsto \det\overline{g_b}(\bar{x})$, where $\bar{x}$ is the image of $x \in X_{h-1}$ in $X_1$ (notation as in Section \ref{sec:projection_to_reductive_quotient}). This function is constant on connected components of $X_{h-1}$, taking values in $\FF_q^{\times}$. In particular, the coefficient of $y_1$ in $P'$ over any point in $X_{h-1}$ is nonzero, so it follows that each fiber of $X_h/\bW_h^{h-1}(\FF_{q^n}) \to X_{h-1}$ is isomorphic to $\bA^{n-1}$.
\end{proof}

\begin{cor}\label{cor:level_compatibility_of_cohomology} There is a natural isomorphism
\begin{equation*}
H_c^i(X_h, \overline \QQ_\ell)^{\bW_h^{h-1}(\FF_{q^n})} \cong H_c^{i + 2(n-1)}(X_{h-1}, \overline \QQ_\ell)(n-1),
\end{equation*}
where $(n-1)$ denotes the Tate twist. 
\end{cor}
This corollary allows to define a \textit{direct} limit of the \textit{homology} groups for $X_h$ (see Section \ref{s:inf ADLV}).

%******************************************************************************************************************************************************
%******************************************************************************************************************************************************

\subsection{$X_h$ as a subscheme of $\bG_h$}
\label{s:Xh def} % $\sL_{0,b}^{\adm, \rat}/\varpi^h$ as a subscheme of $\bG_h$

Let $b$ be a Coxeter-type representative. Let $1 \leq l \leq n$ be an integer satisfying $e_{\kappa,n} l \equiv 1 \pmod n$.  For $x = \sum_{i=1}^n x_i e_i \in \sL_0$ where $x_i \in \cO$, define
\begin{equation*}
\lambda(x) \colonequals \sum_{i=1}^n \frac{1}{\varpi^{\lfloor k_0(i-1)/n_0 \rfloor}} \cdot b^{i-1} \cdot D(x_i),
\end{equation*}
where $D(a) = \diag(a, \sigma^l(a), \ldots, \sigma^{[(n-1)l]}(a))$. 
%The matrix $\lambda(x)$ is strongly related to $g_b^{\rd}(x)$. 
Let $\gamma$ be the inverse of the permutation of the set $\{1, 2, \dots, n\}$ defined by $1 \mapsto 1$ and $i \mapsto [(i-1)e_{\kappa,n}] + 1$ for $2 \leq i \leq n$. Let $\gamma \in \GL_n(K)$ also denote the matrix given by $\gamma(e_i) = e_{\gamma(i)}$. 

\begin{lm}\label{lm:lambda_vs_gbred}
We have 
\[
\lambda(x) = g_b^{\rd}(\gamma^{-1}(x)) \cdot \gamma.
\]
In particular, $\det\lambda(x) = \det g_b^{\rd}(x)$. Moreover, we have $\gamma b_0^{e_{\kappa,n}}\gamma^{-1} = b_0$.
\end{lm}
\begin{proof}
This is a direct computation.
\end{proof}

\begin{ex}
\begin{enumerate}[label=(\roman*)]
\item For $n = 3$, $\kappa = e_{\kappa,n} = 1$, we have $b = \left(\begin{smallmatrix} 0 & 0 & \varpi \\ 1 & 0 & 0 \\ 0 & 1 & 0 \end{smallmatrix}\right)$ and for $x = \left(\begin{smallmatrix} x_1\\ x_2 \\ x_3\end{smallmatrix}\right)$,
\[
\lambda(x) = g_b^{\rd}(x) = g_b(x) = \left(\begin{smallmatrix} x_1 & \varpi \sigma(x_3) & \varpi\sigma^2(x_2) \\ x_2 & \sigma(x_1) & \varpi\sigma^2(x_3) \\ x_3 & \sigma(x_2) & \sigma^2(x_1) \end{smallmatrix}\right).
\]
We have $F(\lambda(x)) \neq \lambda(\sigma(x))$. Thus $\lambda$ is not an $\FF_q$-morphism.
\item For $n=3$, $\kappa = 2$, $e_{\kappa,n} = 2$, we have $b = \left(\begin{smallmatrix} 0 & \varpi & 0 \\ 0 & 0 & \varpi \\ 1 & 0 & 0 \end{smallmatrix}\right)$ and 
\[
g_b^{\rd}(x) = \left(\begin{smallmatrix} x_1 & \varpi \sigma(x_2) & \varpi\sigma^2(x_3) \\ x_2 & \varpi\sigma(x_3) & \sigma^2(x_1) \\ x_3 & \sigma(x_1) & \sigma^2(x_2) \end{smallmatrix}\right) \qquad \text{ and } \qquad \lambda(x) = \left(\begin{smallmatrix} x_1 & \varpi\sigma^2(x_2) & \varpi \sigma(x_3) \\ x_3 & \sigma^2(x_1) & \varpi \sigma(x_2) \\ x_2 & \sigma^2(x_3) & \sigma(x_1) \end{smallmatrix}\right) \in \breve G_{{\bf{x}},0}.
\]
\end{enumerate}
\end{ex}

\begin{proposition-definition}\label{prop:lambda}
The assignment $\lambda$ defines an embedding,
\[
\sL_0 \hookrightarrow M_{n^{\prime}}(\cO_{D_{k_0/n_0}}),
\]
which restricts to
\begin{equation*}
\lambda \colon \sL_{0,b}^{\rm adm} \hookrightarrow \breve G_{\bx,0},
\end{equation*}
Moreover, $\det(\lambda(x)) \in \cO_K^\times$ if and only if $x \in \sL_{b,0}^{\adm,\rat}$. The reduction modulo $\varpi^{h}$ of $\lambda$ induces an $\mathbb{F}_{q^n}$-rational embedding
\[
\sL_{0,b}^{\rm adm, \rm rat}/\sL_0^{(h)} = X_h \hookrightarrow \bG_h.
\]
We denote its image again by $X_h$. This is an $\FF_{q^n}$-subscheme of $\bG_h$.
\end{proposition-definition}

\begin{proof} It is easy to see that $\lambda(e_i) \in M_{n^{\prime}}(\cO_{D_{k_0/n_0}})$ for $i = 1, \dots, n$. This implies that $\lambda(\sL_0) \subseteq M_{n^{\prime}}(\cO_{D_{k_0/n_0}})$. By Lemma \ref{lm:lambda_vs_gbred} it is immediate that $\det(\lambda(x)) \in \cO^{\times}$ if and only if $\det(g_b^{\rm red}(x)) \in \cO^{\times}$ and similarly $\det(\lambda(x)) \in \cO_K^{\times}$ if and only if $\det(g_b^{\rm red}(x)) \in \cO_K^{\times}$. Finally, note that $\lambda$ is a $\FF_{q^n}$-morphism since $\lambda(\sigma^n(x)) =  \sigma^n(\lambda(x)) = F^n(\lambda(x))$.
\end{proof}

%By Lemma \ref{lm:lambda_vs_gbred} $X_h^{\det \equiv 1}\subseteq X_h$ maps under $\lambda$ to the subset of all $\lambda(x)$ with $\det \lambda(x) \equiv 1 \mod \varpi^h$. 

The natural $(\breve G_{\bx,0}^F \times \cO_L^\times)$-action on $\sL_{0,b}^{\rm adm, \rm rat}$ induces a left action of $(G_h \times T_h)$-action on $X_h \subseteq \bG_h$, given by left-multiplication by $G_h = \bG_h(\FF_q)$ and right-multiplication by $T_h = \bT_h(\FF_q)$:
\begin{equation*}
(g,t) \cdot x \colonequals g x t, \qquad \text{for $g \in G_h, \, t \in T_h, \, x \in X_h$.}
\end{equation*}

\subsection{Relation to Deligne--Lusztig varieties for finite rings}\label{sec:rel_Xh_to_DL_Lusztig}
Let $b$ be the Coxeter-type representative.
The following proposition gives a description of $X_h$ reminiscent of Deligne--Lusztig varieties for reductive groups over finite rings \cite{Lusztig_04,Stasinski_09}. Let $U_{up}$ and $U_{low}$ denote the $\breve K$-subgroups of upper and lower triangular unipotent matrices in $J_b$. Consider the unipotent radicals $U \colonequals \gamma^{-1}U_{up}\gamma$, $U^- = \gamma^{-1}U_{low}\gamma$ of opposite Borels (over $\breve K$) in $J_b$ containing the diagonal torus $T$. Let $\bU$ and $\bU^-$ denote the smooth subgroup schemes of $\bG$ whose $\overline{\FF}_q$-points are $U(\breve K) \cap \breve G_{{\bf x},0}$ and $U^-(\breve K) \cap \breve G_{{\bf x},0}$, and let $\bU_h$ and $\bU^-_h$ be the corresponding subgroups of $\bG_h$.

\begin{proposition}\label{p:Xh Ah-}
The subgroup $\bU_h^- \cap F(\bU_h) \subseteq \bG_h$ consists of matrices with $1$'s along the main diagonal and $0$'s outside the first column. 
We have
\begin{align*}
X_h(\overline \FF_q) 
&= \left\{g \in \bG_h(\overline \FF_q) : g^{-1} F(g) \in \bU_h^- \cap F(\bU_h) \right\} \\
&= \left\{g \in \bG_h(\overline \FF_q) : g^{-1} F(g) \in \bU_h^-\right\}/(\bU_h^- \cap F^{-1}(\bU_h^-)).
\end{align*}
\end{proposition}

\begin{proof}
Using $\gamma b_0^{e_{\kappa,n}} \gamma^{-1} = b_0$ and $\gamma t_{\kappa,n} \gamma^{-1} = t_{\kappa,n}$ from Lemma \ref{lm:lambda_vs_gbred}, we compute
\begin{align*}
U^- \cap F(U) &= \gamma^{-1} U_{low} \gamma \cap F(\gamma^{-1}U_{up}\gamma) = 
\gamma^{-1} U_{low} \gamma \cap b_0^{e_{\kappa,n}}t_{\kappa,n} \gamma^{-1}U_{up} t_{\kappa,n}^{-1} b_0^{-e_{\kappa,n}} \gamma \\
&= \gamma^{-1} (U_{low} \cap b_0 U_{up} b_0^{-1} ) \gamma
\end{align*}
and (using $\gamma(e_1) = e_1$) the claim about $\bU_h^- \cap F(\bU_h)$ follows easily.
For any $a \in \bW_h(\overline \FF_q)$, we have that
\begin{equation*}
F(\diag(a, \sigma^{[l]}(a), \ldots, \sigma^{[(n-1)l]}(a))) = \diag(\sigma^n(a), \sigma^{[l]}(a), \ldots, \sigma^{[(n-1)l]}(a)).
\end{equation*}
Thus for any $v = (v_i)_{i=1}^n$ with $v_i \in \bW_h(\overline \FF_q)$ ($i \equiv 1 \pmod{n_0}$) and $v_i \in \bW_{h-1}(\overline \FF_q)$ ($i \not\equiv 1 \pmod{n_0}$),
\begin{equation*}
F(\lambda(x)) = \sum_{i=1}^n \frac{1}{\varpi^{\lfloor k_0(i-1)/n_0 \rfloor}} \cdot b^{i-1} \cdot \diag(\sigma^n(x_i), \sigma^{[l]}(x_i), \ldots, \sigma^{[(n-1)l]}(x_i))
\end{equation*}
differs from $\lambda(x)$ in only the first column. Thus for $x \in \sL_0$, we see that $\lambda(x)^{-1}F(\lambda(x))$ can differ from an element of $\bU_h^- \cap F(\bU_h)$ only in the left upper entry, and this entry is equal to $\det(\lambda(x)^{-1}F(\lambda(x))) = \det(g_b^{\rd}(x)^{-1}\sigma(g_b(x)))$ (Lemma \ref{lm:lambda_vs_gbred}). Now for $x \in \sL_0^{\adm, \rat}$, $\det g_b^{\rd}(x) \in \cO_K^{\times}$. This proves
\begin{equation*}
X_h \subset \left\{g \in \bG_h(\overline \FF_q) : g^{-1} F(g) \in \bU_h^- \cap F(\bU_h)\right\}.
\end{equation*}
To see the other inclusion, observe that if $F(g) = g \cdot u$ for some $u \in \bU_h^- \cap F(\bU_h)$, then comparing the $j$th column for $j \geq 2$ shows that $g$ must necessarily be of the form $\lambda(v)$ for some $v \in \sL_0^{\adm}$. The determinant condition then follows from $\det(u) = 1$. The last equality in the proposition follows from Lemma \ref{lm:F_conj_bijective}.
\end{proof}

\begin{lm}\label{lm:F_conj_bijective}
The morphism
\begin{equation*}
(\bU^-_h \cap F^{-1} \bU^-_h) \times (\bU^-_h \cap F \bU_h) \to \bU^-_h, \qquad (x,g) \mapsto x^{-1} g F(x).
\end{equation*}
is an isomorphism.
\end{lm}

\begin{proof}
We can consider the $\overline{\FF}_q$-scheme $\gamma \bG \gamma^{-1}$, whose $\overline{\FF}_q$-points are $\gamma \bG(\overline{\FF}_q) \gamma^{-1}$, together with a Frobenius isomorphism
\[
F_0 \colon \gamma \bG \gamma^{-1} \stackrel{\sim}{\rightarrow} \gamma \bG \gamma^{-1}, \quad F_0(x) = b_0 (\gamma t_{\kappa,n}\gamma^{-1}) \sigma(x) (b_0(\gamma t_{\kappa,n}\gamma^{-1}))^{-1}
\]
By Lemma \ref{lm:lambda_vs_gbred}, $\gamma b_0^{e_{\kappa,n}} t_{\kappa,n} \gamma^{-1} = b_0(\gamma t_{\kappa,n}\gamma^{-1})$. Thus if $c_{\gamma} \colon \bG \stackrel{\sim}{\rightarrow} \gamma \bG \gamma^{-1}$, $x \mapsto \gamma x \gamma^{-1}$ denotes the conjugation by $\gamma$, we have $c_{\gamma} \circ F = F_0 \circ c_{\gamma}$ (this in particular shows that $F_0$ is an isomorphism).

We will first show that $(\bU^- \cap F^{-1} \bU^-) \times (\bU^- \cap F \bU) \to \bU^-$, $(x,g) \mapsto x^{-1} g F(x)$ is bijective. We have $\bU^-(\overline{\FF}_q) = \gamma^{-1}(\breve U_{low} \cap \gamma \bG(\overline{\FF}_q) \gamma^{-1})\gamma$ and $\bU(\overline{\FF}_q) = \gamma^{-1}(\breve U_{up} \cap \gamma \bG(\overline{\FF}_q) \gamma^{-1})\gamma$. Applying $c_{\gamma}$, we thus have to show that the map 
\begin{align}\label{eq:integral_version_surjectivity}
\nonumber \big(( \breve U_{low} \cap \gamma \bG(\overline{\FF}_q) \gamma^{-1}) \cap F_0(\breve U_{low} \cap \gamma \bG(\overline{\FF}_q) \gamma^{-1}) \big) & \times \big((\breve U_{low} \cap \gamma \bG(\overline{\FF}_q) \gamma^{-1}) \cap F_0(\breve U_{up} \cap \gamma \bG(\overline{\FF}_q) \gamma^{-1})\big) \\
&\qquad\qquad\rightarrow \breve U_{low} \cap \gamma \bG(\overline{\FF}_q) \gamma^{-1}, 
\end{align}
$(x,g) \mapsto x^{-1}gF_0(x)$ is bijective. 
We first show that the following is an isomorphism:
\begin{equation}\label{e:std bij}
(\breve U_{low} \cap b_0^{-1} \breve U_{low} b_0) \times (\breve U_{low} \cap b_0 \breve U_{up} b_0^{-1}) \to \breve U_{low}, \qquad (x, g) \mapsto x^{-1} g F_0(x).
\end{equation}
To do this, it is equivalent to prove that given any $A \in \breve U_{low}$, there exists a unique element $(x,g) \in (\breve U_{low} \cap b_0^{-1} \breve U_{low} b_0) \times (\breve U_{low} \cap b_0 \breve U_{up} b_0^{-1})$ such that $x A = g F_0(x)$. We now compute explicitly and write
\begin{equation*}
x = \left(\begin{matrix}
1 & 0 & 0 & \cdots & \cdots & 0 \\
b_{21} & 1 & 0 & \cdots & \cdots & 0 \\
b_{31} & b_{32} & 1 & \ddots  &  & \vdots \\
\vdots &  & \ddots & \ddots & 0 & \vdots \\
b_{n-1,1} & b_{n-1,2} & \cdots & b_{n-1,n-2} & 1 & 0 \\
0 & \cdots & \cdots & 0 & 0 & 1
\end{matrix}\right), \qquad
g = \left(\begin{matrix}
1 & 0 & 0 & \cdots & 0 \\
c_1 & 1 & 0 & \cdots & 0 \\
c_2 & 0 & 1 & \ddots & \vdots \\
\vdots & \vdots & \ddots & \ddots & 0 \\
c_{n-1} & 0 & \cdots & 0 & 1
\end{matrix}\right).
\end{equation*}
Let $\gamma t_{\kappa,n} \gamma^{-1} = \diag(t_1, t_2, \ldots, t_n)$ so that we have
\begin{equation*}
b_0 t_{\kappa,n} \sigma(x) t_{\kappa,n}^{-1} b_0^{-1}
= 
\left(\begin{matrix}
1 & 0 & 0 & 0 & \cdots & 0 \\
0 & 1 & 0 & 0 & \cdots & 0 \\
0 & \sigma(b_{21}) t_2/t_1 & 1 & 0 &  & 0 \\
0 & \sigma(b_{31}) t_3/t_1 & \sigma(b_{32}) t_3/t_2 & 1 & \ddots &\vdots \\
\vdots & \vdots & \ddots & \ddots & 1 & 0 \\
0 & \sigma(b_{n-1,1}) t_{n-1}/t_1 & \sigma(b_{n-1,2}) t_{n-1}/t_2 & \cdots & \sigma(b_{n-1,n-2}) t_{n-1}/t_{n-2} & 1
\end{matrix}\right).
\end{equation*}
We therefore see that the $(i,j)$th entry of $g F_0(x)$ is
\begin{equation}\label{e:RHS ij}
(g F_0(x))_{i,j} = 
\begin{cases}
1 & \text{if $i = j$,} \\
0 & \text{if $i < j$,} \\
c_{i-1} & \text{if $i > j = 1$,} \\
\sigma(b_{i-1,j-1}) t_{i-1}/t_{j-1} & \text{if $i > j > 1$.}
\end{cases}.
\end{equation}
We also compute the $(i,j)$th entry of $xA$ when $A = (a_{i,j})_{i,j} \in \breve U_{low}$:
\begin{equation}\label{e:LHS ij}
(xA)_{i,j} = 
\begin{cases}
1 & \text{if $i = j$,} \\
0 & \text{if $i < j$,} \\
b_{ij} + \sum_{k=j+1}^{i-1} b_{ik} a_{kj} + a_{ij} & \text{if $j < i \leq n-1$,} \\
a_{nj} & \text{if $j < i = n$.} 
\end{cases}
\end{equation}
We now have $n^2$ equations given by \eqref{e:RHS ij} $=$ \eqref{e:LHS ij}, viewed as equations in the variables $b_{ij}$ and $c_i$. First look at the equations corresponding to $(n,2), (n,3), \ldots, (n,n-1).$ This gives
\begin{equation*}
\sigma(b_{n-1,j-1}) t_{n-1}/t_{j-1} = a_{nj} \qquad \text{for $1 < j < n$.}
\end{equation*}
which uniquely determines $b_{n-1,1}, b_{n-1,2}, \ldots, b_{n-1,n-2}$. Proceding inductively, let $1<i\leq n-1$, and suppose that all $b_{i',j}$ for all $i'\geq i$ and $1 < j \leq i'$ are uniquely determined. Then look at the equations corresponding to $(i-1,2), (i-1,3), \ldots, (i-1,i-2)$. This gives
\begin{equation}\label{eq:recursive_formula_for_bij}
\sigma(b_{i-1,j-1}) t_{i-1}/t_{j-1} = b_{i,j} + \sum_{k=j+1}^{i-1} b_{i,k} a_{kj} + a_{i,j} \qquad \text{for $1 < j < i$,} 
\end{equation}
which uniquely determines $b_{i-1,1}, b_{n-2,2}, \ldots, b_{i-1,i-2}$. This uniquely determines $x$. Finally, by looking at the equations corresponding to $(2,1), (3,1), \ldots, (n,1)$, it is immediately clear that the $c_i$'s are also uniquely determined, so $g$ is as well. This shows the isomorphism \eqref{e:std bij}. 

% 
% First look at the equations corresponding to $(n,2), (n,3), \ldots, (n,n-1).$ This gives
% \begin{equation*}
% \sigma(b_{n-1,j-1}) t_{n-1}/t_{j-1} = a_{nj} \qquad \text{for $1 < j < n$.}
% \end{equation*}
% which uniquely determines $b_{n-1,1}, b_{n-1,2}, \ldots, b_{n-1,n-2}$. Now look at the equations corresponding to $(n-1,2), (n-1,3), \ldots, (n-1,n-2)$. This gives
% \begin{equation*}
% \sigma(b_{n-2,j-1}) t_{n-2}/t_{j-1} = b_{n-1,j} + \sum_{k=j+1}^{n-2} b_{n-1,k} a_{kj} + a_{n-1,j} \qquad \text{for $1 < j < n-1$,}
% \end{equation*}
% which uniquely determines $b_{n-2,1}, b_{n-2,2}, \ldots, b_{n-2,n-3}$. Proceeding, we have now uniquely determined $x$. Finally, by looking at the equations corresponding to $(2,1), (3,1), \ldots, (n,1)$, it is immediately clear that the $c_i$'s are also uniquely determined, so $g$ is as well. This shows the isomorphism \eqref{e:std bij}. 

Now we deduce \eqref{eq:integral_version_surjectivity} from this. Using the same notation as above, assume that $A \in \breve U_{low} \cap \gamma \bG(\overline{\FF}_q)\gamma^{-1}$. Let $\tau_i := \ord(t_i)$ ($1\leq i \leq n$) and $\lambda_{i,j}$ denote the minimum of valuations of all elements of $\gamma \bG(\overline{\FF}_q) \gamma^{-1} \cap \breve U_{\alpha_{i,j}}$, where $\breve U_{\alpha_{i,j}}$ is the root subgroup corresponding to the $(i,j)$th entry ($1 \leq i\neq j \leq n$). Then $\tau_i,\lambda_{ij} \in \{0, 1\}$ for all $i,j$. Moreover, the fact that $F_0$ is an isomorphism shows 
\begin{equation}\label{eq:lambda_tau_connection}
\lambda_{i,j} = \lambda_{i-1,j-1} + \tau_{i-1} - \tau_{j-1}.
\end{equation}
To establish \eqref{eq:integral_version_surjectivity}, we have to show that for all $2\leq j < i \leq n$, we have $\ord(b_{i-1,j-1}) \geq \lambda_{i-1,j-1}$ and $\ord(c_{i-1}) \geq \lambda_{i-1,1}$. 

We first prove the assertion about the $b$'s. As in the proof of \eqref{e:std bij} above, we may proceed inductively on $i$: assuming that the assertion holds for all $i' > i$, we will show that the assertion holds for $i$. (The basic induction step $i=n$ follows from the same argument as below.) Observe that if $\tau_{i-1} = 0$ and $\tau_{j-1} = 1$, then we are done by formula \eqref{eq:recursive_formula_for_bij}. 

Assume that $\tau_{i-1} = \tau_{j-1}$. If $\lambda_{i-1,j-1} = 0$ then by \eqref{eq:recursive_formula_for_bij} there is again nothing to show. Thus we may assume $\lambda_{i-1,j-1} = 1$. By \eqref{eq:recursive_formula_for_bij} we have to check that $\lambda_{i,j} = 1$ and that for each $j+1 \leq k \leq i-1$, either $\lambda_{i,k}=1$ or $\lambda_{k,j}=1$. First, $\lambda_{i,j} = 1$ follows from \eqref{eq:lambda_tau_connection}. Second, $\alpha_{i,k} + \alpha_{k,j} = \alpha_{i,j}$ ($\alpha_{i,j}$ is the root of the diagonal torus of $\GL_n$ corresponding to $(i,j)$th entry). Thus the fact that $\gamma\bG(\overline{\FF}_q)\gamma^{-1}$ is a group implies that $\lambda_{i,k} + \lambda_{k,j} \geq \lambda_{i,j}$ (for all $k$), so $\lambda_{i,k} = 1$ or $\lambda_{k,j} = 1$.

Finally, assume that $\tau_{i-1} = 1$, $\tau_{j-1} = 0$. Then \eqref{eq:lambda_tau_connection} implies  $\lambda_{i-1,j-1} = 0$ and $\lambda_{i,j} = 1$. Then by \eqref{eq:recursive_formula_for_bij} we have to show that $\lambda_{i,j} = 1$ (which we already know) and that for each $j+1\leq k \leq i-1$, we have $\lambda_{i,k} = 1$ or $\lambda_{k,j} = 1$ (which holds for the same reason as above). This completes the proof of the assertion about the $b$'s.

Analogously, one proves the assertion about the $c_{i-1}$'s. Since \eqref{eq:recursive_formula_for_bij} and the equations corresponding to $(2,1), (3,1), \ldots, (n,1)$ uniquely determine the $b$'s and the $c$'s, this establishes bijectivity of \eqref{eq:integral_version_surjectivity}. 

To finish the proof of the lemma, it suffices to check that if $A_1, A_2 \in \gamma \bG(\overline{\FF}_q) \gamma^{-1}$ differ by some element in the normal subgroup $\gamma \ker(\bG(\overline{\FF}_q) \rightarrow \bG_h(\overline{\FF}_q)) \gamma^{-1}$, then the corresponding pairs $(x_1,g_1)$ and $(x_2,g_2)$ with $x_i^{-1}g_iF_0(x_i) = A_i$ ($i=1,2$) satisfy $x_1^{-1}x_2, g_1^{-1}g_2 \in \ker(\bG(\overline{\FF}_q) \rightarrow \bG_h(\overline{\FF}_q))$. Let $\lambda_{i,j}^h \in \{h-1,h\}$ be the smallest possible valuation of an element in $\gamma \ker(\bG(\overline{\FF}_q) \rightarrow \bG_h(\overline{\FF}_q)) \gamma^{-1} \cap \breve U_{\alpha_{i,j}}$. As $F_0$ induces an isomorphism of $\bG_h$, we again have a formula
\begin{equation*}
\lambda_{i,j}^h = \lambda_{i-1,j-1}^h + \tau_{i-1} - \tau_{j-1}.
\end{equation*}
We can once again proceed inductively to deduce that the $b_{i-1,j-1}$ and $c_{i-1}$ are uniquely determined as elements in $\mathfrak{p}^{\lambda_{i-1,j-1}}/\mathfrak{p}^{\lambda_{i-1,j-1}^h}$ by the elements $a_{i,j} \in \mathfrak{p}^{\lambda_{i,j}}/\mathfrak{p}^{\lambda_{i,j}^h}$. 
\end{proof}

\newpage

\part{Alternating sum of cohomology of $X_h$} \label{part:cohomology}

In this part, we study the virtual $G_h$-representations
\begin{equation*}
R_{T_h}^{G_h}(\theta) \colonequals \sum_{i \geq 0} (-1)^i H_c^i(X_h, \overline \QQ_\ell)[\theta],
\end{equation*}
where $\theta$ is a character of $T_h \cong \cO_L^\times/U_L^h = \bW_h^\times(\FF_{q^n})$. 

In Section \ref{s:finite rings}, we prove that if $\theta$ is \textit{primitive}, then $R_{T_h}^{G_h}(\theta)$ is (up to a sign) an irreducible $G_h$-representation (Theorem \ref{t:alt sum Xh}). Our strategy is to extend ideas of Lusztig, who proves the analogous result in the context of division algebras \cite{Lusztig_79} and split groups \cite{Lusztig_04} (see \cite{Stasinski_09} for the mixed characteristic analogue). This is done in Section \ref{sec:lusztig_lemma}. We note that the main result there, Proposition \ref{prop:alternating_sum_for_Sigma}, is more general than Theorem \ref{t:alt sum Xh} in that it works for any Frobenius $F$ on $\bG_h$ and the $F$-fixed points of any $F$-stable maximal torus in $\bG_h$. For example, if we take $F$ to be the twisted Frobenius coming from the Coxeter-type representative, then the $F$-fixed points of the diagonal torus forms the group $\bW_h^\times(\FF_{q^n})$, which exactly gives Theorem \ref{t:alt sum Xh}. On the other hand, if we take $F$ to be the twisted Frobenius coming from the special representative, then the $F$-fixed points of the diagonal torus forms the $n'$-fold product of $\bW_h^\times(\FF_{q^{n_0}})$, which corresponds to the maximally split unramified torus in $G_h$.

In Section \ref{s:finite rings}, we also give a character formula for $R_{T_h}^{G_h}(\theta)$ on certain elements of $T_h$ (Proposition \ref{p:trace alt sum Xh}) and give a geometric interpretation of determinant-twisting on the cohomology groups (Lemma \ref{l:twist compat}). Keeping in mind the remarks in the preceding paragraph, the methods in Section \ref{s:finite rings} primarily use the Coxeter-type representative $b$ (Section \ref{sec:two b}).

In Section \ref{s:cuspidality}, we prove an analogue of a cuspidality result for $R_{T_h}^{G_h}(\theta)$ when $\theta$ is primitive (Theorem \ref{t:Gh cuspidal}). To do this, we perform a character calculation using the geometry of $X_h$. Our approach is a (far-reaching) generalization of the proof in \cite{Ivanov_15_ADLV_GL2_unram} in the special case $G = \GL_2(K)$. We use the special representative $b$ (Section \ref{sec:two b}) as $F$-stable parahoric subgroups are more well-behaved for this choice. We note that although there is no notion of cuspidality for $G_h$-representations, we will see later that Theorem \ref{t:Gh cuspidal} implies the supercuspidality of the corresponding $G$-representation (Theorem \ref{t:RTG irred}). 

\section{Deligne--Lusztig varieties for Moy--Prasad quotients for $\GL_n$} \label{s:finite rings}

We say that a character $\theta \from T_h \cong \bW_h^\times(\FF_{q^n}) \to \overline \QQ_\ell^\times$ is \textit{primitive} if the restriction of $\theta$ to $\bW_h^{h-1}(\FF_{q^n})$ does not factor through any nontrivial norm maps $\bW_h^{h-1}(\FF_{q^n}) \to \bW_h^{h-1}(\FF_{q^r})$ for $r \mid n$, $r < n$.

\subsection{Irreducibility of $R_{T_h}^{G_h}(\theta)$}

\begin{theorem}\label{t:alt sum Xh}
Let $\theta, \theta' \from T_h \to \overline \QQ_\ell^\times$ be two characters and assume $\theta$ is primitive. Then
\begin{equation*}
\left\langle R_{T_h}^{G_h}(\theta), R_{T_h}^{G_h}(\theta') \right\rangle_{G_h} = 
\begin{cases}
1 & \text{if $\theta = \theta'$,} \\
0 & \text{otherwise.}
\end{cases}
\end{equation*}
In particular, the virtual $G_h$-representation $R_{T_h}^{G_h}(\theta)$ is (up to a sign) irreducible.
\end{theorem}

Let $\bU_h,\bU_h^- \subseteq \bG_h$ be as in Section \ref{sec:rel_Xh_to_DL_Lusztig}. Put 
\begin{equation*}
S_h \colonequals \{x \in \bG_h \colon x^{-1}F(x) \in \bU_h^-\}.
\end{equation*}
This has an action of $G_h \times T_h$ by $(g,t) \colon x \mapsto gxt$. Recalling from Proposition \ref{p:Xh Ah-} that $X_h = \{x \in \bG_h : x^{-1} F(x) \in \bU_h^-\}/(\bU_h^- \cap F(\bU_h^-))$, we immediately have the following lemma.

\begin{lm}\label{lm:Xh_and_Xhprime} 
The morphism $X_h \times (\bU_h^- \cap  F\bU_h^-) \stackrel{\sim}{\rightarrow} S_h$ given by $(x,h) \mapsto xh$ is a $(G_h \times T_h)$-equivariant isomorphism, where the action on the left-hand side is given by $(g,t) \colon (x,h) \mapsto (gxt,t^{-1}ht)$. As $\bU_h^- \cap  F\bU_h^-$ is isomorphic to an affine space, for any character $\theta$ of $T_h$, we have $R_{T_h}^{G_h}(\theta) = \sum_i (-1)^i {H}_c^i(S_h, \overline \QQ_\ell)[\theta]$ as virtual $G_h$-representations.
\end{lm}

We show how to reduce Theorem \ref{t:alt sum Xh} to a calculation of the cohomology of 
\begin{equation*}
\Sigma \colonequals \{(x, x', y) \in \bU_h^- \times \bU_h^- \times \bG_h : x F(y) = y x'\},
\end{equation*}
and postpone the study of $\Sigma$ to Section \ref{sec:lusztig_lemma}. Taking for granted Proposition \ref{prop:alternating_sum_for_Sigma}, we give the proof of the main theorem:

\begin{proof}[Proof of Theorem \ref{t:alt sum Xh}]
Let $F$ be the twisted Frobenius given by the Coxeter-type representative $b$ of Section \ref{sec:two b}. Consider the action of $G_h \times T_h \times T_h$ on $S_h \times S_h$ given by $(g, t_1, t_2) \from (x_1, x_2) \mapsto (g x_1 t_1, g x_2 t_2)$. The map
\[
(g,g') \mapsto (x,x',y), x = g^{-1}F(g), x' = g^{\prime -1}F(g'), y = g^{-1} g' 
\]
defines an $T_h \times T_h$-equivariant isomorphism $G_h \backslash S_h \times S_h \cong \Sigma$. We denote by $H_c^i(S_h \times S_h)_{\theta^{-1}, \theta'}$ and $H_c^i(\Sigma)_{\theta^{-1}, \theta'}$ the subspace where $T_h \times T_h$ acts by $\theta^{-1} \otimes \theta'$. We have
\begin{align*}
\langle R_{T_h}^{G_h}(\theta), R_{T_h}^{G_h}(\theta') \rangle_{G_h} 
&= \sum_{i,i' \in \bZ} (-1)^{i+i'} \dim (H_c^i(X_h, \overline{\mathbb{Q}}_{\ell})[\theta^{-1}] \otimes H_c^{i'}(X_h, \overline{\mathbb{Q}}_{\ell})[\theta'])^{G_h} \\
&= \sum_{i,i' \in \bZ} (-1)^{i+i'} \dim (H_c^i(S_h, \overline{\mathbb{Q}}_{\ell})[\theta^{-1}] \otimes H_c^{i'}(S_h, \overline{\mathbb{Q}}_{\ell})[\theta'])^{G_h} && \text{(by Lm \ref{lm:Xh_and_Xhprime})}\\
&= \sum_{i \in \bZ} (-1)^i \dim H_c^i(G_h\backslash (S_h \times S_h), \overline{\mathbb{Q}}_{\ell})_{\theta^{-1},\theta'} \\
&= \sum_{i\in \bZ} (-1)^i \dim H_c^i(\Sigma, \overline{\mathbb{Q}}_{\ell})_{\theta^{-1},\theta'} \\
&= \# \{ \gamma \in \Gal(L/K) \colon \theta \circ \gamma = \theta' \} && \text{(by Prop \ref{prop:alternating_sum_for_Sigma})}
\end{align*}
where in the final equality, we use the fact that $\theta$ is primitive if and only if $\theta$ is regular in the sense of Lusztig \cite[1.5]{Lusztig_04} with respect to the $F$ coming from the Coxeter-type representative $b$. Finally, since the primitivity of $\theta$ implies that the stabilizer of $\theta$ in $\Gal(L/K)$ is trivial, the desired conclusion of Theorem \ref{t:alt sum Xh} now follows.
\end{proof}

\subsection{Traces of very regular elements} \label{sec:traces_of_very_regular_elements}

In Part \ref{part:aut ind and JL}, where we study $R_{T_h}^{G_h}(\theta)$ from the perspective of automorphic induction, we will need to know the trace of \textit{very regular elements} of $\cO_L^\times$; i.e.\ elements $x \in \cO_L^\times$ 
whose image in the residue field generates the multiplicative group $\FF_{q^n}^\times$. 
In fact, we can explicate the character on elements of $\cO_L^\times$ whose image in the residue field has trivial $\Gal(\FF_{q^n}/\FF_q)$-stabilizer.

\begin{proposition}\label{p:trace alt sum Xh}
Let $\theta \from T_h \to \overline \QQ_\ell^\times$ be any character. Then for any element $x \in \cO_L^\times/U_L^h \cong T_h$ in $G_h$ whose image in the residue field has trivial $\Gal(\FF_{q^n}/\FF_q)$-stabilizer,
\begin{equation*}
\Tr\left(x^* ; R_{T_h}^{G_h}(\theta)\right) = \sum_{\gamma \in \Gal(L/K)[n']} \theta^\gamma(x),
\end{equation*}
where $\Gal(L/K)[n']$ is the unique order $n'$ subgroup of $\Gal(L/K)$.
\end{proposition}

\begin{proof}
Let $\zeta_1, \zeta_2 \in T_h$ be $(q^n-1)$th roots of unity, let $t_1, t_2 \in T_h^1$, and assume that the image of $\zeta_1$ modulo $\varpi$ has trivial $\Gal(\FF_{q^n}/\FF_q)$-stabilizer. Note that $(\zeta_1 t_1, \zeta_2 t_2) \in G_h \times T_h$ and therefore acts on $X_h$. By Proposition \ref{p:dim Xh}, $X_h$ is a separated, finite-type scheme over $\FF_{q^n}$. Since $(\zeta_1, \zeta_2) = (\zeta_1 t_1, \zeta_2 t_2)^{q^{n(h-1)}}$ has order prime-to-$p$ and $(t_1, t_2) = (\zeta_1 t_1, \zeta_2 t_2)^N$ (where $N \equiv 1 \pmod{q^{n(h-1)}}$ and $(q^n-1) \mid N$) has order a power of $p$, by the Deligne--Lusztig fixed-point formula \cite[Theorem 3.2]{DeligneL_76},
\begin{equation*}
\sum_i (-1)^i \Tr\left((\zeta_1 t_1, \zeta_2 t_2)^* ; H_c^i(X_h, \overline \QQ_\ell)\right) = \sum_i (-1)^i \Tr\left((t_1, t_2)^* ; H_c^i(X_h^{(\zeta_1, \zeta_2)}, \overline \QQ_\ell)\right).
\end{equation*}
By definition, if $\lambda(x) \in X_h$ corresponds to $x = (x_1, \ldots, x_n) \in \sL_{0,b}^{\rm adm, \rm rat}/\sL_0^{(h)}$, then $(\zeta_1, \zeta_2) x$ corresponds to the tuple $(\zeta_1 \zeta_2 x_1, \sigma^l(\zeta_1) \zeta_2 x_2, \ldots, \sigma^{(n-1)l}(\zeta_1) \zeta_2 x_n)$, where $l$ is the inverse of $e_{\kappa,n}\pmod n$. In particular, we see that if $\zeta_1$ has trivial stabilizer in $\Gal(\FF_{q^n}/\FF_q)$, then the set $X_h^{(\zeta_1, \zeta_2)}$ is nonzero if and only if $\zeta_2^{-1}$ is one of the $n$ distinct elements $\zeta_1, \sigma(\zeta_1), \ldots, \sigma^{n-1}(\zeta_1)$. 

Assume $\zeta_2^{-1} = \sigma^{jl}(\zeta_1)$ with $0 \leq j \leq n-1$, then the elements of $X_h^{(\zeta_1, \zeta_2)}$ correspond to vectors of the shape $x = (0, \ldots, 0, x_{j+1}, 0, \ldots, 0)$. If $n_0$ does not divide $j$, then $\det\lambda(x) \equiv 0 \pmod \varpi$, which contradicts $\det\lambda(x) \in \cO_K^{\times}$. Thus in this case we have $X_h^{(\zeta_1,\zeta_2)} = \varnothing$. Assume $n_0$ divides $j$. Then $x=(0, \ldots, 0, x_{j+1}, 0, \ldots, 0)$ with $x_{j+1} \in \bW_h(\overline{\FF}_q)$ lies in $X_h$ if and only if $\det\lambda(x) = \prod_{i=0}^{n-1} \sigma^i(x_{j+1}) \in (\cO_K/\varpi^h)^{\times}$.
%The preimage of $(\cO_K/\varpi^h)^{\times}$ under the map $(\cO/\varpi^h)^{\times} \rightarrow (\cO/\varpi^h)^{\times}$, $a \mapsto \prod_{i=0}^{n-1}\sigma^i(a)$, is precisely $(\cO_L/\varpi^h)^{\times} = T_h$. 
Thus $X_h^{(\zeta_1,\zeta_2)} = \{x = (0, \ldots, 0, x_{j+1}, 0, \ldots, 0) \colon x_{j+1} \in (\cO_L/\varpi^h)^\times = T_h \}$ is zero-dimensional, and the action of $(t_1,t_2)$ is given by $x_{j+1} \mapsto \sigma^{jl}(t_1) t_2 x_{j+1}$. Thus 
\[
\Tr((t_1,t_2)^{\ast}, {H}_c^0(X_h^{(\zeta_1,\zeta_2)})) = \begin{cases} \# T_h & \text{if $t_2 = \sigma^{jl}(t_1)^{-1}$,} \\ 0 & \text{otherwise.} \end{cases}
\]
From this, we see that 
\begin{align*}
\Tr\left((\zeta_1 t_1, 1)^*; R_{T_h}^{G_h}(\theta)\right)
&= \frac{1}{\# T_h} \sum_{\zeta_2 \in \FF_{q^n}^\times} \sum_{t_2 \in T_h^1} \theta(\zeta_2)^{-1} \theta(t_2)^{-1} \Tr\left((t_1, t_2)^* ; H_c^0(X_h^{(\zeta_1, \zeta_2)}, \overline \QQ_\ell)\right) \\
&= \sum_{\substack{0 \leq j \leq n - 1 \\ n_0 \mid j}} \theta(\sigma^{jl}(\zeta_1)) \theta(\sigma^{jl}(t_1)) = \sum_{\gamma \in \Gal(L/K)[n']} \theta^\gamma(\zeta_1 t_1). \qedhere
\end{align*}
\end{proof}

%************************************************************************************************************************************************************
%************************************************************************************************************************************************************

\subsection{Behavior under twisting of $\theta$} \label{sec:twisting_theta}

\begin{lemma}\label{l:twist compat}
Let $\theta \from T_h \to \overline \QQ_\ell^\times$ be a character with trivial $\Gal(L/K)$-stabilizer and let $\chi \from \bW_h^\times(\FF_q) \to \overline \QQ_\ell^\times$ be any  character. Then as $G_h$-representations,
\begin{equation*}
H_c^i(X_h, \overline \QQ_\ell)[\theta \otimes (\chi \circ \Nm)] \cong H_c^i(X_h, \overline \QQ_\ell)[\theta] \otimes (\chi \circ \det), \qquad \text{for all $i \geq 0$.}
\end{equation*}
\end{lemma}

\begin{proof}
Let $\Gamma_h$ denote the kernel of the natural homomorphism $G_h \times T_h \to \bW_h^\times(\FF_q)$ given by $(g,t) \mapsto \det(t) \Nm(t)$. Recall from \eqref{eq:dec_of_Xh_conn_cpts} that we have a scheme-theoretic morphism $X_h \to \bW_h^\times(\FF_q)$. Write $X_h^{\det \equiv 1}$ for the preimage of the identity. First observe that as $G_h \times T_h$-representations,
\begin{equation*}
\bigoplus_{\substack{\theta' \from T_h \to \overline \QQ_\ell \\ \theta'|_{T_h^\circ} = \theta|_{T_h^\circ}}} H_c^i(X_h, \overline \QQ_\ell)[\theta'] \cong \Ind_{\Gamma_h}^{G_h \times T_h} \left(H_c^i(X_h^{\det \equiv 1}, \overline \QQ_\ell)[\theta|_{T_h^\circ}]\right).
\end{equation*}
Since the number of summands on the left-hand side is equal to the index of $\Gamma_h$ in $G_h \times T_h$, it follows that as representations of $\Gamma_h$,
\begin{equation}\label{eq:Gamma h}
H_c^i(X_h, \overline \QQ_\ell)[\theta'] \cong H_c^i(X_h^{\det \equiv 1}, \overline \QQ_\ell)[\theta|_{T_h^\circ}]
\end{equation}
for any $\theta' \from T_h \to \overline \QQ_\ell^\times$ with $\theta'|_{T_h^\circ} = \theta|_{T_h^\circ}$. In particular, as $\Gamma_h$-representations,
\begin{equation*}
H_c^i(X_h, \overline \QQ_\ell)[\theta \otimes (\chi \circ \Nm)] \cong H_c^i(X_h, \overline \QQ_\ell)[\theta].
\end{equation*}

Now observe that the subgroup of $G_h \times T_h$ generated by $\Gamma_h$ and $1 \times T_h$ is the whole group. For any $g \in G_h$, let $t_g \in T_h$ be any element such that $\det(g) \Nm(t_g) = 1$. Then $(g,t_g) \in \Gamma_h$, and we have
\begin{align*}
\Tr((g,1)^*; &H_c^i(X_h, \overline \QQ_\ell)[\theta \otimes (\chi \circ \Nm)]) \\
&= \Tr((g,t_g)^* ; H_c^i(X_h, \overline \QQ_\ell)[\theta \otimes (\chi \circ \Nm)]) \cdot \theta(t_g^{-1}) \cdot \chi(\Nm(t_g^{-1})) \\
&= \Tr((g,t_g)^* ; H_c^i(X_h, \overline \QQ_\ell)[\theta]) \cdot \theta(t_g^{-1}) \cdot \chi(\Nm(t_g^{-1})) \\
&= \Tr((g,1)^* ; H_c^i(X_h, \overline \QQ_\ell)[\theta]) \cdot \theta(t_g) \cdot \theta(t_g^{-1}) \cdot \chi(\Nm(t_g^{-1})) \\
&= \Tr((g,1)^* ; H_c^i(X_h, \overline \QQ_\ell)[\theta]) \cdot \chi(\det(g)). \qedhere
\end{align*}
\end{proof}

Observe that by Lemma \ref{l:twist compat}, we have that $R_{T_h}^{G_h}(\theta)$ is (up to sign) irreducible if and only if $R_{T_h}^{G_h}(\theta \otimes (\chi \circ \Nm))$ is, where $\chi \from \bW_h(\FF_q)^\times \to \overline \QQ_\ell^\times$. Recall that by Proposition \ref{p:Wh fixed}, if $\theta$ is a character of $T_h$ that factors through the natural surjection $T_h \to T_{h'}$ for some $h' < h$, then $R_{T_h}^{G_h}(\theta) = R_{T_{h'}}^{G_{h'}}(\theta).$ Thus we can strengthen Theorem \ref{t:alt sum Xh} to obtain that $R_{T_h}^{G_h}(\theta \otimes (\chi \circ \Nm))$ is (up to sign) irreducible for any primitive $\theta \from T_{h'} \to \overline \QQ_\ell^\times$ and any $\chi \from \bW_h(\FF_q)^\times \to \overline \QQ_\ell^\times$. Such characters exactly correspond to \textit{minimal admissible} characters of $L^\times$ of level $h$ (see Part \ref{part:aut ind and JL}). This argument will be appear again in the proof of Theorem \ref{t:RTG irred}.

%************************************************************************************************************************************************************
%************************************************************************************************************************************************************

\subsection{Lusztig's theorem} \label{sec:lusztig_lemma}

This is a generalization of \cite{Lusztig_04, Stasinski_09} to non-reductive groups over $\cO$. The Iwahori case (which corresponds to the division algebra setting over $K$) was done in \cite{Lusztig_79} (see also \cite[Section 6.2]{Chan_siDL}) and is a simpler incarnation of these ideas. We keep our notation as close as possible to that of \cite{Lusztig_04, Stasinski_09} as most of the arguments are the same.

\subsubsection{Set-up}

Let $T, T'$ be two maximal $F$-stable tori of $J_b$, split over $\breve{K}$ and let $(U, U^-)$ and $(U', U'{}^-)$ be two pairs of (possibly not $F$-stable) unipotent radicals of opposite Borels containing $T$ and $T$, respectively. (Outside Section \ref{sec:lusztig_lemma} $T$ always denotes a maximal elliptic torus of $G$, but here we want the notation to coincide with \cite{Lusztig_04}). Consider the  intersections of $\breve K$-points of $T,T',U,U^-,U',U^{'-}$ with $\breve G_{\bx,0}$ (Section \ref{sec:definition of GGh}) and denote the corresponding  subgroup schemes in $\bG_h$ by $\bT_h, \bT_h', \bU_h, \bU_h^-, \bU_h', \bU_h'{}^-$. For $1 \leq a \leq h$, let $\bG_h^a \colonequals \ker(\bG_h \to \bG_a)$ be the kernel of the natural projection, and analogously define $\bT_h^a, \bU_h^a,$ and so forth. We set $\bG_h^{a,*} = \bG_h^a \smallsetminus \bG_h^{a+1}$, and analogously for $\bT_h^{a,*}, \bU_h^{a,*}$, and so forth. We use the shorthand $\cT \colonequals \bT_h^{h-1}$.

Let $N(\bT, \bT') = \{g \in \breve G_{\bx,0} : g^{-1} \bT g = \bT'\}$ and $N(\bT_h, \bT_h') = \{g \in \bG_h : g^{-1} \bT_h g = \bT_h'\}$, and define
\begin{equation*}
W(T, T') \colonequals \bT \backslash N(\bT, \bT') = \bT_h \backslash N(\bT_h, \bT_h').
\end{equation*}
Observe that $W(T, T')$ is a principal homogeneous space under the Weyl group of the torus $\bT_1$ in the reductive quotient $\bG_1$ of $G_{\cO}$.

\subsubsection{Roots and regularity} 

Let $\Phi = \Phi(T,J_b)$ denote the set of roots of $T$ in $J_b$. It carries a natural action of $F$. For $\alpha \in \Phi$, let $\bG_h^{\alpha}$ denote the subgroup of $\bG_h$ coming from the root subgroup of $J_b(\breve K) = \GL_n(\breve K)$ corresponding to $\alpha$. For $\alpha \in \Phi$, let $T^{\alpha} \subseteq T$ be the image of the coroot of $T$ in $\GL_n(\breve K)$ corresponding to $\alpha$. It is an one-dimensional subtorus of $T$. We denote by $\bT_h^{\alpha}$ the corresponding subgroup of $\bG_h$. We write $\mathcal{T}^{\alpha} \subseteq \mathcal{T}$ for the one-dimensional subgroup $(\bT_h^{\alpha})^{h-1}$ of $\bT_h^{\alpha}$. 

Following \cite[1.5]{Lusztig_04}, a character $\chi \colon \mathcal{T}^F \rightarrow \overline{\mathbb{Q}}_{\ell}^{\times}$ is called \emph{regular} if for any $\alpha \in \Phi$ and any $m\geq 1$ such that $F^m(\mathcal{T}^{\alpha}) = \mathcal{T}^{\alpha}$, the restriction of $\chi \circ N_F^{F^m} \colon \mathcal{T}^{F^m} \to \overline{\mathbb{Q}}_{\ell}^{\times}$ to $(\mathcal{T}^{\alpha})^{F^m}$ is non-trivial. Here, $N_F^{F^m} \from \cT^{F^m} \to \cT^F$ is the map $t \mapsto t F(t) \cdots F^{m-1}(t)$. A character $\chi$ of $\bT_h^F$ is called \emph{regular} if its restriction $\chi|_{\mathcal{T}^F}$ is regular.

\begin{rem} \label{r:prim reg}
In our situation, when $b$ is a Coxeter-type representative and $T$ is the elliptic diagonal torus of $J_b$, let $\chi$ be a character of $T(K) \cong L^{\times}$ of level $h$. Then the restriction of $\chi$ to $\cO_L^\times$ can be viewed as a character $\chi_h$ of $\bT_h^F \cong (\cO_L/\varpi^h)^{\times}$. A straightforward computation shows: $\chi_h$ is regular in the above sense if and only if it is \textit{primitive}, i.e.\ the restriction of $\chi_h$ to $\mathcal{T}^F \cong \bW_h^{h-1}(\FF_{q^n})$ does not factor through any of the norm maps $\bW_h^{h-1}(\FF_{q^n}) \rightarrow \bW_h^{h-1}(\FF_{q^r})$ for $r \mid n$, $r < n$. We use this in the proof of Theorem \ref{t:alt sum Xh}. \hfill $\Diamond$
\end{rem}

\subsubsection{Bruhat decomposition} For each $w \in W(T,T')$ choose a representative $\dot{w} \in N(T,T')$. We have the Bruhat decomposition $\bG_1 = \bigsqcup_{w\in W(T,T')} \bG_{1,w}$ of the reductive quotient, where $\bG_{1,w} = \bU_1\dot{w}\bT_1'\bU_1'$. Define $\bG_{h,w}$ to be the pullback of $\bG_{1,w}$ along the natural projection $\bG_h \twoheadrightarrow \bG_1$. Thus $\bG_h = \bigsqcup_{w \in W(T,T')} \bG_{h,w}$. Let $\bK_h \colonequals \bU_h^- \cap \dot w \bU_h^{\prime -} \dot w^{-1}$ and $\bK_h^1 \colonequals \bK_h \cap \bG_h^1$.  

\begin{lm}\label{lm:simplification_for_Bruhat_cell}
$\bG_{h,w} = \bU_h \bK_h^1 \dot{w} \bT_h'\bU_h'$.
\end{lm}
\begin{proof} Indeed, we compute 
 \begin{align*}
\bG_{h,w} &= \bU_h \dot w \bT_h' \bG_h^1 \bU_h'  = \bU_h \dot w \bT_h' \left((\bG_h^1 \cap \bT_h') (\bG_h^1 \cap \bU_h^{\prime-}) (\bG_h^1 \cap \bU_h')\right) \bU_h' \\
&= \bU_h \dot w \bT_h' (\bG_h^1 \cap \bU_h^{\prime-}) \bU_h' = \bU_h \left(\dot w (\bG_h^1 \cap \bU_h^{\prime-}) \dot w^{-1}\right) \dot{w} \bT_h' \bU_h' \\
&= \bU_h \left(\bU_h^- \cap \dot w (\bG_h^1 \cap \bU_h^{\prime -}) \dot w^{-1}\right) \dot w \bT_h' \bU_h' = \bU_h \bK_h^1 \dot w \bT_h' \bU_h'. \qedhere
\end{align*}
\end{proof}

\subsubsection{The scheme $\Sigma$}\label{sec:scheme_Sigma}
Define 
\begin{align*}
\Sigma &= \{ (x,x',y) \in F(\bU_h) \times F(\bU_h') \times \bG_h \colon xF(y) = yx' \} \\
\Sigma_w &= \{ (x,x',y) \in F(\bU_h) \times F(\bU_h') \times \bG_h \colon xF(y) = yx', y \in \bG_{h,w} \} \subseteq \Sigma,
\end{align*}
for $w \in W(T,T')$. Set-theoretically, $\Sigma$ is the disjoint union of the locally closed subschemes $\Sigma_w$. The group $\bT_h^F \times \bT_h^{\prime F}$ acts on $\Sigma$ by $(t,t') \colon (x,x',y) \mapsto (txt^{-1}, t'x't^{\prime -1}, tyt^{\prime -1})$ and $\Sigma_w$ is stable under this action for any $w\in W(T,T')$.

\begin{prop}\label{prop:alternating_sum_for_Sigma}
Let $\theta$ and $\theta'$ be characters of $\bT_h^F$ and $\bT_h^{\prime F}$ respectively and assume that $\theta$ is regular. Then 
\[ 
\sum_{i\in \bZ} (-1)^i \dim H_c^i(\Sigma, \overline{\mathbb{Q}}_{\ell})_{\theta^{-1},\theta'} = \# \{ w \in W(T,T')^F \colon \theta \circ Ad(\dot w) = \theta' \}.
\]
\end{prop}

\begin{proof} Using $\Sigma = \bigcup_w \Sigma_w$, it is enough to show that $\sum_{i\in \bZ} (-1)^i \dim H_c^i(\Sigma, \overline{\mathbb{Q}}_{\ell})_{\theta^{-1},\theta'}$ is $1$ if $w \in W(T,T')^F$ and $\theta \circ Ad(\dot w) = \theta'$, and is $0$ otherwise. Fix a $w \in W(T,T')$. Let 
\begin{align*}
\widehat \Sigma_w = \{ (x,x',u,u',z,\tau') \in F(\bU_h) \times F(\bU_h') \times \bU_h \times &\bU_h' \times \bK_h^1 \times \bT_h' \colon \\
&xF(uz\dot{w} \tau' u') = uz\dot{w} \tau' u' x' \}.
\end{align*}
We have the morphism $\widehat \Sigma_w \rightarrow \Sigma_w$, $(x,x',u,u',z,\tau') \mapsto x,x',uz\dot{w}\tau'u'$, which by Lemma \ref{lm:simplification_for_Bruhat_cell} is surjective. Moreover, this map is $\bT_h^F \times \bT_h^{\prime F}$-equivariant, when we endow $\widehat \Sigma_w$ with the $\bT_h^F \times \bT_h^{\prime F}$-action
\begin{equation}\label{eq:torus_action_on_widehatSigma}
(t,t') \colon (x,x',u,u',z,\tau') \mapsto (txt^{-1}, t'x't^{\prime -1}, tut^{-1}, t'u't^{\prime -1}, tzt^{-1}, \dot{w}^{-1}t\dot{w}\tau' t^{\prime -1}).
\end{equation}
As the projection $\widehat \Sigma_w \rightarrow \Sigma_w$ is locally trivial fibration, the cohomology does not change if we pass from $\Sigma_w$ to $\widehat \Sigma_w$. Thus to finish the proof the proposition it is enough to show that
\begin{equation}\label{eq:claim_for_coh_of_widehatSigma}
\sum_{i\in \bZ} (-1)^i \dim H_c^i(\widehat \Sigma, \overline{\mathbb{Q}}_{\ell})_{\theta^{-1},\theta'} = \begin{cases} 1 & \text{if $w \in W(T,T')^F$ and $\theta \circ Ad(\dot w) = \theta'$,} \\ 0 & \text{otherwise.} \end{cases}
\end{equation}
We make the change of variables replacing $xF(u)$ by $x$ and $x'F(u')^{-1}$ by $x'$, and rewrite $\widehat \Sigma_w$ as
\begin{equation*}
\widehat \Sigma_w = \{ (x,x',u,u',z,\tau') \in F(\bU_h) \times F(\bU_h') \times \bU_h \times \bU_h' \times \bK_h^1 \times \bT_h' \colon xF(z\dot{w} \tau') = uz\dot{w} \tau' u' x' \},
\end{equation*}
and the torus action is still given by \eqref{eq:torus_action_on_widehatSigma}.
Define a partition $\widehat \Sigma_w = \widehat \Sigma_w' \sqcup \widehat \Sigma_w^{\prime\prime}$ by
\begin{align*}
\widehat \Sigma_w' &= \{ (x,x',u,u',z,\tau') \in \widehat \Sigma_w \colon z \neq 1 \}, \\
\widehat \Sigma_w^{\prime\prime} &= \{ (x,x',u,u',z,\tau') \in \widehat \Sigma_w \colon z = 1 \}. 
\end{align*}
Both subsets are stable under the $\bT_h^F \times \bT_h^{\prime F}$-action. By Section \ref{sec:widehatSigmadoubleprime},
\begin{equation}\label{e:double prime}
\sum_{i \in \bZ} (-1)^i \dim H_c^i(\widehat \Sigma_w'', \overline{\mathbb{Q}}_{\ell})_{\theta^{-1}, \theta'} = \begin{cases}
1 & \text{if $w \in W(T, T')^F$ and $\theta \circ Ad(\dot w) = \theta'$,} \\
0 & \text{otherwise,}
\end{cases}
\end{equation}
and by Section \ref{sec:widehatSigmaprime}, under the assumption that $\theta$ is regular, 
\begin{equation}\label{e:prime}
\sum_{i \in \bZ} (-1)^i \dim H_c^i(\widehat \Sigma_w', \overline{\mathbb{Q}}_{\ell})_{\theta^{-1}, \theta'} = 0,
\end{equation}
so \eqref{eq:claim_for_coh_of_widehatSigma} holds.
\end{proof}

%%%%%%%%%%%%%%%%%%%%%%%%%%%%%%%%%%%%%%%%%%%%%%%%%%%%%%%%%%%%%%%%%%%%%%%%%%%%%%%%%
%%%%%%%%%%%%%%%%%%%%%%%%%%%%%%%%%%%%%%%%%%%%%%%%%%%%%%%%%%%%%%%%%%%%%%%%%%%%%%%%%

\subsubsection{Cohomology of $\widehat \Sigma_w^{\prime\prime}$}\label{sec:widehatSigmadoubleprime}
We prove \eqref{e:double prime}. This works exactly as in \cite{Lusztig_04} (see the proof of Lemma 1.9, specifically the proof of claim (b) in \textit{op.\ cit.\ }beginning on page 8). For convenience of the reader, we recall the arguments. Consider the closed subgroup 
\[ 
\widetilde{H} = \{ (t,t') \in \bT_h \times \bT_h' \colon tF(t)^{-1} = F(\dot{w}) t'F(t')^{-1} F(\dot{w}^{-1}) \} \subseteq \bT_h \times \bT_h'.
\]
Note that $\widetilde H$ contains $\bT_h^F \times \bT_h'{}^F$ and \eqref{eq:torus_action_on_widehatSigma}
containing $\bT_h^F \times \bT_h^{\prime F}$. The action of $\bT_h^F \times \bT_h^{\prime F}$ on $\widehat \Sigma_w^{\prime\prime}$ extends to an action of $\widetilde H$, still given by \eqref{eq:torus_action_on_widehatSigma}. Let $\bT_{h,\ast}$ and $\bT_{h,\ast}'$ be the reductive part of $\bT_h$ and $\bT_h'$ respectively. 
%(Note that Lusztig's argument also works in mixed characteristic as the morphism of $\FF_q$-group schemes $\bW_h^{\times} \twoheadrightarrow \bW_1^{\times} \cong \bG_m$ is split by the Teichm\"uller lift). 
Set $\widetilde{H}_{\ast} \colonequals \widetilde{H} \cap (\bT_{h,\ast} \times \bT_{h,\ast}')$ and let $\widetilde{H}_{\ast}^0$ be the connected component of $\widetilde{H}_{\ast}$. Then $\widetilde{H}_{\ast}^0$ is a torus acting on $\widehat \Sigma_w^{\prime\prime}$. By \cite[4.5 (and 11.2) and 10.15]{DigneM_91} (compare the similar computation in the proof of \cite[Theorem 3.1]{Stasinski_09}), we have
\[ 
\sum_{i\in \bZ} (-1)^i \dim H_c^i(\widehat \Sigma_w^{\prime\prime}, \overline{\mathbb{Q}}_{\ell})_{\theta^{-1},\theta'} = \sum_{i\in \bZ} (-1)^i \dim H_c^i\left( (\widehat\Sigma_w^{\prime\prime})^{\widetilde H_{\ast}^0}, \overline{\mathbb{Q}}_{\ell} \right)_{\theta^{-1},\theta'}.
\]
Let $(x,x',u,u',1,\tau') \in (\widehat\Sigma_w^{\prime\prime})^{\widetilde H_{\ast}^0}$. By Lang's theorem, $\widetilde H_{\ast} \rightarrow \bT_{h,\ast}$ is surjective and hence (as $\bT_{h,\ast}$ is connected) also $\widetilde{H}_{\ast}^0 \rightarrow \bT_{h,\ast}$ is surjective. Similarly, $\widetilde{H}_{\ast}^0 \rightarrow \bT_{h,\ast}'$ is surjective. Thus for any $t \in \bT_{h,\ast}$, $t' \in \bT_{h,\ast}'$, we have
\[ txt^{-1} = x, \,\, t'x't^{\prime -1} = x', \,\, tut^{-1} = u, \,\, t'u't^{\prime -1} = u'. \]
This implies $x = x' = u = u' = 1$ since $\bT_{h,\ast}$ acts non-trivially on all affine roots subgroups contained in $\bU_h$ (and similarly for $\bT'_h$, $\bU'_h$). Thus 
\begin{equation*}
(\widehat\Sigma_w^{\prime\prime})^{\widetilde H_{\ast}^0} \subseteq \{(1,1,1,1,1,\tau') \colon \tau' \in \bT_h', \, F(\dot w \tau') = \dot w \tau' \},
\end{equation*}
and we deduce
\[
\sum_{i\in \bZ} (-1)^i \dim H_c^i\left( (\widehat\Sigma_w^{\prime\prime})^{\widetilde H_{\ast}^0}, \overline{\mathbb{Q}}_{\ell} \right)_{\theta^{-1},\theta'} = \begin{cases} 1  & \text{if $F(w) = w$ and $\theta \circ Ad(\dot w) = \theta'$}, \\ 
0 & \text{otherwise.} \end{cases}
\]

%%%%%%%%%%%%%%%%%%%%%%%%%%%%%%%%%%%%%%%%%%%%%%%%%%%%%%%%%%%%%%%%%%%%%%%%%%%%%%%%%
%%%%%%%%%%%%%%%%%%%%%%%%%%%%%%%%%%%%%%%%%%%%%%%%%%%%%%%%%%%%%%%%%%%%%%%%%%%%%%%%%

\subsubsection{Some preparations} \label{sec:some_preparations_for_Sigma_prime}

In the next two sections, we make the necessary preparations in order to carry out Lusztig's argument for \eqref{e:prime} in Section \ref{sec:widehatSigmaprime}. Let $N, N^-$ be unipotent radicals of opposite Borel subgroups of $J_b(\breve K) = \GL_n(\breve K)$ containing $T$, and for $h \geq 1$, let $\bN_h$, $\bN_h^-$ be the corresponding subgroups of $\bG_h$. Let $\Phi^+ = \{ \alpha \in \Phi \colon \bG_h^{\alpha} \subseteq \bN_h\}$ and $\Phi^- = \Phi \sm \Phi^+ = \{\alpha \in \Phi \colon \bG_h^{\alpha} \subseteq \bN_h^-\}$. For $\alpha \in \Phi^+$ let ${\rm ht}(\alpha)$ denote the largest integer $m \geq 1$, such that $\alpha = \sum_{i=1}^m \alpha_i$ with $\alpha_i \in \Phi^+$.

We call the roots $\alpha \in \Phi$ for which $\bG_1^{\alpha} \neq 1$ \emph{reductive} and the other roots \emph{non-reductive}. Equivalently, a root $\alpha \in \Phi$ is reductive if and only if $\langle \alpha, \bx \rangle \in \bZ$, where $\bx$ is as in Section \ref{sec:definition of GGh}.

To make explicit calculations, we may assume that $T$ is the diagonal torus in $\GL_n(\breve K)$. For $1\leq i \neq j \leq n$, let $\alpha_{i,j}$ denote the root corresponding to the $(i,j)$th entry of an $n\times n$ matrix. For $1\leq i \leq n$, let $1\leq [i]_{n_0} \leq n_0$ denote its residue modulo $n_0$. Define ${\rm ht}_{n_0}(\alpha_{i,j}) \colonequals [i]_{n_0} - [j]_{n_0}$. Then $\alpha \in \Phi$ is reductive if and only if ${\rm ht}_{n_0}(\alpha) = 0$. If ${\rm ht}_{n_0}(\alpha) > 0$ (resp.\ ${\rm ht}_{n_0}(\alpha) < 0$), we call $\alpha$ non-reductive \emph{of type $1$} (resp.\ \emph{of type $2$}). For any $\alpha = \alpha_{i,j} \in \Phi$ and $1\leq a\leq h$, we have
\begin{equation}\label{eq:expl_description_of_root_filtration}
\left(\bG_h^{\alpha}\right)^a \cong 
\begin{cases} 
\mathfrak{p}^{a-1}/\mathfrak{p}^{h-1} & \text{if ${\rm ht}_{n_0}(\alpha) > 0$,} \\ 
\mathfrak{p}^a/\mathfrak{p}^h & \text{if ${\rm ht}_{n_0}(\alpha) \leq 0$,}  
\end{cases}
\end{equation}
in the sense that $\left(\bG_h^{\alpha}\right)^a$ consists of $n\times n$ matrices with $1$'s on the main diagonal, an element of the subgroup $\mathfrak{p}^{a-1}/\mathfrak{p}^{h-1}$ (resp.\ $\mathfrak{p}^a/\mathfrak{p}^h$) sitting in the $(i,j)$th entry, and $0$'s everywhere else. 

\begin{ex} Let $n= 4$, $\kappa = 2$. Then if $\cA$ is the apartment of $\cB^{\rm red}(\GL_4, \breve K)$ corresponding to the diagonal torus, then $\bx$ is the unique fixed point under the action of $b = b_0\cdot \diag(1,\varpi,1,\varpi) = \left(\begin{smallmatrix} & & & \varpi \\ 1 & & & \\ & \varpi & & \\ & & 1 & \end{smallmatrix}\right)$. Computing, the matrix of inner products for $\alpha_{i,j} \in \Phi$ is
\begin{equation*}
(\langle \alpha_{i,j}, \bx \rangle)_{1 \leq i,j \leq 4} = 
\left(\begin{smallmatrix}
* & -\frac{1}{2} & 0 & -\frac{1}{2} \\
\frac{1}{2} & * & \frac{1}{2} & 0 \\
0 & -\frac{1}{2} & * & -\frac{1}{2} \\
\frac{1}{2} & 0 & \frac{1}{2} & *
\end{smallmatrix}\right).
\end{equation*}
Hence for $h \geq 1$, we have 
\[
\breve G_{\bx,0} = \left(\begin{smallmatrix} \cO & \mathfrak{p} & \cO & \mathfrak{p} \\ \cO & \cO & \cO & \cO \\ \cO & \mathfrak{p} & \cO & \mathfrak{p} \\ \cO & \cO & \cO & \cO \end{smallmatrix}\right)^{\times} \twoheadrightarrow \bG_h(\overline \FF_q) = \left(\begin{smallmatrix} \cO/\mathfrak{p}^h & \mathfrak{p}/\mathfrak{p}^h & \cO/\mathfrak{p}^h  & \mathfrak{p}/\mathfrak{p}^h \\ \cO/\mathfrak{p}^{h-1} & \cO/\mathfrak{p}^h & \cO/\mathfrak{p}^{h-1} & \cO/\mathfrak{p}^h \\ \cO/\mathfrak{p}^h & \mathfrak{p}/\mathfrak{p}^h & \cO/\mathfrak{p}^h  & \mathfrak{p}/\mathfrak{p}^h \\ \cO/\mathfrak{p}^{h-1} & \cO/\mathfrak{p}^h & \cO/\mathfrak{p}^{h-1} & \cO/\mathfrak{p}^h \end{smallmatrix}\right)^{\times},
\]
where the $\times$ superscript means the group of invertible matrices, and for $1 \leq a \leq h$,
\[
\bG_h^a(\overline \FF_q) 
%= \ker(\bG_h(\overline \FF_q) \twoheadrightarrow \bG_a(\overline \FF_q)) 
= \left(\begin{smallmatrix} 1 + \mathfrak{p}^a/\mathfrak{p}^h & \mathfrak{p}^a/\mathfrak{p}^h & \mathfrak{p}^a/\mathfrak{p}^h & \mathfrak{p}^a/\mathfrak{p}^h \\ \mathfrak{p}^{a-1}/\mathfrak{p}^{h-1} & 1 + \mathfrak{p}^a/\mathfrak{p}^h & \mathfrak{p}^{a-1}/\mathfrak{p}^{h-1} & \mathfrak{p}^a/\mathfrak{p}^h \\ \mathfrak{p}^a/\mathfrak{p}^h & \mathfrak{p}^a/\mathfrak{p}^h & 1 + \mathfrak{p}^a/\mathfrak{p}^h & \mathfrak{p}^a/\mathfrak{p}^h \\ \mathfrak{p}^{a-1}/\mathfrak{p}^{h-1} & \mathfrak{p}^a/\mathfrak{p}^h & \mathfrak{p}^{a-1}/\mathfrak{p}^{h-1} & 1 + \mathfrak{p}^a/\mathfrak{p}^h \end{smallmatrix}\right).
\]
\end{ex}

For two elements $z,\xi \in \bG_h$, we write $[\xi,z] = \xi^{-1}z^{-1}\xi z$.

\begin{lm}\label{lm:commutator_trivial} Let $\alpha \in \Phi$. Let $1 \leq a \leq h-1$.
\begin{enumerate}[label=(\roman*)]
\item If $\alpha$ is non-reductive, then $[\bG_h^{a+1}, (\bG_h^{\alpha})^{h-a}] = 1$.
\item If $\alpha$ is reductive, then $[\bG_h^a, (\bG_h^{\alpha})^{h-a}] = 1$.
\end{enumerate}
\end{lm}
\begin{proof}
The computation to show (i) and (ii) is nearly the same. We prove (i). It suffices to check that $[\bT_h^{a+1},(\bG_h^{\alpha})^{h-a}] = 1$ and that $[(\bG_h^{\beta})^{a+1},(\bG_h^{\alpha})^{h-a}] = 1$ for any $\beta \in \Phi$. This is an immediate computation using the explicit description of $\bG_h$ and \eqref{eq:expl_description_of_root_filtration}. The only critical case is when $\alpha$, $\beta$ are both non-reductive of type $1$. Here, it suffices to observe that if $\alpha + \beta$ is again a root, then it is again non-reductive of type $1$.
\end{proof}

Let $(\bN_h^1)^{\leq 0}$ denote the subgroup of $\bN_h^1$ generated by $\bN_h^2$ and all $(\bG_h^\beta)^1$ with $\beta \in \Phi^+$ satisfying ${\rm ht}_{n_0}(\beta) \leq 0$. Obviously $\bN_h^2 \subseteq (\bN_h^1)^{\leq 0} \subseteq \bN_h^1$.

\begin{lm}\label{lm:indep_of_order}
Let $1 \leq a \leq h-1$ and $z \in \bN_h^{a,*}$. Write  $z = \prod_{\beta \in \Phi^+} x_{\beta}^z$ with $x_{\beta}^z \in (\bG_h^{\beta})^a$ for a fixed (but arbitrary) order on $\Phi^+$. For $\beta\in\Phi^+$, let $a \leq a(\beta,z) \leq h$ be the integer such that $x_{\beta}^z \in (\bG_h^{\beta})^{a(\beta,z),\ast}$.
\begin{enumerate}[label=(\roman*)]
\item
If $z \in \bN_h^{a,*} \cap (\bN_h^1)^{\leq 0}$, then the set 
\begin{equation*}
A_z \colonequals \{\beta \in \Phi^+ \colon a(\beta,z) = a\}
\end{equation*}
is independent of the chosen order on $\Phi^+$. 
\item
If $z \in \bN_h^1 \smallsetminus (\bN_h^1)^{\leq 0}$, then the set
\begin{equation*}
A_z \colonequals \{\beta \in \Phi^+ : \text{${\rm ht}_{n_0}(\beta)$ is minimal among those with ${\rm ht}_{n_0} > 0$ and $a(\beta, z) = 1$}\}
\end{equation*}
does not depend on the chosen order on $\Phi^+$. 
\end{enumerate}
%\begin{itemize}
%\item[(i)] Let $1 \leq a \leq h-1$. Let $z \in \bN_h^{a,\ast} \cap (\bN_h^1)^{\leq 0}$. Write $z = \prod_{\beta \in \Phi^+} x_{\beta}^z$ with $x_{\beta}^z \in (\bG_h^{\beta})^a$ for a fixed (but arbitrary) order on $\Phi^+$. For $\beta\in\Phi^+$, let $a \leq a(\beta,z) \leq h$ be the integer such that $x_{\beta}^z \in (\bG_h^{\beta})^{a(\beta,z),\ast}$. The set \begin{equation*}
%A_z \colonequals \{\beta \in \Phi^+ \colon a(\beta,z) = a\}
%\end{equation*}
%is independent of the chosen order on $\Phi^+$. 
%\item[(ii)] Let $z \in \bN_h^1 \sm (\bN_h^1)^{\leq 0}$. Write $z = \prod_{\beta \in \Phi^+} x_{\beta}^z$ with $x_{\beta}^z \in (\bG_h^{\beta})^1$ for a fixed (but arbitrary) order on $\Phi^+$.  For $\beta\in\Phi^+$, let $a \leq a(\beta,z) \leq h$ be the integer such that $x_{\beta}^z \in (\bG_h^{\beta})^{a(\beta,z),\ast}$. The set
%\begin{equation*}
%A_z \colonequals \{\beta \in \Phi^+ : \text{${\rm ht}_{n_0}(\beta)$ is minimal among those with $a(\beta, z) = 1$}\}
%\end{equation*}
%does not depend on the chosen order on $\Phi^+$. 
%\end{itemize}
\end{lm}

\begin{proof}
(i): First let $2 \leq a \leq h-1$. From the explicit description of the root subgroups it follows that the quotient $\bN_h^a/\bN_h^{a+1}$ is abelian (for $a=2$ one needs to use that the sum of two non-reductive roots of type $1$ is again of type $1$ if it is a root), thus its elements are simply tuples $(x_{\beta})_{\beta \in \Phi^+}$ with $x_{\beta} \in (\bG_h^{\beta})^a/(\bG_h^{\beta})^{a+1}$ with entry-wise multiplication. If $\bar{z}= (\bar{x}_{\beta}^z)$ is the image of $z$ in this quotient, then $A_z$ identifies with the set of those $\beta$ for which $\bar{x}_{\beta}^z \neq 1$ (which is obviously independent of the order). Now let $a = 1$. Then $z \in (\bN_h^1)^{\leq 0} \sm \bN_h^2$ and the same arguments apply to the abelian quotient $(\bN_h^1)^{\leq 0}/\bN_h^2$. 

(ii): The group $\bN_h^1/(\bN_h^1)^{\leq 0}$ is not abelian, but is generated by its subgroups $(\bG_h^{\beta})^1/(\bG_h^{\beta})^2$ for $\beta \in \Phi^+$ non-reductive of type $1$. For $m\geq 1$, let $H_m$ be the subgroup generated by all $(\bG_h^{\beta})^1/(\bG_h^{\beta})^2$ with ${\rm ht}_{n_0}(\beta) \geq m$. Since the function ${\rm ht}_{n_0}$ is additive on $\Phi$, the $H_m$ form a filtration of $\bN_h^1/(\bN_h^1)^{\leq 0} = H_1$ with abelian quotients $H_m/H_{m+1} \cong \prod\limits_{\substack{\beta \text{ non-red.\ type $1$} \\ {\rm ht}_{n_0}(\beta) = m}} (\bG_h^{\beta})^1/(\bG_h^{\beta})^2$. Since $z \not\in (\bN_h^1)^{\leq 0}$, there is an $m \geq 1$ such that the image of $z$ in $\bN_h^1/(\bN_h^1)^{\leq 0}$ lies in $H_m \sm H_{m+1}$. Denote by $\bar{z} = (\bar{x}_{\beta}^z)_{\substack{\beta \text{ non-red.\ type $1$} \\ {\rm ht}_{n_0}(\beta) = m}}$ the image of $z$ in $H_m/H_{m+1}$. Now $A_z$ is the set of all $\beta \in \Phi^+$ non-reductive of type $1$ with ${\rm ht}_{n_0}(\beta) = m$ such that $\bar{x}_{\beta}^z \neq 1$. This does not depend on the chosen order. 
\end{proof}

%%%%%%%%%%%%%%%%%%%%%%%%%%%%%%%%%%%%%%%%%%%%%%%%%%%%%%%%%%%%%%%%%%%%%%%%%%%%%%%%%
%%%%%%%%%%%%%%%%%%%%%%%%%%%%%%%%%%%%%%%%%%%%%%%%%%%%%%%%%%%%%%%%%%%%%%%%%%%%%%%%%

\subsubsection{Stratification of $\bK_h^1$} \label{sec:stratification_of_Kh1}

\begin{lm}\label{lm:commutator_decomposition}
Let $1 \leq a \leq h-1$, $z \in \bN_h^{a,\ast}$ and $A_z$ as in Lemma \ref{lm:indep_of_order}.
\begin{enumerate}[label=(\roman*)]
\item If $A_z$ contains a non-reductive root, let $-\alpha \in A_z$ be a non-reductive root  of maximal height and $\alpha \in \Phi^-$ its opposite. Then for any $\xi \in (\bG_h^{\alpha})^{h-a}$, we have $[\xi,z] \in \mathcal{T}^{\alpha}(\bN_h^-)^{h-1}$. Moreover, projecting $[\xi,z]$ into $\mathcal{T}^{\alpha}$ induces an isomorphism
\[
\lambda_z \colon (\bG_h^{\alpha})^{h-a}/(\bG_h^{\alpha})^{h-a+1} \stackrel{\sim}{\rightarrow} \mathcal{T}^{\alpha}
\]
\item If $A_z$ contains only reductive roots, let $-\alpha \in A_z$ be a root of maximal height and $\alpha \in \Phi^-$ its opposite. Then for any $\xi \in (\bG_h^{\alpha})^{h-a-1}$, we have $[\xi,z] \in \mathcal{T}^{\alpha}(\bN_h^-)^{h-1}$. Moreover, projecting $[\xi,z]$ into $\mathcal{T}^{\alpha}$ induces an isomorphism
\[
\lambda_z \colon (\bG_h^{\alpha})^{h-a-1}/(\bG_h^{\alpha})^{h-a} \stackrel{\sim}{\rightarrow} \mathcal{T}^{\alpha}
\]
\end{enumerate}
\end{lm}

\begin{proof}
We first prove (i) when $z \in (\bN_h^1)^{\leq 0}$. Assume first that $A_z$ contains a non-reductive root and let $-\alpha$ be such a root of maximal height and $\alpha \in \Phi^-$ its opposite. By Lemma \ref{lm:commutator_trivial} (applied three times), the commutator map $\bN_h^a \times (\bG_h^{\alpha})^{h-a} \rightarrow \bG_h$ induces a pairing of abelian groups,
\[
\bN_h^a/\bN_h^{a+1} \times (\bG_h^{\alpha})^{h-a}/(\bG_h^{\alpha})^{h-a+1} \rightarrow \bG_h^{h-1}, \quad \bar{x},\bar{\xi} \mapsto [\bar{\xi},\bar{x}].
\]
(If $a=1$, one has to replace $\bN_h^a/\bN_h^{a+1}$ by $(\bN_h^1)^{\leq 0}/\bN_h^2$.) This is bilinear in $\bar{x}$: if $x_1,x_2 \in \bN_h^a$, then
\begin{align*}
[\xi,x_1x_2] &= \xi^{-1} x_2^{-1}x_1^{-1}\xi x_1 x_2 = \xi^{-1} x_1^{-1} x_2^{-1} \xi x_2 x_1 \\
&= \xi^{-1} x_1^{-1} \xi [\xi,x_2] x_1 = [\xi,x_1] [\xi,x_2],
\end{align*}
where the second equality follows from Lemma \ref{lm:commutator_trivial} and $\bN_h^a/\bN_h^{a+1}$ (resp.\ $(\bN_h^1)^{\leq 0}/\bN_h^2$ if $a=1$) being abelian, and the fourth follows from Lemma \ref{lm:commutator_trivial} as $[\xi,x_2] \in \bG_h^{h-1}$. 

Now let $\bar{\xi} \in (\bG_h^{\alpha})^{h-a}/(\bG_h^{\alpha})^{h-a+1}$ and $\bar z \in \bN_h^a/\bN_h^{a+1}$ be the images of $\xi$ and $z$ respectively. Write 
\[ 
\bar{z} = \bar{x}_{-\alpha}^z \prod_{\beta \in \Phi^+ \text{ red.}} \bar{x}_{\beta}^z \cdot  \prod_{\substack{\beta \in \Phi^+ \text{ non-red., } \beta \neq -\alpha \\ {\rm ht}(\beta) \leq {\rm ht}(-\alpha) }} \bar{x}_{\beta}^z.
\]
Then $[\xi,z]$ is the product of $[\bar{\xi},\bar{x}_{-\alpha}]$ with all the $[\bar{\xi},\bar{x}_{\beta}^z]$ in any order. Let $x_{\beta}^z$ be any lift of $\bar{x}_{\beta}^z$ to $(\bG_h^{\beta})^a$. If $\beta$ is reductive and $\alpha$ is (non-reductive) of type $1$, then either $\xi$, $x_{\beta}^z$ commute anyway or $\alpha + \beta$ is again a root (necessarily non-reductive of type $1$) and \eqref{eq:expl_description_of_root_filtration} shows that $[\xi, x_{\beta}^z] = 1$. If $\beta$ is reductive and $\alpha$ is (non-reductive) of type $2$, then \eqref{eq:expl_description_of_root_filtration} shows that $\xi,x_{\beta}^z$ commute. If $\beta \neq -\alpha$ is non-reductive, then by assumption ${\rm ht}(\beta) \leq {\rm ht}(-\alpha)$. Then $[\xi,x_\beta^z] = 1$ unless $\alpha + \beta$ is a root, in which case $[\xi,x_{\beta}^z]\in (\bG_h^{\alpha + \beta})^{h-1}$ by \eqref{eq:expl_description_of_root_filtration}. But the height condition implies that $\alpha + \beta \in \Phi^{-}$. Following this case-by-case examination, the claim about $\lambda_z$ in (i) when $z \in (\bN_h^1)^{\leq 0}$ is then established once we make the following observation: If $\xi$ has $[y] \varpi^{h-a}$ (resp.\ $[y] \varpi^{h-a-1}$, if $a = 1$) and $x_{-\alpha}^z$ has $[u]\varpi^{a-1}$ (resp.\ $[u] \varpi^a$) in their only non-trivial entries, then $[\xi, x_{-\alpha}^z]$ is a diagonal matrix with only two nontrivial entries: $1 \pm [uy] \varpi^{h-1}$. 

In (ii), it is automatic that $z \in (\bN_h^1)^{\leq 0}$, and this case can be proven in exactly the same way as above (and is slightly easier) and we omit the details.

It remains to prove (i) in the case that $z \in \bN_h^1 \sm (\bN_h^1)^{\leq 0}$. In particular, $\xi \in (\bG_h^\alpha)^{h-1}$ since $a = 1$. By construction, $A_z$ consists of non-reductive roots of type $1$, so $\alpha$ must be non-reductive of type $2$. Modulo $(\bN_h^{1})^{\leq 0}$ (which commutes with $\xi$) we may write
\[
z = \Big(\prod_{\substack{\gamma \in \Phi^+ \\ {\rm ht}_{n_0}(\gamma) > {\rm ht}_{n_0}(-\alpha)}} x_{\gamma}^z \Big)  \Big(\prod_{\beta \in A_z \sm \{-\alpha\}} x_{\beta}^z \Big) x_{-\alpha}^z. 
\]
Recall that $A_z \sm\{-\alpha\}$ consists of (necessarily non-reductive, type 1) roots with ${\rm ht}_{n_0}(\beta) = {\rm ht}_{n_0}(-\alpha)$. By construction ${\rm ht}(\gamma), {\rm ht}(\beta) \leq {\rm ht}(-\alpha)$, and so in particular,
\begin{equation*}
s \colonequals \prod_{\beta} [\xi^{-1},(x_{\beta}^z)^{-1}] \in (\bN_h^-)^{h-1}.
\end{equation*}
We claim:
\begin{align}\nonumber
\xi z 
&= \xi \Big(\prod_{\gamma}x_{\gamma}^z\Big) \Big(\prod_{\beta} x_{\beta}^z\Big)  x_{-\alpha}^z \\ \label{e:2}
&= \Big(\prod_{\gamma}x_{\gamma}^z\Big) \xi \Big(\prod_{\beta} x_{\beta}^z\Big)  x_{-\alpha}^z \\\nonumber
&= \Big(\prod_{\gamma}x_{\gamma}^z\Big) \Big(\prod_{\beta} [\xi^{-1},(x_{\beta}^z)^{-1}] x_{\beta}^z \Big) x_{-\alpha}^z \xi [\xi,x_{-\alpha}^z] \\ \label{e:4}
&= \Big(\prod_{\gamma}x_{\gamma}^z\Big) \Big(\prod_{\beta} x_{\beta}^z \Big) x_{-\alpha}^z s \xi [\xi,x_{-\alpha}^z] \\ \nonumber
&= z s \xi [\xi,x_{-\alpha}^z] \\ \label{e:6}
&= z \xi [\xi, x_{-\alpha}^z] s
\end{align}
Here \eqref{e:2} holds as $\alpha + \gamma$ (if it is a root) must be non-reductive of type $1$, and hence $\xi$ and $x_{\gamma}^z$ commute by \eqref{eq:expl_description_of_root_filtration}. To justify \eqref{e:4}, let $\beta \in A_z \sm \{-\alpha\}$. If $\alpha + \beta$ is not a root, then $[\xi,x_{\beta}^z] = 1$. If $\alpha + \beta$ is a root, then $\alpha + \beta$ is reductive (since ${\rm ht}_{n_0}(\beta) = {\rm ht}_{n_0}(-\alpha)$) and $[\xi,x_{\beta}^z] \in (\bG_h^{\alpha + \beta})^{h-1} \subseteq (\bN_h^-)^{h-1}$ (since ${\rm ht}(\beta) \leq {\rm ht}(-\alpha)$ by definition of $\alpha$). But every $\beta' \in A_z$ is non-reductive of type $1$, so we must also have ${\rm ht}_{n_0}(\beta' + (\alpha + \beta)) > 0$, and \eqref{eq:expl_description_of_root_filtration} shows that $[\xi,x_{\beta}^z]$ commute with $x_{\beta'}^z$ for all $\beta' \in A_z$. Finally, \eqref{e:6} follows from the fact that $s \in (\bN_h^-)^{h-1}$ commutes with $\xi \in (\bN_h^-)^{h-1}$ and with $[\xi, x_{-\alpha}^z] \in \mathcal{T}^{\alpha}$. But now we have shown $[\xi,z] = [\xi,x_{-\alpha}^z] s \in \mathcal{T}^{\alpha}(\bN_h^-)^{h-1}$, which finishes the proof of the last remaining assertion of the lemma.
\end{proof}

Let $\bK_h = \bU_h^- \cap \bN_h$. Let $\Phi' = \{\beta \in \Phi^+ \colon \bG_h^{\beta} \in \bK_h\}$. Let $\mathcal{X}$ denote the set of all non-empty subsets $I \subseteq \Phi'$, on which ${\rm ht} \colon \Phi^+ \rightarrow \bZ_{>0}$ is constant.  To $z \in \bK_h^1 \sm \{1\}$ we attach a pair $(a_z,I_z)$ with  $1 \leq a_z \leq h-1$ and $I_z \in \mathcal{X}$. Define $a_z$ by $z \in \bK_h^{a_z,\ast}$. If $A_z$ contains a non-reductive root, let $I_z \subseteq A_z$ be the subset of all non-reductive roots of maximal height. (Note that if $a=1$, then $I_z$ contains only roots of type $1$ if $z \notin (\bN_h^1)^{\leq 0}$ and only contains roots of type $2$ if $z \in (\bN_h^1)^{\leq 0}$.) If $A_z$ contains only reductive roots, let $I_z \subseteq A_z$ be the subset of all roots of maximal height. We have a stratification into locally closed subsets
\begin{equation}\label{eq:stratification_of_Kh1}
\bK_h^1 \sm \{1\} = \bigsqcup_{a,I} \bK_h^{a,\ast,I} \qquad \text{where $\bK_h^{a,\ast,I} = \{z \in \bK_h^1 \sm \{1\} \colon (a_z,I_z) = (a,I)\}$}.
\end{equation}

%%%%%%%%%%%%%%%%%%%%%%%%%%%%%%%%%%%%%%%%%%%%%%%%%%%%%%%%%%%%%%%%%%%%%%%%%%%%%%%%%
%%%%%%%%%%%%%%%%%%%%%%%%%%%%%%%%%%%%%%%%%%%%%%%%%%%%%%%%%%%%%%%%%%%%%%%%%%%%%%%%%

\subsubsection{Cohomology of $\widehat \Sigma_w'$}\label{sec:widehatSigmaprime}

We are now ready to prove \eqref{e:prime} using the same arguments as in the proof of \cite[Lemma 1.9]{Lusztig_04}. To do this, it is enough to show that $H^j_c(\widehat \Sigma'_w)_{\theta,\theta'} = 0$ for all $j \geq 0$. For a $\mathcal{T}^{\prime F}$-module $M$ and a character $\chi$ of $\mathcal{T}^{\prime F}$, write $M_{(\chi)}$ for the $\chi$-isotypic component of $M$. Note that $\mathcal{T}^{\prime F}$ acts on $\widehat\Sigma'_w$ by
\[
t' \colon (x,x',u,u',z,\tau') \mapsto (x, t'x't^{\prime-1}, u, t'u't^{\prime -1}, z, \tau't^{\prime -1}).
\]
Hence $H^j_c(\widehat \Sigma'_w)$ is a $\mathcal{T}^{\prime F}$-module. It is enough to show that $H^j_c(\widehat \Sigma'_w)_{(\chi)} = 0$ for any regular character $\chi$ of $\mathcal{T}^{\prime F}$. Fix such a $\chi$. Define $N = \dot w U^{\prime -} \dot w^{-1}$, $N^- = \dot w U^{\prime} \dot w^{-1}$. Then with notation as in Sections \ref{sec:some_preparations_for_Sigma_prime} and \ref{sec:stratification_of_Kh1}, the stratification of $\bK_h^1 \smallsetminus \{1\}$ given in \eqref{eq:stratification_of_Kh1} induces a stratification of $\widehat \Sigma'_w$ indexed by $1\leq a\leq h-1$ and $I \in \mathcal{X}$:
\begin{equation*}
\widehat\Sigma'_w = \bigsqcup_{a,I} \widehat\Sigma_w^{\prime,a,I} \qquad \text{where $\widehat\Sigma_w^{\prime,a,I} = \{(x,x',u,u',z,\tau') \in \widehat\Sigma'_w \colon z\in \bK_h^{a,\ast,I} \}$}.
\end{equation*}
Note that each $\widehat \Sigma_w^{\prime, a, I}$ is stable under $\mathcal{T}^{\prime F}$. Thus to show \eqref{e:prime}, it is enough to show
\begin{equation}\label{e:Sigma w a I}
H_c^j(\widehat\Sigma_w^{\prime,a,I}, \overline{\mathbb{Q}}_{\ell})_{(\chi)} = 0 \qquad \text{for any fixed $a,I$.}
\end{equation}

Choose a root $\alpha$ such that $-\alpha \in I$. Then $\bG_h^{\alpha} \subseteq \bU_h \cap \dot{w} \bU_h' \dot{w}^{-1}$. For any $z \in \bK_h^{a,\ast,I}$, Lemma \ref{lm:commutator_decomposition} grants us an isomorphism 
\begin{align*}
\lambda_z &\from (\bG_h^\alpha)^{h-a}/(\bG_h^\alpha)^{h-a+1} \stackrel{\sim}{\longrightarrow} \cT^\alpha, && \text{if $\alpha$ is non-reductive,} \\
\lambda_z &\from (\bG_h^\alpha)^{h-a-1}/(\bG_h^\alpha)^{h-a} \stackrel{\sim}{\longrightarrow} \cT^\alpha, && \text{if $\alpha$ is reductive.}
\end{align*}
Let $\pi$ denote the natural projection $(\bG_h^{\alpha})^{h-a} \rightarrow (\bG_h^{\alpha})^{h-a}/(\bG_h^{\alpha})^{h-a+1}$ if $\alpha$ is non-reductive and the natural projection $(\bG_h^{\alpha})^{h-a-1} \rightarrow (\bG_h^{\alpha})^{h-a-1}/(\bG_h^{\alpha})^{h-a}$ if $\alpha$ is reductive. Let $\psi$ be a section to $\pi$ such that $\pi\psi = 1$ and $\psi(1) = 1$. Let
\[
\mathcal{H}' \colonequals \{t' \in \mathcal{T}' \colon t^{\prime -1} F(t') \in \dot{w}^{-1} \mathcal{T}^{\alpha} \dot{w} \}.
\]
This is a closed subgroup of $\mathcal{T}'$. For any $t' \in \mathcal{H}'$ define $f_{t'} \colon \widehat\Sigma_w^{\prime,a,I} \rightarrow \widehat\Sigma_w^{\prime,a,I}$ by
\[f_{t'}(x,x',u,u',z,\tau') = (xF(\xi), \hat{x}^{\prime}, u, F(t')^{-1} u' F(t'), z, \tau'F(t')), \]
where
\[
\xi = \psi\lambda_z^{-1}(\dot{w}F(t')^{-1}t'\dot{w}^{-1}) \in (\bG_h^{\alpha})^{h-a-1} \subseteq \bU_h \cap \dot{w}\bU_h'\dot{w}^{-1}
\]
($(\bG_h^\alpha)^{h-a-1}$ should be replaced by $(\bG_h^\alpha)^{h-a}$ if $\alpha$ is non-reductive), and $\hat{x}' \in \bG_h$ is defined by the condition that
\[
xF(\xi z \dot{w} \tau' F(t')) \in u z \dot{w} \tau' F(t')F(t')^{-1}u'F(t')\hat{x}'. 
\]
To check that this is well-defined one needs to show $\hat{x}' \in F(\bU_h')$. This is done with exactly the same computation as in the proof of \cite[Lemma 1.9]{Lusztig_04}, and we omit this. It is clear that $f_{t'} \colon \widehat\Sigma_w^{\prime,a,I} \rightarrow \widehat\Sigma_w^{\prime,a,I}$ is an isomorphism for any $t' \in \mathcal{H}'$. Moreover, since $\mathcal{T}^{\prime F} \subseteq \mathcal{H}'$ and since for any $t' \in \mathcal{T}^{\prime F}$ the map $f_{t'}$ coincides with the action of $t'$ in the $\mathcal{T}^{\prime F}$-action on $\widehat\Sigma_w^{\prime,a,I}$ (we use $\psi(1) = 1$ here), it follows that we have constructed an action $f$ of $\mathcal{H}'$ on $\widehat\Sigma_w^{\prime,a,I}$ extending the $\mathcal{T}^{\prime F}$-action.

If a connected group acts on a scheme, the induced action in the cohomology is constant. Thus for any $t' \in \mathcal{H}^{\prime 0}$, the induced map $f_{t'}^{\ast} \colon H_c^j(\widehat\Sigma_w^{\prime,a,I}, \overline{\mathbb{Q}}_{\ell}) \rightarrow H_c^j(\widehat\Sigma_w^{\prime,a,I}, \overline{\mathbb{Q}}_{\ell})$ is constant when $t'$ varies in $\mathcal{H}^{\prime 0}$. Hence the restriction of the $\mathcal{T}^{\prime F}$-action on $H_c^j(\widehat\Sigma_w^{\prime,a,I}, \overline{\mathbb{Q}}_{\ell})$ to $\mathcal{T}^{\prime F} \cap \mathcal{H}^{\prime,0}$ is trivial. 

Now we can find some $m\geq 1$ such that $F^m(\dot{w}^{-1}\mathcal{T}^{\alpha}\dot{w}) = \dot{w}^{-1}\mathcal{T}^{\alpha} \dot{w}$. Then 
\[
t' \mapsto t'F(t')F^2(t') \cdots F^{m-1}(t') 
\]
defines a morphism $\dot{w}^{-1}\mathcal{T}^{\alpha}\dot{w} \rightarrow \mathcal{H}'$. Since $\mathcal{T}^{\alpha}$ is connected, its image is also connected and hence contained in $\mathcal{H}^{\prime 0}$. If $t' \in (\dot{w}^{-1}\mathcal{T}^{\alpha}\dot{w})^{F^m}$, then $N_F^{F^m}(t') \in \mathcal{T}^{\prime F}$ and hence also $N_F^{F^m}(t') \in \mathcal{T}^{\prime F} \cap \mathcal{H}^{\prime 0}$. Thus the action of  $N_F^{F^m}(t') \in \mathcal{T}^{\prime F}$ on $H_c^j(\widehat\Sigma_w^{\prime,a,I})$ is trivial for any $t' \in (\dot{w}^{-1}\mathcal{T}^{\alpha}\dot{w})^{F^m}$. 

Finally, observe that if $H_c^j(\widehat\Sigma_w^{\prime,a,I}, \overline{\mathbb{Q}}_{\ell})_{(\chi)} \neq 0$, then the above shows that $t' \mapsto \chi(N_F^{F^m}(t'))$ is the trivial character, which contradicts the regularity assumption on $\chi$. This establishes \eqref{e:Sigma w a I}, which establishes \eqref{e:prime}, which was the last outstanding claim in the proof of Proposition \ref{prop:alternating_sum_for_Sigma}.

\section{Cuspidality} \label{s:cuspidality}

The next theorem (proved in Section \ref{s:Gh cuspidal}) concerns the ``cuspidality'' of the representation $R_{T_h}^{G_h}(\theta)$ for primitive $\theta$. This is the higher-level analogue of Deligne--Lusztig's theorem \cite[Theorem 8.3]{DeligneL_76} required to prove that the induced representation $\cInd_{Z \cdot G_{\cO}}^G\left(|R_{T_h}^{G_h}(\theta)|\right)$ is irreducible and supercuspidal (Theorem \ref{t:RTG irred}). A proof that this induced representation is irreducible supercuspidal when $h = 1$ can be found in \cite[Proposition 6.6]{MoyP_96}, and when $G = \GL_2(K)$ and $h$ arbitrary it was done by the first author in \cite{Ivanov_15_ADLV_GL2_unram}. 

We work with a special representative $b$ as in Section \ref{sec:special_representatives}. 

Let $N'$ be the unipotent radical of any standard parabolic subgroup of $\GL_{n'}$ and let $\breve N$ denote the subgroup of $\GL_n(\breve K)$ consisting of unipotent matrices such that any $(n_0 \times n_0)$-block consists of a diagonal matrix and the $(i,j)$th block can have nonzero entries if and only if the $(i,j)$th entry of an element of $N'$ is nonzero. For each $h \geq 1$, let $\breve N_h$ denote the image of $\breve N \cap \breve G_{\bx,0}$ in $\bG_h(\overline \FF_q)$. Define $N_h \colonequals \breve N_h^F$ and $N_h^{h-1} \colonequals \ker(N_h \to N_{h-1})$.

\begin{theorem}\label{t:Gh cuspidal}
Assume $\theta \from T_h \to \overline \QQ_\ell^\times$ is primitive. Then the restriction of $|R_{T_h}^{G_h}(\theta)|$ to $N_h^{h-1}$ does not contain the trivial representation.
\end{theorem}

%*******************************************************************************************
%*******************************************************************************************

\subsection{Proof of Theorem \ref{t:Gh cuspidal}} \label{s:Gh cuspidal}

\subsubsection{}\label{ss:1}

We retain notation as in the statement of the Theorem and the set-up directly proceeding it. Let $J \colonequals \{\alpha = (i,j) : U_\alpha \subset N'\}$ be the set of roots of the diagonal torus in $\GL_{n'}$ occurring in $N'$. Let $l$ be the inverse of $k_0$ modulo $n_0$ and let $[a]_{n_0}$ denote the residue of $a \in \bZ$ in $1 \leq [a] \leq n_0$. The elements of $N_h^{h-1}$ consist of $n\times n$-matrices, whose $(i,j)$th $(n_0 \times n_0)$-block is the identity matrix if $i = j$, is zero if $i\neq j$ and $(i,j) \not\in J$, and is of the form $\diag(\varpi^{h-1}u, \varpi^{h-1}\sigma^{[l]_{n_0}}(u), \varpi^{h-1}\sigma^{[2l]_{n_0}}(u), \dots, \varpi^{h-1}\sigma^{[(n_0-1)l]_{n_0}}(u))$ for some $u \in \FF_{q^{n_0}}$ if $(i,j) \in J$.  Observe that it is sufficient to show that the theorem holds under the assumption that $N'$ is the unipotent radical of a maximal proper parabolic; that is, 
\begin{equation*}
J = \{(i,j)\colon 1\leq i \leq n' -\ell, \, n'-\ell+1 \leq j \leq n, \} \qquad \text{for some $1\leq \ell \leq n'.$} 
\end{equation*}
(This will be used only in the proof of Lemma \ref{lm:which_W_can_occur}.)

%*******************************************************************************************
%*******************************************************************************************

\subsubsection{}

Our main tool will be a close variant of \cite[Lemma 2.12]{Boyarchenko_12}. The set-up and proof of Lemma \ref{l:char fixed point} is nearly the same as the proof of \textit{op.\ cit.}\ verbatim. Assume that $X$ is a separated scheme of finite type over $\FF_q$ and we are given an automorphism $\varphi$ of $X$ and a right action of a finite group $A$ on $X$ that commute with $\varphi$. For each character $\chi \from A \to \overline \QQ_\ell^\times$, we write $H_c^i(X, \overline \QQ_\ell)[\chi]$ for the subspace of $H_c^i(X, \overline \QQ_\ell)$ on which $A$ acts by $\chi$. Note that this subspace is invariant under the action of $\varphi^* \from H_c^i(X, \overline \QQ_\ell)[\chi] \stackrel{\cong}{\to} H_c^i(X, \overline \QQ_\ell)[\chi]$.

\begin{lemma}\label{l:char fixed point}
Let $\chi \from A \to \overline \QQ_\ell^\times$ be a character. Assume that $\Fr_q$ acts on $\sum_i (-1)^i H_c^i(X, \overline \QQ_\ell)[\chi]$ by a scalar $\lambda$. Then
\begin{equation*}
\Tr\left(\varphi^*; \textstyle \sum\limits_i (-1)^i(X, \overline \QQ_\ell)[\chi]\right) = \frac{1}{\lambda \cdot \#A} \cdot \sum_{a \in A} \chi(a) \cdot \#\{x \in X(\overline \FF_q) : \varphi(\Fr_q(x))= x \cdot a\}.
\end{equation*}
\end{lemma}

\begin{proof}
For each $a \in A$, let $\rho_a \from X \to X$ denote the automorphism $x \mapsto x \cdot a$ and write $\varphi_a = \varphi \circ \rho_a$. Then $\rho_a$ is a finite-order automorphism of $X$ and (as in the proof of \cite[Proposition 3.3]{DeligneL_76})
\begin{equation*}
\textstyle\sum\limits_i (-1)^i \Tr(\Fr_q \circ \, \varphi_a^*; H_c^i(X, \overline \QQ_\ell))  = \#\{x \in X(\overline \FF_q) : \varphi(\Fr_q(x)) = x \cdot a^{-1}\}.
\end{equation*}
Hence averaging over $\chi^{-1}(a)$, we have
\begin{align*}
\frac{1}{\# A} \cdot \sum_{a \in A} \chi^{-1}(a) \cdot \#\{x \in X(\overline \FF_q) : \varphi(\Fr_q(x)) = x \cdot a^{-1}\} 
&= \textstyle\sum\limits_i (-1)^i \Tr\left(\Fr_q \circ \, \varphi^* ; H_c^i(X, \overline \QQ_\ell)[\chi]\right) \\
&= \lambda \Tr\left(\varphi^* ; \textstyle\sum\limits_i (-1)^i H_c^i(X, \overline \QQ_\ell)[\chi]\right). \qedhere
\end{align*}
\end{proof}

%*******************************************************************************************
%*******************************************************************************************

\subsubsection{}

Now fix a character $\theta \from T_h \to \overline \QQ_\ell^\times$ as in the theorem. Recall from \eqref{eq:dec_of_Xh_conn_cpts} that 
\begin{equation*}
X_h = \bigsqcup_{a \in (\mathcal{O}_K/\varpi^h)^{\times}} g_a.X_h^{\det \equiv 1}, \qquad \text{where $X_h^{\det \equiv 1} = \{x \in X_h : \det g_b^{\rm red}(x) \equiv 1 \!\!\!\! \pmod{\varpi^h}\}.$}
\end{equation*}
Note that $T_h$ transitively permutes the components $g_a.X_h^{\det \equiv 1}$ ($a \in (\cO_K/\varpi^h)^{\times}$) and let $T_h^\circ \subseteq T_h$ denote the stabilizer of a (any) component. Since the composition $H_c^i(X_h^{\det \equiv 1})[\theta|_{T_h^\circ}] \hookrightarrow H_c^i(X_h)[\theta|_{T_h^\circ}] \twoheadrightarrow H_c^i(X_h)[\theta]$ is bijective, it must be an isomorphism of $N_h^{h-1}$-representations (see also \eqref{eq:Gamma h}). Hence to show the theorem, it is enough to show that the trivial character of $N_h^{h-1}$ does not occur in $\sum_i (-1)^i H_c^i(X_h^{\det \equiv 1})[\theta|_{T_h^\circ}]$; that is,
\begin{align}
\nonumber
\Big\langle {\rm triv}, &\textstyle\sum_i (-1)^i H_c^i(X_h^{\det \equiv 1}, \overline \QQ_\ell)[\theta|_{T_h^\circ}] \Big\rangle_{N_h^{h-1}} \\
\label{e:triv inner prod}
&= \frac{1}{\# N_h^{h-1}}\sum_{g \in N_h^{h-1}} \Tr\left(g ; \textstyle \sum_i (-1)^i H_c^i(X_h^{\det \equiv 1}, \overline \QQ_\ell)[\theta|_{T_h^\circ}]\right) = 0.
\end{align}
We now apply Lemma \ref{l:char fixed point} to the $\FF_{q^n}$-scheme $X_h^{\det \equiv 1}$ with $A = T_h^\circ$ and $\varphi \from X_h^{\det \equiv 1} \to X_h^{\det \equiv 1}$ given by $x \mapsto g \cdot x$ for some $g \in N_h^{h-1}$. We see that to show \eqref{e:triv inner prod}, we must show
\begin{equation*}\label{eq:sum_of_traces_vanishes}
\sum_{g \in N_h^{h-1}} \sum_{t \in T_h^\circ} \theta(t) \cdot \# S_{g,t} = 0,  \qquad \text{where $S_{g,t} \colonequals \{x \in X_h^{\det \equiv 1}(\overline \FF_q) : g \cdot \Fr_{q^n}(x) = x \cdot t\}$.}
\end{equation*}

\begin{lm}\label{lm:congruence_of_g_and_t}
Let $g \in N_h^{h-1}$ and $t \in T_h^\circ$ such that $S_{g,t} \neq \varnothing$. Then $t \equiv (-1)^{n'-1} \pmod{\varpi^{h-1}}$ and $\sigma^n(x) \equiv (-1)^{n'-1}x \pmod{\varpi^{h-1}}$ for all $x \in S_{g,t}$.
\end{lm}
\begin{proof}
An element $y \in \sL_0$ lies in $\sL_{0,b}^{\rm adm,rat}$ if and only if $\det g_b^{\red}(y) \in \cO_K^{\times}$, or equivalently, $\ord \det g_b(y) = \ord\det(D_{\kappa,n}) =: c$ and $\sigma(\det g_b(y)) = \det g_b(y)$.
Multiplying by $b$ on both sides, we see that these conditions are equivalent to
\[
\det(b\sigma(y) \,|\, (b\sigma)^2(y) \,|\, \dots \,|\, (b\sigma)^n(y)) = \det(b) \det(g_b(y))\in \varpi^{c + \kappa} \cO^{\times}.
\]
As $b$ is the special representative, $\det(b) = ((-1)^{(n_0 - 1)k_0} \varpi^{k_0})^{n'} = (-1)^{\kappa(n_0 - 1)} \varpi^{\kappa}$, and moreover, $b^n = \varpi^{\kappa}$. Thus the above is equivalent to
\[
\varpi^{\kappa} \det(b\sigma(y) \,|\, (b\sigma)^2(y) \,|\, \dots \,|\, \sigma^n(y)) = (-1)^{\kappa(n_0 - 1)} \varpi^{\kappa} \det(g_b(y)) \in \varpi^{c + \kappa} \cO^{\times}, 
\]
An elementary computation shows $(-1)^{n-1 - \kappa(n_0 - 1)} = (-1)^{n' - 1}$, thus the above is equivalent to
\begin{equation}\label{eq:reformulation_y_in_Ladmrat}
(-1)^{n'-1} \det(\sigma^n(y) \,|\, b\sigma(y) \,|\, (b\sigma)^2(y) \,|\, \dots \,|\, (b\sigma)^{n-1}(y)) =  \det(g_b(y)) \in \varpi^c \cO^{\times}. 
\end{equation}
Let now $x \in S_{g,t} \subseteq X_h$. Denote by $y \in \sL_0^{\rm adm,rat}$ a lift of $x$. As $g \equiv 1 \pmod{\varpi^{h-1}}$, we by assumption have $\sigma^n(y) \equiv yt \pmod{\varpi^{h-1}}$. Thus replacing in \eqref{eq:reformulation_y_in_Ladmrat} $\sigma^n(y)$ by $ty + \varpi^{h-1}\ast$ for some $\ast \in \sL_0$, using the linearity of the determinant in the first column, and the fact that each entry of the $i$th column ($2 \leq i \leq n$) of the matrix on the left hand side of \eqref{eq:reformulation_y_in_Ladmrat} is in $\cO$ divisible by $\varpi^{\left\lfloor \frac{(i-1)\kappa_0}{n_0} \right\rfloor}$ (and $\sum_{i=2}^n \left\lfloor \frac{(i-1)\kappa_0}{n_0} \right\rfloor = c$), we deduce that
\[
(-1)^{n' - 1}t \equiv 1 \pmod{\varpi^{h-1}}.
\]
If $\tilde x \in X_{h-1}^+$ denotes the image of $x$ modulo $\varpi^{h-1}$, we obtain $\sigma^n(\tilde x) = (-1)^{n'-1}\tilde x$.
\end{proof}

Thus for $g,t$ as in the lemma, $S_{g,t} \neq \varnothing$ implies
\[ 
t \in (-1)^{n'-1}T_h^{h-1} \cap T_h^\circ = \{(-1)^{n'-1}(1 + \varpi^{h-1}[a]) \colon a \in \FF_{q^n}, {\rm tr}_{\FF_{q^n/\FF_q}}(a) = 0 \}.
\]
so that, after factoring out the constant $\theta(-1)^{n'-1}$, it remains to show:
\begin{equation}\label{e:fixed point goal}
\sum_{g \in N_h^{h-1}} \sum_{a \in \ker(\FF_{q^n} \rightarrow \FF_q)} \theta(1 + \varpi^{h-1}[a]) \cdot \# S_{g,(-1)^{n'-1}(1 + \varpi^{h-1} [a])} = 0.
\end{equation}

%*******************************************************************************************
%*******************************************************************************************

\subsubsection{}\label{ss:t condition}

Before we can prove \eqref{e:fixed point goal}, we need some preparations. Recall from Section \ref{sec:fibers_as_hypersurface} that one has an intermediate scheme $X_h \twoheadrightarrow X_{h-1}^+ \twoheadrightarrow X_{h-1}$. Define $X_{h-1}^{+,\det \equiv 1}$ to be the subscheme of $X_{h-1}^+$ consisting of $x \in X_{h-1}^+$ with $\det(g_b^{\red}(x)) \equiv 1$ modulo $\varpi^{h-1}$. Then we have a surjection
\begin{equation*}
f \from X_h^{\det \equiv 1} \to X_{h-1}^{+, \det \equiv 1}
\end{equation*}
and by Proposition \ref{prop:polynomial_P_describing_the_fiber}, $X_h^{\det \equiv 1} \hookrightarrow X_{h-1}^{+, \det \equiv 1} \times \bA^{n'}$ is the (relative) hypersurface given by 
\begin{equation*}
\sum_{i=0}^{n_0-1} \sigma^i(P_1) + c = 0,
\end{equation*}
where $c \from X_h \to \bA^1$ factors through $f$ and $P_1$ is a polynomial over $X_{h-1}^+$ in the variables $x_{i,h-1}$ for $i \equiv 1 \pmod{n_0}$.

%*******************************************************************************************
%*******************************************************************************************

\subsubsection{}\label{ss:x'}

By Lemma \ref{lm:congruence_of_g_and_t}, for $g \in N_h^{h-1}$ and $t \in (-1)^{n'-1}(T_h^{h-1} \cap T_h^\circ)$ with $S_{g,t} \neq \varnothing$, we have $S_{g,t} \subseteq f^{-1}(S_{h-1})$, where 
\begin{equation*}
S_{h-1} \colonequals \left\{\widetilde x \in X_{h-1}^{+, \det\equiv1}: \sigma^n(\widetilde x) = (-1)^{n'-1} \widetilde x\right\} \subset X_{h-1}^{+, \det=1}
\end{equation*}
is a finite set of points. Regard $S_{h-1}$ as a (zero-dimensional, reduced) subscheme of $X_{h-1}^{+,\det \equiv 1}$. Consider the $S_{h-1}$-morphism $S_{h-1} \times \bA^{n'} \stackrel{\sim}{\rightarrow} S_{h-1} \times \bA^{n'}$, which is the linear change of variables defined by
\begin{equation*}
(x_{1,h-1}, x_{n_0+1,h-1}, \dots, x_{n_0(n'-1)+1,h-1})^{\trans} = \overline{g_b}(\bar x) (z_1, z_2, \ldots, z_{n'})^{\trans},\end{equation*}
where $\bar{x}$ is the image of $\tilde{x} \in S_{h-1}$ in $X_1$ and $\overline{g_b}(\bar x)$ is as in Section \ref{sec:projection_to_reductive_quotient}. 

\begin{claim*}
$X_h^{\det \equiv 1}$ is a (relative) hypersurface over $X_{h-1}^{+,\det\equiv1}$ defined by an equation $\sum_{i=0}^{n-1} z_1^{q^i} = c$.
\end{claim*}

It is enough to show that in the new coordinates $z_1, \ldots, z_{n'}$, the polynomial $P_1$ as in Proposition \ref{prop:polynomial_P_describing_the_fiber}  takes the form $P_1 = \sum_{i=0}^{n'-1} z_1^{q^{n_0i}}$. We prove this now.

Recall the $n'$-dimensional $\FF_{q^{n_0}}$-vector space $\overline{V}$ with its distinguished basis $\{\bar e_{n_0(i-1)+1}\}_{1\leq i \leq n'}$ and the $\FF_{q^{n_0}}$-linear morphism $\overline{\sigma_b}$ of $\overline{V}$ from Section \ref{lm:defin_of_V}. To simplify notation, we write $\bar{x}_j$ instead of $x_{n_0(j-1)+1,0}$ for $x \in X_h$ and $1\leq j \leq n'-1$ in what follows (i.e., the image of $x \in X_h$ in $\overline{V}$ is $\bar{x} = (\bar{x}_i)_{i=1}^{n'}$). Recall from Section \ref{sec:projection_to_reductive_quotient} that for $\bar x \in \overline{V}$ the $i$th column of the $(n' \times n')$-matrix $\overline{g_b}(\bar{x})$ is  $\overline{\sigma_b}^{i-1}(\bar x)$. Let $m_i$ denote the $i$th row of the adjoint matrix $(m_{ij})$ of $\overline{g_b}(\bar x)$. Then the above change of variables gives
\begin{align*} \label{eq:poly_P_basechanged_in_general}
P_1 &= (m_1 \cdot \bar{x})z_1  + (m_2 \cdot \bar{\sigma}_b(\bar x))z_1^{q^{n_0}}  +  (m_3 \cdot \bar\sigma_b^2(\bar x)) z_1^{q^{2n_0}} + \dots +  (m_{n'} \cdot \bar\sigma_b^{n'-1}(\bar x)) z_1^{q^{n_0(n'-1)}}  \\
&+  (m_1 \cdot \overline\sigma_b(\bar x)) z_2 +  (m_2 \cdot \overline{\sigma_b}^2(\bar x)) z_2^{q^{n_0}} +  (m_3 \cdot \bar\sigma_b^3(\bar x)) z_2^{q^{2n_0}} + \dots +  (m_{n'} \cdot \bar\sigma_b^{n'}(\bar x)) z_2^{q^{n_0(n'-1)}}  \\
  &+ \cdots +  (m_1 \cdot \overline{\sigma_b}^{n'-1}(\bar x)) z_{n'} +  (m_2 \cdot \overline{\sigma_b}^{n'}(\bar x)) z_{n'}^{q^{n_0}} +   \dots +  (m_n \cdot \overline{\sigma_b}^{2n'-2}(\bar x)) z_{n'}^{q^{n_0(n'-1)}}.
\end{align*}
(Here $\cdot$ denotes the matrix product.) But $\sigma^n(\bar x) = (-1)^{n'-1} \bar x$ (where $\sigma^n$ is applied entry-wise), and hence from the explicit form of $\overline{\sigma_b}$ we deduce that $\overline{\sigma_b}^{n'}(\bar x) = (-1)^{n'-1}\bar{x}$. As $(m_{ij})$ is adjoint to $\overline{g_b}(\bar x)$ and $\det(\overline{g_b}(\bar x)) = 1 \in \FF_q^{\times}$, we have 
\begin{equation*}
m_i \cdot \overline{\sigma_b}^j(\bar x) = \begin{cases}
1 & \text{if $j = i-1$,} \\
0 & \text{otherwise.}
\end{cases}
\end{equation*}
This shows that all coefficients are equal to $1$ in the first line of the above expression and vanish in lines $2, \dots, n'$. This completes the proof of the claim.

%*******************************************************************************************
%*******************************************************************************************

\subsubsection{}\label{sec:actions_unipotent_traces}

Note that $N_h^{h-1} \times (T_h^{h-1} \cap T_h^\circ)$ stabilizes $S_{h-1}$ and acts trivially on it.
We describe the action of $N_h^{h-1} \times (T_h^{h-1} \cap T_h^\circ)$ on the new coordinates $z_1, \ldots, z_{n'}$. 

Let $g \in N_h^{h-1}$ and for $(i,j) \in J$ as in Section \ref{ss:1}, let $[u_{i,j}]\varpi^{h-1}$ denote the upper left entry of the $(i,j)$th $n_0\times n_0$-block of $g$. Recall that the action of $g$ on $f^{-1}(S_{h-1})$ in the old coordinates $x_{1,h-1}, \ldots, x_{n_0(n'-1) + 1, h-1}$ is given by
\begin{equation*}
g.(x_{n_0(i-1),h-1})_{i=1}^{n'} = \Big( x_{n_0(i-1),h-1} + \textstyle\sum\limits_{\substack{ 1\leq j \leq n' \\ (i,j) \in J}} u_{i,j} \bar x_j \Big)_{i=1}^{n'}.
\end{equation*}
Since $\det(\overline{g_b}(\overline x)) = 1$ by assumption, the adjoint matrix $(m_{ij})_{ij}$ of $\overline{g_b}(\overline x)$ is in fact the inverse, so that $z_i = \sum_{j=1}^{n'} m_{ij}x_{n_0(j-1) + 1,h-1}$. Thus the action of $g$ on the new coordinates is given by
\begin{equation*}
z_i \mapsto z_i + \textstyle\sum\limits_{\substack{1\leq k,j\leq n' \\ (k,j) \in J}} m_{ik} u_{kj}  \bar x_j.
\end{equation*}

We now describe the action of $(T_h^{h-1} \cap T_h^\circ) = \{1 + [a] \varpi^{h-1} \colon \Tr_{\FF_{q^n}/\FF_q}(a) = 0\}$. For $a \in \ker(\Tr \from \FF_{q^n} \to \FF_q)$, the action of $1 + [a] \varpi^{h-1}$ on the old coordinates is given by 
\begin{equation*}
x_{n_0(i-1) + 1,h-1} \mapsto 
x_{n_0(i-1) + 1,h-1} + a \bar x_i.
\end{equation*}
Since $\sum_{k=1}^n m_{ik} \bar x_k$ is equal to $1$ when $i = 1$ and equal to $0$ when $i > 1$, the action of $1 + [a] \varpi^{h-1}$ on the new coordinates is given by
\begin{equation*}
z_i \mapsto
\begin{cases}
z_1 + a & \text{if $i = 1$,} \\
z_i & \text{if $i = 2, \ldots, n'$.}
\end{cases}
\end{equation*}
Moreover, $x \mapsto \sigma^n(x) \cdot (-1)^{n'-1}$ defines an isomorphism of each fiber $f^{-1}(\tilde x)$, and one computes that in coordinates $z_i$ it is given by $z_i \mapsto \sigma^n(z_i)$. Thus for $t = (-1)^{n'-1}(1+\varpi^{h-1}[a]) \in (-1)^{n'-1}(T_h^{h-1} \cap T_h^\circ)$, the assignment $x \mapsto \sigma^n(x)\cdot t$ defines an isomorphism of $f^{-1}(\tilde x)$ which in the coordinates $z_i$ is given by $z_1 \mapsto \sigma^n(z_1) + a$, $z_i \mapsto \sigma^n(z_i)$ for $2 \leq i \leq n'$.

%*******************************************************************************************
%*******************************************************************************************

\subsubsection{}\label{ss:c}

We next claim that $c(\widetilde x) \in \FF_q$ for $\tilde{x} \in S_{h-1}$. Consider the ``extension by zero'' morphism $\bW_{h-1} \to \bW$ given by $\sum_{i=0}^{h-2} [a_i] \varpi^i \mapsto \sum_{i=0}^{h-2} [a_i] \varpi^i$. It defines a map $\sL_0/\varpi^{h-1} \sL_0 \to \sL_0$, $y \mapsto [y,0]$. To show the claim it is sufficient to show that $[\tilde x, 0]$ lies in $\sL_{0,b}^{\rm adm, rat}$. Obviously, $\det g_b^{\red}([\tilde x, 0]) \in \cO$. Now, note that as $\tilde x \in X_{h-1}^+$, there exists some lift $z \in \sL_{0,b}^{\rm adm, rat}$ of $\tilde x$. This gives in particular $\det g_b^{\red}(z) \equiv \det g_b^{\red}([\tilde x, 0]) \pmod{\varpi^{h-1}}$. We deduce $\det g_b^{\red}([\tilde x, 0]) \in \cO^{\times}$. It remains to show that $\det g_b^{\red}([\tilde x, 0]) \in K$. To do this, it suffices to prove that $\det g_b([\tilde x, 0]) \in K$, as $\det g_b^{\red}(\cdot)$ and $\det g_b(\cdot)$ differ only  by a power of $\varpi$. But as $\sigma^n(\widetilde x) = (-1)^{n'-1} (\widetilde x)$, we have $\sigma^n([\widetilde x,0]) = (-1)^{n'-1} [\widetilde x,0]$. Using this and $\det(b) = (-1)^{\kappa(n_0-1)}\varpi^k$ we compute:
\begin{align*}
(-1)^{\kappa(n_0-1)}\varpi^k \sigma(\det g_b([\tilde x, 0])) &= \det b\sigma(g_b([\tilde x, 0])) \\
&= \det \left( b\sigma([\tilde x, 0]) \, | \, (b\sigma)^2([\tilde x, 0]) \, | \, \dots \, | \, \varpi^k \sigma^n([\tilde x, 0])\right) \\
&= (-1)^{(n-1) + (n'-1)} \varpi^{\kappa} \det g_b([\tilde x, 0]).
\end{align*}
But as in the proof of Lemma \ref{lm:congruence_of_g_and_t}, we have $(-1)^{\kappa(n_0-1)} = (-1)^{(n-1) + (n'-1)}$. This shows the claim.

%*******************************************************************************************
%*******************************************************************************************

\subsubsection{}\label{ss:nonempty condition}

Fix $\tilde x \in S_{h-1}$ and $t = (-1)^{n'-1}(1+ \varpi^{h-1} [a])$ with $\Tr_{\FF_{q^n}/\FF_q}(a) = 0$ (as in Equation \eqref{e:fixed point goal}). We see that a point $x \in f^{-1}(\tilde x)$ with coordinates $(z_i)_{i=1}^{n'}$ as in Section \ref{ss:x'} lies in $S_{g,t} \cap f^{-1}(\widetilde x)$ if and only if 
\begin{equation*}
g \cdot \sigma^n(x) \cdot t^{-1} = x \qquad \text{and} \qquad
z_1 + z_1^q + \cdots + z_1^{q^{n-1}} = c(\widetilde x).
\end{equation*} 
By Section \ref{sec:actions_unipotent_traces}, the first equation is equivalent to (use that $\sigma^n(m_{i,k}) = -m_{i,k}$, $\sigma^n(\bar{x}_j) = (-1)^{n'-1}\bar{x}_j$)
\begin{equation*}
z_1^{q^n} + \sum_{(k,j) \in J} m_{1k}u_{kj} \bar x_j = z_1 + a,
\end{equation*}
along with similar equations for the $(z_i)_{i=2}^{n'}$ (of the form $z_i^{q^n} + \text{(sum of terms)} = z_i$). Since $c(\widetilde x) \in \FF_q$ by Section \ref{ss:c}, the second equation implies $z_1 = z_1^{q^n}$, which eliminates $z_1$ from the first equation. Therefore $S_{g,t} \cap f^{-1}(\widetilde x) \neq \varnothing$ if and only if
\begin{equation}\label{e:nonempty condition}
\psi(\widetilde x,g) = a, \qquad \text{where 
$\psi(\widetilde x,g) \colonequals \textstyle\sum_{(k,j) \in J} m_{1k}u_{kj} \bar x_j.$}
\end{equation}
Moreover, since the $n'-1$ equations for $(z_i)_{i=2}^{n'}$ is a separable polynomial in $z_i$, each gives precisely $q^n$ choices for $z_i$, $2 \leq i \leq n'$, with no further conditions. Thus
\begin{align*}
S_{g,t} \cap f^{-1}(\widetilde x) \neq 0 \quad &\Longleftrightarrow \quad \text{\eqref{e:nonempty condition} holds} \\
&\Longleftrightarrow \quad \#(S_{g,t} \cap f^{-1}(\widetilde x)) = \underbrace{q^{n-1}}_{\text{for $z_1$}}(\underbrace{q^n}_{\text{for $z_i$, $2 \leq i \leq n'$}})^{n'-1} = q^{nn'-1}.
\end{align*}
This shows the following lemma.

\begin{lemma}\label{l:claim i}
For $g \in N_h^{h-1}$, $t = (-1)^{n'-1} (1 + \varpi^{h-1}[a])$ with $\Tr_{\FF_{q^n}/\FF_q}(a) = 0$, and $\tilde{x} \in S_{h-1}$, 
\[
\# S_{g,t} \cap f^{-1}(\tilde x) = \begin{cases} q^{nn' - 1} & \text{if $\psi(g, \tilde x) = a$,} \\ 0 & \text{otherwise.} \end{cases}
\]
\end{lemma}

For $a \in \ker(\Tr \colon \FF_{q^n} \rightarrow \FF_q)$, put
\[
B_{g,a} \colonequals \{\tilde x \colon \psi(g,\tilde x) = a \} \subseteq S_{h-1}.
\]
As $S_{g,t} = \bigsqcup_{\widetilde x \in S_{h-1}} S_{g,t} \cap f^{-1}(\widetilde x)$, Lemma \ref{l:claim i} implies that 
\[ \# S_{g,(-1)^{n'-1}(1 + \varpi^{h-1}[a])} = q^{nn' - 1} \cdot \# B_{g,a}.\]
Thus the left hand side of \eqref{e:fixed point goal} is 
\begin{equation}\label{eq:cusp_goal_2}
q^{nn'-1} \cdot \sum_{a \in \ker(\FF_{q^n} \rightarrow \FF_q)} \sum_{g \in N_h^{h-1}} \theta(1 + \varpi^{h-1}[a]) \cdot \# B_{g,a}.
\end{equation}

%*******************************************************************************************
%*******************************************************************************************

\subsubsection{}\label{ss:Sgt_vanishes_for_t_not_kernn0}
We have the following lemma.

\begin{lm}\label{lm:Sgt_empty_for_t_not_in_kernn0}
Let $g \in N_h^{h-1}$ and let $t = (-1)^{n'-1}(1 + \varpi^{h-1}[a])$ with $\Tr_{\FF_{q^n}/\FF_q}(a) = 0$. We have $S_{g,t} = \varnothing$, unless $\Tr_{\FF_{q^n}/\FF_{q^{n_0}}}(a) = 0$.
\end{lm}

\begin{proof}
It is enough to show that if $\Tr_{\FF_{q^n}/\FF_{q^{n_0}}}(a) \neq 0$, then $S_{g,t} \cap f^{-1}(\tilde x) = \varnothing$ for all $\tilde x \in S_{h-1}$. By Lemma \ref{l:claim i}, it is enough to show that for all $g \in N_h^{h-1}$ and $\tilde x \in S_{h-1}$, we have $\psi(\tilde x, g) \in \ker(\Tr_{\FF_{q^n}/\FF_{q^{n_0}}})$. Fix such $g$ and $\tilde x$ and let $\varpi^{h-1} u_{k,j}$ ($(k,j) \in J$) denote the entries of $g$ (as in beginning of Section \ref{sec:actions_unipotent_traces}). As $u_{k,j} \in \FF_{q^{n_0}}$, and as $k\neq j$ holds for all pairs $(k,j) \in J$, it suffices to show that $m_{1,k}\bar x_j \in \ker(\Tr_{\FF_{q^n}/\FF_{q^{n_0}}})$ if $k \neq j$. Since $\tilde{x} \in S_{h-1}$, one computes $m_{k,l} = m_{k-1,l}^{q^{n_0}}$. Thus $\Tr(m_{1,k}\bar x_j)$ is precisely the $(j,k)$th entry of the matrix $\overline{g_b}(\bar{x}) \cdot m$, which is equal $0$. 
\end{proof}

By Lemma \ref{lm:Sgt_empty_for_t_not_in_kernn0} and \eqref{eq:cusp_goal_2}, we have reduced showing \eqref{e:fixed point goal} to showing 
\begin{equation}\label{e:fixed point goal_3}
\sum_{a \in \ker(\FF_{q^n} \rightarrow \FF_{q^{n_0}})} \sum_{g \in N_h^{h-1}}  \theta(1 + \varpi^{h-1}[a]) \cdot \# B_{g,a} = 0
\end{equation}

%*******************************************************************************************
%*******************************************************************************************

\subsubsection{}\label{ss:claim ii}

For $\tilde x \in S_{h-1}$, consider the $\FF_{q^{n_0}}$-vector subspace $W(\tilde{x}) \colonequals \langle m_{1,i} \bar{x}_j | (i,j) \in J \rangle \subseteq \ker(\FF_{q^n} \rightarrow \FF_{q^{n_0}})$. The left hand side of \eqref{e:fixed point goal_3} is
\begin{align}
\nonumber &\sum_{a \in \ker(\FF_{q^n} \rightarrow \FF_{q^{n_0}})} \sum_{g \in N_h^{h-1}}  \theta(1 + \varpi^{h-1}[a]) \cdot \# B_{g,a}   \\
\label{eq:trace_comp_unipotents_interm} &\qquad= \sum_{a\in \ker(\FF_{q^n} \rightarrow \FF_{q^{n_0}})} \sum_{\tilde x \in S_{h-1}} \theta(1 + \varpi^{h-1}[a]) \cdot \# \{ g \in N_h^{h-1} \colon \psi(g,\tilde x) = a \} \\
\nonumber &\qquad= \sum_{a\in \ker(\FF_{q^n} \rightarrow \FF_{q^{n_0}})} \sum_{ W \subseteq \ker(\FF_{q^n} \rightarrow \FF_{q^{n_0}})} \sum_{\tilde x \colon W(\tilde{x}) = W} \theta(1 + \varpi^{h-1}[a]) \cdot \# \{ g \in N_h^{h-1} \colon \psi(g,\tilde x) = a \}. 
\end{align}
Now fix some $W$ and $\tilde{x} \in S_{h-1}$ such that $W(\tilde x) = W$. Then $\{m_{1,i}\bar{x}_{j}\}_{(i,j) \in J}$ span the $\FF_{q^{n_0}}$-vector space $W$, and from the explicit form \eqref{e:nonempty condition} of $\psi(g,a)$, it is clear that 
\[
\#\{g \in N_h^{h-1} \colon \psi(g,\tilde{x}) = a \} = \begin{cases} q^{n_0(\# J - \dim W)} & \text{if $a \in W$,}\\ 0 & \text{otherwise.} \end{cases}
\]
Note that $\#\{g \in N_h^{h-1} : \psi(g, \tilde x) = a\}$ depends only on $W(\tilde x)$ and not on $\tilde x$ itself. Thus, if we set $S_{h-1,W} \colonequals \{\tilde x \in S_{h-1} \colon W(\tilde{x}) = W\}$, then \eqref{eq:trace_comp_unipotents_interm} is equal to 
\begin{align*}
\sum_W \#S_{h-1,W} \cdot q^{n_0(\#J - \dim W)} \cdot \sum_{a \in W} \theta(1 + \varpi^{h-1}[a]).
\end{align*}
But as $\theta$ is assumed to be primitive in the theorem, this expression is equal to $0$ once we show the following lemma:
%, once we can show that the only $\FF_{q^{n_0}}$-spaces $W \subseteq \ker({\rm tr}_{n/n_0} \colon \FF_{q^n} \rightarrow \FF_{q^{n_0}})$ for which there exists a $\tilde x \in S_{h-1}$ with $W(\tilde x) = W$, are of the form $\ker({\rm tr}_{n/n_0r} \colon \FF_{q^n} \rightarrow \FF_{q^{n_0r}})$ for some $r<n'$, dividing $n'$. But this is precisely Lemma \ref{lm:which_W_can_occur}.

\begin{lm}\label{lm:which_W_can_occur}
Let $\tilde x \in S_{h-1}$. Then there is some $r \mid n'$, $r < n'$ such that $W(\tilde x) = \ker(\Tr_{\FF_{q^n}/\FF_{q^{n_0r}}})$.
\end{lm}

\begin{proof}
Write $W = W(\tilde x)$. Consider the perfect symmetric $\FF_{q^{n_0}}$-bilinear trace pairing 
\begin{equation*}
\FF_{q^n} \times \FF_{q^n} \to \FF_{q^{n_0}}, \qquad (x,y) \mapsto \Tr_{\FF_{q^n}/\FF_{q^{n_0}}}(xy).
\end{equation*}
It is an immediate computation that $\ker(\Tr_{\FF_{q^n}/\FF_{q^{n_0r}}})^\perp = \FF_{q^{n_0r}}$ for any divisor $r$ of $n'$, so we need to show that $W^{\perp}$ is of the form $\FF_{q^{n_0r}}$ for some $r<n'$. For this, it suffices to show that $W^{\perp}$ is an $\FF_{q^{n_0}}$-algebra, which is properly contained in $\FF_{q^n}$. 

First of all note that $W^{\perp}$ contains $1$ since for all $(i,j) \in J$, $m_{1,i} \bar{x}_j \in \ker(\Tr_{\FF_{q^n}/\FF_{q^{n_0}}})$ (as in the proof of Lemma \ref{lm:Sgt_empty_for_t_not_in_kernn0}). Since $W^\perp$ is an $\FF_{q^{n_0}}$-vector space and contains $1$, it must contain $\FF_{q^{n_0}}$. It remains to show that $W^{\perp}$ is closed under multiplication. We now use that $J$ is of the form
\[ 
J = \{(i,j)\colon 1\leq i \leq n' -\ell, \, n'-\ell+1 \leq j \leq n, \} 
\]
for some $1\leq \ell \leq n'$ (see Section \ref{ss:1}). For a fixed $n'-\ell +1 \leq j \leq n'$, let 
\[
L_j \colonequals {\rm span}_{\FF_{q^{n_0}}} \left\langle m_{1,i} \bar{x}_j \colon 1\leq i \leq n'-\ell \right\rangle^{\perp}. 
\]
Observe that the $m_{1,i}$ are all $\FF_{q^{n_0}}$-linearly independent (since $b$ is the special representative) and hence $L_j$ has dimension $n'-\ell$. For $1 \leq i \leq n' - \ell$ and $n' - \ell + 1 \leq i' \leq n'$ we have
\[
\Tr_{\FF_{q^n}/\FF_{q^{n_0}}} \left( m_{1,i}\bar x_j \cdot \frac{\bar x_{i'}}{\bar x_j} \right) = \Tr_{\FF_{q^n}/\FF_{q^{n_0}}}( m_{1,i}\bar x_{i'}) = 0,
\]
as in the proof of Lemma \ref{lm:Sgt_empty_for_t_not_in_kernn0}. This implies the inclusion ``$\supseteq$'' in the formula
\[
L_j = {\rm span}_{\FF_{q^{n_0}}} \left\langle \frac{\bar{x}_{i'}}{\bar{x}_j} \colon n' -\ell + 1 \leq i' \leq n' \right\rangle.
\]
The other inclusion follows by dimension reasons. As $W$ is generated by all $L_j^\perp$ ($n' - \ell + 1\leq j \leq n'$), we have $W^{\perp} = \bigcap_{j=n-\ell + 1}^n L_j$. Let $v,w \in W^{\perp}$. We need to show that $vw \in W^{\perp}$, i.e., that for all $(i_0,j_0) \in J$ we have $\Tr_{\FF_{q^n}/\FF_{q^{n_0}}}(m_{1,i_0}\bar{x}_{j_0} v w) = 0$. As $v\in L_{j_0}$, we may write $v = \sum_{(a,j_0) \in J} v_a \cdot \frac{\bar x_i}{\bar x_{j_0}}$ with $v_a \in \FF_{q^{n_0}}$. Then
\begin{align*} 
\Tr_{\FF_{q^n}/\FF_{q^{n_0}}}(m_{1,i_0}\bar{x}_{j_0} v w) &= 
\Tr_{\FF_{q^n}/\FF_{q^{n_0}}}\Big(m_{1,i_0}\bar{x}_{j_0} \Big(\textstyle\sum\limits_{(a,j_0) \in J} v_a \cdot \frac{\bar x_a}{\bar x_{j_0}}\Big) w\Big) \\ 
&= \textstyle\sum\limits_{(a,j_0) \in J} v_a \cdot \Tr_{\FF_{q^n}/\FF_{q^{n_0}}}\left(m_{1,i_0} \bar x_a w \right) = 0,
\end{align*}
where the last equality holds since $w \in W^{\perp}$ is orthogonal to each $L_a^{\perp}.$
\end{proof}

%************************************************************************************************************************************************************
%************************************************************************************************************************************************************

\newpage

\part{Automorphic induction and the Jacquet--Langlands correspondence}\label{part:aut ind and JL}

In this part, we use the results of Parts \ref{part:GLn DL} and \ref{part:cohomology} to study the $\ell$-adic homology groups of the semi-infinite Deligne--Lusztig variety $\dot X_{\dot w}^{DL}(b)$, which by Theorem \ref{cor:infty_level_adlv} along with Corollary \ref{cor:scheme_structure_on the_inverse_limit} is isomorphic to the affine Deligne--Lusztig variety at infinite level $\dot X_w^\infty(b)$ constructed in Section \ref{sec:comparison_DL_ADLV_isocrystal}. In Section \ref{s:henniart}, we recall methods of Henniart characterizing certain representations by considering the action of very regular elements. In Section \ref{s:inf ADLV}, we define the homology of $\dot X_w^\infty(b) \cong X_w^{DL}$ and give a representation-theoretic description of 
\begin{equation*}
R_T^G(\theta) \colonequals \sum_{i \geq 0} (-1)^i H_i(\dot X_w^\infty(b), \overline \QQ_\ell)[\theta] \qquad \text{for $\theta \from T = L^\times \to \overline \QQ_\ell^\times$ smooth}
\end{equation*}
in terms of the cohomology of the finite-type variety $X_h$ studied in the previous two parts of the paper. Using methods of Henniart as reviewed in Section \ref{s:henniart}, we prove Theorem \ref{t:LLC JLC gen}: if $|R_T^G(\theta)|$ is irreducible supercuspidal, then the assignment $\theta \mapsto |R_T^G(\theta)|$ realizes automorphic induction. To finish, we prove in Section \ref{s:minimal} that when $\theta \from L^\times \to \overline \QQ_\ell^\times$ is \textit{minimal admissible}, then $|R_T^G(\theta)|$ is irreducible supercuspidal.

\mbox{}

We now give some basic definitions which we will use throughout the next few sections. Recall that for any smooth character $\theta \from L^\times \to \overline \QQ_\ell^\times$, there exists an integer $h \geq 1$ such that $\theta$ is trivial on $U_L^h = 1 + \varpi^h \cO_L$. We call the smallest such $h$ the \textit{level of $\theta$}. We say that $\theta$ is \textit{in general position} if its stabilizer in $\Gal(L/K)$ is trivial. Let $\sX$ denote the set of such characters.

We say that an element $x$ of $L^\times$ is \textit{very regular} if $x \in \cO_L^\times$ and its image in the residue field $\FF_{q^n}$ generates its multiplicative group $\FF_{q^n}^{\times}$.

We say that a virtual representation is a \textit{genuine representation} if it is a nonnegative linear combination of irreducible representations. If $R$ is a virtual representation that is $\pm \pi$, where $\pi$ is a genuine representation, we write $|R| = \pi$.

\section{Results of Henniart on the Local Langlands Correspondence} \label{s:henniart}

In this section, we review the methods of Henniart \cite{Henniart_92,Henniart_93} characterizing certain cases of automorphic induction by considering the action of very regular elements. We give a generalization of the discussions of \cite{BoyarchenkoW_13} to all inner forms of $\GL_n(K)$. There are no technical difficulties in doing this, but we provide it for completeness of our paper.

Fix a character $\epsilon$ of $K^\times$ with $\ker(\epsilon) = \Nm_{L/K}(L^\times)$, and let $\cG_K^\epsilon(n)$ denote the set of irreducible $n$-dimensional representations $\sigma$ of the Weil group $\cW_K$ such that $\sigma \cong \sigma \otimes (\epsilon \circ {\rm rec}_K^{-1})$, where ${\rm rec}_K \colon K^{\times} \rightarrow \cW_K^{ab}$ is the reciprocity isomorphism from local class field theory.  It is known that every element of $\cG_K^\epsilon(n)$ is of the form $\Ind_{\cW_L}^{\cW_K}(\theta)$ for some character $\theta \in \sX$. However, it is also known that automorphic induction is not compatible with induction on Weil groups in the sense that the Langlands parameter may have a twist by a rectifying character. Hence the approach we take is via the $\chi$-datum of Langlands--Sheldstad \cite[Section 2.5]{LanglandsS_87}. Because $L/K$ is unramified, there is a canonical choice of $\chi$-datum, and this gives rise to a bijection
\begin{equation*}
\sX/\Gal(L/K) \to \cG_K^\epsilon(n), \qquad \theta \mapsto \sigma_\theta.
\end{equation*}
See \cite[Section 7.2]{Chan_DLII} for an exposition and an explicit discussion of the unramified setting. Note that $\sigma_\theta$ differs from the notation of \cite{BoyarchenkoW_13} by a rectifying character.

Let $\cA_K^\epsilon(\GL_n)$ denote the set of isomorphism classes of irreducible supercuspidal representations $\pi$ of $\GL_n(K)$ such that $\pi \cong \pi \otimes (\epsilon \circ \det)$. There is a canonical bijection
\begin{equation*}
\cG_K^\epsilon(n) \stackrel{\rm LLC}{\longrightarrow} \cA_K^\epsilon(\GL_n), \qquad \sigma_\theta \mapsto \pi_\theta
\end{equation*}
satisfying certain properties. By work of Henniart, the character of $\pi_\theta$ is very nicely behaved on certain elements of $\GL_n(K)$. 

Now let $G$ be an inner form of $\GL_n(K)$ so that $G \cong \GL_{n'}(D_{k_0/n_0})$, where $D_{k_0/n_0}$ is the division algebra of dimension $n_0^2$ over $K$ with Hasse invariant $k_0/n_0$. Let $\cA_K^\epsilon(G)$ denote the set of isomorphism classes of irreducible supercuspidal representations $\pi'$ of $G$ such that $\pi' \cong \pi' \otimes (\epsilon \circ \det)$. By the Jacquet--Langlands correspondence, there is a canonical bijection
\begin{equation*}
\cA_K^\epsilon(\GL_n) \stackrel{\rm JLC}{\longrightarrow} \cA_K^\epsilon(G), \qquad \pi \mapsto \pi' \colonequals \JL(\pi)
\end{equation*}
such that the central characters of $\pi$ and $\pi'$ match and such that their characters on regular semisimple elements differs by $(-1)^{n-n'}$.

\begin{remark}
We remark that the notation $\pi_\theta$ agrees with the $\pi(\theta)$ of \cite{Henniart_93}, but with the $\pi'(\theta)$ (rather than the $\pi(\theta)$) of \cite{Henniart_92}. When $n$ is odd, there is no discrepancy, but when $n$ is even, our $\pi_\theta$ is the representation $\pi_{\theta \omega} = \pi'(\theta)$ in \cite{Henniart_92}, where $\omega$ is the unique unramified character of $L^\times$ of order $2$. \hfill $\Diamond$
\end{remark}

The following theorem can be found in \cite[Section 3.14]{Henniart_92}.

\begin{theorem}[Henniart]\label{t:henniart without c}
For each $\theta \in \sX$, there exists a constant $c_\theta = \pm 1$ such that
\begin{equation*}
\Tr \JL(\pi_\theta)(x) = c_\theta \cdot \sum_{\gamma \in \Gal(L/K)} \theta^\gamma(x)
\end{equation*}
for every very regular element $x \in L^\times \subset \GL_n(K)$.
\end{theorem}

As we will see momentarily, one can even go the other direction: the trace of $\pi \in \cA_K^\epsilon(\GL_n)$ on very regular elements of $L^\times$ characterizes $\pi$.  Furthermore, $c_\theta$ can be pinpointed for $\GL_n(K)$ by \cite[Theorem 3.14]{Henniart_92} and extended to any inner form of $\GL_n(K)$ via the character condition of the Jacquet--Langlands correspondence.  For each positive integer $r$ and each $\theta \in \sX$, consider the subgroup of $\Gal(L/K)$ given by $\sG_{\theta,r} \colonequals \{\gamma \in \Gal(L/K) : a(\gamma) \leq r\}$, where $a(\gamma)$ is the level of $\theta/\theta^\gamma$. 

\begin{theorem}[Henniart]\label{t:henniart c}
The constant $c_\theta$ of Theorem \ref{t:henniart without c} satisfies
\begin{equation*}
(-1)^{n-n'} c_\theta = \begin{cases}
+ 1 & \text{if $n$ is odd,} \\
+ 1 & \text{if $n$ is even and $s$ is even,} \\
- 1 & \text{if $n$ is even and $s$ is odd,}
\end{cases}
\end{equation*}
where $s$ is such that $\sG_{\theta,s} \smallsetminus \sG_{\theta,s-1}$ contains the unique element of order $2$ in $\Gal(L/K)$.
\end{theorem}

\begin{lemma}[Henniart]
\label{l:theta theta' gen}
Let $\theta \in \sX$ and suppose that there exists a character $\theta'$ of $L^\times$ (a priori, not necessarily in $\sX$) such that $\theta(\varpi) = \theta'(\varpi)$ and 
\begin{equation}\label{e:galois average}
c \cdot \sum_{\gamma \in \Gal(L/K)} \theta^\gamma(x) = c' \cdot \sum_{\gamma \in \Gal(L/K)} \theta'{}^\gamma(x)
\end{equation}
for all very regular elements of $x \in L^\times$. Assume in addition that $c = c'$ in the special case $n = 2$, $q = 3$, and $\theta|_{U_L^1}$ factors through the norm $U_L^1 \to U_K^1$ (i.e. $\theta \in \sX^0$ with notation from Section \ref{s:minimal}). Then 
\begin{equation*}
\theta' = \theta^\gamma \qquad \text{for some $\gamma \in \Gal(L/K)$.}
\end{equation*}
\end{lemma}

\begin{proof}
We provide the proof in the case that $\theta|_{U_L^1}$ has trivial $\Gal(L/K)$-stabilizer, following \cite[Section 5.3]{Henniart_93} (see also \cite[Lemma 1.7]{BoyarchenkoW_13}). This is the simplest setting. In \cite[Section 5.3]{Henniart_93}, Henniart proves the lemma for $\theta \in \sX$ in the case $[L:K]$ is prime by essentially the arguments presented here. A significantly more involved incarnation of these arguments is used in \cite[Identity (2.5), Sections 2.6--2.12]{Henniart_92} to prove the lemma in full generality as stated.

We first show that the conclusion holds on $U_L^1$. Fix a very regular element $x \in L^\times$. Since every element of $x U_L^1 \subset L^\times$ is a very regular element, the assumption implies that we have an equation of linear dependence between the $2n$ characters of $U_L^1$ given by the restrictions of the $\Gal(L/K)$-translates of $\theta$ and $\theta'$. Explicitly: on $U_L^1$, we have
\begin{equation*}
\theta' = c'{}^{-1} \theta'(x)^{-1} \cdot \left(\sum_{\gamma \in \Gal(L/K)} c\theta^\gamma(x) \cdot \theta^\gamma -  \sum_{1 \neq \gamma \in \Gal(L/K)} \theta'(x)^{-1} \theta'{}^\gamma(x) \cdot \theta'{}^\gamma\right).
\end{equation*}
Considering the character inner product of $\theta'$ with $\theta^{\gamma'}$ on $U_L^1$ for some fixed $\gamma' \in \Gal(L/K)$, we have:
\begin{equation*}
\langle \theta', \theta^{\gamma'} \rangle = \frac{c \cdot \theta^{\gamma'}(x)}{c' \cdot \theta'(x)} - c' \cdot \sum_{1 \neq \gamma \in \Gal(L/K)} \theta'{}^\gamma(x) \langle \theta'{}^\gamma, \theta^{\gamma'} \rangle.
\end{equation*}
If $\langle \theta'{}^\gamma, \theta^{\gamma'} \rangle = 1$ for some $1 \neq \gamma \in \Gal(L/K)$, then we are done. Otherwise, we must have $c' \theta'(x) = c \theta^{\gamma'}(x)$ and $\theta' = \theta^{\gamma'}$ on $U_L^1$ since $\theta, \theta'$ agree on $K^\times$. 

We have now shown that there exists a $\gamma \in \Gal(L/K)$ such that $\theta'(x) = \theta^\gamma(x)$ for any very regular element $x \in L^\times$. But now it follows that $\theta' = \theta^\gamma$ on $\cO_L^\times$ since any very regular element together with $U_L^1$ generate $\cO_L^\times$. The desired conclusion now follows by the assumption $\theta(\pi) = \theta'(\pi)$ since $\langle \varpi \rangle \cdot \cO_L^\times = L^\times$.
\end{proof}

From Lemma \ref{l:theta theta' gen}, we obtain the following result:

\begin{proposition}[Henniart, Boyarchenko--Weinstein]\label{p:very regular determines}
Let $\theta \in \sX$ and let $G$ be any inner form of $\GL_n(K)$. Assume that $\pi$ is an irreducible supercuspidal representation of $G$ with central character $\theta|_{K^\times}$ satisfying:
\begin{enumerate}[label=(\roman*)]
\item
$\pi \cong \pi \otimes (\epsilon \circ \det)$, 
\item
there exists a constant $c \neq 0$ satisfying $\Tr \pi(x) = c \cdot \sum_{\gamma \in \Gal(L/K)} \theta^\gamma(x)$ for each very regular element $x \in L^\times$.
\end{enumerate}
If $n = 2$, $q = 3$, and $\theta|_{U_L^1}$ factors through the norm $U_L^1 \to U_K^1$ (i.e. $\theta \in \sX^0$ with notation from Section \ref{s:minimal}), assume in addition that $c = \begin{cases} -1 & \text{if $G \cong \GL_2(K)$} \\ +1 & \text{if $G \cong D_{1/2}^\times$} \end{cases}$. Then $\pi$ corresponds to $\theta$ under automorphic induction and the Jacquet--Langlands correspondence:
\begin{equation*}
\pi \cong \JL(\pi_\theta).
\end{equation*}
\end{proposition}

\begin{proof}
This is \cite[Proposition 1.5]{BoyarchenkoW_13} (combined with the remarks of Section 1.4 of \textit{op.\ cit.}) when $G \cong \GL_n(K)$ or $G \cong D_{1/n}^\times$. The proof extends to the general situation with no complications. 
\end{proof}

\section{Homology of affine Deligne--Lusztig varieties at infinite level}\label{s:inf ADLV}

We explain how the results of Part \ref{part:cohomology} on the cohomology of the finite-type of $\FF_{q^n}$-schemes $X_h$ for $h \geq 1$ (Proposition-Definition \ref{prop:lambda}, Proposition \ref{p:Xh Ah-}) allows one to define and determine homology groups of the schemes $\dot X_w^\infty(b)$. 

\subsection{Definition of the homology groups}\label{s:def H_i}

Following \cite{Lusztig_79}, for any smooth $\overline \FF_q$-scheme $S$ of pure dimension $d$, we set
\begin{equation*}
H_i(S, \overline \QQ_\ell)\colonequals H_c^{2d-i}(S, \overline \QQ_\ell)(d),
\end{equation*}
where $(d)$ denotes the $d$th Tate twist. Recall from Proposition \ref{p:dim Xh} that for any $h \geq 1$, the $\FF_{q^n}$-scheme $X_h$ is smooth of pure dimension $(n-1)(h-1) + (n'-1)$.

By Proposition \ref{p:conn comps}, Corollary \ref{cor:scheme_structure_on the_inverse_limit} and \eqref{eq:two_prolimit_presentations}, we have
\begin{equation*}
\dot X_w^\infty(b) = \bigsqcup_{g \in G/G_\cO} \varprojlim_{r > m \geq 0} g \cdot \dot X_{\dot w_r}^m(b)_{\sL_0} = \bigsqcup_{g \in G/G_\cO} \varprojlim_h X_h.
\end{equation*}
By Proposition \ref{p:Wh fixed}, we have the natural inclusion 
\begin{equation*}
H_i(X_{h-1}, \overline \QQ_\ell) = H_i(X_h, \overline \QQ_\ell)^{\bW_h^{h-1}(\FF_{q^n})} \subseteq H_i(X_h, \overline \QQ_\ell).
\end{equation*}
We may therefore define
\begin{align*}
H_i\Big(\dot X_w^\infty(b)_{\sL_0}, \overline \QQ_\ell\Big) &= H_i\Big(\varprojlim_h X_h, \overline \QQ_\ell\Big)  \colonequals \varinjlim_{h} H_i\left(X_h, \overline \QQ_\ell\right), \\
H_i\left(\dot X_w^\infty(b), \overline \QQ_\ell\right) &= \bigoplus_{G/G_{\cO}} H_i\left(g \cdot \dot X_w^\infty(b)_{\sL_0}, \overline \QQ_\ell\right).
\end{align*}

Recall that in Theorem \ref{cor:infty_level_adlv} we extended the action of $\cO_L^{\times}$ on $\dot X_w^\infty(b)$ to an action of $T=L^{\times}$.
\begin{definition}\label{d:homology ADLV}
For any (smooth) character $\theta \from T \to \overline \QQ_\ell^\times$, define the virtual $G$-representation
\begin{equation*}
R_T^G(\theta) \colonequals \sum_{i \geq 0} (-1)^i H_i(\dot X_w^\infty(b), \overline \QQ_\ell)[\theta],
\end{equation*}
where $[\theta]$ denotes the subspace where $T$ acts by $\theta$. 
\end{definition}

Let $Z$ denote the center of $G$.  
\begin{theorem}\label{t:RTG desc}
Let $\theta \from T = L^\times \to \overline \QQ_\ell^\times$ be a character of level $h \geq 1$. Then as $G$-representations,
\begin{equation}\label{e:ind from Gh}
R_T^G(\theta) \cong \cInd_{Z \cdot G_\cO}^G\left(R_{T_h}^{G_h}(\theta)\right),
\end{equation}
where we view the (virtual) $G_h$-representation $R_{T_h}^{G_h}(\theta)$ as a $G_\cO$-representation by pulling back along the natural surjection $G_\cO \to G_h$ and then extend to $Z$ by letting $\varpi$ act by $\theta(\varpi)$. Furthermore, for any very regular element $x \in L^\times$,
\begin{equation*}
\Tr\left(x^* ; R_T^G(\theta)\right) = \sum_{\gamma \in \Gal(L/K)} \theta^\gamma(x).
\end{equation*}
\end{theorem}

\begin{proof}
The stabilizer of $\dot X_w^{\infty}(b)_{\sL_0} \subseteq \dot X_w^{\infty}(b)$ in $G$ is $G_{\cO}$. Let $T_{\cO}$ be the preimage of $\cO_L^{\times}$ under $T \cong L^{\times}$. It is easy to see that the stabilizer of $\dot X_w^{\infty}(b)_{\sL_0} \subseteq \dot X_w^{\infty}(b)$ in $G\times T$ is the subgroup $\Gamma$ generated by $G_\cO \times T_\cO$ and $(\varpi, \varpi^{-1})$.
% Recall that by Proposition \ref{p:conn comps}, we have that $\dot X_w^\infty(b)$ is a union of $G$-translates of $\dot X_w(b)_{\sL_{b,0}^{\rm adm}} = \varprojlim_h X_h$ and the stabilizer of $\dot X_w(b)_{\sL_{b,0}^{\rm adm}}$ is exactly $G_\cO$. But we also have an action by $T$, and it is easy to check that the stabilizer of $\dot X_w(b)_{\sL_{b,0}^{\rm adm}}$ in $G \times T$ is the subgroup $\Gamma$ generated by $G_\cO \times T_\cO$ and $(\varpi, \varpi^{-1})$. 
Hence as representations of $G \times T$, we have
\begin{equation*}
\sum (-1)^i H_i(\dot X_w(b), \overline \QQ_\ell) \cong \cInd_{\Gamma}^{T \times G}\left(\sum (-1)^i H_i(\dot X_w(b)_{\sL_0}, \overline \QQ_\ell)\right).
\end{equation*}
Now let $\widetilde \Gamma$ be the subgroup of $G \times T$ generated by $\Gamma$ and $\{1\} \times T$. Note that $\widetilde \Gamma \cong Z G_\cO \times T$. The isomorphism \eqref{e:ind from Gh} follows from the above together with the definition of the homology groups of $\dot X_w(b)_{\sL_0}$ in terms of the cohomology of $X_h$ (remembering that $\theta$ has level $h$ by assumption).

It remains to determine the character on very regular elements of $L^\times.$ We use \eqref{e:ind from Gh} together with the corresponding character formula result for $X_h$ (Proposition \ref{p:trace alt sum Xh}). By Lemma \ref{lm:descr_normalizer}, we know that for each $\varphi \in \Gal(L/K)$, there exists an element $g_\varphi \in N_G(G_\cO)$ satisfying $g_\varphi x g_\varphi^{-1} = \varphi(x)$ for all $x \in L^\times$ and that if $\varphi \in \Gal(L/K)[n']$, one can choose $g_\varphi \in G_\cO$. By Section \ref{sec:cartan}, we know that $N_G(G_\cO)/G_\cO \cong \bZ/n_0 \bZ$, and therefore using the fact that
\begin{equation*}
\Tr\left(x^* ; R_{T_h}^{G_h}(\theta)\right) = \sum_{\varphi \in \Gal(L/K)[n']} \theta(g_\varphi x g_\varphi^{-1})
\end{equation*}
by Proposition \ref{p:trace alt sum Xh}, we have:
\begin{equation*}
\Tr\left(x^* ; R_T^G(\theta)\right) = \sum_{\substack{g \in G/Z G_\cO \\ g x g^{-1} \in Z G_\cO}} \Tr\left(x^* ; R_{T_h}^{G_h}(\theta)\right) = \sum_{\varphi \in \Gal(L/K)} \theta(g_\varphi x g_\varphi^{-1}). \qedhere
\end{equation*}
\end{proof}

\begin{theorem}\label{t:LLC JLC gen}
Let $\theta \in \sX$. If $|R_T^G(\theta)|$ is irreducible supercuspidal, then the assignment $\theta \mapsto |R_T^G(\theta)|$ is a geometric realization of automorphic induction and the Jacquet--Langlands correspondence. That is,
\begin{equation*}
|R_T^G(\theta)| \cong \JL(\pi_\theta),
\end{equation*}
where $\JL$ denotes the Jacquet--Langlands transfer of the $\GL_n(K)$-representation $\pi_\theta$ to the (possibly split) inner form $G$ of $\GL_n(K)$. Moreover, writing $|R_T^G(\theta)| = c_\theta' R_T^G(\theta)$ for $c_\theta' \in \{\pm 1\}$, we have $c_\theta' = c_\theta$.
\end{theorem}

\begin{remark}\label{rem:geom sign}
If $|R_T^G(\theta)|$ is irreducible supercuspidal, then  $R_T^G(\theta) = (-1)^{r_\theta} \pi$, where $\pi$ is an irreducible supercuspidal representation occurring in $H_{r_\theta}(\dot X_w^\infty(b), \overline \QQ_\ell)[\theta]$ for some $r_\theta \in \bZ$. (There may be other degrees where $\pi$ contributes, but they all cancel out. In particular, there may be more than one choice of $r_\theta$, but the parity of $r_\theta$ is invariant.) Then by Theorem \ref{t:LLC JLC gen} implies that $c_\theta = (-1)^{r_\theta}$, which gives a \textit{geometric} interpretation of Henniart's sign $c_\theta$ in terms of the surviving cohomological degree in the alternating sum $R_T^G(\theta)$. \hfill $\Diamond$
\end{remark}

\begin{proof}
Write $|R_T^G(\theta)| = c_\theta' R_T^G(\theta)$ for some $c_\theta' = \pm 1$. If we can show that $|R_T^G(\theta)|$ satisfies the hypotheses of Proposition \ref{p:very regular determines}, then we are done. By assumption, $|R_T^G(\theta)|$ is an irreducible cuspidal representation and by definition of the $G \times T$ action on $R_T^G(\theta)$, the central character of $|R_T^G(\theta)|$ must be $\theta|_{K^\times}$.

To see that (i) of Proposition \ref{p:very regular determines} holds, note that since $L/K$ is unramified we have $\langle \varpi^n \rangle \cO_K^\times = \Nm_{L/K}(L^\times) = \ker(\epsilon)$. In particular, we see that $\epsilon \circ \det$ is trivial on $Z G_\cO$ and so by Theorem \ref{t:RTG desc}, we have $|R_T^G(\theta)| \cong |R_T^G(\theta)| \otimes (\epsilon \circ \det)$.

We now establish (ii) of Proposition \ref{p:very regular determines} and the additional assumption in the special case $n=2$ and $q=3$. By Theorem \ref{t:RTG desc}, we have that for any very regular element $x \in L^\times$,
\begin{equation*}
\Tr(x^* ; |R_T^G(\theta)|) = c_\theta' \cdot \sum_{\gamma \in \Gal(L/K)} \theta^\gamma(x).
\end{equation*}
If $n=2$, $q = 3$, and $\theta \in \sX^0$, then by Theorem \ref{t:depth zero}, we know in addition that
\begin{equation*}
c_\theta' = (-1)^{n'-1} = \begin{cases} -1 & \text{if $G \cong \GL_2(K)$,} \\ +1 & \text{if $G \cong D_{1/2}^\times$.} \end{cases}
\end{equation*}
We have now established all the conditions required by Proposition \ref{p:very regular determines} to conclude that $c_\theta' = c_\theta$ and $|R_T^G(\theta)| \cong \JL(\pi_\theta).$
\end{proof}

\section{   A geometric realization of automorphic induction and Jacquet--Langlands}\label{s:minimal}

In this section, we write down the cases in which we can prove Theorem \ref{t:LLC JLC gen} unconditionally. To this end, we consider the following two subsets of $\sX$:
\begin{align*}
\sX^0 &\colonequals \{\theta \in \sX : \text{$\theta|_{U_L^1}$ factors through the norm map $U_L^1 \to U_K^1$}\} \\
\sX^{\rm min} &\colonequals \{\theta \in \sX : \text{$\theta$ is minimal admissible}\} \\
&= \{\theta \in \sX : \text{the $\theta/\theta^\gamma$ have the same level for any $1 \neq \gamma \in \Gal(L/K)$}\}
\end{align*}
Note that $\sX^0 \subseteq \sX^{\rm min}$ is the ``depth zero'' part of $\sX^{\rm min}$.

\begin{remark}
Let $\theta \in L^\times \to \overline \QQ_\ell^\times$ be a smooth character with trivial $\Gal(L/K)$-stabilizer. Then its restriction to $\cO_L^\times$ must have trivial $\Gal(L/K)$-stabilizer. For the reader's convenience, we summarize the relation between minimal admissibility and similar notions in the literature:
\begin{enumerate}[label=$\cdot$,leftmargin=*]
\item
$\theta$ is minimal admissible if and only if $(L/K, \theta)$ form a minimal admissible pair, which happens if and only if $\theta$ has only one ``jump'' in the sense of Bushnell--Henniart \cite[Section 1.1]{BushnellH_05}.
\item
$\theta$ is minimal admissible if and only if it can be written in the form $\theta_{\rm prim} \cdot (\chi\circ\Nm_{L/K})$ for some smooth $\chi \from K^\times \to \overline \QQ_\ell^\times$, where $\theta_{\rm prim}$ is \textit{primitive} in the sense of Boyarchenko--Weinstein \cite[Section 7.1]{BoyarchenkoW_16} (see also Section \ref{s:finite rings} of the present paper).
\item
Let $h$ be such that $\theta|_{U_L^h} = 1$ and $\theta|_{U_L^{h-1}} \neq 1$. Then $\theta$ is primitive if and only if $\theta$ is \textit{regular} as a character of $\cO_L^\times/U_L^h$ in the sense of Lusztig \cite[Section 1.5]{Lusztig_04}, when $\cO_L^\times/U_L^h$ is the $F$-fixed points of a maximal torus (see Remark \ref{r:prim reg}). \hfill $\Diamond$
\end{enumerate}
\end{remark}

\subsection{Depth zero representations}\label{s:depth zero}

In this section we only consider characters $\theta \in \sX^0$ and give a nonvanishing result for the individual cohomology groups $H_i(\dot X_w^\infty(b), \overline \QQ_\ell)[\theta]$. Since each $\theta \in \sX^0$ is of the form $\theta_0 \cdot (\chi \circ \Nm_{L/K})$, where $\theta_0 \in \sX^0$ and $\theta_0|_{U_L^1} = 1$, determining when $H_i(\dot X_w^\infty(b), \overline \QQ_\ell)[\theta] \neq 0$ can be reduced to the corresponding question for the cohomology of classical Deligne--Lusztig varieties. Recall from Proposition \ref{p:dim Xh} that $\dim X_h = (n-1)(h-1)+(n'-1)$.

\begin{theorem}\label{t:depth zero}
Fix $\theta \in \sX^0$ of level $h$ and write $\theta = \theta_0 \cdot (\chi \circ \Nm_{L/K})$ for some $\theta \in \sX^0$ of level $1$ and some character $\chi$ of $K^\times$ of level $h$. Then:
\begin{enumerate}[label=(\roman*)]
\item
the cohomology groups $H_c^i(X_h, \overline \QQ_\ell)[\theta]$ are concentrated in a single degree and
\begin{equation*}
|R_{T_h}^{G_h}(\theta)| = H_c^{2(n-1)(h-1) + n'-1}(X_h, \overline \QQ_\ell)[\theta] \cong H_c^{n'-1}(X_1, \overline \QQ_\ell)[\theta_0] \otimes (\chi \circ \det)
\end{equation*}
\item
the homology groups $H_i(\dot X_w^\infty(b), \overline \QQ_\ell)[\theta]$ are concentrated in a single degree and
\begin{equation}\label{e:depth zero ind}
|R_T^G(\theta)| = H_{n'-1}(\dot X_w^\infty(b), \overline \QQ_\ell)[\theta] \cong \cInd_{Z \cdot \GL_{n'}(\cO_{D_{k_0/n_0}})}^{\GL_{n'}(D_{k_0/n_0})}\left(\rho_\theta\right)
\end{equation}
is an irreducible supercuspidal representation of $G$. Here $\rho_\theta$ is the extension of the $\GL_{n'}(\cO_{D_{k_0/n_0}})$-representation $H_c^{n'-1}(X_1, \overline \QQ_\ell)[\theta_0] \otimes (\chi \circ \det)$ obtained by letting $\varpi \in Z = K^\times$ act by $\theta(\varpi)$.
\end{enumerate}
Moreover, $|R_{T_h}^{G_h}(\theta)| = (-1)^{n'-1} R_{T_h}^{G_h}(\theta)$ and $|R_T^G(\theta)| = (-1)^{n'-1} R_T^G(\theta).$
\end{theorem}

\begin{proof}
By Lemma \ref{l:twist compat} and Proposition \ref{p:Wh fixed}, we have, as $G_h$-representations
\begin{equation*}
H_c^i(X_h, \overline \QQ_\ell)[\theta_0 \circ (\chi \circ \Nm_{L/K})] \cong H_c^i(X_h, \overline \QQ_\ell)[\theta_0] \otimes (\chi \circ \det) \cong H_c^{i-2(n-1)(h-1)}(X_1, \overline \QQ_\ell)[\theta_0]
\end{equation*}
for all $i \geq 2(n-1)(h-1)$. This reduces the cohomology calculation to a statement about $X_1$, which is a classical Deligne--Lusztig variety attached to the maximal torus $\FF_{q^n}^\times$ in $\GL_{n'}(\FF_{q^{n_0}})$. By \cite[Corollary 9.9]{DeligneL_76}, 
\begin{equation*}
H_c^i(X_1, \overline \QQ_\ell)[\theta] \neq 0 \qquad \Longleftrightarrow \qquad i = n' - 1.
\end{equation*}
This proves (i). Since $\dim X_h = (n-1)(h-1) + n'-1$ by Proposition \ref{p:dim Xh}, we now also see the nonvanishing assertion of (ii) and $H_{n'-1}(\dot X_w^\infty(b), \overline \QQ_\ell)[\theta]$ has the form \ref{e:depth zero ind} by Theorem \ref{t:RTG desc}. It is well known that this representation is irreducible and supercuspidal (see also Theorem \ref{t:RTG irred}). For example, one can show by hand (by the first part of the proof of Theorem \ref{t:RTG irred}) that the induction to the normalizer of $Z G_\cO$ is irreducible, and then the conclusion follows from \cite[Proposition 6.6]{MoyP_96}.
\end{proof}

\begin{theorem}\label{t:depth zero LLC}
For $\theta \in \sX^0$, the assignment $\theta \mapsto H_{n'-1}(\dot X_w^\infty(b), \overline \QQ_\ell)[\theta]$ is a geometric realization of automorphic induction and the Jacquet--Langlands correspondence. That is,
\begin{equation*}
H_{n'-1}(\dot X_w^\infty(b), \overline \QQ_\ell)[\theta] = (-1)^{n'-1} R_T^G(\theta) = |R_T^G(\theta)| \cong \JL(\pi_\theta).
\end{equation*}
\end{theorem}

\begin{proof}
By Theorem \ref{t:depth zero}, we know that $|R_T^G(\theta)| = (-1)^{n'-1} R_T^G(\theta)$ is an irreducible supercuspidal representation, and by Theorem \ref{t:RTG desc}, we know that for any very regular element $x \in L^\times$,
\begin{equation*}
\Tr(x^* ; |R_T^G(\theta)|) = (-1)^{n'-1} \cdot \sum_{\gamma \in \Gal(L/K)} \theta^\gamma(x).
\end{equation*}
By definition $\epsilon$ is a finite-order character of $K^\times$ with $\ker(\epsilon) = \Nm_{L/K}(K^\times)$. Since $L/K$ is unramified, $\ker(\epsilon)$ contains $\cO_K^\times$, and therefore $\epsilon \circ \det$ is trivial on $Z \cdot \GL_{n'}(\cO_{D_{k_0/n_0}})$. Hence $|R_T^G(\theta)| \otimes (\epsilon \circ \det) \cong |R_T^G(\theta)|$. We can now apply Proposition \ref{p:very regular determines}, noting that in the case $n=2,$ $q=3$, we have the correct sign $c_\theta$ (compare with Theorem \ref{t:henniart c}) as required by the proposition.
\end{proof}

\begin{remark}
Observe that as in Remark \ref{rem:geom sign}, the nonvanishing degree $n'-1$ of the homology of $\dot X_w^\infty(b)$ gives a geometric interpretation of Henniart's sign $c_\theta$ from Theorem \ref{t:henniart c}. \hfill $\Diamond$
\end{remark}

\subsection{Representations corresponding to minimal admissible characters}

We now prove the supercuspidality of $|R_T^G(\theta)|$ for $\theta \in \sX^{\rm min}$. The main technical inputs are the irreducibility of $|R_{T_h}^{G_h}(\theta)|$ (Section \ref{s:finite rings}) and a ``cuspidality'' result for $|R_{T_h}^{G_h}(\theta)|$ (Theorem \ref{t:Gh cuspidal}).

\begin{theorem}\label{t:RTG irred}
If $\theta \in \sX^{\rm min}$, then $|R_T^G(\theta)|$ is irreducible supercuspidal.
\end{theorem}

\begin{proof}
We first establish some notation. If $\pi \from H \to \GL(V)$ is a representation of a subgroup $H \subset G$, then for any $\gamma \in G$, we define ${}^\gamma \pi \from \gamma H\gamma^{-1} \to \GL(V)$ by ${}^\gamma \pi(g) \colonequals \pi(\gamma^{-1} g \gamma)$. Assume that $\theta$ is minimal admissible of level $h$. By definition, we can write $\theta = \theta' \otimes (\chi \circ \Nm)$, where $\theta'$ is a primitive character of $L^\times$ of level $h' \leq h$, $\chi$ is any character of $K^\times$ of level $h$, and $\Nm \from L^\times \to K^\times$ is the usual norm. Denoting by $\theta, \theta', \chi$ the corresponding restrictions to the unit groups, by Proposition \ref{p:Wh fixed} and Lemma \ref{l:twist compat}, we have
\begin{equation*}
R_{T_h}^{G_h}(\theta) = R_{T_h}^{G_h}(\theta' \otimes (\chi \circ \Nm)) \cong R_{T_h}^{G_h}(\theta') \otimes (\chi \circ \det) \cong R_{T_{h'}}^{G_{h'}}(\theta') \otimes (\chi \circ \det).
\end{equation*}
In particular, by Theorem \ref{t:RTG desc}, we see that
\begin{equation*}
R_T^G(\theta) \cong R_T^G(\theta') \otimes (\chi \circ \det).
\end{equation*}
Since twists of irreducible supercuspidal representations are again irreducible supercuspidal, it suffices to prove the theorem for primitive characters $\theta$.

Assume now that $\theta$ is a primitive character of level $h$. By Theorem \ref{t:alt sum Xh}, $|R_{T_h}^{G_h}(\theta)|$ is irreducible. Recall that there is a natural surjection $G_\cO \to G_h$ so that we may view $|R_{T_h}^{G_h}(\theta)|$ as a representation of $G_\cO$. We extend this to a representation of $Z \cdot G_\cO = \langle \varpi \rangle \cdot G_\cO$ by letting $\varpi$ act on $|R_{T_h}^{G_h}(\theta)|$ by $\theta(\varpi)$. We first claim that
\begin{equation*}
\rho_\theta \colonequals \cInd_{Z \cdot G_{\cO}}^{Z \cdot N_{G}(G_{\cO})}\left(|R_{T_h}^{G_h}(\theta)|\right)
\end{equation*}
is irreducible. Recall from Section \ref{sec:cartan} that $\# N_G(G_{\cO})/G_{\cO} = n_0$ and let $\{1, \varphi_2, \ldots, \varphi_{n_0}\}$ denote a complete set of coset representatives of $\Gal(L/K)/\Gal(L/K)[n']$.  By Lemma \ref{lm:descr_normalizer}, there exists $g_{\varphi_i} \in N_{G}(G_{\cO})$ such that $g_{\varphi_i}^{-1} x g_{\varphi_i} = \varphi_i(x)$ for all $x \in \cO_L^\times$. By Mackey's irreducibility criterion, it suffices to show that
\begin{equation}\label{e:normal hom 0}
\Hom_{G_{\cO}}\left(|R_{T_h}^{G_h}(\theta)|, {}^{g_{\varphi_i}}|R_{T_h}^{G_h}(\theta)|\right) = 0, \qquad \text{for $i = 2, \ldots, n_0$.}
\end{equation}
Fix some $i$ with $2 \leq i \leq n_0$. By Proposition \ref{p:trace alt sum Xh}, for any very regular element $x \in \cO_L^\times$, 
\begin{equation*}
\Tr\left(x^* ; {}^{g_{\varphi_i}} R_{T_h}^{G_h}(\theta)\right) = \Tr\left((g_{\varphi_i}^{-1} x g_{\varphi_i})^* ; R_{T_h}^{G_h}(\theta)\right) = \sum_{\gamma \in \Gal(L/K)[n']} \theta^\gamma(\varphi_i(x)).
\end{equation*}
Applying Lemma \ref{l:theta theta' gen} to the case when $\theta' = \theta^{\varphi_i}$ and the base field $K$ is replaced by the unique subfield of index $n'$ in $L$ containing $K$, we see that
\begin{equation*}
\Tr\left(x^* ; {}^{g_{\varphi_i}}R_{T_h}^{G_h}(\theta)\right) \neq \Tr\left(x^* ; R_{T_h}^{G_h}(\theta)\right).
\end{equation*}
But now ${}^{g_{\varphi_i}}R_{T_h}^{G_h}(\theta)$ and $R_{T_h}^{G_h}(\theta)$ are irreducible representations of $G_{\cO}$ whose characters differ from each other, and so necessarily \eqref{e:normal hom 0} holds and $\rho_\theta$ is irreducible.

We now fix $\gamma \in G \smallsetminus N_G(Z \cdot G_{\cO})$. Once again by Mackey's criterion, to complete the proof we must show that
\begin{equation}\label{e:intertwiner}
\Hom_{\gamma Z G_{\cO} \gamma^{-1} \cap Z G_{\cO}} \left(|R_{T_h}^{G_h}(\theta)|, {}^\gamma|R_{T_h}^{G_h}(\theta)|\right) = 0.
\end{equation}
At this point, let $b$ be a special representative. By Section \ref{sec:cartan}, we may assume that $\gamma = \Pi_0^\nu$, where $\nu = (\nu_1, \ldots, \nu_1, \nu_2, \ldots, \nu_2, \ldots, \nu_{n'}, \ldots, \nu_{n'})$ (each $\nu_i$ repeated $n_0$ times) for $0 = \nu_1 \leq \nu_2 \leq \cdots \leq \nu_{n'}$, and $\Pi_0^\nu$ is the block-diagonal matrix whose $i$th $n_0 \times n_0$ block is given by $\left(\begin{smallmatrix} 0 & \varpi \\ 1_{n_0 - 1} & 0 \end{smallmatrix}\right)^{\nu_i}$. Observe that if $(A_{i,j})_{1 \leq i, j \leq n'} \in \GL_n(\breve K)$, where each $A_{i,j}$ is a $(n_0 \times n_0)$-matrix, then
\begin{equation}\label{e:Pi nu conj}
\Pi_0^{-\nu} \cdot (A_{i,j})_{1 \leq i, j \leq n'} \cdot \Pi_0^\nu = (\Pi_0^{-\nu_i} A_{i,j} \Pi_0^{\nu_j})_{1 \leq i, j \leq n'}.
\end{equation}

For a parabolic subgroup $P'$ of $\GL_{n'}$ containing the upper triangular matrices, let $\breve N_{P'}$ be its unipotent radical. Let $\breve N_P$ denote the subgroup of $\GL_n(\breve K)$ such that each $(n_0 \times n_0)$-block consists of a diagonal matrix and the $(i, j)$th block is nonzero if and only if the $(i, j)$th entry of an element of $\breve N_{P'}$ is nonzero. Write $N_P = \breve N_P^F \cap G_\cO$. For $h\geq 1$ let $N_P^h = N_P \cap \ker(G_\cO \to G_h)$. 

We claim that there exists a parabolic $P' \subseteq \GL_{n'}$ as above, such that $\Pi_0^{-\nu} N_P^{h-1} \Pi_0^\nu \subset \ker(G_\cO \to G_h)$. Let $1 \leq i_0 \leq n'$ be the last $\nu_{i_0} = 0$ so that $\nu_{i_0} < \nu_{i_0+1}$, and let $P'$ be the minimal parabolic corresponding to the partition $i_0 + (n'-i_0)$. Let $(A_{i,j})_{1 \leq i,j \leq n'} \in N_P^{h-1}$ so that each $A_{i,j}$ is a diagonal $n_0 \times n_0$ matrix whose entries all lie in $\bW_h^{h-1}(\overline \FF_q)$. By \eqref{e:Pi nu conj}, we see that the $(i,j)$th block of $\Pi_0^{-\nu} \cdot (A_{i,j})_{1 \leq i,j \leq n'} \Pi_0^{\nu}$ is $\Pi_0^{-\nu_i} A_{i,j} \Pi_0^{\nu_j}$, so that in particular, if $1 \leq i \leq i_0$ and $i_0 + 1 \leq j \leq n'$, then $\nu_j - \nu_i > 0$. By definition of $G_h$ (Section \ref{sec:definition of GGh}), we now have that $\Pi_0^{-\nu} \cdot (A_{i,j})_{1 \leq i,j \leq n'} \cdot \Pi_0^{\nu} \in \ker(G_\cO \to G_h)$.

The above implies that the restriction of ${}^{\Pi_0^{\nu}} |R_{T_h}^{G_h}(\theta)|$ to $N_P^{h-1}$ is trivial. On the other hand, by Theorem \ref{t:Gh cuspidal}, the restriction of $|R_{T_h}^{G_h}(\theta)|$ to $N_P^{h-1}$ does not contain the trivial representation. Therefore:
\begin{align*}
\dim \Hom_{\gamma Z G_{\cO} \gamma^{-1} \cap Z G_{\cO}} &(|R_{T_h}^{G_h}(\theta)|, {}^\gamma|R_{T_h}^{G_h}(\theta)|) \\
&\leq \dim \Hom_{N_P^{h-1}}(|R_{T_h}^{G_h}(\theta)|, {}^\gamma|R_{T_h}^{G_h}(\theta)|) \\
&\leq \dim \Hom_{N_P^{h-1}}(|R_{T_h}^{G_h}(\theta)|, \text{triv}) = 0. \qedhere
\end{align*}
\end{proof}

Combining Theorems \ref{t:LLC JLC gen} and \ref{t:RTG irred} proves:

\begin{theorem}\label{t:LLC JLC min}
If $\theta \in \sX^{\rm min}$, then the assignment $\theta \mapsto |R_T^G(\theta)|$ is a geometric realization of automorphic induction and the Jacquet--Langlands correspondence.
\end{theorem}

\bibliography{bib_ADLV_CC}{}
\bibliographystyle{alpha}

\end{document}